\documentclass[12pt,a4paper,bibliography=totoc,abstract=on,twoside]{scrreprt}
\usepackage[english]{babel}
\usepackage{amsmath}
\usepackage{amsfonts}
\usepackage{amssymb}
\usepackage{amsthm}
\usepackage{graphicx}
\usepackage{wasysym}
\usepackage{listings}
\usepackage{setspace}
\usepackage{mathrsfs}
\usepackage[titles]{tocloft}
\usepackage{tikz}
\usepackage[a4paper,inner=3.5cm,outer=3cm,top=3cm,bottom=4cm,pdftex]{geometry}

\setkomafont{sectioning}{\rmfamily\bfseries} 

\theoremstyle{plain}
\newtheorem{thm}{Theorem}[subsection]
\newtheorem{lem}[thm]{Lemma}
\newtheorem{cor}[thm]{Corollary}
\newtheorem{prop}[thm]{Proposition}
\newtheorem{constr}[thm]{Construction}
\theoremstyle{remark}
\newtheorem{rem}[thm]{Remark}
\newtheorem{note}[thm]{Note}
\newtheorem{rmind}[thm]{Reminder}
\theoremstyle{definition}
\newtheorem{defn}[thm]{Definition}
\newtheorem{nota}[thm]{Notation}
\newtheorem{exmp}[thm]{Example}
\numberwithin{equation}{section}
\numberwithin{figure}{section}

\DeclareMathOperator{\PCon}{PCon}
\DeclareMathOperator{\Con}{Con}
\renewcommand{\Im}{\operatorname{Im\,}}

\allowdisplaybreaks
\begin{document}
\pagenumbering{arabic}

\title{Algebraic~Aspects~in Tropical~Mathematics}
\author{Tal Perri \\ Mathematics Department, Bar-Ilan University \\ Under the supervision of Professor Louis Rowen}%
\maketitle

\makeatletter
\let\thetitle\@title
\let\theauthor\@author
\makeatother

\begin{abstract}
 Much like in the theory of algebraic geometry, we develop a correspondence between certain types of  algebraic and geometric objects. The basic algebraic environment we work in is the a semifield of fractions $\mathbb{H}(x_1,...,x_n)$ of the polynomial semidomain $\mathbb{H}[x_1,...,x_n]$, where $\mathbb{H}$ is taken to be a bipotent semifield, while for the geometric environment we have the space $\mathbb{H}^n$ (where addition and scalar multiplication are defined coordinate-wise). We show that taking $\mathbb{H}$ to be bipotent makes both $\mathbb{H}(x_1,...,x_n)$ and $\mathbb{H}^n$ idempotent which turn out to satisfy many desired properties that we utilize for our construction.\\
The fundamental algebraic and geometric objects having interrelations are called  \emph{kernels}, encapsulating congruences over semifields (analogous to ideals in algebraic geometry) and \emph{skeletons} which serve as the analogue for zero-sets of algebraic geometry. As an analogue for the celebrated Nullstellenzats theorem which provides a correspondence between radical ideals and zero sets, we develop a correspondence between skeletons and a family of kernels called \emph{polars} originally developed in the theory of lattice-ordered groups.
For a special kind of skeletons, called \emph{principal skeletons}, we have simplified the correspondence by restricting our algebraic environment to a very special semifield which is also a kernel of $\mathbb{H}(x_1,...,x_n)$.\\
After establishing the linkage between kernels and skeletons we proceed to construct a second linkage, this time between a family of skeletons and what we call  `corner-loci'. Essentially a corner locus is what is called a tropical variety in the theory of tropical geometry, which is a set of corner roots of some set of tropical polynomials. The relation between a skeleton and a corner-locus is that they define the exact same subset of $\mathbb{H}^n$, though in different ways: while a corner locus is defined by corner roots of tropical polynomials, the skeleton is defined by an equality, namely, equating fractions from $\mathbb{H}(x_1,...,x_n)$ to $1$. All the connections presented above form a path connecting a tropical variety to a kernel. In this paper we also develop a correspondence between supertropical varieties (generalizing tropical varieties) introduced by Izhakian, Knebusch and Rowen to the lattice of principal kernels. Restricting this correspondence to a sublattice of kernels, whom we call regular kernels, yields the correspondence described above. The research we conduct shows that the theory of supertropical geometry is in fact a natural generalization of tropical geometry.\\
We conclude by developing some algebraic structure notions such as composition series and convexity degree, along with some notions holding a geometric interpretation, like reducibility and hyperdimension.
\end{abstract}

\tableofcontents
\clearpage

\section{Overview}

In the following overview, we will establish a linkage between our construction and the widely known theory of algebraic geometry, in order to give the reader some additional insights.\\

As noted in the abstract,  we develop a geometric structure corresponding to a semifield of fractions $\mathbb{H}(x_1,...,x_n)$ over a bipotent semifield $\mathbb{H}$. We concentrate on the case where $\mathbb{H}$ is a divisible archimedean semifield, motivated by our interest to link the theory to tropical geometry. We also consider a divisible bipotent archimedean semifield $\mathscr{R}$, which is also complete (and thus isomorphic to the positive reals with max-plus operations). \\

The role of ideals is played by an algebraic structures called kernels. While ideals encapsulate the structure of the preimages of zero of rings homomorphisms, kernels encapsulate the preimages of $\{1\}$ with respect to semifield homomorphisms with the substantial difference that considering homomorphisms one gets a sublattice of kernels with respect to intersection and multiplication, $(\Con(\mathscr{R}(x_1,...,x_n)), \cdot, \cap)$.\\

For each kernel $K$ of $\Con(\mathscr{R}(x_1,...,x_n))$, we define an analogue for zero sets corresponding to ideals, namely, a skeleton $Skel(K) \subseteq \mathscr{R}^n$, which is defined to be the set of all points in $\mathscr{R}^n$ over which all the elements of the kernel $K$ are evaluated to be $1$. \\

While in the classical theory all ideals are finitely generated, in our case we explicitly consider the finitely generated kernels in $\Con(\mathscr{R}(x_1,...,x_n))$. This family of finitely generated kernels in $\Con(\mathscr{R}(x_1,...,x_n))$ forms a sublattice of $(\Con(\mathscr{R}(x_1,...,x_n)), \cdot, \cap)$, which we denote as $\PCon(\mathscr{R}(x_1,...,x_n))$. At this point, the theory is considerably simplified, since a  well-known result in the theory of semifields states that every finitely generated kernel is principal, i.e., generated as a kernel by a single element.

The generator of a principal kernel is not unique, though there is a designated set of generators which provides us with many tools for implementing the theory. \\

First we establish a Zariski-like correspondence between a special kind of kernels called polars and general skeletons. Then we find some special semifield of $\mathscr{R}(x_1,...,x_n)$ over which the above correspondence gives rise to a simplified correspondence between principal (finitely generated) kernels of and principal (finitely generated) skeletons which hold most of our interest. We show that a skeleton is uniquely defined by any of its corresponding kernel generators, define a maximal kernel, and show that the maximal principal kernels correspond to points in $\mathscr{R}^n$.  Along the way, we provide some theory concerning general (not necessarily principal) kernels and skeleton.\\

The geometric interpretation established for the theory of semifields provides us with a topology, called the Stone Topology, defined on the family of the so-called  irreducible kernels (analogous to prime ideals) and the family of maximal kernels. This topology is in essence the semifields version of the famous Zariski topology.\\

After establishing this geometric framework, we introduce a map linking  supertropical varieties (cf. \cite{SuperTropicalAlg}), which we call  `Corner loci', to a subfamily of principal skeletons. We use the latter to construct a correspondence between principal supertropical varieties (the analogue for hypersurfaces of algebraic geometry) and a lattice of  kernels generated by special principal kernels which we call corner-integral. In addition,  we characterize a family of skeletons that coincide with the family of `regular' tropical varieties via which we obtain a correspondence with a sublattice of principal kernels called regular kernels.\\

Finally, we establish a notion of reducibility and some other notions and properties of kernels and their corresponding skeletons, such as convex-dependence, dimensionality etc.\\
We begin our thesis by introducing the relevant results in the theory of semifields.

\newpage
\section{Background}
\ \\

\subsection{Basic notions in lattice theory}

\ \\

\begin{defn}
A poset $(X, \leq)$ is \emph{directed} if
for any pair of elements $a, b \in X$ there exists  $c \in X$ such that $a \leq c$ and $b \leq c$, i.e., $c$ is an upper bound for $a$ and $b$.\\
A poset $(X, \leq)$ is a \emph{lattice} or \emph{lattice-ordered set} if
for $a, b \in X$, the set $\{a, b\}$ has a join $a \vee b$ (also known as the least upper bound, or the supremum)
and a meet $a \wedge b$ (also known as the greatest lower bound, or the infimum).\\
Equivalently, a lattice can be defined as a directed poset $(X, \vee, \wedge )$, consisting of a set $X$ and two associative and commutative binary operations $\vee$ and $\wedge$ defined on $X$, such that for all elements $a, b \in X$.
$$a \vee ( a \wedge b) = a \wedge (a \vee b) = a.$$
The first definition is derived from the second by defining a partial order on $X$ by
$$a \leq  b \Leftrightarrow a = a \wedge b \ \ \text{or equivalently,} \ a \leq  b \Leftrightarrow b = a \vee b.$$
\end{defn}

\begin{defn}
A lattice $(X, \vee, \wedge )$ is said to be \emph{distributive} if the following condition holds for any $a,b,c \in X$:
$$a \wedge (b \vee c) = (a \wedge b) \vee (a \wedge c).$$
A fundamental result in lattice theory states that this condition is equivalent to the condition
$$a \vee (b \wedge c) = (a \vee b) \wedge (a \vee c).$$
\end{defn}

\begin{defn}\label{defn_cond_complete}
A \emph{conditionally complete} lattice is a poset $P$ in which every nonempty subset that has an upper bound in $P$ has a least upper bound (i.e., a supremum) in $P$ and every nonempty subset that has an lower bound in $P$ has an infimum in $P$. A lattice $P$ is said to be \emph{complete} if all its subsets have both a join and a meet.
\end{defn}

\begin{exmp}
The poset of the real numbers $\mathbb{R}$ is conditionally complete, and become complete when adjoining ${-\infty, \infty}$. Likewise the positive reals $\mathbb{R}^{+}$ is conditionally complete, and become complete when adjoining ${0, \infty}$. Note that for example the set $\{1/n : n \in \mathbb{N} \} \subset \mathbb{R}^{+}$ is not bounded below in $\mathbb{R}^{+}$ since $0 \not \in \mathbb{R}^{+}$.
\end{exmp}

\begin{defn}\label{defn_completely_closed}
The subset $A$ of the poset $P$ is said to be \emph{completely closed} in $P$ if
$A$ contains the least upper bound or greatest lower bound of any of its nonempty
subsets, if either exists in $P$.
\end{defn}

\begin{defn}
Given two lattices $(X, \vee_X, \wedge_X)$ and $(Y, \vee_Y, \wedge_Y)$, a \emph{homomorphism of lattices} or \emph{lattice homomorphism} is a function $f : X \rightarrow Y$ such that
$$f(a \vee_X b) = f(a) \vee_Y f(b), \ \text{and} \ f(a \wedge_X b) = f(a) \wedge_Y f(b).$$
\end{defn}

\ \\
\subsection{Semifields}

\ \\

\subsubsection*{Basic setting and assumptions}
\ \\

\begin{defn}
A \emph{semiring}  $\mathbb{S}$ is a set $S$ equipped with two binary operations $+$ and $\cdot$, such that $(S, +)$ is a commutative monoid with identity element $0$, $(S, \cdot)$ is a monoid with identity element $1$, multiplication left and right distributes over addition and multiplication by $0$ annihilates $S$. $\mathbb{S}$ is called commutative when $(S, \cdot)$ is commutative.\\
$\mathbb{S}$ is called a \emph{domain} when $\mathbb{S}$ is multiplicatively cancellative.
\end{defn}

\begin{defn}\label{defn_semifield}
A \emph{semifield} is a semiring  $(\mathbb{H},+,\cdot)$ in which all nonzero elements have a multiplicative inverse. A semifield is said to be \emph{proper} if it is not a field.
\end{defn}

\begin{note}
Though in general, a semifield is not assumed to be commutative, in the scope of our study we consider commutative semifields. Thus, we always assume a semifield to be commutative. Nevertheless, we introduce some definitions and results in the wider context in which a semifield is not necessarily commutative. For example, the notion of a semifield-kernel is defined as a normal subgroup implying the definition refers to the wider context.
\end{note}

From now on, unless stated otherwise we assume a semifield to be commutative and proper.

\begin{lem}\label{lem_proper_semifield}
If \ $\mathbb{H}$ is a proper semifield then $a + b \neq 0$ for all $a,b \in \mathbb{H} \setminus \{0\}$.
\end{lem}
\begin{proof}
Let $\mathbb{H}$ be a semifield and let $a \in \mathbb{H}$ be any element of $\mathbb{H}$. If there exists $b \in \mathbb{H}$ such that $a+b = 0$, then $1+a^{-1}b = 0$. Thus $-1 = a^{-1}b \in \mathbb{H}$ and thus for any $c \in \mathbb{H}$, we have $c + (-1)c = 0$ and $-c \in \mathbb{H}$, yielding that $\mathbb{H}$ is a field.
\end{proof}

\begin{note}
Throughout this dissertation, we assume a semifield to be a proper semifield.
In view of Lemma \ref{lem_proper_semifield}, this implies that every element of a semifield is not invertible with respect to addition. This makes the zero element somewhat redundant. Thus we generally assume a proper semifield to have no zero element. Whenever we choose to adjoin such an element we will indicate it.
\end{note}
\ \\

\begin{exmp}\label{exmp_semifield_examples}
The following are some well-known examples for semifields:
\begin{enumerate}
  \item The positive real numbers with the usual addition and multiplication form a (commutative) semifield.
  \item The rational functions of the form $f/g$, where $f$ and $g$ are polynomials in one variable with positive coefficients, comprise a (commutative) semifield.
  \item The max-plus algebra, or the tropical semiring, $(\mathbb{R}, max, +)$ is a semifield. Here the sum of two elements is defined to be their maximum, and the product to be their ordinary sum.
  \item If $(A,\leq)$ is a lattice ordered group then $(A,+,·)$ is an additively idempotent semifield. The semifield sum is defined to be the sup of two elements. Conversely, any additively idempotent semifield $(A,+,·)$ defines a lattice-ordered group $(A,\leq)$, where $a \leq b$ if and only if $a + b = b$.
\end{enumerate}
\end{exmp}

\begin{rem}
As stated in Example \ref{exmp_semifield_examples}(4) (additively) idempotent semifields correspond to lattice ordered groups. In the stated correspondence, the addition operation $+$ of the semifield coincides with the so-called 'join' operation $\vee$ defined on the underlying lattice structure of the lattice ordered group. Due to this and in order to emphasize the underlying lattice structure of the semifield,  when considering idempotent semifields we denote addition by $\dotplus$.
\end{rem}

\begin{defn}
An \emph{idempotent semiring} is a semiring $\mathbb{S}$ such that $$\forall a \in \mathbb{S} : a + a = a.$$
An \emph{idempotent semifield} is  an idempotent semiring which is a semifield. Note that considering a semifield
the condition of idempotency is equivalent to demanding that $1 + 1 = 1$, as $a = 1a = (1+1)a = a+a$.
\end{defn}

\begin{defn}
 We define a \emph{bipotent semiring} $\mathbb{S}$ to be a semiring with 1 (multiplicative identity element) admitting bipotent addition, i.e.,
 $ \alpha + \beta  \in  \{\alpha, \beta \}  $ for any $\alpha, \beta \in \mathbb{S}$.
 When $\mathbb{S}$ is a semifield, i.e., every nonzero element of $\mathbb{S}$ is invertible with respect to multiplication, we say that $\mathbb{S}$ is a \emph{bipotent semifield}.
\end{defn}

\begin{rem}
A bipotent semifield (semiring) is a special case of an idempotent semifield (semiring).
\end{rem}

Motivated by tropical geometry, we have a special interest in the following particular semiring.
\begin{defn}\label{defn_max_semifield}
Let $(H,\cdot,1)$ be a lattice ordered monoid. Define addition on $H$ to be the operation of supremum, denoted by $\dotplus$, i.e., for any $a,b \in H$
\begin{equation}
a \dotplus b = \sup(a,b).
\end{equation}
Adjoin a zero element $0$ to $H$ such that $\forall a \in H \ : \ a \dotplus 0 = 0 \dotplus a = a$.
Then $\mathbb{H} = (H,\cdot,\dotplus, 1, 0)$ is a semiring. If $\mathbb{H}$ is totally ordered then supremum is maximum and the semiring is bipotent.
Taking $\mathbb{H}$ to be a multiplicative group, $\mathbb{H}$ becomes an idempotent semifield.
In such a case, for $a,b \in \mathbb{H}$, $(a \dotplus  b)^{-1} = \inf(a^{-1}, b^{-1})$. Thus $\dotplus$ induces a infimum operation, to be denoted by $\wedge$, such that $$a \wedge b = (a^{-1} \dotplus b^{-1})^{-1} = \inf(a,b).$$ In particular, for any $a \in \mathbb{H}$, $(a \dotplus a^{-1})^{-1} = a \wedge a^{-1}$.\\
If $\mathbb{H}$ is totally ordered infimum is minimum.
\end{defn}

\begin{rem}
The semifield described in Definition \ref{defn_max_semifield}, is idempotent (bipotent) and thus can be viewed as a commutative lattice ordered-group (totally ordered) $(\mathbb{H}, \dotplus , \cdot)$ where for $a,b \in \mathbb{H}$,  $a \leq b \Leftrightarrow a \dotplus b = b$.
\end{rem}


\begin{flushleft}The following lemma and the subsequent remarks establish the connection between a semifield addition and a natural order it induces.\end{flushleft}

\begin{lem}\label{lem_quasi_identity}
For every proper semifield $\mathbb{H}$, the following quasi-identity holds:
\begin{equation}
a+b+c= a \Rightarrow a+b=a
\end{equation}
for $a,b,c \in \mathbb{H}.$
\end{lem}
\begin{proof}
Let $a+b+c = a$. Multiplying both sides of the equation by $a^{-1}ba^{-1}$ and then adding $ca^{-1}$ to both parts yield
that $(ba^{-1} + ca^{-1})(1+ba^{-1}) = ba^{-1} + ca^{-1}$. Hence, as multiplicative cancellation holds, we get $1+ba^{-1} = 1$ and therefore $a+b = a$.
\end{proof}

\begin{rem}\label{rem_natural_order_on_semifield}
Every (commutative with respect to addition) semifield $\mathbb{H}$ is endowed with a partial order defined by
\begin{equation}\label{order1}
a \leq b \Leftrightarrow a=b \ \text{or} \  a+c = b \ \text{for some} \ c \in \mathbb{H}
\end{equation}
and is ordered with respect to a natural order, i.e., $a \leq b$ implies that $a+c \leq b+c$, $ac \leq bc$ and $ca \leq cb$ for all $a,b,c \in \mathbb{H}$.
In the special case of idempotent semifields, the relation~\eqref{order1} can be rephrased as
\begin{equation}\label{order2}
a \leq b \Leftrightarrow a+b = b.
\end{equation}
\end{rem}

\begin{note}
In the literature (for example in \cite{Hom_Semifields}), semifields are sometimes not \linebreak assumed to be commutative with respect to addition and thus do not always have a natural order.
 Since the semifields we consider are additively commutative (abelian), a semifield in our scope is always partially ordered with respect to the natural order.
\end{note}

\begin{note}
For two elements $a,b$ of a semifield $\mathbb{H}$, we say that $a$ and $b$ are \emph{comparable} (or $a$ is comparable to $b$) if $a \leq b$ or $b \leq a$. \\
A delicate point arises when considering functions over some semifield. For example, consider the semifield of fractions in one variable $\mathbb{H}(x)$ with $\mathbb{H}$ a semifield. Although $x + 1 \geq x$ and $x \neq x + 1$, one has $x+1 \not > x$ (for example take $x=1$).
\end{note}

\begin{defn}\label{defn_frob_property}
A semiring $S$ is said to satisfy the \emph{Frobenius property} if for every $m \in \mathbb{N}$ and for every $a,b \in S$:
$$(a+b)^m = a^{m} + b^{m}.$$
\end{defn}
\begin{lem}
Every bipotent semifield satisfies the Frobenius property.
\end{lem}
\begin{proof}
Since a bipotent semifield is totally ordered, for any pair of elements $a,b$ of the semifield one has that $a \leq b$ or $b \leq a$. Assume that $a \leq b$. Then $a^{j_1}b^{k_1} \leq a^{j_2}b^{k_2}$ for $j_1 + k_1 = j_2 + k_2 = m$ and $j_1 \geq j_2$ which in turn implies that $(a+b)^m = b^{m}$ and so, as $a^m \leq b^m$ we can write  $(a+b)^m = a^{m} + b^{m}$. Analogously $b \leq a$ implies that $(a+b)^m = b^{m} = a^{m} + b^{m}$.
\end{proof}

\begin{rem}
The converse implication does not hold, namely, a semifield satisfying the Frobenius property is not necessarily bipotent. See Example \ref{exmp_frob_not_bipotent}.
\end{rem}
%

\begin{defn}
Let $S_1,S_2$ be semirings. A map $\phi : S_1 \rightarrow S_2$ is a \emph{semiring homomorphism} if for any
$a,b \in S_1$ the following conditions hold:
$$\phi(a \cdot b) = \phi(a) \cdot \phi(b)  \  \text{and} \  \phi(a+b) = \phi(a) + \phi(b).$$
\end{defn}

\begin{rem}
A semiring homomorphism is order preserving, in the sense that for a semiring homomorphism $\phi: S_1 \rightarrow S_2$, if $a,b \in S_1$ such that $a \leq b$, then $\phi(a) \leq \phi(b)$.\\
Indeed, $a \geq b$ yields that there exists $c \in S_1$ such that $a = b + c$ thus $\phi(a) \geq \phi(b)$ since $\phi(a) = \phi(b+c) = \phi(b) + \phi(c)$.
\end{rem}

\begin{prop}\cite[Proposition (2.4)]{Hom_Semifields}
Let $\mathbb{S}_1,...,\mathbb{S}_t$, $t \in \mathbb{N}$, be proper semifields. Then their direct
product $\mathbb{S} = \mathbb{S}_1 \times \dots \times \mathbb{S}_t$, defined as the set
$$\{ (s_l,... ,s_t) :  s_i \in S_i, \ 1 \leq i \leq t\}$$
with component-wise addition and multiplication, is also a proper semifield.
\end{prop}

\begin{defn}
A semiring $\mathbb{H}$ is \emph{divisible} (also called \emph{radicalizable}) if for any $n \in \mathbb{N}$ and $\alpha \in \mathbb{H}$, there exists some $\beta \in \mathbb{H}$ such that $\beta^n = \alpha$.
\end{defn}

\begin{rem}
A homomorphic image of a divisible semifield is divisible.
\end{rem}
\begin{proof}
Let $\mathbb{H}$ be a semifield and let $\phi : \mathbb{H} \rightarrow \Im(\phi)$ be a semifield homomorphism.
Since a homomorphic image of a semifield is a semifield, we only need to show that $\Im(\phi)$ is divisible.
Let $a = \phi(\alpha) \in \Im(\phi)$. Since $\mathbb{H}$ is divisible for any $n \in \mathbb{N}$ there exists some $\beta~\in~\mathbb{S}$ such that $\beta^n = \alpha$. Taking $b = \phi(\beta) \in \Im(\phi)$, \ $b^n = \phi(\beta)^n = \phi(\beta^n) = \phi(\alpha) = a$, thus $\Im(\phi)$ is divisible.
\end{proof}

\begin{defn}
A po-group (partially ordered group)  $(G, \cdot)$ is called \emph{archimedean} if $a^{\mathbb{Z}} \leq b$ implies that $a = 1$.
A semifield $(\mathbb{H}, + , \cdot)$ is said to be \emph{archimedean} if $\mathbb{H} \setminus \{0\}$ archimedean as a po-group.
\end{defn}

\begin{note}
The archimedean property is widely used in the context of  totally ordered groups, where every pair of elements are comparable. Those who are used to working in the total order setting, may find the implications of this property in the wider context of partial order groups somewhat confusing.
\end{note}

\begin{defn}\label{defn_divisible_bipotent_base_semifield}
Throughout this chapter we denote by $\mathscr{R}$ the \emph{bipotent} semifield defined in Definition \ref{defn_max_semifield} with the supplementary properties of being divisible, \linebreak archimedean (as a po-group) and complete in the sense that the underlying lattice (with $\vee$ ($\dotplus$) and $\wedge$ as its operations) is conditionally complete.
\end{defn}

\begin{defn}
A \emph{semimodule} $\mathbb{M}$ over a semifield $\mathbb{H}$ is a semigroup $(\mathbb{M},+)$ \linebreak endowed with scalar multiplication such that for every $\alpha \in \mathbb{H}$ and $a \in \mathbb{M}$, $\alpha~\cdot~a~\in~\mathbb{M}$.
A \emph{semialgebra} $\mathbb{A}$ over a semifield $\mathbb{H}$ is a semimodule $(\mathbb{A},+)$ endowed with multiplication, such that $(\mathbb{A}, \cdot)$ is a semigroup and distributivity of multiplication over addition holds.
\end{defn}

\begin{rem}\label{rem_inverse_free}
 If $\mathbb{M}$ is a semimodule over an idempotent semifield $\mathbb{H}$, then $\mathbb{M}$ lacks inverses with respect to addition. \\
Indeed, the idempotency  of $\mathbb{H}$ implies that $\mathbb{M}$ is idempotent with respect to addition. Thus if for some $u \in \mathbb{M}$ there exists $v \in \mathbb{M}$ such that $u+v = 0$ we have that $u = u + 0 =  u+(u+v) = (u+u)+v = u(1+1) + v = u + v = 0$, proving our claim.
\end{rem}
In view of Remark \ref{rem_inverse_free}, any semialgebra (in particular a semifield) over a bipotent semifield is inverse free with respect to addition.

\begin{defn}
A semiring $\mathbb{D}$ is said to be an \emph{extension} of a semifield $\mathbb{H}$ if $\mathbb{D} \supseteq \mathbb{H}$
and $\mathbb{H}$ is a subsemiring of $\mathbb{D}$.
\end{defn}

\begin{defn}\label{defn_affine_semifield_extension}
Let $\mathbb{H}$ be a semifield, and let $\mathbb{D}$ be a semiring extending $\mathbb{H}$.
We say that $\mathbb{D}$ is generated by a subset $A \subset \mathbb{D}$ over $\mathbb{H}$ if every element $a \in \mathbb{D}$ is of the form $\sum_{i=1}^{n}\alpha_i \prod_{j=1}^{m}a_{i,j}^{k_{i,j}}$ with $a_{i,j} \in A $ and $k_{i,j} \in \mathbb{N}$. $\mathbb{D}$ is said to be \emph{affine} over $\mathbb{H}$, or an \emph{affine extension} of $\mathbb{H}$, if $A$ is finite.
\end{defn}

\begin{note}
Later in this section we introduce the notion of a `semifield with a generator'. In this definition, we define generators of a semifield to be elements generating it as a kernel. In order to avoid ambiguity,  when we refer to generation as in Definition~\ref{defn_affine_semifield_extension} we will indicate it explicitly. In any other case we consider generation as a kernel.
\end{note}

\begin{defn}\label{def_semiring_of_functions}
For any set $X$ and any semiring $\mathbb{S}$, $Fun(X,\mathbb{S})$ denotes the set of functions $f : X \rightarrow  \mathbb{S}$.
$Fun(X,\mathbb{S})$ also is a semiring, whose operations are given pointwise:
$$(fg)(a) = f(a)g(a), \ \ (f+g)(a) = f(a) + g(a)$$
for all $a \in X$. The unit element of $Fun(X,\mathbb{S})$ is the constant function always taking on the value $1_\mathbb{S}$.
\end{defn}

\begin{rem}
If $\mathbb{S}$ is a semifield without a zero element then $Fun(X,\mathbb{S})$ is a semifield (without a zero element).
Indeed, for any $f \in Fun(X,\mathbb{S})$ we have that $f^{-1}(x) = f(x)^{-1}$ is the inverse function of $f$.
\end{rem}

Recall that we always assume a semifield does not contain a zero element, unless stated otherwise.

\begin{rem}
For a semiring $\mathbb{S}$, there are two distinct semiring structures arising on $\mathbb{S}[x_1,...,x_n]$.
The first is obtained by considering $\mathbb{S}[x_1,...,x_n]$ as a subset of $Fun(\mathbb{D}^n,\mathbb{D})$ with $\mathbb{D}$ taken to be any extension of $\mathbb{S}$, i.e., the elements of $\mathbb{S}[x_1,...,x_n]$ are considered as functions defined over $\mathbb{D}^n$. The second way is to consider the variables $x_1,...,x_n$ as symbols rather than functions and taking the formal addition and multiplication operations on $\mathbb{S}[x_1,...,x_n]$.
\end{rem}

\begin{note}\label{note_semiring_of_functions}
We always consider the polynomial semiring $\mathbb{S}[x_1,...,x_n]$ (and its semifield of fractions)  mapped to the semiring of functions.
\end{note}

\begin{exmp}\label{exmp_frob_not_bipotent}
For a nontrivial bipotent semifield $\mathbb{H}$, the semiring $\mathbb{H}[x_1,...,x_n]$ (considered as a subsemiring of $Fun(\mathbb{H}^n,\mathbb{H})$) is idempotent but not bipotent (for example the constant function $\alpha$ for any $\alpha \in \mathbb{H}$ and the function $x$ are incomparable). Since $\mathbb{H}$ is bipotent, for any pair of polynomials $f,g \in \mathbb{H}[x_1,...,x_n]$ we have that  $(f(a) + g(a))^m = f(a)^m + g(a)^m$ at any given point $a \in \mathbb{H}^n$. Thus $(f + g)^m = f^m + g^m$ globally over $\mathbb{H}^n$, i.e., as elements of $\mathbb{H}[x_1,...,x_n]$. So $\mathbb{H}[x_1,...,x_n]$ satisfies the Frobenius property though it is not bipotent. Note that the arguments introduced above apply more generally to the semiring of functions yielding that the Frobenius property holds there too.
\end{exmp}

\subsubsection*{The structure of a semifield}

Like a normal subgroup in group theory and an ideal in commutative ring theory, the kernel encapsulates relations on semifields. In the following few paragraphs,  we \linebreak introduce the notion of a kernel along with some of its properties.

\begin{defn}
A subset $K$ of a semifield $\mathbb{H}$ is a \emph{semifield-kernel} of $\mathbb{H}$ if $K$ is a normal subgroup in $\mathbb{H}$ with the convexity property that for every $x,y \in \mathbb{H}$ such that $x+y = 1$,
\begin{equation}\label{defn_ker_convexity}
a,b \in K  \ \Rightarrow xa + by \in K.
\end{equation}
The set of all the kernels of a semifield $\mathbb{H}$ is denoted by $\Con(\mathbb{H})$.
\end{defn}

\begin{note}
We note that the name `kernel' and the notation $Con$ are customary in previous study of semifields.
From now on we refer to a semifield-kernel simply as a `kernel'. In places where confusion arises with the notion of a kernel of a homomorphism, we provide clarification.
\end{note}

\begin{note}
\begin{itemize}
  \item Some may find the name `kernel' not necessarily  the best choice for a name for the above structure. We presume the motivation for the name is that a `kernel' is a kernel of a semifield homomorphism. It obviously might cause a little confusion. When  such confusion may occur we explicitly indicate to which notion of kernel we refer.
  \item Though the so-called `convexity' condition \eqref{defn_ker_convexity} may give a somewhat \linebreak misleading impression, kernels are nothing but a special kind of groups inside the semifield $\mathbb{H}$.
\end{itemize}
\end{note}

\ \\

\begin{rem}\cite[Proposition 1.1]{Prop_Semifields}
An equivalent definition of a kernel of a semifield $\mathbb{H}$ is the class $[1]_{\rho}$ of an arbitrary congruence
$\rho$ on $\mathbb{H}$.
\end{rem}

\begin{rem}
If $\mathbb{S}$ is an idempotent semifield since $1 + 1 = 1$, we get that for any kernel $K$ of $\mathbb{S}$,  $a,b \in K \Rightarrow a + b = 1a + 1b \in K$, yielding that $K$ is itself a semifield. A particular case of interest is the semifield of fractions $\mathbb{H}(x_1,...,x_n)$ of $\mathbb{H}[x_1,...,x_n]$. If $\mathbb{H}$ is idempotent, then $\mathbb{H}(x_1,...,x_n)$ is idempotent which yields that the kernels of $\mathbb{H}(x_1,...,x_n)$ are subsemifields of $\mathbb{H}(x_1,...,x_n)$.
\end{rem}

\begin{note}
Throughout this dissertation, we work with an underlying semifield $\mathbb{H}$ of $\mathbb{H}(x_1,...,x_n)$ which is both idempotent and archimedean. In particular, under the assumption of idempotency, all kernels are also semifields.
\end{note}

\begin{thm}\cite[Theorem 3.6]{Hom_Semifields}
The set $\Con(\mathbb{H})$ of all kernels of a semifield $\mathbb{H}$ forms a full modular lattice with respect to the operations of multiplication and the \linebreak intersection of kernels, canonically isomorphic to the lattice of all possible
congruencies on $\mathbb{H}$.
\end{thm}

\begin{rem}\label{rem_ker_latice_operations}\cite{Hom_Semifields}
Let $K_1$ and $K_2$ be kernels of the semifield $\mathbb{H}$. Then $K_1 \cap K_2$ and $K_1 \cdot K_2$ are kernels of $\mathbb{H}$. Moreover $K_1 \cdot K_2$ is the smallest kernel in $\mathbb{H}$ containing $K_1 \cup K_2$.
\end{rem}

\begin{note}
Note that multiplication of kernels is formulated by
$$K_1 \cdot K_2 = \{ ab \ : \ a \in K_1, b \in K_2 \}$$
as customary in the theory of groups.
\end{note}

\begin{lem}\label{lem_kernels_algebra}\cite[Lemma 4.1]{Prop_Semifields}
The following equalities hold for arbitrary kernels $A,B$ and $K$ of a semifield $\mathbb{H}$, among which at least one is a semifield:
\begin{equation}
AK \cap BK  = (A \cap B)K ;
\end{equation}
\begin{equation}
(A \cap K)(B \cap K)= AB  \cap K.
\end{equation}
\end{lem}

\begin{cor}
Since every kernel of an idempotent semifield is also a semifield, we have by Lemma \ref{lem_kernels_algebra} that its lattice of kernels  is distributive. Thus every idempotent semifield is distributive.
\end{cor}

\begin{flushleft}The following are the three fundamental isomorphism theorems.\end{flushleft}

\begin{thm}\label{thm_kernels_hom_relations}\cite[Theorems 3.4 and 3.5]{Hom_Semifields}
Let $\mathbb{H}_1, \mathbb{H}_2$ be semifields and let $R \subset \mathbb{H}_1$ be a subsemifield of $\mathbb{H}_1$. Let $\phi : \mathbb{H}_1 \rightarrow \mathbb{H}_2$ be a semiring homomorphism and let $K$ be the homomorphism kernel of $\phi$. Then the following hold:
\begin{enumerate}
  \item $\phi(R) \subset \mathbb{H}_2$ is a subsemifield of $\mathbb{H}_2$. The homomorphism kernel of the restriction $\phi:R \rightarrow \phi(R)$ is $R \cap K$.
  \item $\phi^{-1}(\phi(R))= KR$ which is a subsemifield of \ $\mathbb{H}_1$.
  \item For any kernel $L$ of \ $\mathbb{H}_1$, $\phi(L)$ is a kernel of \ $\phi(\mathbb{S}_1)$.
  \item For a kernel $K$ of \ $\phi(\mathbb{H}_1)$, $\phi^{-1}(K)$ is a kernel of \ $\mathbb{H}_1$. In particular, for any kernel $L$ of \ $\mathbb{H}_1$ we have that $\phi^{-1}(\phi(L)) = K\L$ is a kernel of $\mathbb{H}_1$.
\end{enumerate}
\end{thm}

\begin{cor}
As a special case of (4), taking $K = \{1\}$, we have that the homomorphism kernel $\phi^{-1}(1)$ of a semifield homomorphism $\phi : \mathbb{H}_1 \rightarrow \mathbb{H}_2$ is a kernel.
\end{cor}

\begin{thm}\cite{KA_Semifields}\label{thm_nother_1_and_3}
Let $\mathbb{H}$ be a semifield and $K$ a kernel of \ $\mathbb{H}$.
\begin{enumerate}
  \item If $\mathbb{U}$ is a subsemifield of  $\mathbb{H}$, then $\mathbb{U} \cap K$ is a kernel of $\mathbb{U}$ and $K$ a kernel of the subsemifield $\mathbb{U} \cdot K = \{u \cdot k \ : \ u \in \mathbb{U}, \ k \in K \}$ of \ $\mathbb{H}$ and one has the isomorphism $$\mathbb{U}/(\mathbb{U} \cap K) \cong \mathbb{U} \cdot K/K.$$
  \item If $L$ is a kernel of \ $\mathbb{H}$, then $L \cap K$ is a kernel of $L$ and $K$ a kernel of $L \cdot K$. Now, one has in general only the group isomorphism $$L/(L \cap K) \cong L \cdot K/K$$ which is a semifield isomorphism exactly in the case when $L$ is also a \linebreak subsemifield of \ $\mathbb{H}$.
\end{enumerate}
\end{thm}

\begin{thm}\cite{KA_Semifields}\label{thm_nother2}
Let  $\mathbb{H}$ be a semifield and let $K$ and $L$ be kernels of \ $\mathbb{H}$ satisfying $K \subseteq L$. Then $L/K$ is a kernel of \ $\mathbb{H}/K$ and one has the semifield isomorphism $$\mathbb{H}/L \cong (\mathbb{H}/K)/(L/K).$$
\end{thm}

The following result concerns the induced order of the quotient semifield. It holds for any idempotent semifield $\mathbb{S}$, and in particular for $\mathbb{H}(x_1,...,x_n)$ with $\mathbb{H}$ an idempotent semifield.\\

\begin{thm}\label{cor_qoutient_corr}
Let $\mathbb{H}$ be a semifield and let $L \in \Con(\mathbb{H})$ be a kernel of $\mathbb{H}$. Every kernel of $\mathbb{H}/L$ has the form $K/L$ for some kernel $K \in \Con(\mathbb{H})$ uniquely determined such that $K \supseteq L$, and there is a $1:1$ correspondence
$$ \{ \text{Kernels of} \ \mathbb{H}/L \} \rightarrow  \{ \text{Kernels of} \ \mathbb{H} \ \text{containing} \ L \}$$
given by $K/L \mapsto K$.
\end{thm}
\begin{proof}
From the theory of groups we have that there is such a bijection for normal subgroups. To apply the theorem for kernels, we only need to show that a homomorphic image and preimage of a kernel are kernels, which in turn is true by \linebreak Theorem \ref{thm_kernels_hom_relations}.
\end{proof}

\ \\

\begin{rem}
Let $K$ be a kernel of an idempotent semifield $\mathbb{S}$. Then the induced order on the quotient semifield $\mathbb{S}/K$ is such that $aK \leq bK \Leftrightarrow (1 + ab^{-1}) \in K$.
\end{rem}
\begin{proof}
Let $a, b$ be elements of $\mathbb{H}$. The induced order on $\mathbb{H}/K$ is given
by: $aK \leq bK$ if and only if there exists some $c \in K$ such that $a \leq cb$. Now, the following identity holds in $\mathbb{H}$: \ $a = (a \wedge b)(ab^{-1} \dotplus 1)$ ( distributing the right hand side and the left hand side one at a time give opposite weak inequalities).
Consequently, if $ (ab^{-1} \dotplus 1) \in K$ then $a =  (a \wedge b)(ab^{-1} \dotplus 1)$, implying that $aK = (a \wedge b)K \leq bK$. Conversely, let $c \in K$ such that $a \le cb$. Then we have
$(a \wedge b) \leq a \leq cb \Rightarrow (a \wedge b)b^{-1} \leq ab^{-1} \leq c \Rightarrow
ab^{-1} \dotplus 1 \leq ab^{-1} \leq c \Rightarrow 1 \dotplus (ab^{-1} \dotplus 1) \leq (ab^{-1} \dotplus 1) \leq 1 \dotplus c \Rightarrow
1 \leq (ab^{-1} \dotplus 1) \leq 1 \dotplus c$. Since $1, 1 \dotplus c \in K$ we have that $(ab^{-1} \dotplus 1) \in K$.
\end{proof}

\begin{flushleft}We conclude this part with the definition of a large kernel.\end{flushleft}

\begin{defn}
Let $\mathbb{S}$ be a semifield.
A kernel $K$ of a semifield $\mathbb{S}$ is said to be \emph{large} in $\mathbb{S}$ if  $L \cap K \neq \{1 \}$ for each kernel $L \neq \{ 1 \}$ of $\mathbb{S}$.
\end{defn}

\newpage

\subsection{Semifield with a generator and generation of kernels}

\ \\

\begin{defn}\label{defn_generation_of_kernels}
Let $A$ be a subset of a semifield $\mathbb{H}$. Denote by $\langle A \rangle$ the smallest kernel in $\mathbb{H}$ containing~$A$. It is equal to the intersection of all kernels in $\mathbb{H}$ containing $A$. If $\mathbb{H} = \langle A \rangle$, then $A$ is called a set of generators of the semifield $\mathbb{H}$ (as a kernel). \\
A kernel $K$ is said to be \emph{finitely generated}  if $K = \langle A \rangle$ where $A$ is a finite set of elements of $\mathbb{H}$.  By Remark \ref{rem_ker_latice_operations}, if $K$ is generated by $\{a_1,...,a_t \} \subset \mathbb{H}$ then $K = \langle a_1 \rangle \cdot \dots \cdot \langle a_t \rangle$
(the smallest kernel containing $\{a_1,...,a_t \}$). In such a case, we write $K = \langle a_1,...,a_t \rangle$ to indicate that $K$ is generated by $\{a_1,...,a_t \}$. If $K = \langle a \rangle$ for some $a \in \mathbb{H}$, then $K$ is called a \emph{principal kernel}.
A semifield is said to be \emph{finitely generated} if it is finitely generated as a kernel.
If $\mathbb{H} = \langle a \rangle$ for some $a \in \mathbb{H}$, then $a$ is called a \emph{generator} of $\mathbb{H}$ and $\mathbb{H}$ is said to be a \emph{semifield with a generator}. In other words, a semifield with a generator is a semifield which is principal as a kernel of itself.
\end{defn}

\begin{lem}\label{lem_generators}\cite[Property 2.3]{Prop_Semifields}
Let $K$ be a kernel  of a semifield $\mathbb{H}$. Then for $a,b \in \mathbb{H}$,
\begin{equation}
a + a^{-1} \in K \ \ \text{or} \ \ a + a^{-1}+ b \in K \ \ \Rightarrow a \in K.
\end{equation}
\end{lem}
\begin{proof}
Let $a + a^{-1} \in K$. Then in the factor semifield $\mathbb{H}/K$, we get $w + w^{-1} = 1$ for each element $w = aK$
which yields that $w \leq  1, w^{-1} \leq 1$, and thus  $w = 1$, i.e., $aK = K$ and so $a \in K$. For the second condition just apply Lemma \ref{lem_quasi_identity}.
\end{proof}

\begin{prop}\label{prop_absolute_generator}\cite{Prop_Semifields}
Let $\mathbb{H}$ be a semifield. Then for any $a \in \mathbb{H}$ such that the kernel generated by $a$ is a semifield, the following equality holds:
\begin{equation*}
\langle a \rangle = \langle a + a^{-1} \rangle.
\end{equation*}
In words, the kernel generated by $a$ coincides with the kernel generated by $a + a^{-1}$.
\end{prop}
\begin{proof}
A direct consequence of Lemma \ref{lem_generators}, which implies that $a \in \langle a + a^{-1} \rangle$, and thus $\langle a \rangle  \subseteq \langle a + a^{-1} \rangle$. The converse inclusion follows from the fact that $\langle a \rangle$ is a semifield.
\end{proof}

\ \\
\ \\

\begin{rem} \label{rem_kernel_is_convex}\cite[Property 2.4]{Prop_Semifields}
Every kernel $K$ of a semifield $\mathbb{H}$ is convex with respect to the natural order on $\mathbb{H}$:
for $a,c \in K$ and  $b \in \mathbb{H}$
\begin{equation}
a \leq b \leq c \ \Rightarrow \ b \in K.
\end{equation}
\end{rem}
\begin{proof}
$a \leq b \leq c$ implies that $a + u = b$ and $b + v =c$ for $u,v \in \mathbb{H}$, i.e., $a+u+v=c$. The equality $ K + uK + vK = K$ holds in the factor semifield $\mathbb{H}/K = \{ yK \ : \ y \in \mathbb{H} \} $. Then $bK = (a + u)K = aK + uK = K$ by Lemma \ref{lem_quasi_identity}. Thus $b \in K$.
\end{proof}

\begin{prop}\label{prop_full_ordered_group_is_kernel}
Let $N$ be a convex normal subgroup of a semifield $\mathbb{H}$. If each $a \in N$ is comparable with $1$, i.e., if $\leq$ is a total order on $N$, then $N$ is a kernel of $\mathbb{H}$.
\end{prop}
\begin{proof}
In order to prove $N$ is a kernel, we need to show that for $a \in N$ and $s,t \in \mathbb{H}$ such that $s+t=1$, $s+th \in N$. By assumption, we have that $1 \leq a$ or $a \leq 1$. In the former case, we have $t \leq ta$ and $s \leq sa$, thus $1 = s+t \leq s +ta \leq sa + ta = (s + t)a = a$. Since $N$ is convex and $a,1 \in N$ we get that $s+ta \in N$. The latter case yields that $a \leq s +ta \leq 1$ which implies the same.
\end{proof}

\ \\

The following remark yields an important property of a kernel, which we call \linebreak `power-radicality', to be introduced shortly.
\begin{rem}\label{rem_torsion_free}
The multiplicative group of every proper semifield $\mathbb{H}$ is a torsion-free group, i.e., all of its elements that are not equal to $1$ have infinite order.
\end{rem}
\begin{proof}
If $a^{n} = 1$ for $a \in \mathbb{H}$ and $n \in \mathbb{N}$, then
$$a(a^{n-1} + a^{n-2} + \dots + a + 1) = a^{n} + (a^{n-1} + \dots + a) = 1 + (a^{n-1} + \dots + a)$$
$$=a^{n-1} + \dots + a + 1$$ which yields that $a = 1$.
\end{proof}

The following remark is a straightforward consequence of Remark \ref{rem_torsion_free}.
\begin{rem}\label{rem_radicality_of_ker}
Let $K$ be a kernel of a proper semifield $\mathbb{H}$. For every $a \in \mathbb{H}$, if $a^{n} \in K$ for some $n \in \mathbb{N}$ then $a \in K$. We refer to this property of kernels by saying that a kernel is \emph{power-radical}.
\end{rem}
\begin{proof}
Indeed, if there exists an element $a \in \mathbb{H}$ not admitting the stated property, then its image $\phi(a)$ in $\mathbb{H}/K$ under the quotient homomorphism is torsion, which by Remark \ref{rem_torsion_free} is not possible since $\mathbb{H}/K$ is a semifield.
\end{proof}

The following subsequent statements establish a connection between the (normal) group generated by a set of elements and the kernel generated by the set. Recall that, kernels are a specific kind of group, the (normal) subgroup generated by a set of elements need not be a kernel.

\begin{prop}\label{prop_ker_stracture_by group}\cite[Proposition (3.13)]{Hom_Semifields}
Let $\mathbb{H}$ be a semifield and let $N$ be a normal subgroup of $( \mathbb{H}, \cdot)$. Then the smallest kernel containing $N$ is given by
\begin{equation}
\mathcal{K}(N) = \left\{ \sum_{i=1}^{n}s_ih_i \ : \ n \in \mathbb{N}, \ h_i \in  N, \ s_i \in \mathbb{H} \ \text{such that} \ \sum_{i=1}^{n}s_i = 1 \right\}.
\end{equation}
\end{prop}
\ \\

\begin{rem}
Let $S \subset \mathbb{H}$. The kernel generated by $S$ is $\langle S \rangle = \mathcal{K}(\mathcal{G}(S))$ where $\mathcal{G}(S)$ is the (multiplicative) group generated by $S$.
\end{rem}
\begin{proof}
As a kernel is defined to be a multiplicative (normal) group, the assertion is immediate.
\end{proof}

\begin{rem}
Let $S_1,...,S_r \subset \mathbb{H}$ and let $G_1,...,G_r$ be the groups generated by $S_1,...,S_r$ respectively.
Then $\langle \bigcup_{i=1}^{r}{S_i} \rangle = \mathcal{K}(\prod_{i=1}^{r}G_i) = \prod_{i=1}^{r}\mathcal{K}(G_i) = \prod_{i=1}^{r}\langle S_i \rangle$
\end{rem}
\begin{proof}
By definition, $\prod_{i=1}^{r}\mathcal{K}(G_i)$ is a kernel, and thus a group, which contains all the groups $G_i$ for $i = 1,...,r$, thus also contains the group $\prod_{i=1}^{r}G_i$. Since
$\mathcal{K}(\prod_{i=1}^{r}G_i)$ is the smallest kernel containing the group $\prod_{i=1}^{r}G_i$, we get that
$\mathcal{K}(\prod_{i=1}^{r}G_i) \subseteq \prod_{i=1}^{r}\mathcal{K}(G_i)$. Now, since $G_i \subseteq \mathcal{K}(\prod_{i=1}^{r}G_i)$ for every $i = 1,...,r$, we have that $\bigcup_{i=1}^{r}G_i \subseteq \mathcal{K}(\prod_{i=1}^{r}G_i)$. As $\prod_{i=1}^{r}\mathcal{K}(G_i)$ is the smallest kernel containing $\bigcup_{i=1}^{r}G_i$ (see Remark \ref{rem_ker_latice_operations}), we get the converse inclusion and thus equality.
All other equalities are group theoretic basic equalities.
\end{proof}

\begin{prop}\label{prop_principal_ker}\cite[Proposition (3.1)]{Prop_Semifields}
Let $K = \langle a \rangle$, a principal kernel in a semifield $\mathbb{H}$ with $a \in \mathbb{H}$ such that $a \geq 1$. Then
\begin{equation}
K = \{ x \in \mathbb{H} \ : \ \exists n \in \mathbb{N} \  \ \ \text{such that} \ \ \  a^{-n} \leq x \leq a^{n} \}.
\end{equation}
\end{prop}

\begin{cor}\label{cor_positive_generator_of_a_semifield}
Every nontrivial semifield $\mathbb{H}$ with a generator has a generator $a \geq 1$ such that $a \neq 1$.
\end{cor}
\begin{proof}
Let $u \in \mathbb{H} \setminus \{1\}$ be a generator of $\mathbb{H}$. By Lemma \ref{lem_generators}, the element $u + u^{-1}$ is also a generator of $\mathbb{H}$ which yields that the element $a = (u + u^{-1})^2 = u^2 + u^{-2} + 1 \geq 1$ is a generator of $\mathbb{H}$ too by Proposition \ref{prop_principal_ker}.
\end{proof}

\begin{rem}\label{rem_non_trivial_element}
Every semifield $\mathbb{H}$ such that $\mathbb{H} \neq \{ 1 \}$ has an element
$a \in \mathbb{H}$ such that $a > 1$.
\end{rem}
\begin{proof}
Indeed, $\mathbb{H} \neq \{ 1 \}$, so there exists $a \in \mathbb{H} \setminus \{1 \}$. Now, if $a$ is not comparable with $1$ then take $1+a$. Note that $1 + a \neq 1$ since otherwise, it would imply that $a \leq 1$, contradicting our assumption that $a$ and $1$ are not comparable. Thus $1+a >1$. On the other hand, if $a$ is comparable with $1$ and if $a < 1$, take $1 < a^{-1} \in \mathbb{H}$.
\end{proof}


\begin{cor}\label{cor_principal_ker_by_order}
If $\mathbb{H}$ is a semifield, then for any element $a \in \mathbb{H}$ we have that
\begin{equation}
\langle a \rangle = \{ x \in \mathbb{H} \ : \ \exists n \in \mathbb{N} \ \text{such that} \  (a+a^{-1})^{-n} \leq x \leq (a + a^{-1})^{n} \}.
\end{equation}
\end{cor}
\begin{proof}
This is a direct consequence of Proposition \ref{prop_principal_ker} and Proposition \ref{prop_absolute_generator}.
\end{proof}

\begin{note}
Note that for an idempotent semifield the equality introduced in Corollary \ref{cor_principal_ker_by_order} can be restated as
\begin{equation}
\langle a \rangle = \{ x \in \mathscr{R} \ : \ \exists n \in \mathbb{N} \ \text{such that} \  (a \wedge a^{-1})^{n} \leq x \leq (a \dotplus a^{-1})^{n} \}
\end{equation}
using the underlying lattice operation $\wedge$.
\end{note}

By Corollary \ref{cor_principal_ker_by_order} we have
\begin{rem}\label{rem_kernel_by_abs_value}
For any element $a \in \mathbb{H}$ we have that
\begin{equation}\label{eq_kernel_by_abs_value}
\langle a \rangle = \{ x \in \mathbb{H} \ : \ \exists n \in \mathbb{N} \  \text{such that} \  (x + x^{-1}) \leq (a + a^{-1})^{n} \}.
\end{equation}
\end{rem}
\begin{proof}
Just take inverses in equation $(a+a^{-1})^{-n} \leq x \leq (a + a^{-1})^{n}$ and sum up both sides of resulting weak inequalities.
\end{proof}

\begin{note}
Equality \eqref{eq_kernel_by_abs_value} of Remark \ref{rem_kernel_by_abs_value} can be written as
\begin{equation}\label{eq_kernel_by_abs_value1}
\langle a \rangle = \{ x \in \mathbb{H} \ : \ \exists n \in \mathbb{N} \  \text{such that} \  |x| \leq |a|^n \}.
\end{equation}
\end{note}

\begin{rem}\label{rem_homomorphic_image_of_semifield_with_generator}\cite[Property (3.2)]{Prop_Semifields}
The homomorphic image of a generator is a generator of the image. In particular,
a homomorphic image of a semifield with a generator is a semifield with a generator.
\end{rem}
\begin{proof}
If $\phi : \mathbb{H} \rightarrow \mathbb{U}$ is a homomorphism of a semifield $\mathbb{H}$ onto a semifield $\mathbb{U}$ and $\mathbb{H} = \langle a \rangle$, then $\mathbb{U} = \langle \phi(a) \rangle$, because the preimage of a kernel at a homomorphism of semifields is always a kernel.
\end{proof}

\begin{rem}\label{rem_image_of_principal_kernel}
Let $\langle a \rangle$ be a principal kernel of a semifield $\mathbb{H}$, which is also a semifield and let $\phi : \mathbb{H} \rightarrow \mathrm{U}$ be a semifield epimorphism. Then
$$\phi(\langle a \rangle) = \langle \phi(a) \rangle = \{ b \in U \ : \ \exists n \in \mathbb{N} \  \text{such that} \  |b| \leq |\phi(a)|^n \},$$
i.e., the image of $\langle a \rangle$ is the kernel generated by $\phi(a)$ in $U$.
\end{rem}
\begin{proof}
As $\langle a \rangle$ is also a semifield, it is a semifield with a generator $a$. Thus by \linebreak Remark~\ref{rem_homomorphic_image_of_semifield_with_generator} its homomorphic image is also a semifield with a generator $\phi(a)$ and by Theorem \ref{thm_kernels_hom_relations} it is a kernel. Thus, its homomorphic image is a principal kernel $\langle \phi(a) \rangle$ which is also a semifield.
\end{proof}
\ \\

We can apply Remark \ref{rem_image_of_principal_kernel} and get
\begin{cor}\label{cor_homomorphic_image_of_principal_kernel}
 Let $\phi : \mathscr{R}(x_1,...,x_n) \rightarrow \mathrm{U}$ be a semifield epimorphism. Then for every principal kernel $\langle f \rangle$ of $\mathscr{R}(x_1,...,x_n)$, one has that
\begin{equation}\label{eq_homomorphic_image_of_principal_kernel}
\phi(\langle f \rangle) = \langle \phi(f) \rangle_{U} = \{ g \in \phi( \mathscr{R}(x_1,...,x_n)) \ : \ \exists n \in \mathbb{N} \  \text{such that} \  |g| \leq |\phi(f)|^n \}.
\end{equation}
\end{cor}

\begin{note}
Note that if $\phi$ is not onto $U$, then the kernel generated by $\phi(f)$ in $U$, $\langle \phi(f) \rangle_{U}$, may contain elements that are not in the image of $\phi$. In general one has that $\langle \phi(f) \rangle_{\Im(\phi)} \subseteq \langle \phi(f) \rangle_{U}$.
\end{note}

\begin{thm}\label{thm_semi_with_a_gen}
If a semifield $\mathbb{H}$ has a finite number of generators, then $\mathbb{H}$ is a semifield with a generator.
\end{thm}
\begin{proof}
Let $\mathbb{H} = \langle u_1 \rangle \cdot \dots \cdot \langle u_n \rangle$ with the finite set of generators $\{ u_1,...,u_n \}$. By Remark~\ref{lem_generators}, $u_1,...,u_n$ are contained in the kernel $K = \langle u \rangle \subseteq \mathbb{H}$ where $ u  = u_1 + u_1^{-1} + \dots + u_n + u_n^{-1}$, thus $\mathbb{H} = \langle u \rangle$ as desired.
\end{proof}

\begin{rem}\label{rem_affine_semifields_as_images}
Let $\mathbb{H}$ be an idempotent semifield.
Let $\mathbb{K}=\mathbb{H}(a_1,...,a_n)$ be an affine semifield extension of the semifield $\mathscr{R}$.
Then $\mathbb{K} \cong \mathbb{H}(x_1,...,x_n)/K$ for some \linebreak $K \in \Con(\mathbb{H}(x_1,...,x_n))$.
\end{rem}
\begin{proof}
 Let $\phi : \mathbb{H}(x_1,...,x_n) \rightarrow \mathbb{K}$ be the substitution map sending $x_i \mapsto a_i$. Then~$\phi$~is an epimorphism. Taking $K = Ker\phi$ we have by the first isomorphism theorem that $\mathbb{K} \cong \mathbb{H}(x_1,...,x_n)/K$.
\end{proof}

\begin{cor}\label{affine_semifield_has_a_generator}
Every affine semifield over an idempotent semifield $\mathbb{H}$ is a semifield with a generator.
\end{cor}
\begin{proof}
By Remark \ref{rem_affine_semifields_as_images} an affine semifield is a homomorphic image of of the semifield of fractions, which is a semifield with a generator, thus by Remark \ref{rem_homomorphic_image_of_semifield_with_generator} is also a semifield with a generator.
\end{proof}

\begin{rem}\label{rem_the contant_generated_kernel}
Let $\mathbb{H}$ be an archimedean semifield. Then
$$\langle \alpha \rangle = \langle \beta \rangle \in \PCon(\mathbb{H}(x_1,...,x_n))$$
for any $\alpha, \beta \in \mathbb{H} \setminus \{1 \}$.
\end{rem}
\begin{proof}
Indeed, since $\mathbb{H}$ is archimedean, Corollary \ref{cor_principal_ker_by_order} implies that $\alpha \in \langle \beta \rangle$ and \linebreak $\beta \in \langle \alpha \rangle$ so  $\langle \alpha \rangle = \langle \beta \rangle$.
\end{proof}

\begin{nota}
As it does not depend on the choice of constant generator $\alpha \in \mathbb{H} \setminus \{1\}$, we denote the kernel generated by $\alpha$ by $\langle \mathbb{H} \rangle$.
\end{nota}

Note that if $\mathbb{H}$ is an idempotent semifield then the semifield $\mathbb{H}(x_1,...,x_n)$ is also idempotent, so $\langle \mathbb{H} \rangle \in \PCon(\mathbb{H}(x_1,...,x_n))$ is a subsemifield of  $\mathbb{H}(x_1,...,x_n)$. Also note that the elements of $\langle \mathbb{H} \rangle$ are rational functions which are not necessarily constant. We discuss the structure of $\langle \mathbb{H} \rangle$ thoroughly in the subsequent sections.

\begin{note}
\textbf{Henceforth we always assume affine extensions, in particular the semifield of fractions, to be defined over an idempotent semifield, which make the extensions idempotent.
}
\end{note}

\ \\

\subsection{Simple semifields}

\ \\

\begin{defn}
A kernel $K$ of a semifield $\mathbb{H}$ which contain no kernels but the trivial ones, $\{1\}$ and $K$ itself, is called \emph{simple}.
A semifield is \emph{simple} if it is simple as a kernel of itself.
\end{defn}

\begin{rem}\label{rem_order_simple}
Every totally (linearly) ordered archimedean semifield (i.e., bipotent semifield) $\mathbb{H}$ has no kernels but the trivial ones, i.e., is simple.
\end{rem}
\begin{proof}
We may assume $\mathbb{H} \neq \{ 1 \}$. Let $a \in \mathbb{H}$ such that $a > 1$ (there exists such $a$ by Remark \ref{rem_non_trivial_element}). Now, since $\mathbb{H}$ is a linearly (totally) ordered semifield, for every $b \in \mathbb{H}$ there exists $m \in \mathbb{N}$ such that $a^{-m} \leq b \leq a^{m}$. Then by Proposition \ref{prop_principal_ker} we have that $b \in \langle a \rangle$. Thus $\langle a \rangle = \mathbb{H}$ and our claim is proved.
\end{proof}

\begin{rem}\label{rem_simple_is_with_generator}
Any simple semifield is a semifield with a generator.
\end{rem}
\begin{proof}
Indeed, if $\mathbb{H}$ is trivial then the assertion is obvious. Assume $\mathbb{H} \neq \{1\}$, then there exist some $\alpha \in \mathbb{H} \setminus \{1 \}$ and so $\langle 1 \rangle \subset \langle \alpha \rangle \subseteq \mathbb{H}$. Since $\mathbb{H}$ is simple we have that $\langle \alpha \rangle = \mathbb{H}$, so $\mathbb{H}$ is a semifield with a generator.
\end{proof}

\begin{cor}
The semifield $\mathscr{R}$ is simple and thus by Remark \ref{rem_simple_is_with_generator} a semifield with a generator.
\end{cor}

\newpage

\subsection{Irreducible kernels, maximal kernels and the Stone topology}

\ \\

\begin{defn}\label{defn_irreducible_maximal_kernels}
A proper (non-trivial) kernel $K$ of a semifield $\mathbb{H}$ is said to be \linebreak \emph{irreducible} if for any pair of kernels $A,B$ of $\mathbb{H}$
\begin{equation}
A \cap B \subseteq K \Rightarrow A \subseteq K \ \text{or} \ B \subseteq K.
\end{equation}
A kernel $K$ is called \emph{weakly irreducible}  if for any pair of kernels $A,B$ of $\mathbb{H}$
\begin{equation}
A \cap B = K \Rightarrow A = K \ \text{or} \ B = K.
\end{equation}
$K$ is called \emph{maximal} if for any kernel $A$ of $\mathbb{H}$
\begin{equation}
K \subseteq A \Rightarrow K=A \ \text{or} \ A = K.
\end{equation}
\end{defn}

%
%

\begin{defn}\label{defn_reduced_semifield}
A semifield $\mathbb{H}$ is said to be \emph{reduced} if for any pair of kernels $A$ and~$B$ of $\mathbb{H}$,
$A \cap B = \{1\}$ implies that $A = \{1\}$ or $B = \{1\}$.
\end{defn}

\begin{rem}
If $P$ be an irreducible kernel of $\mathbb{H}$, then the quotient semifield $\mathbb{H}/P$ is reduced.
\end{rem}
\begin{proof}
Let $\phi : \mathbb{H} \rightarrow \mathbb{H}/P$ be the quotient map and let $A \neq \{1 \} ,B \neq \{1 \}$  kernels of $\mathbb{H}/P$ such that $A \cap B = \{ 1 \}$. Then the kernels $A'=\phi^{-1}(A)$ and $B'=\phi^{-1}(B)$
admit $A' \cap B' \supseteq  \phi^{-1}(A \cap B) = \phi^{-1}(\{1\}) = P$ which yields that either $A' \subseteq P$ or $B' \subseteq P$ so $A = \phi(A') \subseteq \phi(P) = \{1\}$ or $B = \phi(B') \subseteq \phi(P) = \{1\}$, contradicting our assumption that $A \neq \{1 \} ,B \neq \{1 \}$. Thus $A = \{1\}$ or $B = \{1 \}$ and $\mathbb{H}/P$ is reduced.
\end{proof}

\begin{thm}\label{thm_irr_kernels}\cite[Theorem (4.1)]{Prop_Semifields}
Let $K$ be a proper kernel of a semifield $\mathbb{H}$. Then there exists at least one irreducible kernel $P$ of $\mathbb{H}$ such that $K \subseteq P$.
\end{thm}

An immediate consequence of Theorem \ref{thm_irr_kernels} is
\begin{rem}\label{cor_max_kernels}
Every maximal kernel is irreducible.
\end{rem}

\ \\

In the following we prove some assertions concerning maximal kernels.

\begin{rem}
Let $\mathbb{H}$ be a semifield. Let $K$ be a kernel of $\mathbb{H}$ and $S \subseteq \mathbb{H}$ a subset. By Remark \ref{rem_ker_latice_operations},  The smallest kernel of $\mathbb{H}$ containing both $K$ and $S$ is $\langle M \cup S \rangle = M \cdot \langle S \rangle$.
\end{rem}

\begin{rem}\label{rem_maximal_kernel_property}
Let $M$ be a kernel of a semifield $\mathbb{H}$. $M$ is maximal if and only if for any $a \in \mathbb{H} \setminus M$, $\langle M \cup \{ a \} \rangle = M \cdot \langle a \rangle = \mathbb{H}$.
\end{rem}
\begin{proof}
Let $M$ be maximal kernel. Since $a \not \in M$ and since $M$ is a kernel, we have that $M \subset \langle M \cup \{ a \} \rangle$ and thus $\langle M \cup \{ a \} \rangle = \mathbb{H}$. Conversely, assume $M$ is not maximal, then there exists a kernel $N \neq \mathbb{H}$ such that $M \subset N$, thus there exists $a \in N$ such that $a \not \in M$ and so we get $M \subset \langle M \cup \{ a \} \rangle = \mathbb{H} \subseteq N$, contradicting $N$ being a maximal kernel.
\end{proof}

\begin{cor}\label{cor_max_ker_simple_corr}
For any semifield $\mathbb{H}$ and a kernel $K$ of $\mathbb{H}$, $K$ is a maximal kernel if and only if $\mathbb{H}/K$ is simple.
\end{cor}
\begin{proof}
 If $\mathbb{H}/K$ is not simple, then there exists a kernel $\{ 1\} \subset B \subset  \mathbb{H}/K$.\\ If $\phi:~\mathbb{H}~\rightarrow~\mathbb{H}/K$ is the quotient homomorphism, then by Theorem \ref{thm_kernels_hom_relations}, $\phi^{-1}(B)$ is a kernel of $\mathbb{H}$ and $ K =\phi^{-1}(\{1\}) \subset \phi^{-1}(B) \subset \phi^{-1}(\mathbb{H}/K) = \mathbb{H}$, so $\phi^{-1}(B)$ is proper and contains $K$.
Assume $K$ is not maximal, then there is some kernel $M$ of $\mathbb{H}$ containing (not equal to) $K$. Now, by  Theorem \ref{thm_kernels_hom_relations}(4), $K \subset \phi^{-1}(\phi(M)) = KM = M$ (since \ $K \subset M$) which in turn yields that $\phi(M) \subset \mathbb{H}/K$ is a proper kernel of $\mathbb{H}/K$ (for otherwise $\phi^{-1}(\phi(M)) = \mathbb{H}$) and $\phi(M) \neq \{1\}$ (for otherwise by the above $M=K$) thus $\mathbb{H}/K$ is not simple.
\end{proof}

\begin{defn}
The set \emph{Spec($\mathbb{H}$)} of all irreducible kernels of a semifield $\mathbb{H}$ is called the \emph{irreducible spectrum} of $\mathbb{H}$. The subset of $Spec(\mathbb{H})$ consisting of all maximal kernels \emph{Max($\mathbb{H}$)} is called the \emph{maximal spectrum} of $\mathbb{H}$.
\end{defn}

We now introduce the Stone topology defined on $Spec(\mathbb{H})$:
\begin{rem}
The sets $D(A) = \{P \in Spec(\mathbb{H}) \ : \ A \not \subseteq P \}$ with $A$ a kernel of $\mathbb{H}$ are open in the Stone topology. Denote $D(\langle u \rangle)$ for $u \in \mathbb{H}$ by $D(u)$. Then $D(1) = \emptyset$, $D(\mathbb{H})=Spec(\mathbb{H})$, and $D(\prod A_i) = \bigcup D(A_i)$ for any family $\{A_i \}$ of kernels of $\mathbb{H}$. Moreover, $D(A \cap B) = D(A) \cap D(B)$ for any kernels $A,B$ of $\mathbb{H}$. Thus the collection \linebreak $\{ D(u) \ : \ u \in \mathbb{H} \}$ is a basis of a topology on $Spec(\mathbb{H})$ called the \emph{Stone topology}.
\end{rem}

\begin{rem}
$Spec(\mathbb{H})$ is a topological space with respect to the Stone topology and $Max(\mathbb{H})$ is a subspace of $Spec(\mathbb{H})$ (by definition, with respect to the induced topology).
\end{rem}

\begin{cor}
If $A$ is a kernel of a semifield $\mathbb{H}$, then $D(A) = Spec(\mathbb{H})$ implies $A =\mathbb{H}$.
\end{cor}

\begin{thm}\label{thm_semifield_with_gen_property}\cite[Theorem (4.2)]{Prop_Semifields}
The following conditions are equivalent for any semifield $\mathbb{H}$:
\begin{enumerate}
  \item $\mathbb{H}$ is a semifield with a generator.
  \item $Spec(\mathbb{H})$ is compact.
  \item $Max(\mathbb{H})$ is compact, and every proper kernel of $\mathbb{H}$ is contained in some maximal kernel.
\end{enumerate}
\end{thm}

\begin{prop}\cite[Proposition (4.1)]{Prop_Semifields}
Any irreducible kernel in a semifield $\mathbb{H}$ contains a minimal irreducible kernel.
\end{prop}


\begin{prop}\label{prop_maximal_kernels_in_semifield_of_fractions_part1}
Let $\mathbb{H}(x_1,...,x_n)$ be the semifield of fractions where $\mathbb{H}$  is a bipotent semifield. Then
for any $\gamma_1,....,\gamma_n \in \mathbb{H}$ the kernel $\frac{x_1}{\gamma_1},...,\frac{x_n}{\gamma_n} \rangle$ is a maximal kernel of $\mathbb{H}(x_1,...,x_n)$.
\end{prop}
\begin{proof}
Let $\mathbb{H}(x)$ be the semifield of fractions of $\mathbb{H}[x]$ where $\mathbb{H}$  is a bipotent semifield.
We will now show that $\langle x \rangle$ is a maximal kernel of $\mathbb{H}(x)$. Consider the substitution homomorphism $\psi : \mathbb{H}(x) \rightarrow \mathbb{H}$ defined by mapping $x \mapsto 1$. Write $a = \frac{\sum_{i=1}^{m}\alpha_i x^{i}}{\sum_{j=1}^{k}\beta_i x^{j}}$  with $\alpha_i, \beta_j \in \mathbb{H}$ for a general element of $\mathbb{H}(x)$. Then $\psi(a) = \frac{\sum_{i=1}^{m}\alpha_i}{\sum_{j=1}^{k}\beta_j}$ and thus $\psi(a)= 1$ if and only if $\sum_{i=1}^{m}\alpha_i=\sum_{j=1}^{k}\beta_j$ . We will show that if $\psi(a)=1$ then $a \in \langle x \rangle$. Indeed, assume $\psi(a)=1$, denote $\alpha =\sum_{i=1}^{m}\alpha_i, \beta = \sum_{j=1}^{k}\beta_j$ , then by the above  $\alpha = \beta$. We have   $\frac{\sum_{i=1}^{m}\alpha_ix^{i}}{\sum_{j=1}^{k}\beta_j x^{j}} = (\frac{\alpha}{\beta}) \frac{\sum_{i=1}^{m}(\frac{\alpha_i}{\alpha})x^{i}}{\sum_{j=1}^{k}(\frac{\beta_j}{\beta})x^{j}}$. Now, since $x^{i},x^{j} \in \langle x \rangle$, since $\langle x \rangle$ is multiplicative group and since  $\sum_{i=1}^{m}\frac{\alpha_i}{\alpha} = \sum_{j=1}^{k}\frac{\beta_j}{\beta} = 1$, we have that $\sum_{i=1}^{m}\frac{\alpha_i}{\alpha}x^{i}, (\sum_{j=1}^{k}\frac{\beta_j}{\beta}x^{j})^{-1} \in \langle x \rangle$. By assumption $\frac{\alpha}{\beta} = 1$, and thus $a \in \langle x \rangle$ as desired. \\
As $\mathbb{H}$ is a bipotent semifield and thus completely ordered and hence simple, we can use the assertions above and  Corollary \ref{cor_max_ker_simple_corr}, to deduce that $\langle x \rangle$ is a maximal kernel of $\mathbb{H}(x)$.\\

Applying a change of variable $y= \frac{x}{\gamma}$ for each $\gamma \in \mathbb{H}$, the above proof implies that the kernel $\langle \frac{x}{\gamma} \rangle$ is the kernel of the substitution map $\psi_{\gamma} : \mathbb{H}(x) \rightarrow \mathbb{H}$ defined by $ x \mapsto \gamma$. Since $\mathbb{H} \neq \{1\}$ is simple it is generated by any $\gamma \in \mathbb{H}$ so $\psi_{\gamma}$ is onto and consequently, by the same argument used above, $\langle \frac{x}{\gamma} \rangle$ is maximal for any $\gamma \in \mathbb{H}$.\\

Taking a general element $\frac{\sum_{i \in I} \alpha_{i}x^{I(i)}}{\sum_{j \in J} \beta_{j}x^{J(j)}}$ of $\mathbb{H}(x_1,...,x_n)$ where $I, J$ are both finite sets of multi-indices and $x = (x_1,...,x_n)$ and performing the exact same procedure described above, substituting $ \vec{1}$ for $x$, and then $(\frac{x_1}{\gamma_1},...,\frac{x_n}{\gamma_n})$ for $x$. We get that for any choice of $\vec{\gamma}=(\gamma_1,...,\gamma_n) \in \mathbb{H}^n$  the kernel $\langle \frac{x_1}{\gamma_1},...,\frac{x_n}{\gamma_n} \rangle$ is a maximal kernel of $\mathbb{H}(x_1,...,x_n)$ and it corresponds to the substitution map $\mathbb{H}(x_1,...,x_n) \rightarrow \mathbb{H}$, defined by $$(x_1,...,x_n) \mapsto (\gamma_1,...,\gamma_n).$$

The case where $\mathbb{H} = \{1\}$ is trivial since any homomorphism (in particular substitution) has all the domain $\mathbb{H}(x_1,...,x_n)$ as its kernel.\\
\end{proof}

\begin{exmp}
Let $\mathbb{H}$ be an idempotent semifield. For \ $a = (\alpha_1,...,\alpha_n) \in \mathbb{H}^n$  let
$$\phi_a: \mathbb{H}(x_1,...,x_n) \rightarrow \mathbb{H}$$
be the substitution homomorphism defined by $f \mapsto f(a)$. Then we have that \linebreak $Ker(\phi_a) = L_a = \langle \alpha_1 x_1, \dots, \alpha_n x_n \rangle$ . Taking the constant fractions in $\mathbb{H}(x_1,...,x_n)$, we have that $\phi_a$ is onto. By Theorem \ref{thm_kernels_hom_relations} we have that $$\mathbb{H}(x_1,...,x_n) = \phi_a^{-1}(\mathbb{H}) = \mathbb{H}\cdot L_a = \mathbb{H}\cdot \langle \alpha_1 x_1, \dots, \alpha_n x_n \rangle.$$
Intersecting both sides of the last equality with $\langle \mathbb{H} \rangle $, we get
\begin{align*}
\langle \mathbb{H} \rangle = \ & (\mathbb{H} \cdot L_a) \cap \langle \mathbb{H} \rangle =
\mathbb{H} \cdot \langle \alpha_1 x_1, \dots, \alpha_n x_n \rangle \cap \langle \mathbb{H} \rangle
= \mathbb{H} \cdot (\langle \alpha_1 x_1, \dots, \alpha_n x_n \rangle \cap \langle \mathbb{H} \rangle) \\
= \ & \mathbb{H} \cdot \langle  |\alpha_1 x_1| \wedge |\alpha|, \dots, |\alpha_n x_n| \wedge |\alpha| \rangle
\end{align*}

for any $\alpha \in \mathscr{R} \setminus \{1\}$. 
\end{exmp}

\ \\

\newpage\subsection{Distributive semifields}

\ \\

As will be shown in the section concerning idempotent semifields and lattice-ordered groups, the lattice of kernels every of an idempotent semifield is distributive. Idempotent semifields play an important role in our theory.
Here we introduce some of the properties of semifields which have a distributive lattice of kernels.

\ \\

\begin{defn}
A semifield $\mathbb{H}$ is called \emph{distributive} if the lattice $\Con(\mathbb{H})$ is distributive.
\end{defn}

\begin{prop}\label{prop_dist_semifield1}\cite[Proposition (4.2)]{Prop_Semifields}
If $\mathbb{H}$ is a distributive semifield, then the following statements hold:
\begin{enumerate}
  \item All weakly irreducible kernels of the semifield $\mathbb{H}$ are irreducible.
  \item All proper kernels of $\mathbb{H}$ are intersections of its irreducible kernels.
  \item $D(A) \subseteq D(B) \Leftrightarrow A \subseteq B$ for any pair of kernels $A,B$ of \ $\mathbb{H}$.
  \item $D(A) = D(B) \Leftrightarrow A = B$ for any pair of kernels $A,B$ of \ $\mathbb{H}$.
\end{enumerate}
\end{prop}

As every reducible kernel is weakly irreducible, the first assertion of Proposition~\ref{prop_dist_semifield1} states that
\begin{note}
In a distributive semifield $\mathbb{H}$, a kernel $K$ of $\mathbb{H}$ is irreducible if and only if
\begin{equation}
A \cap B = K \Rightarrow A = K \ \text{or} \ B = K
\end{equation}
for any pair of kernels $A,B$ of $\mathbb{H}$.
\end{note}

The following is a nice example of utilizing of irreducible kernels.
\begin{rem}\label{rem_prod_principal_kernels}
Let $\mathbb{H}$ be a distributive semifield. Then, for any $a,b \in \mathbb{H}$ such that  \linebreak $\langle a \rangle \cap \langle b \rangle = \{ 1 \}$,
\begin{equation}
\langle a \rangle \langle b \rangle = \langle ab \rangle.
\end{equation}
\end{rem}
\begin{proof}
As $ab \in \langle a \rangle \langle b \rangle$, obviously $\langle ab \rangle \subseteq \langle a \rangle \langle b \rangle$. Since $\langle a \rangle \cap \langle b \rangle = \{ 1 \}$, we have that $a \in P$ or $b \in P$ for any $P \in Spec(\mathbb{H})$ (by definition of irreducibility of a kernel and the fact that $\{1\} \subseteq P$). If $ab \in P$, then if $a \in P$ we get $b = a^{-1}(ab) \in P$ and if $b \in P$ we get $a = (ab)b^{-1} \in P$. In any case both $a$ and $b$ are in $P$ and thus we have shown that \linebreak $\langle ab \rangle \subseteq P $ implies $\langle a \rangle \langle b \rangle \subseteq P$. Now, as $\mathbb{H}$ is distributive, we have by Proposition~\ref{prop_dist_semifield1}~(1) that $\langle a \rangle \langle b \rangle = \langle ab \rangle$ as they are both intersections of the irreducible kernels containing them.
\end{proof}

Kernels having  trivial intersection ( $\{ 1 \}$ ) play a very important role in our theory.

\begin{rem}\cite{Prop_Semifields}
Let $\mathbb{H}$ be a distributive semifield. Then any proper kernel of $\mathbb{H}$ is the intersection of all irreducible kernels containing it. In particular, the intersection of all the irreducible kernels of $\mathbb{H}$ equals $\{ 1 \}$.
\end{rem}


\ \\
\subsection{Idempotent semifields : Part 1}

\ \\

In the following subsection, we concentrate our attention on the theory of \linebreak idempotent semifields. We will continue the study of idempotent in subsequent \linebreak sections after introducing the theory of lattice-ordered groups.\\

It turns out that this special kind of semifield has some very nice additional \linebreak properties to those of general semifields that make it quite easy to work with. \\

As we have already noted, kernels of an idempotent semifield are themselves \linebreak semifields. We now introduce a few
more interesting properties of such semifields.\\


\begin{rem}\label{rem_finitely_generated_are_principal}
Finitely generated kernels of an idempotent semifield are principal.
\end{rem}
\begin{proof}
This follows directly from Theorem \ref{thm_semi_with_a_gen}, as a kernel of an idempotent \linebreak semifield is itself a semifield.
\end{proof}

\begin{rem}\label{rem_a_generator_comparable to 1}
Since every finitely generated kernel of an idempotent semifield is itself a semifield with a generator,  by Corollary \ref{cor_positive_generator_of_a_semifield}, we have that for any such kernel we can choose a generator $a$ such that $|a|= a \geq 1$. This issue is important, as for any such an element $1 \dotplus a = \sup(1,a) = \max(1,a) = a$ and $1 \wedge a = \inf(1,a) = \min(1,a) = 1$, which imply that $1 \dotplus (a^{-1}) =  \sup(1,a^{-1}) = \max(1,a^{-1}) = 1$ and that $1 \wedge (a^{-1}) = \inf(1,a^{-1}) = \min(1,a^{-1}) = a^{-1}$. By convention, we always take a generator of a kernel of an idempotent semifield to be such a `positive' generator unless stated otherwise. The importance of such generator is that it is \emph{comparable} to $1$.
\end{rem}

\begin{rem}\label{rem_sum_prod_absolute_values}
For any kernel $K$ of an idempotent semifield the following holds:
\begin{equation}\label{eq_sum_prod_absolute_values}
|g||h| \in K \ \ \Leftrightarrow \ \ |g| \dotplus |h| \in K.
\end{equation}
\end{rem}
\begin{proof}
Indeed, Since $|g|, |h| \geq 1$, we have that $|g| \leq |g||h|$ and $|h| \leq |g||h|$ thus
$$|g| \dotplus |h| = \sup(|g|,|h|) \leq |g||h|.$$
On the other hand, we have that
$$(|g| \dotplus |h|)^2 = |g|^2 \dotplus |g||h| \dotplus |h|^2 \geq |g||h|.$$
So, by Remark \ref{rem_kernel_by_abs_value} we have that $\langle |g||h| \rangle = \langle |g| + |h| \rangle$ and equality \eqref{eq_sum_prod_absolute_values} follows.
\end{proof}

\begin{rem}\label{rem_sum_prod_absolute_values2}
For any kernel $K$ of an idempotent semifield the following holds:
\begin{equation*}
 g,h \in K \ \ \Leftrightarrow \ \ |g| \dotplus |h| \in K;
\end{equation*}
\begin{equation*}
 g,h \in K \ \ \Leftrightarrow \ \ |g||h| \in K.
\end{equation*}
\end{rem}
\begin{proof}
By Remark \ref{rem_sum_prod_absolute_values} we only have to prove the first equality. If $g,h \in K$ then $|g|,|h| \in K$ thus $|g| \dotplus |h| \in K$. On the other hand, since $|g|,|h| \leq |g| \dotplus |h|$, \linebreak by  Remark~\ref{rem_kernel_by_abs_value}, we have that $|g| \dotplus |h| \in K$ implies $|g|, |h| \in K$, which in turn yields that $g,h \in K$ proving our claim.
\end{proof}

\begin{prop}\label{prop_algebra_of_generators_of_kernels}\cite[Theorem 2.2.4(d)]{OrderedGroups3}
Let $\mathbb{S}$ be an idempotent semifield. For $X,Y \subset \mathbb{S}$, denote by $\langle X \rangle$ and $\langle Y \rangle$ the kernels generated by $X$ and $Y$, respectively. Then for any $X,Y \subset \mathbb{S}$,  $K \in \Con(\mathbb{S})$ and $a,b \in \mathbb{S}$ the following statements hold:
\begin{enumerate}
  \item $\langle X \rangle \cdot \langle Y \rangle = \langle X \cup Y \rangle = \langle \{|x| \dotplus |y| : x \in X, y \in Y \} \rangle$.
  \item $\langle X \rangle \cap \langle Y \rangle = \langle \{|x| \wedge |y| : x \in X, y \in Y \} \rangle$.
  \item $\langle K, a \rangle \cap \langle K ,b \rangle  = K \cdot \langle |a| \wedge |b| \rangle$, where $\langle K, a \rangle$ denotes the kernel generated by the set $K \cup \{a\}$.
\end{enumerate}
\end{prop}

\begin{cor}\label{cor_max_semifield_principal_kernels_operations}
Let $\mathbb{S}$ be an idempotent semifield. Then the intersection and product of two principal kernels are principal kernels. Namely,
for every $f,g \in \mathbb{S}$
\begin{equation}\label{eq_max_principal_operations}
\langle f \rangle \cap \langle g \rangle = \langle (f \dotplus f^{-1}) \wedge (g \dotplus g^{-1}) \rangle \ ; \
\langle f \rangle \cdot \langle g \rangle = \langle (f \dotplus f^{-1})(g \dotplus g^{-1}) \rangle.
\end{equation}
\end{cor}
\begin{proof}
Taking $X = \{f\}$ and $Y = \{g\}$ in Proposition \ref{prop_algebra_of_generators_of_kernels} yield the equalities,
where for the second equality, applying Remark \ref{rem_sum_prod_absolute_values}  yields that
$$\langle f \rangle \cdot \langle g \rangle = \langle (f \dotplus f^{-1}) \dotplus (g \dotplus g^{-1}) \rangle,$$
 and from there we apply the proposition.
\end{proof}

A direct consequence of Corollary \ref{cor_max_semifield_principal_kernels_operations} is
\begin{cor}\label{cor_principal_kernels_sublattice}
The set of principal kernels of an idempotent semifield forms a \linebreak sublattice of the lattice of kernels (i.e., a lattice with respect to intersection and \linebreak multiplication)
\end{cor}

\begin{defn}\label{defn_principal_kernels_in_semifield_of_fractions}
Denote the collection of principal kernels of an idempotent \linebreak semifield $\mathbb{S}$ by $\PCon(\mathbb{S})$.
In particular we denote the collection of principal kernels of $\mathbb{H}(x_1,...,x_n)$ by
$\PCon(\mathbb{H}(x_1,...,x_n))$ (where $\mathbb{H}$ is idempotent).
\end{defn}

\begin{rem}\label{rem_principal_kernels_algebra}
Let $\mathbb{S}$ be an idempotent semifield. Then
the following equalities hold for $\langle f \rangle, \langle g \rangle, \langle h \rangle \in \PCon(\mathbb{S})$:
\begin{equation*}
(\langle f \rangle \cdot \langle h \rangle) \cap (\langle g \rangle \cdot \langle h \rangle)  = (\langle f \rangle \cap \langle g \rangle) \cdot \langle h \rangle;
\end{equation*}
\begin{equation*}
\langle (|f| \dotplus |h|) \wedge (|g| \dotplus |h|) \rangle  = \langle (|f| \wedge |g|)  \dotplus  |h| \rangle;
\end{equation*}
\begin{equation*}
(\langle f \rangle \cap \langle h \rangle) \cdot (\langle g \rangle \cap \langle h \rangle)= (\langle f \rangle \cdot \langle g \rangle)  \cap \langle h \rangle;
\end{equation*}
\begin{equation*}
\langle (|f| \wedge |h|) \dotplus (|g| \wedge |h|) \rangle= \langle (|f| \dotplus |g|) \wedge |h| \rangle.
\end{equation*}
\end{rem}
\begin{proof}
Follows directly from Lemma \ref{lem_kernels_algebra} and Corollary \ref{cor_max_semifield_principal_kernels_operations} above. The proof can also be found in \cite{OrderedGroups3}
\end{proof}

\begin{rem}\label{rem_iso_thms_for_PCon}
As principal kernels form a sublattice of $\PCon(\mathbb{S})$ for any \linebreak idempotent semifield $\mathbb{S}$ and since for a semifield homomorphism $\phi$ whose kernel is a principal kernel, both homomorphic images and preimages (which are then a product of principal kernels by Theorem \ref{thm_kernels_hom_relations}) with respect $\phi$  are principal kernels, we have that all three isomorphism theorems apply to the restriction of $\Con(\mathbb{S})$ to $\PCon(\mathbb{S})$.
\end{rem}

\begin{cor}\label{cor_nother_4_for_principal_kernels}
In the setting of Theorem \ref{cor_qoutient_corr}, let $L = \langle a \rangle \in \PCon(\mathbb{H})$ and let $\phi_L : \mathbb{H} \rightarrow \mathbb{H}/L$ be the quotient epimorphism. Then by Remark \ref{rem_image_of_principal_kernel} the image of a principal kernel of $\mathbb{H}$ under $\phi_L$ is a principal kernel of $\mathbb{H}/L$  and the preimage of a principal kernel $\langle b \rangle/L \in \PCon(\mathbb{H}/L)$ is $\langle b \rangle \cdot L = \langle b \rangle \cdot \langle a \rangle = \langle |b|+|a| \rangle$  which is a principal kernel. Thus we have that the correspondence of Theorem \ref{cor_qoutient_corr} applies to the principal kernels in $\PCon(\mathbb{H})$ and the principal kernels in $\PCon(\mathbb{H}/L)$ containing $L$, namely, there is a $1:1$ correspondence
$$ \{ \text{Principal kernels of} \ \mathbb{H}/L \} \rightarrow  \{ \text{Principal kernels of} \ \mathbb{H} \ \text{containing} \ L \}$$
given by $\langle b \rangle / \langle a \rangle \mapsto \langle b \rangle$.
\end{cor}

\begin{prop}\label{prop_frac_semifield_finitely_gen}
Let $\mathbb{H}$ be a bipotent archimedean semifield.
Let $\mathbb{H}(x_1,...,x_n)$ be the semifield of fractions of $\mathbb{H}[x_1,...,x_n]$. Then $\mathbb{H}(x_1,...,x_n)$ is finitely generated by $\{x_1,...,x_n\}$ as a semifield over $\mathbb{H}$.
$\mathbb{H}(x_1,...,x_n)$ is also finitely generated by $\{e ,x_1,...,x_n\}$ as a kernel
for any $e \in \mathbb{H}$ (by Remark \ref{rem_order_simple} in case $\mathbb{H}$ is trivial we can omit $e$, otherwise we can choose $e > 1$ such that $ \mathbb{H} \subseteq \langle e \rangle$). In fact, $\mathbb{H}(x_1,...,x_n) = \langle e \rangle \cdot \prod_{i=1}^{n} \langle x_i \rangle$ and by Theorem \ref{thm_semi_with_a_gen} we have that  $\mathbb{H}(x_1,...,x_n)$ is a semifield with a generator $\sum_{i=1}^n|x_{i}|+ |e|$.
\end{prop}
\begin{proof}
First note that as $\langle e \rangle \cdot \prod_{i=1}^{n} \langle x_i \rangle$ is closed under multiplication and addition (since it is a semifield), by Remark \ref{rem_kernel_by_abs_value} it is enough to prove that for any monomial $f$ in $\mathbb{H}[x_1,...,x_n]$ there exists some $k \in \mathbb{N}$ such that
$$|f| \leq (\sum_{i=1}^n|x_{i}| \dotplus |e|)^k  = \sum_{i=1}^n|x_{i}|^k \dotplus |e|^k.$$
Let $f(x_1,...,x_n) = \alpha x_{1}^{p_i}\cdot \dots \cdot x_{n}^{p_n}$ where $\alpha \in \mathbb{H}$. Since $\mathbb{H}$ is a semifield with a generator, and $e$ is a generator of $\mathbb{H}$, we have that $\alpha \in \langle e \rangle$. Thus there exists some $s$ such that $|\alpha| \leq |e|^{s}$. Now, as $|x_1^{p_1} \cdot \dots \cdot x_n^{p_n}| \leq |x_1 \dotplus \dots  \dotplus x_n|^{|p_1| \dotplus \dots \dotplus |p_n|} = |x_1|^{|p_1| + \dots + |p_n|} \dotplus \dots \dotplus |x_n|^{|p_1| + \dots + |p_n|}$. For $k = \max( s, |p_1| + \dots + |p_n|)$, we get $|f| \leq \sum_{i=1}^n|x_{i}|^k \dotplus |e|^k$.
\end{proof}

\ \\

\begin{exmp}\label{exmp_frac_semifield_construction}
Let $\mathbb{H}(x)$ be the semifield of fractions in one variable of $\mathbb{H}[x]$, where $\mathbb{H}$ is a bipotent semifield. Consider the kernel, $\langle x \rangle$, generated by $x$ in  $\mathbb{H}(x)$. Since $\{ x^{k} \ : \ k \in \mathbb{Z} \}$ is the (normal) group generated by $x$,  by Proposition \ref{prop_ker_stracture_by group}, $\langle x \rangle$ is
\begin{equation}
 \left\{ \sum_{i=1}^{m}f_{i}(x)x^{I(i)} \ : \ m \in \mathbb{N},  \ f_i \in \mathbb{H}(x) \ \text{such that} \ \sum_{i=1}^{m}f_i = 1 \right\}
\end{equation}
where $ I \subset \mathbb{Z}$ containing $n$ elements.\\
First, note that since $\mathbb{H}$ is idempotent ($1+1 =1$), $\langle x \rangle$ is closed with respect to addition and thus a semifield.
Now, $\sum_{i=1}^{n}f_i(x) = 1$ yields by the natural order that $f_i \leq 1$ for every $i =1,...,n$. Write $f_i(x) = \frac{h_i(x)}{g_i(x)}$, and let $g(x) = \prod_{i=1}^{n}g_i(x)$ and $s_i(x) = \prod_{j \neq i}g_j(x)$, then $\frac{\sum_{i=1}^{n}h_i(x)s_i(x)}{g(x)} = 1$ so $g(x)=\sum_{i=1}^{n}h_i(x)s_i(x)$ and $f_i(x) = \frac{h_i(x)s_i(x)}{g(x)} = \frac{h_i(x)s_i(x)}{\sum_{i=1}^{n}h_i(x)s_i(x)}$. Thus we can assume $f_i(x)$ is of the form $\frac{r_i(x)}{\sum_{i=1}^{n}r_i(x)}$ with $r_i(x) \in \mathbb{H}[x]$.\\
For simplicity of notation, denote $x_0 = e$, where as above $e$ is the generator of $\mathbb{H}$ as a kernel over itself.\\
The exact same construction yields that the kernel $\langle x_k \rangle$ in $\mathbb{H}(x_1, ... , x_n)$ with $k = 0,1,...,n$ is of the form
\begin{equation}\label{eq_kernel_of_field_of_fructions}
 \left\{ \frac{\sum_{i=1}^{m} x_{k}^{I(i)}\cdot f_{i}(x_1,...,x_n)}{\sum_{i=1}^{m}f_{i}(x_1,...,x_n)} \ : \ m \in \mathbb{N},  \ f_i \in \mathbb{H}[x_1,...,x_n]  \right\}
\end{equation}
where $ I \subset \mathbb{Z}$ containing $n$ elements.
Notice that in the special case where $e=1$ the kernel in \eqref{eq_kernel_of_field_of_fructions} degenerates to $\{ 1 \}$. Further, \eqref{eq_kernel_of_field_of_fructions} applies to any element $f(x_1,...,x_n)$ of $\mathbb{H}(x_1,...,x_n)$ taken in place of $x_k$.\\
Now, for each $x_i$, \ $i = 0,1,....,n$  since $\langle x_i \rangle$ are principal kernels and since none of the elements in $\{x_0=e,x_1,...,x_n\}$ are comparable to each other, Proposition \ref{prop_principal_ker} yields
that $x_i \not \in \langle x_j \rangle$ for any $j \neq i$. Thus, none of these kernels is contained in another.
\end{exmp}

\begin{exmp}
Consider the semifield of fractions $\mathscr{R}(x,y)$ and the substitution \linebreak homomorphism $\phi : \mathscr{R}(x,y) \rightarrow \mathscr{R}(y)$ defined by $x \mapsto 1$. Then $\phi(\mathscr{R}(x,y)) = \mathscr{R}(y)$.

1. Since  $\phi(y) = y$, we have that
   $$\phi(\langle y \rangle_{\mathscr{R}(x,y)} ) = \langle \phi(y) \rangle_{\mathscr{R}(y)}  = \langle y \rangle_{\mathscr{R}(y)} = \{ f \in \mathscr{R}(y) \ : \ \exists n \in \mathbb{N} \  \text{such that} \  |f| \leq |y|^n \}$$
   where $\langle y \rangle_{\mathscr{R}(x,y)} = \{ f \in \mathscr{R}(x,y) \ : \ \exists n \in \mathbb{N} \  \text{such that} \  |f| \leq |y|^n \}$.\\
2. Since $\phi(|x| \wedge |y|) = |\phi(x)| \wedge |\phi(y)| = 1 \wedge |y| = 1$, we have that
   $$\phi(\langle |x| \wedge |y| \rangle_{\mathscr{R}(x,y)} ) = \langle \phi(|x| \wedge |y|) \rangle_{\mathscr{R}(y)} = \langle 1 \rangle_{\mathscr{R}(y)};$$
   This is expected as $|x| \wedge |y| \in \langle x \rangle = Ker\phi$. \\
3. Since $\phi(|x| \dotplus |y|) = |\phi(x)| \dotplus |\phi(y)| = 1 \dotplus |y| = |y|$, we have that
   $$\phi(\langle |x| \dotplus |y| \rangle_{\mathscr{R}(x,y)} ) = \langle \phi(|x| \dotplus |y|) \rangle_{\mathscr{R}(y)} = \langle |y| \rangle_{\mathscr{R}(y)} = \langle y \rangle_{\mathscr{R}(y)}.$$
   Thus we have that $\phi(\langle y \rangle_{\mathscr{R}(x,y)} ) = \phi(\langle |x| \dotplus |y| \rangle_{\mathscr{R}(x,y)} )$.\\
   By Theorem \ref{thm_kernels_hom_relations} \ $\phi^{-1}(\langle y \rangle_{\mathscr{R}(y)}) = \langle x \rangle_{\mathscr{R}(x,y)} \cdot \langle y \rangle_{\mathscr{R}(x,y)} = \langle |x| \dotplus |y| \rangle_{\mathscr{R}(x,y)}$, thus \linebreak $\langle |x| \dotplus |y| \rangle$ is a preimage of $\phi$ while $\langle y \rangle$ is not.

\end{exmp}

\begin{prop}\label{prop_induced_stone_topology}
Since $\PCon(\mathscr{R}(x_1,...,x_n))$ forms a sublattice of the lattice of  kernels in  $\mathscr{R}(x_1,...,x_n)$ with respect to multiplications and intersections, and since the kernels $1 = \langle 1 \rangle$ and  $\mathscr{R}(x_1,...,x_n)$ are principal (since $\mathscr{R}(x_1,...,x_n)$ is a semifield with a generator), the Stone topology induces a topology on \linebreak  $\PCon(\mathscr{R}(x_1,...,x_n))$.\\
The collection \emph{PSpec($\mathscr{R}(x_1,...,x_n)$)} of all irreducible principal kernels of the semifield $\mathscr{R}(x_1,...,x_n)$ is called the \emph{irreducible principal spectrum} of $\mathscr{R}(x_1,...,x_n)$. It is a topological space with respect to the principal Stone topology. Its subspace \linebreak \emph{PMax($\mathscr{R}(x_1,...,x_n)$)}, consisting of all maximal principal kernels, is called the \emph{maximal principal spectrum}.\\
As noted above, $1,\mathscr{R}(x_1,...,x_n) \in PSpec(\mathscr{R}(x_1,...,x_n))$. Thus we have that $D(1)~=~\emptyset$, $D(\mathscr{R}(x_1,...,x_n))=PSpec(\mathscr{R}(x_1,...,x_n))$ are in the topology.
The sets $$D(\langle f \rangle)~=~\{P~\in~Spec(\PCon(\mathscr{R}(x_1,...,x_n))) \ : \ \langle f \rangle \not \subseteq P \}$$ are open in the induced topology.
\end{prop}

\begin{defn}\label{defn_principal_stone_topology}
We call the induced topology on $\PCon(\mathscr{R}(x_1,...,x_n))$ introduced in Proposition \ref{prop_induced_stone_topology}, the \emph{principal Stone topology}.
\end{defn}

\begin{rem}
The Stone topology on $\Con(\mathscr{R}(x_1,...,x_n))$ gives rise to the induced topology on the kernels of the semifield with a generator $\langle \mathscr{R} \rangle$,  the kernels of which, $\Con(\langle \mathscr{R} \rangle)$, form a lattice of kernels which embeds  as a sublattice of  $\Con(\mathscr{R}(x_1,...,x_n))$.
\end{rem}

\begin{rem}
In the next section, we show that
$$\PCon(\langle \mathscr{R} \rangle) = \{ \langle f \rangle \cap \langle \mathscr{R} \rangle : f \in \PCon(\mathscr{R}(x_1,...,x_n)) \} \supset \PCon(\mathscr{R}(x_1,...,x_n))$$ which forms a sublattice of $\PCon(\mathscr{R}(x_1,...,x_n)$. Thus the principal Stone topology induces a topology on $\PCon(\langle \mathscr{R} \rangle)$.
\end{rem}

\ \\

\subsection{Affine extensions of idempotent archimedean semifields}

\ \\

In the following, we characterize the  affine idempotent semifield extensions of an idempotent archimedean semifield $\mathbb{H}$ which are divisible over $\mathbb{H}$ (divisible extensions are the analogue for algebraic extensions in ring theory).

\begin{defn}
For an idempotent archimedean semifield $\mathbb{H}$ satisfying the Frobenius property, we denote by $\overline{\mathbb{H}}$ the divisible closure of $\mathbb{H}$, i.e., the smallest divisible semifield containing $\mathbb{H}$.
\end{defn}

\begin{note}
In Section \ref{Section:lattice_ordered_groups}, concerning lattice ordered groups, we revisit the notion of divisible closure and show that the semifield $\overline{\mathbb{H}}$ exists and is also idempotent and archimedean.
\end{note}

%

\begin{lem}
If $\mathbb{H}$ is an idempotent semifield satisfying the Frobenius property, then $\overline{\mathbb{H}}$ also satisfies the Frobenius property.
\end{lem}
\begin{proof}
For $a,b \in \overline{\mathbb{H}}$ there exist some (minimal) $k,m \in \mathbb{N}$ such that $a^k = \alpha, b^m = \beta$ are elements of $\mathbb{H}$. First note that for every $\alpha, \beta \in \mathbb{H}$ and $t \in \mathbb{N}$,  $\alpha^s\beta^{t-s} \leq \alpha^{t} +\beta^{t}$ for any $0 \leq s \leq t$ since $(\alpha + \beta)^t = \alpha^t + \beta^{t}$. Now $(a^{s} b^{t-s})^{km} = \alpha^{sm} \beta^{k(t-s)} \leq (\alpha^{m} + \beta^{k})^t$ by the former observation. Since both $\alpha^m \in \mathbb{H}$ and $\beta^k \in \mathbb{H}$ we have that the Frobenius property holds, so that $(\alpha^{m} + \beta^{k})^t = \alpha^{mt} + \beta^{kt} = a^{kmt} + b^{kmt}$. As $a^{kmt} + b^{kmt} \leq (a^{t} + b^{t})^{km}$, we conclude that $(a^{s} b^{t-s})^{km} \leq (a^{t} + b^{t})^{km}$ which implies that $a^{s} b^{t-s} \leq a^{t} + b^{t}$ for any $0 \leq s \leq t$ (since $\overline{\mathbb{H}}$ is archimedean) and so we get that $(a+b)^t = \sum_{s=0}^{t}a^{s} b^{t-s} = a^{t} + b^{t}$, as desired.
\end{proof}

\begin{rem}\label{rem_divisibe_extensions}
Let $\mathbb{H}$ be an idempotent semifield satisfying the Frobenius property.
Let $\mathbb{K}=\mathbb{H}(a_1,...,a_n)$ be an affine semifield extension of $\mathbb{H}$ with $a_1,...,a_n \in \overline{\mathbb{H}}$.
The substitution map $\phi : \mathbb{H}(x_1,...,x_n) \rightarrow \mathbb{K}$  sending $x_i \mapsto a_i$  is an epimorphism. For each $a_i$ let $\alpha_i$ be such that $\alpha_i = a_i^{k(i)}$ with $k(i) \in \mathbb{N}$ minimal such that $a_i^{k(i)} \in \mathbb{H}$.
Consider the kernel $$K = \langle \alpha_{1}^{-1} x_{1}^{k(1)}, ... ,\alpha_{n}^{-1} x_{n}^{k(n)} \rangle  \in \PCon(\mathbb{H}(x_1,...,x_n)).$$ $K$ is a principal kernel. We claim that $Ker\phi = K$.\\
First, as all generators of $K$ are mapped to $1$ we have that $K \subseteq Ker\phi$.
Let $\overline{\mathbb{H}}$ be the divisible closure of $\mathbb{H}$. If $\mathbb{H} = \overline{\mathbb{H}}$ (i.e., $\mathbb{H}$ is divisible) then for every $1 \leq i \leq n$ we have that $\alpha_i = a_i$. Thus by Proposition \ref{prop_maximal_kernels_in_semifield_of_fractions_part1} we have that $K$ is a maximal kernel and $K = ker\phi$. Let us denote this last kernel of $\overline{\mathbb{H}}(x_1,...,x_n)$ by $\overline{K}$. Now, for each $1 \leq i \leq n$ we have that  $\alpha_{i}^{-1} x_{i}^{k(i)} = a_i^{-k(i)}x_{i}^{k(i)} = (a_i^{-1}x_{i})^{k(i)}$ thus $\alpha_{i}^{-1}x_{i}^{k(i)}$ is a generator of the kernel $\langle a_i^{-1} x_{i} \rangle$. Consider the restriction $$\tilde{\phi} = \phi|_{\mathbb{H}(x_1,...,x_n)} : \mathbb{H}(x_1,...,x_n) \rightarrow \mathbb{H}(a_1,...,a_n)$$ of $\phi : \overline{\mathbb{H}}(x_1,...,x_n) \rightarrow \overline{\mathbb{H}}(a_1,...,a_n)$ to $\mathbb{H}(x_1,...,x_n)$. Then $\tilde{\phi}$ is epimorphism and by Theorem \ref{thm_nother_1_and_3} (1) taking $\alpha_{i}^{-1}x_{i}^{k(i)}$ as a generator of $\overline{K}$ we have that \linebreak $Ker\tilde{\phi} = \overline{K} \cap \mathbb{H}(x_1,...,x_n) = K$. By the first isomorphism theorem we have that
$$\mathbb{H}(x_1,...,x_n)/K \cong \mathbb{H}(a_1,...,a_n),$$
proving our claim.
\end{rem}

\newpage

\section{Lattice-ordered groups, idempotent semifields and the semifield of fractions $\mathbb{H}(x_1,...,x_n)$.}\label{Section:lattice_ordered_groups}

\ \\

An affine semifield over an idempotent semifield $\mathbb{H}$, in particular, the semifield of fractions $\mathbb{H}(x_1,...,x_n)$, is an idempotent semifield. It turns out that $\mathbb{H}(x_1,...,x_n)$ and generally every idempotent semifield can be considered as a special kind of group, called a lattice-ordered group, or $\ell$-group, in which the notion of kernels coincides with the notion of normal and convex $\ell$-subgroups.\\
Although we consider $\mathbb{H}(x_1,...,x_n)$, all results given below hold for any idempotent semifield, in particular for the kernels of  $\mathbb{H}(x_1,...,x_n)$ each of which is an idempotent semifield in its own right. A particularly important such  kernel is $\langle \mathbb{H} \rangle$ which denotes the kernel of $\mathbb{H}(x_1,...,x_n)$ generated by any $\alpha \in \mathbb{H} \setminus \{1 \}$.\\

Due to the central role idempotent semifields play in our theory, we hereby introduce some basic notions and some important results in the theory of lattice-ordered groups.
All relevant definitions can be found in \cite{OrderedGroups3}. We note that in \cite{OrderedGroups3} the group operation is taken to be addition while in our context we take it to be multiplication, as we use $\dotplus$ for $\sup$ , i.e., $\vee$. The order preserving map $x \mapsto a^x$ with $a$ a symbol is used to translate the statements presented there to our language where the group is taken to be multiplicative. In particular $(\mathbb{R},+)$ is translated to $(\mathbb{R}^{+},\cdot)$.

\newpage

\subsection{Lattice-ordered groups}

\ \\


\begin{defn}
A \emph{partially-ordered group} (p.o group) is a group $G$ endowed with a partial order such that the group operation preserves the order on $G$, that is,
$$a \leq  b \Rightarrow \forall g \in G \ (ag \leq bg \ \ \text{and} \ \ ga \leq gb).$$
$$a \leq  b \Rightarrow \forall g \in G \ ( g (a \wedge b) = ga \wedge gb \ \text{and} \  (a \wedge b) g = ag \wedge bg).$$
If the partial order on $G$ is directed, then $G$ is a \emph{directed group}.
If the partial order on $G$ is a lattice order, then $G$ is a \emph{lattice-ordered group} or
$\ell$-group. If the order on $G$ is a linear order, then $G$ is called a \emph{totally-ordered
group} or o-group.
\end{defn}

\begin{defn}
An $\ell$-subgroup of a po-group is a subgroup which is also a sublattice.
\end{defn}

\begin{rem}\label{rem_sufficient_condition_for_l-subgroup}
Since $x \vee y = (xy^{-1} \vee 1)y$ and $x \wedge y = (x^{-1} \vee y^{-1})^{-1}$, $L$ is an $\ell$-subgroup of a po-group $G$ exactly when $z \in L$ implies that $z \vee 1 \in L$ (whenever $z \vee 1$ exists in $G$).
\end{rem}

\begin{defn}
If $G$ and $H$ are $\ell$-groups, then a group homomorphism \\
$\phi : G \rightarrow H$ which is also a lattice homomorphism (preserves $\wedge$ and $\vee$ ($\dotplus$ in our notation)) is called an \emph{$\ell$-homomorphism}.
\end{defn}

\begin{defn}
A subset $S$ of a poset $P$ is said to be \emph{convex} if $a \leq p \leq b$ with $a,b \in S$ implies that $p \in S$.
\end{defn}

\begin{rem}\cite[Theorem (2.2.3)]{OrderedGroups3}
All fundamental group isomorphism theorems hold for normal convex $\ell$-subgroups of an $\ell$-group.
\end{rem}

\begin{thm}\label{thm_coresspodence_of_l_groups_lattices}\cite[Theorem (2.2.3)]{OrderedGroups1}
Let $N$ be  a normal convex $\ell$-subgroup of an \linebreak $\ell$-group $G$.
The mapping $A \mapsto A/N$ is a lattice isomorphism between the lattice of
convex $\ell$-subgroups of $G$ that contain $N$ and the lattice of convex $\ell$-subgroups of $G/N$.
\end{thm}

\begin{rem}\label{rem_transitivity_for_l_groups}\cite[Proposition (4.3)]{OrderedGroups1}
Let $G$ be an $\ell$-group. Let $K$ be a normal convex $\ell$-subgroup of $G$ and $L$ a normal convex $\ell$-subgroup of $K$. Then $L$ is a normal convex $\ell$-subgroup of $G$ if and only if $L$ is a normal subgroup of $G$.
\end{rem}

\ \\

\begin{prop}\label{prop_convex_l_subgroups_distributive_lattice}\cite[Theorem (2.2.5)]{OrderedGroups3}
Let $G$ be an $\ell$-group.
The lattice of convex $\ell$-subgroups of $G$ is a complete distributive sublattice of the lattice of subgroups of $G$,
and, in fact, it satisfies the infinite distributive law
$$A \cap \left( \bigvee_{i} B_i \right) = \bigvee_{i}(A \cap B_i).$$
\end{prop}

\begin{defn}\label{defn_subdirect_prod}
An $\ell$-group $G$ is a \emph{subdirect product} of the family $\{ G_i : i \in I \}$ of
$\ell$-groups if there is a monomorphism $\phi : G \rightarrow \prod G_i$ such that each composite $\pi_i \circ \phi$ is
an epimorphism, and $G$ is \emph{subdirectly irreducible} if, in any such representation of $G$
there is an index $i$ such that $\pi_i \circ \phi$ is an isomorphism.
We indicate that  $G$ is a \emph{subdirect product} of the family $\{ G_i : i \in I \}$ by writing $G \xrightarrow[s.d]{} \prod G_i$.
\end{defn}

\begin{note}
We use the same notation presented in Definition \ref{defn_subdirect_prod} in the context of idempotent semifields.
\end{note}

\begin{rem}\label{rem_sub_direct_prod_construction}
A function $\phi : G \rightarrow \prod G_i$ on an $\ell$-group $G$ is uniquely determined by the family $ \{ \pi_i \circ \phi : i \in I\}$, and is an
$\ell$-homomorphism if and only if each $\pi_i \circ \phi$ is an $\ell$-homomorphism. In this case $Ker\phi =
\bigcap_{i} Ker (\pi_i \circ \phi)$ and $G/Ker(\pi_i \circ \phi) \cong \Im(\pi_i \circ \phi)$. Consequently, each family $\{ N_i : i \in I \}$ of normal convex $\ell$-subgroups of the $\ell$-group $G$ determines a homomorphism
$G \rightarrow \prod G/N_i$ with kernel $N = \bigcap_{i}N_i$, and $G/N$ is a subdirect product of the family
$\{ N_i : i \in I \}$ and all subdirect product representations of $G/N$ essentially arise in
this way. Clearly, a nonzero $\ell$-group $G$ is subdirectly irreducible if and only if it has
a smallest nontrivial (i.e., $\neq \{1\}$) normal convex $\ell$-subgroup.
\end{rem}

\begin{thm}\cite{OrderedGroups3}
Each $\ell$-group is a subdirect product of a family of subdirectly irreducible
$\ell$-groups.
\end{thm}
\begin{proof}
Let $G$ be an $\ell$-group. For $a \in G \setminus \{1 \}$ let $N_a$ be a normal convex $\ell$-group of
$G$ which is maximal with respect to excluding $a$. The existence of $N_a$ is given by
Zorn's Lemma. Since each normal convex $\ell$-subgroup of $G$ that properly contains
$N_a$ must contain $a$, $G/N_a$ is subdirectly irreducible (it has a smallest nontrivial (contains the coset $aN_a$) normal convex $\ell$-subgroup. But $\bigcap_{a \neq 1}N_a = \{1\}$, so $G$ is isomorphic to a subdirect product of the $G/N_a$.
\end{proof}

Shortly we will introduce a much stronger result for commutative $\ell$-groups.

\begin{defn}\cite[Section (2.4)]{OrderedGroups3}
A convex $\ell$-subgroup $P$ of the $\ell$-group $G$ is called a \emph{prime subgroup} if whenever
$a, b \in G$ with $a \wedge b \in P$, then $a \in P$ or $b \in P$.
\end{defn}

\begin{note}
In the context of idempotent semifields, we use the notion `irreducible' for `prime'.
\end{note}

\begin{prop}\cite[Theorem (2.4.1)]{OrderedGroups3}
The following statements are equivalent for the convex $\ell$-subgroup
$P$ of $G$:
\begin{enumerate}
  \item $P$ is a prime subgroup.
  \item If $a,b \in G$ with $a \wedge b = 1$ then $a \in P$ or $b \in P$.
  \item The lattice of (left) cosets  $G/P$ is totally ordered.
  \item The lattice of convex $\ell$-subgroups of $G$ that contain $P$ is totally ordered.
\end{enumerate}
\end{prop}

\begin{rem}\cite{OrderedGroups3}
Each subgroup $L$ that contains a prime subgroup is an $\ell$-subgroup, and hence is
itself prime if it is convex. For if $a \in L$, then, since $(a \vee 1) \wedge ( a^{-1} \vee 1) = 1$, $(a \vee 1) = a ( a^{-1} \vee 1) \in L$ and so by Remark \ref{rem_sufficient_condition_for_l-subgroup} $L$ is an $\ell$-subgroup.
It can be easily seen that the intersection of any chain of prime subgroups is prime.
In particular, if $P$ is a prime subgroup and $\{ P \}$ is enlarged to a maximal chain
$\{P_i : i \in I \}$ of prime subgroups, then $\bigcap P$ is a minimal prime contained in $P$.
\end{rem}

\begin{prop}
Each convex $\ell$-subgroup is the intersection of prime subgroups.
\end{prop}
\begin{proof}
See \cite{OrderedGroups3} Theorem (2.4.2).
\end{proof}

\begin{defn}
Let $G$ be an $\ell$-group. For $a,b \in G$, $a$ and $b$ are said to be \emph{disjoint} or \emph{orthogonal}
if $|a| \wedge |b| = 1$.
\end{defn}

\ \\

\begin{rem}\label{rem_polar_kernels}
Let $G$ be a commutative $\ell$-group. For any $a \in G$, the set
$$\perp a = \{ x \in G : |a| \wedge |x| = 1 \}$$
is a convex $\ell$-subgroup of $G$.
\end{rem}

\ \\

The following theorems can be found in \cite{OrderedGroups4} (Theorems XIII.21 and XIII.22).

\begin{thm}\label{thm_l_group_main_1}
A commutative $\ell$-group is either linearly ordered or subdirectly \linebreak reducible.
\end{thm}

\begin{thm}{(Clifford)}\label{thm_l_group_main_2}
Any commutative $\ell$-group is a subdirect product of subdirectly irreducible
linearly ordered $\ell$-groups.
\end{thm}

\ \\

\subsection{Idempotent semifields versus lattice-ordered groups}

\ \\

Endowed with the natural order given in Remark \ref{rem_natural_order_on_semifield}, an idempotent semifield of fractions $\mathbb{H}(x_1,...,x_n)$  can be considered as a $\ell$-group.
The following results establish the connection between $\ell$-groups and additively commutative and idempotent semifields in general, and in particular $\mathbb{H}(x_1,...,x_n)$.\\

There exists a correspondence between $\ell$-groups and additively commutative and idempotent semifields.\\

The following results can be found in section 4 of \cite{Semifields_LO_Groups}.
\begin{prop}\cite{Semifields_LO_Groups}\label{prop_l-group_idempotent_semifield_corr_1}
1. Let $(A, \leq , .)$ be an $\ell$-group and define
$$a + b = a \vee b = sup\{a, b\} \ \text{for all} \ a, b \in A.$$
Then $(A, +, .)$ is a semifield such that $(A, +, \cdot)$ is commutative and idempotent.\\
2. Conversely, let $(A, +, .)$ be commutative and idempotent semifield and define $a \leq b$ whenever
$a + b = b$ for $a, b \in A$. Then $(A, \leq , .)$ is an $\ell$-group satisfying $a \vee b = a + b$
and $a \wedge b = (a^{-1} + b^{-1})^{-1}$.\\
3. This correspondence is bijective.
Moreover, the following statements are equivalent in this context:
\begin{itemize}
  \item $\phi$ is a semifield homomorphism of $(A, +, \cdot)$ into $(B, +, \cdot)$.
  \item $\phi$ is a $\vee$ -preserving group homomorphism of $(A, \leq , \cdot)$ into $(B, \leq , \cdot)$.
  \item $\phi$ is a $\wedge$-preserving group homomorphism of $(A, \leq , \cdot)$ into $(B, \leq , \cdot)$.
\end{itemize}
\end{prop}

\begin{prop}\label{prop_l-group_idempotent_semifield_corr_2}
Each kernel $K$ of an additively commutative and idempotent \linebreak
semifield $(A,+, .)$ is a normal and convex subgroup of $(A, \leq , .)$ such that
$a \vee b \in K$ holds for all $a, b \in K$, and conversely. In fact, $K$ is a subsemifield of
$(A,+, .)$ and a sublattice of $(A, \vee, \wedge)$.
\end{prop}

The following is a direct consequence of Remark \ref{rem_transitivity_for_l_groups}:

\begin{cor}\label{rem_transitivity_for_idemotent_kernels}
Let $\mathbb{H}$ be an idempotent commutative semifield. If $K$ is a kernel of $\mathbb{H}$ and $L$ is a kernel of $K$ (viewing $K$ as an idempotent semifield) then $L$ is a kernel of~$\mathbb{H}$.
\end{cor}

\ \\
\textbf{
In view of the above, all statements presented in \ref{prop_l-group_idempotent_semifield_corr_1} and \ref{prop_l-group_idempotent_semifield_corr_2}  referring to \\ $\ell$-groups are true for idempotent semifields, changing notation from $\vee$ to $\dotplus$.
}

\ \\

\subsection{Idempotent semifields: Part 2}

\ \\

The structure of lattice-ordered groups is thoroughly studied in \cite{OrderedGroups1} and \cite{OrderedGroups2}, and
many results given there are applicable for a semifield of fractions $\mathbb{H}(x_1,...,x_n)$ with $\mathbb{H}$ a bipotent semifield and generally for affine semifield extensions of $\mathbb{H}$ (which are just quotients of $\mathbb{H}(x_1,...,x_n)$ by Remark \ref{rem_affine_semifields_as_images}).

The results we introduce in the remainder of this section are derived from known results in the theory
of lattice-ordered groups.

The notions of an $\ell$-group and  normal convex $\ell$-subgroups  correspond to an \linebreak idempotent semifield and their kernels. In our case the idempotent semifield is also commutative, thus normality of an $\ell$-subgroup is insignificant.

%
\ \\

\begin{defn}\label{defn_positive_cone of_fractions}
Let $\mathbb{S}$ be an idempotent semifield (equivalently, an $\ell$-group). Define the \emph{positive cone} of $\mathbb{S}$ to be
$$\mathbb{S}^{+} = \left\{ s \in \mathbb{S} : s \geq 1 \right\}$$ and the negative cone  of $\mathbb{S}$ to be $$\mathbb{S}^{-} = \left\{s^{-1} : s  \in \mathbb{S}^{+} \right\} = \left\{s  \in \mathbb{S} : s  \leq 1 \right\}.$$\\
In the special case of the idempotent semifield $\mathbb{H}(x_1,...,x_n)$, we call $\mathbb{H}(x_1,...,x_n)^{+}$ and $\mathbb{H}(x_1,...,x_n)^{-}$ the \emph{positive cone of fractions} and the \emph{negative cone of fractions} and denote them by $\mathcal{P}^{+}$ and $\mathcal{P}^{-}$, respectively.
\end{defn}

As the following Theorem demonstrates, the partial orders of the group \\ $(\mathbb{H}(x_1,...,x_n), \cdot)$ that make it into a po-group are in one-to-one correspondence with the positive cones of $\mathbb{H}(x_1,...,x_n)$.

\begin{thm}\label{thm_cones_and_orderes}\cite[Theorem (2.1.1)]{OrderedGroups3}
The following statements hold for the positive and negative cones of an idempotent semifield $\mathbb{S}$ (equivalently an $\ell$-group):
\begin{enumerate}
  \item $(\mathbb{S}^{+},\cdot)$ is a (normal) subsemigroup of $\mathbb{S}^{+}$,
  \item $\mathbb{S}^{+} \cap \mathbb{S}^{-1} = \{1\}$,
  \item For every $f,g \in \mathbb{S}$, \ $f \leq g \Leftrightarrow gf^{-1} \in \mathbb{S}^{+}$.
\end{enumerate}
Conversely, if $P$ is a (normal) sub semigroup of \ $\mathbb{S}$ that satisfies (2),
then the relation defined in (3) is a partial order of \ $\mathbb{S}$ which makes $\mathbb{S}$ into a partially ordered group with positive cone $P$.
\end{thm}

\begin{note}
Theorem \ref{thm_cones_and_orderes} gives another perspective for viewing quotient semifields.
\end{note}

%

\begin{thm}\label{thm_cones_as_generators} \cite[Theorem (2.1.2)]{OrderedGroups3}
Let $G$ be a po-group.
\begin{itemize}
  \item $G$ is totally ordered if and only if  $G = G^{+} \cup G^{-}$.
  \item $G$ is directed if and only if  $G^{+}$ generates $G$. Moreover, if $G^{+}$ generates $G$, then $G =
 G^{+} \cdot G^{-} = \{ a b^{-1} \ : \ a,b \in G^{+} \}$.
  \item $G$ is an $\ell$-group if and only if $G^{+}$ is a lattice and generates $G$ as a group.
\end{itemize}
\end{thm}


\begin{cor}
$\mathcal{P}^{+}$ is a lattice and generates $\mathbb{H}(x_1,...,x_n)$.
\end{cor}
\begin{proof}
Note that $\mathcal{P}^{+}$ being a lattice means that for any $f,g \in \mathcal{P}^{+}$ both $f \wedge g$ and $f \dotplus g$ (i.e., $f \vee g$) are in $\mathcal{P}^{+}$. The assertion follows directly from Theorem \ref{thm_cones_as_generators}.
\end{proof}
%

\begin{defn}
A lattice  $L$ is infinitely distributive if whenever $\{ x_i \}$ is a subset of $L$ for which
$\bigvee x_i$ exists, then, for each $y \in L$,
$\bigvee (y \wedge x_i)$ exists and $$y \wedge \bigvee x_i=\bigvee (y \wedge x_i).$$ The dual
also holds.
\end{defn}

The following is a consequence of $\mathbb{H}(x_1,...,x_n)$ being an $\ell$-group.
\begin{prop}\cite[Theorem (2.1.3)]{OrderedGroups3}
\begin{enumerate}
  \item $\mathbb{H}(x_1,...,x_n)$ is infinitely distributive and hence is distributive.
  \item $\mathbb{H}(x_1,...,x_n)$ satisfies the property  $f^k \geq 1 \Rightarrow f \geq 1$ for any $f \in \mathbb{H}(x_1,...,x_n)$ and any integer $k \geq 1$, hence it has no nonzero elements of finite order.
\end{enumerate}
\end{prop}

\begin{rem}\cite[Chapter 2]{OrderedGroups1}\label{rem_wedge_vee_operations_relations}
For any $f,g \in \mathbb{H}(x_1,...,x_n)$ the following equalities hold:
\begin{enumerate}
  \item $(f + g)^{-1} = f^{-1} \wedge g^{-1}$,
  \item $(f \wedge g)^{-1} = f^{-1} + g^{-1}$,
  \item $ f (f + g)^{-1} g = f \wedge g$,
  \item $f (f \wedge g)^{-1} g = f + g$,
  \item $(f (f \wedge g)^{-1}) \wedge (g (f \wedge g)^{-1}) = 1$,
  \item $|f g^{-1}| = (f + g)(f \wedge g)^{-1}$.
\end{enumerate}
\end{rem}

\begin{lem}\label{lem_plus_wedge_distributive_relation}
If $\mathbb{H}$ is bipotent then $\dotplus$ is distributive over $\wedge$ in $\mathbb{H}(x_1,...,x_n)$, i.e.,
for any $f,g,h \in \mathbb{H}(x_1,...,x_n)$ the following equality holds:
\begin{equation}\label{eq_distributive_identity_of_wedge_and_plus}
f \dotplus ( g \wedge h) = (f \dotplus g) \wedge (f \dotplus h).
\end{equation}
\end{lem}
\begin{proof}
We can rewrite \eqref{eq_distributive_identity_of_wedge_and_plus} as $\max(f,\min(g,h)) = \min(\max(f,g),\max(f,h))$.
Let $x \in \mathbb{H}^n$. If $f(x) \geq g(x),h(x)$, $h(x) \geq f(x) \geq g(x) $ or $g(x) \geq f(x) \geq h(x) $, then both sides of the equality are equal to $f(x)$. If $g(x) \geq h(x) \geq f(x)$  then both sides of the equality are equal to $h(x)$. If $h(x) \geq g(x) \geq f(x)$  then both sides of the equality are equal to $g(x)$.\\
As the above holds for every $x \in \mathbb{H}^n$, \eqref{eq_distributive_identity_of_wedge_and_plus} is an identity.
\end{proof}

\begin{rem}\label{rem_wedge_is gcd}
Let $g,h \in \mathbb{H}(x_1,...,x_n)$. For every $f \in \mathbb{H}(x_1,...,x_n)$ such that \linebreak $f \geq |g| \wedge |h|$ (thus $f \geq 1$ and $f = |f|$), if $f \leq |g|$ and $f \leq |h|$ then $f = |g| \wedge |h|$.  \\
Indeed, $f = f \dotplus   (|g| \wedge |h|) = (f \dotplus |g|) \wedge (f \dotplus |h|) = |g| \wedge |h|$. The first equality from the right follows the assumption that $f \geq |g| \wedge |h|$, the second follows Lemma \ref{lem_plus_wedge_distributive_relation} and the last follows the assumptions $|f| \leq |g|$ and $|f| \leq |h|$.
\end{rem}

\begin{defn}\label{defn_max_semifield_of_fractions_operations}
For any $f \in \mathbb{H}(x_1,...,x_n)$ we define
\begin{equation}
|f|_{\geq}(x) = f(x) \dotplus 1 = \begin{cases}
1  &  \    f(x) \leq  1 ;\\
f(x)  & \ f(x) > 1,
\end{cases}
\end{equation}
and
\begin{equation}
|f|_{\leq}(x) = (f(x) \wedge 1)^{-1} = \begin{cases}
f(x)^{-1}  &  \    f(x) \leq  1 ; \\
1  & \ f(x) > 1.
\end{cases}
\end{equation}
By definition $|f|_{\geq},|f|_{\leq} \in \mathbb{H}(x_1,...,x_n)$.
\end{defn}
\ \\

\begin{rem}
By Remark \ref{rem_wedge_vee_operations_relations} we have that
$$|f|_{\leq} = (f \wedge 1)^{-1} = f^{-1} \dotplus 1.$$
\end{rem}

\ \\

\begin{rem}\cite[Chapter 2]{OrderedGroups1}
For any $f,g \in \mathbb{H}(x_1,...,x_n)$ the following statements hold:
\begin{enumerate}
  \item $|fg|_{\geq} \leq |f|_{\geq} |g|_{\geq}$.
  \item $|fg|_{\geq} = f(f \wedge g^{-1})^{-1}$.
  \item $|f^{n}|_{\geq} = |f|_{\geq}^n$ and $|f^{n}|_{\leq} = |f|_{\leq}^n$ for any $n \in \mathbb{N}$.
\end{enumerate}
\end{rem}

\ \\

\begin{rem}\label{rem_absolute_operators_relations}
For any $f,g \in \mathbb{H}(x_1,...,x_n)$ the following statements hold:
\begin{enumerate}
  \item $|f|_{\geq} \cdot |f|_{\leq}^{-1}  = f$.
  \item $|f|_{\geq} \cdot |f|_{\leq} = |f|_{\geq} \dotplus |f|_{\leq} = |f|$.
  \item $|f|_{\geq} \wedge |f|_{\leq} = 1$.
  \item $| | f | | = |f|$, $||f| \wedge |g|| = |f| \wedge |g|$, $||f| \dotplus |g|| = |f| \dotplus |g|$.
  \item $| |f|_{\geq} | = |f|_{\geq}$, $| |f|_{\leq} | = |f|_{\leq}$.
  \item $\langle f \rangle = \langle |f| \rangle$.
  \item $|f|^{k} = |f^{k}|$ for any $k \in \mathbb{N}$.

\end{enumerate}
\end{rem}
\begin{proof}
1. Since $f \cdot |f|_{\leq}^{-1}  = f(f \wedge 1)^{-1} = f ( f^{-1} + 1) = 1 + f = |f|_{\geq}$ we have that  $|f|_{\geq} \cdot |f|_{\leq}  = f$.\\
3. $ |f|_{\geq} \wedge |f|_{\leq}^{-1} = (|f|_{\geq} |f|_{\leq} |f|_{\geq}^{-1}) \wedge |f|_{\geq}^{-1} = (|f|_{\geq}|f|_{\leq} \wedge 1 ) |f|_{\geq}^{-1} = (f \wedge 1 ) |f|_{\geq}^{-1} = |f|_{\geq}|f|_{\geq}^{-1} = 1$.
The sixth statement follows from Proposition \ref{prop_absolute_generator}. The last statement is due to the fact that $|f(x)|^{k}= (f(x) + (f(x))^{-1})^{k} = f(x)^k + (f(x))^{-k}$ for any $x \in \mathbb{H}^n$.
The rest of the statements are obvious and can be found in chapter 2 of \cite{OrderedGroups1}.
\end{proof}

\ \\

\begin{rem}\label{rem_Riesz and disjoint}
Suppose $f,g,h \in \mathbb{H}(x_1,...,x_n)$
\begin{enumerate}
  \item The following are equivalent:
    \begin{itemize}
      \item $f \wedge g = 1$.
      \item $f g = f + g$.
      \item $f = |f g^{-1}|_{\geq}$ \ and \ $g = |f g^{-1}|_{\leq}$.
    \end{itemize}

  \item (The Riesz decomposition property) If $f,g,h \in \mathcal{P}^{+}$ and $f \leq g h$, then $f = g' h'$ where $1 \leq g' \leq g$ and $1 \leq h' \leq h$.
  \item If $f,g,h \in \mathcal{P}^{+}$, then $f \wedge (g h) \leq  (f \wedge g)(f \wedge h)$.
  \item If $f \wedge g = 1$ and $f \wedge h = 1$, then $f \wedge (g h) = 1$.
\end{enumerate}
\end{rem}
\begin{proof}
We will only give the proof for the Riesz decomposition property. For all other properties
see \cite{OrderedGroups3} (2.1.4).
Let $g' = f \wedge g$ and $h' = (g')^{-1} f$. Then $1 \leq  g' \leq g$,  $h' = (f^{-1} + g^{-1})f = 1 + g^{-1}f \leq h$, $h' \geq 1$, and $f = g' h'$.
\end{proof}

\begin{rem}
For any $f,g \in \mathbb{H}(x_1,...,x_n)$ the following relations hold:
\begin{enumerate}
  \item $|f + g | \leq |f| + |g|, \ |f| \cdot |g|$.
  \item $| f g^{-1} | = (f + g)(f \wedge g)^{-1}$.
  \item $| f g | \leq |f| \cdot |g|$.
\end{enumerate}
\end{rem}
\begin{proof}
See \cite{OrderedGroups1}, Chapter 2. We note that the third statement is true due to the \linebreak commutativity
of $\mathbb{H}(x_1,...,x_n)$.
\end{proof}

%

\ \\
\subsubsection{Structure of affine extensions of a bipotent semifield}

\ \\

Due to the correspondence between idempotent semifields and $\ell$-groups, we can apply
Theorems \ref{thm_l_group_main_1} and \ref{thm_l_group_main_2} to $\mathbb{H}(x_1,...,x_n)$ and deduce

\begin{cor}\label{cor_l_group_main_deduction}
$\mathbb{H}(x_1,...,x_n)$ is a subdirect product of simple semifields.
\end{cor}
\begin{proof}
First note that any kernel of an idempotent semifield is itself an idempotent semifield.
Now, linearly (totally) ordered $\ell$-groups have no proper convex $\ell$-subgroups and thus correspond to  semifield having no proper kernels, i.e., simple semifields (see Remark \ref{rem_order_simple}).
\end{proof}

\begin{rem}\cite[Lemmas (3.1.8),(3.1.16)]{Mark}\cite[Sec. 2.1]{OrderedGroups3}\label{rem_divisible_hull}
For a torsion free partially ordered abelian group $G$ there exists a group $\overline{G}$ which is the smallest divisible group containing $G$ extending its order. If $G$ is lattice ordered, directed or totally ordered then so is $\overline{G}$, lattice ordered, directed or totally ordered, respectively. $\overline{G}$ is called the \emph{divisible hull} of $G$.
\end{rem}

\begin{defn}
Viewing an idempotent semifield $\mathbb{H}$ as an $\ell$-group, we define the divisible closure $\overline{\mathbb{H}}$ of $\mathbb{H}$ to be its divisible hull. By the above, $\overline{\mathbb{H}}$ is also an idempotent semifield (as it is an $\ell$-group).
\end{defn}

The following two results will be considered and proved in our subsequent discussions:

\begin{prop}\label{prop_sub_direct_prod_of_semifield_of_fractions}
Let $\mathbb{H}$ be a bipotent semifield. As we have previously shown, any bipotent semifield is totally (linearly) ordered and thus a simple semifield.\\
For each point $a = (\alpha_1,...,\alpha_n) \in \overline{\mathbb{H}}^n$, let $$L_a = \langle (\beta_1)^{-1} x_1^{k_1} ,...., (\beta_n)^{-1}x_n^{k_n} \rangle$$ where $k_i \in \mathbb{N}$ is minimal such that $\beta_i = (\alpha_i)^{k_i} \in \mathbb{H}$.  Then $L =\bigcap_{ \ a \in \overline{\mathbb{H}}^n}L_a = \{1\}$ and $$\mathbb{H}(x_1,...,x_n) = \mathbb{H}(x_1,...,x_n)/L \xrightarrow[s.d]{} \prod_{ \ a \in \overline{\mathbb{H}}^n}Q_a$$ where  $Q_a = \mathbb{H}(x_1,...,x_n)/L_a$. Each $Q_a$ is an affine bipotent (thus totally ordered) semifield extension of $\mathbb{H}$ namely $\mathbb{H}(\alpha_1,...,\alpha_n)$ which is a simple.
In the case where $\mathbb{H}$ is divisible, we have each $Q_a \cong \mathbb{H}$.
\end{prop}
\begin{proof}
Let $f \in L =\bigcap_{ \ a \in \overline{\mathbb{H}}^n}L_a$. Then for each $a \in \overline{\mathbb{H}}^n$ since $f \in L_a$ we have that \linebreak $f(a) = 1$. Thus, since $\overline{\mathbb{H}}$ is divisible and $f$ coincides with $1$ over $\overline{\mathbb{H}}^n$, we have that $f = 1$. By Remark \ref{rem_divisibe_extensions}, we have that $Q_a \cong \mathbb{H}(\alpha_1,...,\alpha_n)$ is thus a bipotent semifield (extending $\mathbb{H}$).
\end{proof}

\begin{rem}
The analogous construction holds for the subsemifield  $\langle\mathbb{H}\rangle$ of \\
$\mathbb{H}(x_1,...,x_n)$ which is the principal kernel $\langle |\gamma| \rangle$ with $\gamma$ any element of $\mathbb{H}\setminus \{1\}$,
taking $$L_a \cap \langle \mathbb{H} \rangle = \langle |(\beta_1)^{-1} x_1^{k_1}| \wedge |\gamma| ,...., |(\beta_n)^{-1}x_n^{k_n}| \wedge |\gamma| \rangle$$
with $\gamma \in \mathbb{H} \setminus \{1\}$.
\end{rem}

\ \\

\subsection{Archimedean idempotent semifields}

\ \\

All the statements introduced in this section were stated originally for additive \linebreak $\ell$-groups and have been translated by us to the language of commutative idempotents semifields, where the group operation is multiplication instead of addition.

Recall that an idempotent semifield $\mathbb{H}$ is said to be \emph{archimedean} if it is archimedean as a po-group,
i.e., for $a,b \in \mathbb{H}$ if $a^{\mathbb{Z}} \leq b$ then $a = 1$.

\begin{prop}\label{prop_semifield_of_fractions_is_archimedean}
If $\mathbb{H}$ is archimedean, then $\mathbb{H}(x_1,...,x_n)$ is \linebreak
archimedean. Moreover, every $K \in \Con(\mathbb{H}(x_1,...,x_n))$ is archimedean as a \linebreak subsemifield of $\mathbb{H}(x_1,...,x_n)$.
\end{prop}
\begin{proof}
Let $f,g \in \mathbb{H}(x_1,...,x_n)$ such that $f^{\mathbb{Z}} \leq g$ . If $a \in \mathbb{H}^n$ then
by our assumption $f(a)^{k} \leq g(a)$ for all $k \in \mathbb{Z}$. Since $\mathbb{H}$ is archimedean and $f(a),g(a) \in \mathbb{H}$, we have that $f(a) = 1$. Since this holds for any $a \in \mathbb{H}^n$ we have $f(\mathbb{H}^n)=1$ implying  $f = 1$. The arguments above hold for any kernel of $\mathbb{H}(x_1,...,x_n)$.
\end{proof}

\begin{rem}
The arguments given for Proposition \ref{prop_semifield_of_fractions_is_archimedean} yield that the semifield $Fun(\mathbb{H}^n,\mathbb{H})$ is archimedean.
\end{rem}

\begin{defn}
A idempotent semifield $\mathbb{H}$ is said to be \emph{complete} if its underlying lattice is conditionally complete, cf. Definition \ref{defn_cond_complete}.
\end{defn}

\begin{rem}\label{rem_complete_semifield_of_functions}
If the idempotent semifield $\mathbb{H}$ is complete, then $Fun(\mathbb{H}^n,\mathbb{H})$ is complete.
\end{rem}
\begin{proof}
Let $X \subset Fun(\mathbb{H}^n,\mathbb{H})$ be bounded from below, say by $h \in Fun(\mathbb{H}^n,\mathbb{H})$.
Then for any $a \in X$ the set $\{f(a) : f \in X \}$ is bounded from below by $h(a)$ thus has an infimum $\bigwedge_{f \in X}f(a)$. It is readily seen that the function $g \in Fun(\mathbb{H}^n,\mathbb{H})$ defined by $g(a) = \bigwedge_{f \in X}f(a)$ is an infimum for $X$, i.e., $g = \bigwedge_{f \in X}f$. Analogously, if $X$ is bounded from above then $(\bigvee_{f \in X}f)(a) = \bigvee_{f \in X}f(a)$ is the supremum of $X$.
\end{proof}

\begin{defn}
A \emph{completion} of the idempotent semifield $\mathbb{H}$ is a pair $(H,\theta)$ where $H$ is a complete idempotent semifield and $\theta : \mathbb{H}  \rightarrow H$ is a monomorphism whose image is dense in $H$.
\end{defn}

The following theorem states that each archimedean idempotent semifield has a unique completion.

\begin{thm}\label{thm_completion_of_l_group}\cite[Theorem 2.3.4]{OrderedGroups3}
An idempotent semifield has a completion if and only if it is archimedean. If
$(A, \theta_A)$ and $(B, \theta_B )$ are two completions of the idempotent semifield $\mathbb{H}$, then there is a unique isomorphism
$\rho : A \rightarrow B$ such that $\rho \circ \theta_A = \theta_B$.
\end{thm}

\begin{defn}\label{defn_coinitial}
A subset $A$ of the poset $P$ is called \emph{co-final} in $P$ if for every $x \in P$ there exists some $a \in A$ such that $x \leq  a$. $A$ is said to be \emph{co-initial} in $P$ if for every $x \in P$ there exists some $a \in A$ such that $a \leq  x$.
\end{defn}

\begin{defn}\label{defn_lef_right_dense}
The subset $A$ of the idempotent semifield $\mathbb{H}$ is called \emph{left dense} in~$\mathbb{H}$ if $A^{+} \setminus \{1\}$ (where $A^{+} = \{ a \in A : a \geq 1 \}$) is co-initial in $\mathbb{H}^{+} \setminus \{1 \}$, and $A$ is called \emph{right dense} in H if $A^{+}$ is co-final in $\mathbb{H}^{+}$.
\end{defn}

\begin{thm}\cite[Theorem (2.3.6)]{OrderedGroups3}\label{thm_completion_of_a_semifield}
Let $\mathbb{H}$ be an archimedean subsemifield of the complete idempotent semifield $H$. Then the following
statements are equivalent:
\begin{enumerate}
  \item $H$ is the completion of $\mathbb{H}$.
  \item $\mathbb{H}$ is left dense in $H$, and if $A$ is an idempotent subsemifield of $H$ that is complete and
contains $\mathbb{H}$, then $A = H$.
\end{enumerate}
\end{thm}

\begin{cor}\label{cor_completion_of_kernels}\cite{OrderedGroups3}
Suppose that $\mathbb{H}$ is a left dense archimedean idempotent subsemifield of the complete idempotent semifield $H$. Then if $A$ is a kernel of $\mathbb{H}$, then the completion of $A$ is the kernel of $H$ generated by $A$.
\end{cor}

We now state the well-known result by H\"{o}lder for $\ell$-groups:

\begin{thm}\cite[Theorem (2.3.10)(H\"{o}lder)]{OrderedGroups3}
The following statements are equivalent for an $\ell$-group $G$.
\begin{enumerate}
  \item The only convex $\ell$-subgroups of $G$ are $1$ and $G$.
  \item $G$ is totally ordered and archimedean.
  \item $G$ can be embedded in $(\mathbb{R}, + )$.
\end{enumerate}
\end{thm}

\begin{prop}\cite{OrderedGroups3}
A divisible totally ordered archimedean group which is \linebreak complete is isomorphic to $(\mathbb{R},+)$.
\end{prop}

Translating the above to the language of idempotent semifields and using the \linebreak isomorphism $(\mathbb{R}, +) \rightarrow (\mathbb{R}^{+}, \cdot)$ defined by $x \mapsto e^x$  (where $\mathbb{R}^{+}$ is the set of positive real numbers) yield

\begin{thm}\label{thm_holder_translated}
The following statements are equivalent for an idempotent \linebreak semifield $\mathbb{H}$.
\begin{enumerate}
  \item $\mathbb{H}$ is simple.
  \item $\mathbb{H}$ is totally ordered and archimedean.
  \item $\mathbb{H}$ can be embedded in $(\mathbb{R}^{+}, \dotplus, \cdot)$.
\end{enumerate}
\end{thm}

\begin{cor}\label{cor_totaly_ordered_isomorphic_to_reals}
A divisible totally ordered archimedean idempotent semifield which is complete is isomorphic to $(\mathbb{R}^{+}, \dotplus, \cdot)$.
\end{cor}

In view of the above, we may regard the designated semifield $\mathscr{R}$ as being $(\mathbb{R}^{+}, \dotplus, \cdot)$.

\newpage

\section{Skeletons and kernels of skeletons}
\   \\

In this section we introduce the notions of `skeletons' and `kernels of skeletons'. We define a pair of operators
$Skel$ and $Ker$ where $Skel$ maps a kernel to its skeleton and $Ker$ maps a skeleton to its corresponding kernel. Although we define these \linebreak operators with respect to the semifield of fractions $\mathbb{H}(x_1,...,x_n)$, since we only use the fact that $\mathbb{H}(x_1,...,x_n)$ is a semifield to prove our assertions, all the properties proved in subsections \ref{Subsection:Skeletons} and \ref{Subsection:Kernels_of_Skeletons} hold replacing $\mathbb{H}(x_1,...,x_n)$ by any subsemifield of the semifield of functions $Fun(\mathbb{H}^n, \mathbb{H})$.

\subsection{Skeletons}\label{Subsection:Skeletons}
\    \\
We now define a geometric object in the semifield $\mathbb{H}^{n}$ to which we aim to associate to a kernel of the semifield of fractions $\mathbb{H}(x_1,...,x_n)$. Namely, we define an operator
$$Skel : \mathbb{P}(\mathbb{H}(x_1,...,x_n)) \rightarrow \mathbb{H}^{n}$$
which associates a subset of $\mathbb{H}^{n}$ to every kernel of $\mathbb{H}(x_1,...,x_n)$.

\begin{note}
In the field of computer vision the notion of skeleton (or topological \linebreak skeleton) of a shape is a thin version of that shape that is equidistant to its boundaries. \linebreak Resembling  the tropical variety in its shape,  we have decided to call the next geometric object , which will be shown to generalize the notion of tropical variety, a `skeleton'.
\end{note}
\ \\

\begin{defn}
Let $S$ be a subset of \ $\mathbb{H}(x_1,...,x_n)$. Define the subset $Skel(S)$ of $\mathbb{H}^n$ \linebreak  to be
\begin{equation}
Skel(S) = \{(a_1,...,a_n) \in \mathbb{H}^n  \ : \  f(a_1,...,a_n) = 1,  \ \forall f \in S \}.
\end{equation}
\end{defn}

\begin{defn}
A subset $\mathcal{S}$ in $\mathbb{H}^n$ is said to be a \emph{skeleton} if there exists a subset $S \subset \mathbb{H}(x_1,...,x_n)$  such that $\mathcal{S} = Skel(S)$.
\end{defn}

\begin{defn}
A skeleton $\mathcal{S}$ in $\mathbb{H}^n$ is said to be a  \emph{principal skeleton}, if there exists $f \in \mathbb{H}(x_1,...,x_n)$ such that $\mathcal{S} = Skel(\{f\})$.
\end{defn}

\begin{prop}\label{prop_skel_property1}
For $S_i \subset \mathbb{H}(x_1,...,x_n)$ the following statements hold:
\begin{enumerate}
  \item $S_1 \subseteq S_2 \Rightarrow Skel(S_2) \subseteq Skel(S_1)$.
  \item $Skel(S_1) = Skel(\langle S_1 \rangle)$.
  \item $\bigcap_{i \in I}Skel(S_i) = Skel(\bigcup_{i \in I}S_i)$ for any index set $I$ and in particular, \linebreak $Skel(S)= \bigcap_{f \in S}Skel(f)$.
\end{enumerate}
\end{prop}
\begin{proof}
The first statement is set theoretically obvious and since $S_1 \subseteq \langle S_1 \rangle$, it implies that $Skel(\langle S_1 \rangle) \subseteq Skel(S_1)$ in the second assertion. To prove the opposite inclusion, we note that $(fg)(x) = f(x)g(x) = 1 \cdot 1 = 1$, $f^{-1}(x)=f(x)^{-1} = 1^{-1}=1$ and  $(f+g)(x) = f(x) + g(x) = 1+1=1$, for any $f,g$ such that $f(x)=g(x)=1$
thus proving the skeleton of the semifield generated by $S_1$ is in $Skel(S_1)$. Now, we have to show that convexity is preserved. Since for $k_1,...,k_t \in \mathbb{H}(x_1,...,x_n)$ s.t. $\sum_{i=1}^{t}k_i=1$ and for any $s_1,...,s_t$ in the semifield generated by $S_1$ we have $(\sum_{i=1}^{t}k_is_i)(x) = \sum_{i=1}^{t}k_i(x)s_i(x)= \sum_{i=1}^{t}(k_i(x)\cdot 1)= \sum_{i=1}^{t}k_i(x) = 1(x)=1$, we have that convexity holds.
For the last assertion, by (1) we have that $Skel(\bigcup_{i \in I}S_i) \subseteq Skel(S_i)$ for each $i \in I$, thus $Skel(\bigcup_{i \in I}S_i) \subseteq \bigcap_{i \in I}Skel(S_i)$. The converse direction is obvious.
\end{proof}

\begin{rem}\label{rem_max_semifield_principal_skeletons}
For $\mathbb{H}$ a bipotent semifield, the following statements hold for \linebreak $f,g \in \mathbb{H}(x_1,...,x_n)$:
\begin{enumerate}
  \item $Skel(fg), Skel(f \dotplus g), Skel( f \wedge g) \supseteq Skel(f) \cap Skel(g)$.
  \item $Skel(fg) = Skel(f \dotplus g)  = Skel(f) \cap Skel(g)$ for all $f,g \geq 1$.
  \item $Skel(f \wedge g)  = Skel(f) \cup Skel(g)$ for all $f,g \leq 1$.
  \item $Skel(f) = Skel(f^{-1}) = Skel(f \dotplus f^{-1})  = Skel(f \wedge f^{-1})$.
\end{enumerate}
\end{rem}
\begin{proof}
(1) If $x \in Skel(f) \cap Skel(g)$ then $f(x)=g(x)=1$, so $f(x) \dotplus g(x) =  f(x) \wedge g(x) =f(x)g(x) = 1$.\\
(2) We have that $(f \dotplus g)(x) = f(x) \dotplus g(x) = \sup(f(x),g(x)) = 1$. Since $f,g \geq 1$ we have that $\sup(f(x),g(x)) = 1$ if and only if $f(x) = 1$ and $g(x)=1$. Analogously the same holds for $f(x)g(x) = 1$ when both $f,g \geq 1$.\\
(3) We have that $(f \wedge g)(x) = f(x) \wedge g(x) = \inf(f(x),g(x)) = 1$. Since $f,g \geq 1$ we have that $\inf(f(x),g(x)) = 1$ if and only if $f(x) = 1$ or $g(x)=1$. \\
(4) follows Proposition \ref{prop_skel_property1}(2) since $\langle f \rangle = \langle f^{-1} \rangle = \langle f \dotplus f^{-1} \rangle = \langle f \wedge f^{-1} \rangle$ (recalling that $f \wedge f^{-1} =  (f \dotplus f^{-1})^{-1}$).
\end{proof}

\begin{defn}\label{defn_principal_kernels_in_semifield_of_fractions}
Denote the collection of skeletons in $\mathbb{H}^n$ by $Skl(\mathbb{H}^n)$
and the collection of principal skeletons in $\mathbb{H}^n$ by $PSkl(\mathbb{H}^n)$.
\end{defn}

\ \\

\subsection{Kernels of skeletons}\label{Subsection:Kernels_of_Skeletons}

\ \\

In the following discussion, we construct an operator $Ker : \mathbb{H}^n \rightarrow \mathbb{P}(\mathbb{H}(x_1,...,x_n))$ which associates a kernel of the semifield of fractions $\mathbb{H}(x_1,...,x_n)$ to any skeleton in $Skl(\mathbb{H}^n)$. Then we will proceed to study the relation between this operator $Ker$ and the operator $Skel$ defined in the previous subsection.\\
Throughout this section, we take $\mathbb{H}$ to be a bipotent (totally ordered) semifield.
\ \\

\begin{defn}
Given a subset $Z$ of $\mathbb{H}^n$ define following subset of $\mathbb{H}(x_1,...,x_n)$:
\begin{equation}
Ker(Z) = \{ f \in \mathbb{H}(x_1,...,x_n) \ : \  f(a_1,...,a_n) = 1, \ \forall (a_1,...,a_n) \in Z    \}
\end{equation}
\end{defn}

\begin{rem}\label{rem_ker_properties_1}
For $Z,Z_i \subset \mathbb{H}^n$ the following statements hold:
\begin{enumerate}
  \item $Ker(Z)$ is a kernel of $\mathbb{H}(x_1,...,x_n)$.
  \item If $Z_1 \subseteq Z_2$, then $Ker(Z_2) \subseteq Ker(Z_1)$.
  \item $Ker(\bigcup_{i \in I}Z_i) = \bigcap_{i \in I}Ker(Z_i)$.
  \item $K \subseteq Ker(Skel(K))$ for any kernel $K$ of $\mathbb{H}(x_1,...,x_n)$.
  \item $Z \subseteq Skel(Ker(Z))$.
\end{enumerate}
\end{rem}
\begin{proof}
The first assertion follows from the proof of (2) of Proposition \ref{prop_skel_property1}.
The \linebreak second assertion is a trivial set theoretic fact, which in turn implies that \linebreak
$Ker(\bigcup_{i \in I}Z_i) \subseteq \bigcap_{i \in I}Ker(Z_i)$. The second inclusion of (3) is trivial.
Assertions (4) and (5) are straightforward consequences of the definitions of $Skel$ and $Ker$.
\end{proof}

\begin{prop}\label{prop_skel_ker_relations}
\begin{enumerate}
  \item $Skel(Ker(Z)) = Z$ if $Z$ is a skeleton.
  \item $Ker(Skel(K)) = K$, \ if $K = Ker(Z)$ for some $Z \subseteq \mathbb{H}^{n}$.
\end{enumerate}
\end{prop}

\begin{proof}
In view of Proposition \ref{prop_skel_property1} (5), we need to show that $Skel(Ker(Z)) \subseteq Z$. \linebreak Indeed, writing $Z = Skel(S)$ we have by Proposition \ref{prop_skel_property1} that $S \subset Ker(Skel(S))$ and thus $$Z =Skel(S) \supseteq Skel(Ker(Skel(S))) = Skel(Ker(Z)).$$
For the second assertion, by Remark \ref{rem_ker_properties_1} (5), $Z \subseteq Skel(Ker(Z))$, so
$$Ker(Skel(K)) = Ker(Skel(Ker(Z))) \subseteq Ker(Z) = K.$$
The reverse inclusion follows Remark \ref{rem_ker_properties_1} (4).
\end{proof}

\begin{defn}\label{defn_k_kernels}
A \emph{$\mathcal{K}$-kernel} of $\mathbb{H}(x_1,...,x_n)$ is a kernel of the form $Ker(Z)$, where $Z$ is a suitable skeleton.
\end{defn}

By the above assertions we have
\begin{prop}\label{prop_correspondence_between_K_kernels_and_skeletons}
There is a $1:1$ order reversing correspondence
\begin{equation}
\{ \text{skeletons of } \ \mathbb{H}^{n}  \} \rightarrow \{ \mathcal{K}- \text{kernels of } \ \mathbb{H}(x_1,...,x_n) \},
\end{equation}
given by $Z \mapsto Ker(Z)$;  the reverse map is given by $K \mapsto Skel(K)$.
\end{prop}

\begin{prop}\label{prop_ker_skel_correspondence}
Let $E_1$ and $E_2$ be kernels in $\mathbb{H}(x_1,...,x_n)$, and let $\mathcal{S}_1~=~Skel(E_1)$ and $\mathcal{S}_2 = Skel(E_2)$ be their corresponding skeletons. Then the following statements hold:
\begin{equation}
Skel(E_1 \cdot E_2) = \mathcal{S}_1 \cap \mathcal{S}_2;
\end{equation}
\begin{equation}
\mathcal{S}_1 \cup \mathcal{S}_2 = Skel(E_1 \cap E_2).
\end{equation}
\end{prop}
\begin{proof}
For the first assertion, denote $K = E_1 \cdot E_2$. Since $E_1$ and $E_2$ are kernels, by Remark \ref{rem_ker_latice_operations} $K$ also is a kernel and $E_1, E_2 \subseteq K$. Thus, by Proposition \ref{prop_skel_property1} (1), $Skel(K) \subseteq Skel(E_1)$ and $Skel(K) \subseteq Skel(E_2)$, so $Skel(K) \subseteq \mathcal{S}_1 \cap \mathcal{S}_2$. Conversely, if $x~\in~\mathcal{S}_1 \cap \mathcal{S}_2$ then $\forall f \in E_1$ and $\forall g \in E_2$, \ $f(x) = g(x) =1$ thus $(fg)(x) = f(x)g(x) = 1 \cdot 1 = 1$. Consequently,
\begin{align*}
\mathcal{S}_1 \cap \mathcal{S}_2 = \ &  \left\{  x \in \mathbb{H}^n \ : \ f(x) = 1 \ \forall f \in E_1 \right\} \cap \left\{  x \in \mathbb{H}^n \ : \ g(x) = 1 \ \forall g \in E_2 \right\} \\
= \ & \left\{  x \in \mathbb{H}^n \ : \ f(x) = 1, g(x) = 1 \  \forall f \in E_1 \ \forall g \in E_2 \right\} \\
\subseteq \ & \left\{ x \in \mathbb{H}^n \ : \ (fg)(x) =f(x)g(x) =  1 \  \forall f \in E_1 \ \forall g \in E_2 \right\}\\
= \ & \left\{ x \in \mathbb{H}^n \ : \ h(x) = 1 \ \forall h \in K \right\} =Skel(K),
\end{align*}
as desired.\\
For the second assertion, denote $K = E_1 \cap E_2$. Since $K \subseteq E_1, E_2$ we have by Proposition \ref{prop_skel_property1} that $\mathcal{S}_1 =Skel(E_1)\subseteq Skel(K)$ and $\mathcal{S}_2 = Skel(E_2) \subseteq Skel(K)$ thus $\mathcal{S}_1 \cup \mathcal{S}_2~\subseteq~Skel(K)$. Conversely, since $E_1$ and $E_2$ are kernels $\langle E_1 \rangle = E_1$ and $\langle E_2 \rangle = E_2$, thus Proposition \ref{prop_algebra_of_generators_of_kernels} yields that $K =  \langle \{|f| \wedge |g| : f \in E_1, g \in E_2 \} \rangle$ and so, by Proposition \ref{prop_skel_property1}(2), \ $Skel(K) = Skel\left( \{|f| \wedge |g| : f \in E_1, g \in E_2 \} \right)$. In view of the above we have that
$$x \in Skel(K) \ \  \Leftrightarrow  \ \  |f(x)| \wedge |g(x)| = 1 \ \ \forall f \in E_1, g \in E_2.$$
Now, Let $x \in Skel(K)$ and assume to the contrary that $x \not \in Skel(E_1)$ and \linebreak $x \not \in Skel(E_2)$. Then there are $f' \in E_1$ and $g' \in E_2$ such that $|f'(x)| > 1$ and $|g'(x)| > 1$. Since $\mathbb{H}$ is bipotent and $|f'(x)|,|g'(x)| \in \mathbb{H}$ we have that $(|f'| \wedge |g'|)(x) = |f'(x)| \wedge |g'(x)|  = \min(|f'(x)|, |g'(x)|)> 1$ which yields that $x \not \in Skel(K)$. A contradiction.
\end{proof}

As a special case of Proposition \ref{prop_ker_skel_correspondence} we have

\begin{cor}\label{cor_principal_skel_correspondence}
For $f,g \in \mathbb{H}(x_1,...,x_n)$
\begin{equation}\label{eq_1_prop_principal_skel_correspondence}
Skel(\langle f \rangle \cdot \langle g \rangle ) = Skel(f) \cap Skel(g);
\end{equation}
\begin{equation}\label{eq_2_prop_principal_skel_correspondence}
Skel(\langle f \rangle \cap \langle g \rangle) = Skel(f) \cup Skel(g).
\end{equation}
Note that by convention $Skel(f) = Skel(\langle f \rangle)$.
\end{cor}

\newpage

\section{The structure of  the semifield of fractions}

\ \\
\subsection{Bounded rational functions in the semifield of fractions}

\ \\
Throughout this section we assume $\mathbb{H}$ to be a bipotent and divisible semifield. Any supplemental assumptions regarding $\mathbb{H}$ will be explicitly stated.

\begin{flushleft}We begin by introducing an example which motivates our subsequent discussion in this section. \end{flushleft}

\begin{exmp}
Consider the principal kernel $\langle x \rangle$. Its corresponding skeleton is \linebreak defined by the equation $x=1$. Let $\alpha \in \mathscr{R}$ be such that $\alpha \neq 1$. The principal \linebreak kernel  $ \langle x \rangle \cap \langle \alpha \rangle = \langle |x| \wedge |\alpha| \rangle$ where $|x| \wedge |\alpha| = \frac{|x| |\alpha|}{|x| + |\alpha| } = \min(|x|,|\alpha|)$. Since \linebreak $ \min(|x|,|\alpha|) = 1 \Leftrightarrow |x| = 1$ we get that $\langle |x| \wedge |\alpha| \rangle$ defines exactly the same skeleton as $\langle x \rangle$.  As $x \not \in \langle x \rangle \cap \langle \alpha \rangle$ ($x$ cannot be bounded by a bounded function)  we have \linebreak that $\langle x \rangle \supset  \langle x \rangle \cap \langle \alpha \rangle$.\\
The cause of this phenomenon is the kernels of the form $\langle \alpha \rangle$ with $\alpha~\in~\mathscr{R}~\setminus~ \{1\}$. These kernels are not kernels of $\mathscr{R}$-homomorphisms, since for any such homomorphism $\phi$, one must have $\alpha =\alpha \phi(1) = \phi(\alpha)$ but as $\alpha \in Ker\phi$ we have that $\phi(\alpha) = 1$ too.
Thus kernels containing them  also fail to be kernels of $\mathscr{R}$-homomorphisms.
\end{exmp}

\begin{defn}
$f \in \mathscr{R}(x_1,...,x_n)$ is said to be \emph{bounded from below}  if there exists some $\alpha > 1$ in $\mathscr{R}$
such that $|f| \geq \alpha$.
\end{defn}

\begin{rem}\label{rem_bounded_from_below_contain_H_kernel}
Let $\langle f \rangle$ be a principal kernel of $\mathbb{H}(x_1,...,x_n)$. Then $f$ is bounded from below if and only if $\langle f \rangle \supseteq \langle \alpha \rangle$ for some $\alpha > 1$ in $\mathbb{H}$.
\end{rem}
\begin{proof}
If $\langle f \rangle \supseteq \langle \alpha \rangle$ then in particular $\alpha \in \langle f \rangle$ thus there exists some $k \in \mathbb{N}$ such that $|\alpha| \leq |f|^k$ thus $\beta \leq |f|$ for $\beta \in \mathbb{H}$ such that $\beta^k = |\alpha|$ (such $\beta$ exists as $\mathbb{H}$ is divisibly closed). Thus $f$ is bounded from below. Conversely, if $f$ is bounded from below, then there exists some $\alpha > 1$ in $\mathbb{H}$ such that $|f| \geq \alpha = |\alpha|$ which yields that $\alpha \in \langle f \rangle$ and thus $\langle \alpha \rangle \subseteq \langle f \rangle$.
\end{proof}

\ \\

\begin{rem}\label{rem_generator_of_down_bounded_kernel}
 Let $\langle f \rangle$ be a principal kernel such that $f$ is bounded from below. Then any generator $g \in \langle f \rangle$ is bounded from below.
\end{rem}
\begin{proof}
 Since $g$ is a generator, we in particular have that $f \in \langle g \rangle$. Then, By \linebreak Remark~\ref{rem_kernel_by_abs_value}, there exists some $k \in \mathbb{N}$ such that $|f| \leq |g|^k$. Since $f$ is bounded from below there exists some $\alpha > 1$ in $\mathbb{H}$ such that $|f| \geq \alpha$, thus $|g|^k \geq |f| \geq \alpha$, which yields that  $g$ is bounded from below by $\beta \in \mathbb{H}$ such that $\beta^k = \alpha$ (and so $\beta > 1$).
\end{proof}

\begin{defn}
A principal kernel $\langle f \rangle$  of $\mathbb{H}(x_1,...,x_n)$ is said to be \emph{bounded from below}, if it is generated by a function bounded from below.
\end{defn}

Let $\mathbb{H}$ be a divisible bipotent semifield. We will begin by characterizing the principal kernels $K = \langle f \rangle \subseteq \mathbb{H}(x_1,...,x_n)$ for which $Skel(f) = \emptyset$.

\begin{rem}\label{rem_emptyset_kernels}
Let $\langle f \rangle$ be a principal kernel of $\mathbb{H}(x_1,...,x_n)$. Then $Skel(f) = \emptyset$ \linebreak if and only if
$|f(x)|  > 1$ for every $x \in \mathbb{H}^n$.
\end{rem}
\begin{proof}
$Skel(f) = Skel(|f|) = \{ x \in \mathbb{H}^n : h(x) = 1 \ \forall h \in \langle |f| \rangle \} \subseteq \{ x \in \mathbb{H}^n : |f(x)| = 1 \} = \emptyset$. Conversely, as the skeleton is determined by any generator of $\langle f \rangle$, in particular $|f|$, $Skel(f) = \emptyset$ implies that there is no solutions to the equation $|f(x)| = 1$. Since $|f(x)| \geq 1$ for every $x \in \mathbb{H}^n$, this implies that $|f(x)|  > 1$ for every $x \in \mathbb{H}^n$.
\end{proof}

\begin{prop}\label{prop_emptyset_kernels}
Let $\langle f \rangle$ be a principal kernel in $\mathbb{H}(x_1,...,x_n)$. Then the following statements are equivalent:
\begin{enumerate}
  \item $Skel(\langle f \rangle) = \emptyset.$
  \item There exists a generator $f'$ of $\langle f \rangle$ such that $f' = f' \dotplus \gamma$ with $\gamma > 1$.
  \item There exists a generator of $\langle f \rangle$ which is bounded from below.
\end{enumerate}
\end{prop}
\begin{proof}
$(2) \Leftrightarrow (3)$ since $f' = f' \dotplus \gamma \Leftrightarrow |f'| = f' \geq \gamma > 1$. \\
We now prove  $(1) \Leftrightarrow (2)$. By Remark \ref{rem_emptyset_kernels}, $Skel(\langle f \rangle) = \emptyset$ if and only if $|f(x)|~>~1$ for every $x \in \mathbb{H}^n$. We claim that  $|f(x)|>1$ for every $x \in \mathbb{H}^n$ if and only if $|f| = |f| \dotplus \gamma$ with $\gamma > 1$ or in other words, that $|f| \geq \gamma$.
If $|f| = |f| \dotplus \gamma$,  then $|f(x)| > \gamma > 1$  for every $x \in \mathbb{H}^n$ and so $Skel(f) = \emptyset$. Conversely, let $|f| = \frac{h}{g} = \frac{\sum_{i=1}^{k}h_i}{\sum_{j=1}^{m}g_j}$ with $h_i$ and $g_j$ monomials in $\mathbb{H}[x_1,...,x_n]$. $|f| : \mathbb{H}^n \rightarrow \mathbb{H}$ defines a partition of $\mathbb{H}^n$ to a finite number of regions. Over each of these regions, $|f|$ has the form of a Laurent monomial  $l = \frac{h_{i_0}}{g_{j_0}}$ where $h = h_{i_0}$ and $g = g_{j_0}$. Each region $R$ is either closed and bounded or unbounded. In the former case $|f|$, being continuous, attains a minimum value over $R$; thus the image of $|f|$ is bounded from below by some $\gamma > 1$. In the latter case, if $l$ is constant then it trivially attains a minimal value and is bounded from below by some $\gamma > 1$. The only possibility we are left with is that $l$ is not constant. Assume the image of $l$ is not bounded from below by any $\gamma > 1$. Then there exists a one dimensional curve over which $l$ is monotonically decreasing  in a constant rate and thus must obtain the value $1$ contradicting the assumption that $|f(x)| > 1$ for all $x \in \mathbb{H}^n$. ( It is convenient to consider $l$ in logarithmic scale so that $l$ takes the form of a linear form and the curve is a one dimensional affine space (a straight line) over which $l$ has a constant slope.)
\end{proof}

%
%

\begin{cor}\label{cor_empty_kernels_correspond_to_bfb_kernels}
Let $\langle f \rangle$ be a principal kernel in $\mathbb{H}(x_1,...,x_n)$. Then \linebreak $Skel(\langle f \rangle) = \emptyset$ if and only if $\langle f \rangle$ is bounded from below.
\end{cor}
\begin{proof}
If $Skel(\langle f \rangle) = \emptyset$ then by Proposition \ref{prop_emptyset_kernels} $\langle f \rangle$ is generated by a bounded from below element, thus is a bounded from below kernel. Moreover, by Remark \ref{rem_generator_of_down_bounded_kernel} we have that any generator of $\langle f \rangle$ is bounded from below. Conversely, if $\langle f \rangle$ is a bounded from below kernel, then $f'$ is bounded from below for any generator $f'$ of $\langle f \rangle$, in particular for $f' = f$. Thus there exists some $\alpha > 1$ in $\mathbb{H}$ such that $|f| \geq \alpha$,  thus, by Remark \ref{rem_emptyset_kernels}, we have that $Skel(\langle f \rangle)= \emptyset$.
\end{proof}

\begin{defn}
$f \in \mathbb{H}(x_1,...,x_n)$ is said to be \emph{bounded from above} (or simply \emph{bounded}) if there exists some $\alpha \geq 1$ in $\mathbb{H}$
such that $|f| \leq \alpha$.
\end{defn}

\begin{rem}\label{rem_generator_of_up_bounded_kernel}
 Let $\langle f \rangle$ be a principal kernel such that $f$ is bounded from above. Then any $g \in \langle f \rangle$ is bounded from above. In particular, any generator of $\langle f \rangle$ is bounded from above.
\end{rem}
\begin{proof}
By Remark \ref{rem_kernel_by_abs_value} for any $g \in \langle f \rangle$ there exists some $k \in \mathbb{N}$ such that $|g|~\leq~|f|^k$. Since $f$ is bounded there exists some $\alpha \geq 1$ in $\mathbb{H}$ such that $|f| \leq \alpha$, thus $|g| \leq |f|^k \leq \alpha^k$, which yields that  $g$ is bounded.
\end{proof}

\begin{defn}
A principal kernel $\langle f \rangle$  of $\mathbb{H}(x_1,...,x_n)$ is said to be \emph{bounded from above} if it is generated by a function bounded from above.
\end{defn}

\begin{rem}\label{rem_up_bounded_is_in_H_kernel}
Let $\langle f \rangle$ be a principal kernel of $\mathbb{H}(x_1,...,x_n)$. Then $f$ is bounded from above if and only if $\langle f \rangle \subseteq \langle \alpha \rangle$ for some $\alpha \in \mathbb{H}$.
\end{rem}
\begin{proof}
If $\langle f \rangle \subseteq \langle \alpha \rangle$ then in particular $f \in \langle \alpha \rangle$. Thus there exists some $k \in \mathbb{N}$ such that $|f| \leq |\alpha|^k$, thus $f$ is bounded from above. Conversely, if $f$ is bounded from above, then there exists some $\alpha \in \mathbb{H}$ ( $\alpha \geq 1$) such that $|f| \leq \alpha = |\alpha|$ which yields that $f \in \langle \alpha \rangle$ and thus $\langle f \rangle \subseteq \langle \alpha \rangle$.
\end{proof}

\begin{rem}\label{rem_down_bounded_kernel_contains_H}
Let $\langle f \rangle$ be a principal kernel of $\mathbb{H}(x_1,...,x_n)$. Then $\langle f \rangle$ is bounded from below if and only if $\mathbb{H} \subseteq \langle f \rangle$, if and only if there exists some $\alpha > 1$ in $\mathbb{H}$ such that $\alpha \in \langle f \rangle$.
\end{rem}
\begin{proof}
If $\langle f \rangle$ is bounded from below, then there exists some $\alpha > 1$ in $\mathbb{H}$ such that $|f| \geq \alpha$ thus by Remark \ref{rem_kernel_by_abs_value} we have that $\alpha \in \langle f \rangle$. Since any $\alpha \neq 1$ is a generator of $\mathbb{H}$, we get that $\mathbb{H} = \langle \alpha \rangle \subseteq \langle f \rangle$.
Conversely, if $\mathbb{H} \subseteq \langle f \rangle$ then in particular $\alpha \in \langle f \rangle$ for any $\alpha > 1$ in $\mathbb{H}$. If $f$ is not bounded from below, then for any $\beta > 1$ there exists some $a_{\beta} \in \mathbb{H}^n$ such that $|f(a_{\beta})| < \beta$. Now, as $\alpha \in \langle f \rangle$, there exists some $k \in \mathbb{N}$ such that $\alpha \leq |f|^k$. Thus $\beta \leq |f|$ for $\beta \in \mathbb{H}$ such that $\beta^k = \alpha$ (such $\beta$ exists as $\mathbb{H}$ is divisibly closed, and $\beta > 1$ since $\alpha > 1$), a contradiction, since $|f(a_{\beta})| < \beta$.
\end{proof}

\begin{rem}\label{rem_bound_wrt_homomorphism}
Let $\phi : \mathbb{H}(x_1,...,x_n) \rightarrow \mathbb{H}(x_1,...,x_n)$ be an $\mathbb{H}$-homomorphism. If \linebreak $f \in \mathbb{H}(x_1,...,x_n)$ is bounded from below then $\phi(f)$ is bounded from below, and if $f$ is bounded from above then $\phi(f)$ is bounded from above.
\end{rem}
\begin{proof}
Let $f \in \mathbb{H}(x_1,...,x_n)$ be bounded from below. Then there exists some $\alpha > 1$ such that $|f| \dotplus \alpha = |f|$. As $|\phi(f)| = \phi(|f|) = \phi(|f| \dotplus \alpha) = \phi(|f|) \dotplus \phi(\alpha) = |\phi(f)| \dotplus \alpha$,  we have that $\phi(f)$ is bounded from below.\\ Let $f$ be bounded from above. Then there exists some $\alpha \geq 1$ such that $|f| \dotplus \alpha = \alpha$. As $\alpha = \phi(\alpha) = \phi(|f| \dotplus \alpha) = |\phi(f)| \dotplus \alpha$  we have that $\phi(f)$ is bounded from above.
\end{proof}

\begin{prop}
Let $\phi : \mathbb{H}(x_1,...,x_n) \rightarrow \mathbb{H}(x_1,...,x_n)$ be an $\mathbb{H}$-homomorphism. Then $\phi$ sends bounded from above principal kernels to bounded from above \linebreak principal kernels and bounded from below principal kernels to bounded from below principal kernels.
\end{prop}
\begin{proof}
The image of a kernel under a homomorphism is a kernel. Now, since \linebreak $\phi(\langle f \rangle) = \langle \phi(f) \rangle$ and since a principal kernel is bounded from above or bounded from below if and only if its generator is, the result follows from Remark \ref{rem_bound_wrt_homomorphism}.
\end{proof}

\begin{rem}\label{rem_bounded_from_above_elements_kernel}
$$\langle \mathbb{H} \rangle = \{ f \in \mathbb{H}(x_1,...,x_n) \ : \ f \ \text{is bounded from above} \ \}.$$
\end{rem}
\begin{proof}
First note that for $\alpha \in \mathbb{H}$ such that $\alpha \neq 1$, $\alpha$ generates $\mathbb{H}$ as a semifield and thus  $\langle \alpha \rangle = \langle \mathbb{H} \rangle$.
The assertion follows from Remark \ref{rem_up_bounded_is_in_H_kernel}, since $f \in \langle \alpha \rangle = \langle \mathbb{H} \rangle$ if and only if $\langle f \rangle \subseteq  \langle \alpha \rangle$.
\end{proof}

\ \\

As we shall now show, there is a close connection between bounded from above and bounded from below kernels:\\

Let $\langle f \rangle \in \PCon(\mathbb{H}(x_1,...,x_n))$ be a principal kernel. By Theorem \ref{thm_nother_1_and_3}(2) we have that
$$ \langle f \rangle / (\langle f \rangle \cap \langle \mathbb{H} \rangle) \cong \langle f \rangle \cdot \langle \mathbb{H} \rangle / \langle \mathbb{H} \rangle.$$
In view of the preceding discussion in this section, $\langle f \rangle \cdot \langle \mathbb{H} \rangle$ is the smallest bounded from below kernel containing $\langle f \rangle$ while $\langle f \rangle \cap \langle \mathbb{H} \rangle$ is the largest bounded from above kernel contained in $\langle f \rangle$. One can consider the quotient $ \langle f \rangle / (\langle f \rangle \cap \langle \mathbb{H} \rangle)$ as a measure for ``how far a kernel is from being bounded".

The kernel $\langle \mathscr{R} \rangle$ plays an important role in our theory. It actually contains all the \linebreak kernels needed to form a correspondence  between principal kernels and principal skeletons. In essence, every principal kernel $\langle f \rangle$ in $\PCon(\mathscr{R}(x_1,...,x_n))$ has an bounded from above ``copy" in the family of principal kernels contained in $\langle \mathscr{R} \rangle$, which possesses the exact same skeleton. Our so called ``Zariski correspondence" will take place between the principal kernels contained in $\langle \mathscr{R} \rangle$ and the skeletons in~$\mathscr{R}^n$.

\ \\

\subsection{The structure of $\langle \mathbb{H} \rangle$}

\ \\

Let $\mathbb{H}$ be a bipotent divisible archimedean semifield. In the following discussion, we study the structure of the subsemifield and kernel $\langle \mathbb{H} \rangle$ of $\mathbb{H}(x_1,...,x_n)$, which by Remark \ref{rem_the contant_generated_kernel} is generated, as a kernel, by any $\alpha \in \mathbb{H} \setminus \{1 \}$. As we have shown in Remark \ref{rem_bounded_from_above_elements_kernel}, $\langle \mathbb{H} \rangle$  is comprised of bounded from above elements of $\mathbb{H}(x_1,...,x_n)$. Note that $\mathbb{H}$ is always assumed to be a bipotent divisible semifield while any additional assumptions regarding $\mathbb{H}$ will be stated when necessary.\\

\begin{prop}\label{prop_unbounded_generator}
For any principal kernel $\langle f \rangle \in \PCon(\langle \mathbb{H} \rangle)$ bounded from above, there exists an unbounded from above kernel $\langle f' \rangle \in \PCon(\mathbb{H}(x_1,...,x_n)$ such that $$ \langle f \rangle = \langle f' \rangle \cap \langle \mathbb{H} \rangle.$$
In particular, $\langle f' \rangle \supset \langle f \rangle$  and $Skel(f') = Skel(f)$.
\end{prop}
\begin{proof}
Let $f(x_1,...,x_n) \in \langle \mathbb{H} \rangle$ be bounded from above. Then there exists some $\beta_1 \in \mathbb{H}$ such that
$|f(x_1,...,x_n)| = |\alpha_1| \in \mathbb{H}$  for every $|x_1| \geq \beta_1$ for otherwise $f$ will not be bounded from above since for every $\gamma \in \mathbb{H}$ there exists some $a =(a_1,...,a_n) \in \mathbb{H}^n$ with $|a_1| > \beta(\gamma)$ such that $|f(a,x_2,....,x_n)| > \gamma$. Similarly for each $2 \leq i \leq n$ there exists some $\beta_i \in \mathbb{H}$ such that $|f(x_1,...,x_n)| = |\alpha_i| \in \mathbb{H}$ for every $|x_i| \geq \beta_i$. As~$|f|$ is continuous we have that $\alpha_i = \alpha$ are all the same. Now define the following function
$$f' = | (\beta^{-1}|x_1| \wedge .... \wedge \beta^{-1}|x_n|+1)| + |f(x_1,...,x_n)|$$
where $\beta = \sum_{i=1}^{n}|\beta_i|$. Write $g(x_1,...x_n) = \beta^{-1}|x_1| + .... + \beta^{-1}|x_n|+1$. Let
$$S = \{x =(x_1,...,x_n) \in \mathbb{H}^n :  |x_i| > \beta \ \forall i \}.$$ Then for every $a \in S$, $f'(a)=g(a)+|\alpha|$. Moreover, for every $b=(b_1,...,b_n) \not \in S$ there exists some $j$ such that $|b_j| \leq \beta$   thus we have that $(\beta^{-1}|b_1| + .... + \beta^{-1}|b_n|) \leq 1$ and so $g(b) = 1$. By construction $Skel(f) \subseteq Skel(g)$, so  $Skel(f') = Skel(|g| + |f|) = Skel(g) \cap Skel(f) = Skel(f)$. Finally as $|g|$ is not bounded and $|f'| = |g| + |f| \geq |g|$, we have that $f'$ is not bounded.
Now, as $|f'| = |g| + |f|$ we have that $|f| \leq |f'|$, so $f \in \langle f' \rangle$. On the other hand, since $f'$ is not bounded from above, Remark \ref{rem_generator_of_up_bounded_kernel} implies that $f' \not \in \langle f \rangle$. Finally,  \ $g(a) \geq 1$ for any $a \in S$. Thus,  $f'(a) \wedge |\alpha| = (g(a) + |\alpha|) \wedge |\alpha| = |\alpha|$, while for $a \not \in S$ \ $f'(a) \wedge |\alpha| = (g(a) + |f(a)|) \wedge |\alpha| = (1 + |f(a)|) \wedge |\alpha| =  |f(a)|  \wedge |\alpha| = |f(a)|$, since $|f| \leq |\alpha|$. So we get that $|f'| \wedge |\alpha| = |f|$  which means that $\langle f \rangle = \langle f' \rangle \cap \langle \mathbb{H} \rangle$ \linebreak (Note that $f' = |f'|$ by definition, since $f' \geq 1$.)
\end{proof}

\begin{cor}\label{cor_unbounded_copy}
For every  $f \in \mathbb{H}(x_1,...,x_n)$ there exists some $f' \in \mathbb{H}(x_1,...,x_n)$ such that $f'$ is not bounded from above, $\langle f \rangle \subseteq \langle f' \rangle$ and $Skel(f') = Skel(f)$.
\end{cor}
\begin{proof}
By Proposition \ref{prop_unbounded_generator} for every $f$ bounded from above there exists such $f'$. On the other hand, if $f$ is not bounded from above, take $f' = f$.
\end{proof}

The following definition generalizes the notion of principal kernels bounded from above and from below introduced earlier.

\begin{defn}\label{defn_general_bounded_kernels}
Let $\mathbb{H}$ be a bipotent semifield.
A kernel $K \in \Con(\mathbb{H}(x_1,...,x_n))$ is said to be  \emph{bounded}
if $K = K \cap \langle \mathbb{H} \rangle$. $K$ is said to be \emph{bounded from below} if $\langle \mathbb{H} \rangle \subseteq K$ or equivalently $K = K \cdot \langle \mathbb{H} \rangle$.
\end{defn}

By definition \ref{defn_general_bounded_kernels} every kernel $K \subseteq \langle \mathbb{H} \rangle$ is a bounded kernel and vice versa.\\

\begin{rem}
Bounded kernels form a sublattice of $(\PCon(\mathbb{H}(x_1,...,x_n)), \cdot, \cap)$ (with respect to multiplication and intersection) and principal bounded kernels form a sublattice of $(\PCon(\mathbb{H}(x_1,...,x_n)), \cdot, \cap)$.
\end{rem}
\begin{proof}
For any $A,B \in \Con(\mathbb{H}(x_1,...,x_n)$
$$(A \cap \langle \mathbb{H} \rangle ) \cap (B \cap \langle \mathbb{H} \rangle) = (A \cap B) \cap \langle \mathbb{H} \rangle.$$
By Lemma \ref{lem_kernels_algebra}
$$(A \cap \langle \mathbb{H} \rangle )(B \cap \langle \mathbb{H} \rangle)= A \cdot B  \cap \langle \mathbb{H} \rangle.$$
Thus $(\mathrm{B},  \cdot, \cap)$  is a sublattice of $(\Con(\mathbb{H}(x_1,...,x_n)), \cdot, \cap)$.\\
Let $\langle f \rangle  \in \PCon(\mathbb{H}(x_1,...,x_n))$. Then $\langle f \rangle \cap \langle \mathbb{H} \rangle = \langle |f| + |\alpha| \rangle$ for any $\alpha \in \mathbb{H} \setminus \{1 \}$, thus  $\langle f \rangle \cap \langle \mathbb{H} \rangle$ is a bounded principal kernel.
\end{proof}

\begin{rem}\label{rem_wedge_preserves_1}
For any $f \in \mathbb{H}(x_1,...,x_n)$ and $\alpha \in \mathbb{H} \setminus \{1\}$
$$|f| \wedge |\alpha| = 1 \Leftrightarrow f = 1.$$
\end{rem}
\begin{proof}
$(\Leftarrow)$ \ is obvious. \\
$(\Rightarrow)$ \ is true since for any  $a \in \mathbb{H}^n$  \  $|f(a)| \wedge |\alpha| = 1$ if and only if $|f(a)|= 1$.
\end{proof}

\begin{rmind}
Let $\mathbb{S}$ be a semifield.
A kernel $K$ of a semifield $\mathbb{S}$ is said to be \emph{large} in $\mathbb{S}$ if  $L \cap K \neq \{1 \}$ for each kernel $L \neq \{ 1 \}$ of $\mathbb{S}$.
\end{rmind}

\begin{cor}\label{cor_intersection_with_the kernel_of_H}
Let $\langle f \rangle \in \PCon(\mathbb{H}(x_1,...,x_n))$. Then
$$\langle f \rangle \cap \langle \mathbb{H} \rangle = \{ 1 \}  \Leftrightarrow \langle f \rangle = \{ 1 \}.$$
Consequently,
$$K \cap \langle \mathbb{H} \rangle = \{ 1 \}  \Leftrightarrow K = \{ 1 \}$$
for any $K \in \Con(\mathbb{H}(x_1,...,x_n))$. Thus $\langle \mathbb{H} \rangle$ is large kernel in $\mathbb{H}(x_1,...,x_n).$
\end{cor}
\begin{proof}
The first statement follows from the equality $\langle f \rangle \cap \langle \mathbb{H} \rangle  = \langle |f| \wedge |\alpha| \rangle$ and \linebreak Remark~\ref{rem_wedge_preserves_1}. For the second statement, if $K \neq \{1\}$ then taking $f \in K \setminus \{1\}$ we have $$K \cap  \langle \mathbb{H} \rangle \supseteq \langle f \rangle \cap \langle \mathbb{H} \rangle \neq \{ 1 \}.$$
\end{proof}


\begin{rem}\label{rem_the_lattice_of_kernels_of_the_bounded_kernel}
The following hold for the kernels of the semifield $\langle \mathbb{H} \rangle$:
$$\Con(\langle \mathbb{H} \rangle) =  \{K \cap \langle \mathbb{H} \rangle : K \in \Con(\mathbb{H}(x_1,...,x_n)) \},$$
and
$$\PCon(\langle \mathbb{H} \rangle) = \{ \langle f \rangle \cap \langle \mathbb{H} \rangle : \langle f \rangle \in \PCon(\mathbb{H}(x_1,...,x_n)) \}.$$
Moreover, every kernel of $\langle \mathbb{H} \rangle$ is also a kernel of $\mathbb{H}(x_1,...,x_n).$
\end{rem}
\begin{proof}
Since $\langle \mathbb{H} \rangle$ is a subsemifield of $\mathbb{H}(x_1,...,x_n)$
by Theorem  \ref{thm_nother_1_and_3}(1) for any kernel $K$ of $\mathbb{H}(x_1,...,x_n)$, \ $K \cap \langle \mathbb{H} \rangle$ is a kernel of $\langle \mathbb{H} \rangle$, i.e., $K \cap \langle \mathbb{H} \rangle \in \Con(\langle \mathbb{H} \rangle$. Conversely, as $\mathbb{H}(x_1,...,x_n)$ is idempotent and commutative, by Remark \ref{rem_transitivity_for_idemotent_kernels} we have that any kernel $L$ of $\langle \mathbb{H} \rangle$ is a kernel of $\mathbb{H}(x_1,...,x_n)$. Moreover, since $L \subseteq \langle \mathbb{H} \rangle$ we have that $L = L \cap \langle \mathbb{H} \rangle$. The second equality holds since for any principal kernels $\langle f \rangle \in \PCon(\mathbb{H}(x_1,...,x_n))$, $\langle f \rangle \cap \langle \mathbb{H} \rangle = \langle |f| \wedge |\alpha| \rangle \in \PCon(\langle \mathbb{H} \rangle)$.
\end{proof}

\newpage

\section{The polar-skeleton correspondence}

In subsection \ref{Subsection:Kernels_of_Skeletons} we introduced the notion of a $\mathcal{K}$-kernels of  $\mathbb{H}(x_1,...,x_n)$ whom have been shown to correspond to skeletons in $\mathbb{H}^n$ (Proposition \ref{prop_correspondence_between_K_kernels_and_skeletons}). Our next aim is to characterize these special kind of kernels. In analogy to algebraic geometry, we are looking for our `radical ideals' which are shown to correspond to algebraic sets (zero sets) by the Nullstellensatz theorem. Unfortunately in our context `radicality' does help much since by Remark \ref{rem_radicality_of_ker} every kernel is power-radical. It actually turns out that in our theory, the role of radical ideals is played by polars, a notion originating from the theory of lattice ordered groups. One can think of the theory introduced in this section as the analogue to the celebrated Nullstellensatz theorem.\\

We begin our discussion here considering general idempotent semifields. We will later restrict ourselves to the designated semifield of fractions $\mathbb{H}(x_1,...,x_n)$ with $\mathbb{H}$ a bipotent semifield.

In the following section, $\mathbb{S}$ is always assumed to be an idempotent semifield.

\subsection{Polars}

\ \\

In this subsection we introduce the notion of a polar, borrowed from the theory of lattice ordered groups (\cite[section (2.2)]{OrderedGroups3}).
We will see that polars are a special kind of kernels and use them to construct the so called $\mathcal{K}$-kernels introduced in the previous section.
\ \\

Translating Proposition \ref{prop_convex_l_subgroups_distributive_lattice} to the language of idempotent semifields we have:
\begin{prop}\label{prop_distributive_lattice_of_kernels}
The lattice of kernels $\Con(\mathbb{S})$ of an idempotent semifield $\mathbb{S}$ is complete distributive and it satisfies the infinite distributive law:
\begin{equation}\label{eq_distributivity_law}
A \cap \left( \prod_{i} B_i \right) = \prod_{i}(A \cap B_i)
\end{equation}
for any $A,B_i \in \Con(\mathbb{S})$.
\end{prop}
\begin{proof}
Viewing $\mathbb{S}$ as an $\ell$-group, since $\mathbb{S}$ is commutative every convex $\ell$-subgroup is normal and thus a kernel. Thus Proposition \ref{prop_convex_l_subgroups_distributive_lattice} applies to $\Con(\mathbb{S})$.
\end{proof}

\begin{defn}\label{def_polar}
Let $A$ be a subset of $\mathbb{S}$. The \emph{polar} of $A$ is the set
\begin{equation}
A^{\bot} = \{ x  \in \mathbb{S} : |x| \wedge |a| = 1 \ ; \ \forall a \in A \}.
\end{equation}
For $a \in \mathbb{S}$  we shortly write $a^{\bot}$ for $\{ a \}^{\bot}$.
The set of all polars in $\mathbb{S}$ is denoted by $\mathscr{B}(\mathbb{S})$.
\end{defn}

\begin{rem}\label{rem_basic_properites_of_polar}
The following statements are direct consequences of Definition \ref{def_polar}.
For any  $K,L \subset \mathbb{S}$
\begin{enumerate}
  \item $K \subseteq L \Rightarrow K^{\bot} \supseteq L^{\bot}$.
  \item $K \subseteq K^{\bot \bot}$.
  \item $K^{\bot} = K^{\bot \bot \bot}$.
\end{enumerate}
\end{rem}

The following definition generalizes Definition \ref{defn_positive_cone of_fractions}:
\begin{defn}\label{defn_positive_cone}
Let $K$ be a kernel of an idempotent semifield $\mathbb{S}$. The \emph{positive cone} of $K$ is the set
$$K^{+} = \{ k \in K : k \geq 1 \}.$$
In particular the positive cone of $\mathbb{S}$ is $\mathbb{S}^{+} = \{ h \in \mathbb{S} : h \geq 1 \}$.
\end{defn}

\begin{rem}\label{rem_inf_prod_of_groups}
For a family of groups $\{ G_i : i \in I \}$, with an arbitrary index set $I$
$$\prod_{i \in I}G_i = \{ g_{i_1} \cdot \dots \cdot  g_{i_k} :  g_{i_j} \in G_{i_j}, k \in \mathbb{N} \}.$$
Let $G$ be a group. A subgroup $H \leq G$ is said to be generated by a family of subgroups $\{ G_i : i \in I \}$ of $G$, with an arbitrary index set $I$, if
$$H = \prod_{i \in I}G_i.$$
In particular for every $h \in H$ there are some $i_1,...,i_k$ such that $h = g_{i_1} \cdot \dots \cdot  g_{i_k}$ where $g_{i_1} \in G_{i_1}, ...,g_{i_k} \in G_{i_k}$.
\end{rem}

\begin{prop}\label{kernel_as_a_product}\cite[Theorem (2.2.4)(c)]{OrderedGroups3}
The subgroup $K$ of an idempotent \linebreak semifield $\mathbb{H}$ generated by a family of kernels $\{ K_i : i \in I \}$ (where $I$ is an arbitrary index set) is a kernel,
and its positive cone $K^{+} = \{ k \in K : k \geq 1 \}$ is the subsemigroup of $\mathbb{H}^{+}$ generated by
the corresponding family of positive cones.
\end{prop}

\begin{cor}
For any kernel $K$ of an idempotent semifield $\mathbb{H}$
$$K = \prod_{f \in K} \langle f \rangle = \langle \bigcup_{f \in K} \{f \} \rangle.$$
\end{cor}
\begin{proof}
For $a \in \prod_{f \in K} \langle f \rangle$, there are some $f_1,...,f_k \in K $ such that $a = g_{1} \cdot \dots \cdot  g_{k}$   where $g_{1} \in \langle f_{1} \rangle, ...,g_{k} \in \langle f_{k} \rangle.$
Since $K$ is a kernel, for any $f \in K$, \ $\langle f \rangle \subseteq K$ so \linebreak $a = g_{1} \cdot \dots \cdot  g_{k} \in K$. So $K \supseteq \prod_{f \in K} \langle f \rangle$. The opposite inclusion is obvious as each \linebreak $f \in K$ is by definition in $\prod_{f \in K} \langle f \rangle$.  Since the kernel generated by a kernel $K$ (as a set) is the kernel $K$ itself, the last equality holds.
\end{proof}

\begin{rem}
Let $\mathbb{H}$ be an idempotent semifield and let $g \in \mathbb{H}$. Then \\ $K \cap \langle g \rangle = \{1\}$ if and only if $\langle f \rangle \cap \langle g \rangle = \{1\}$ for every $f \in K$.
\end{rem}
\begin{proof}
If $K = \bigcup_{f \in K}\{ f \}$ then $K = \bigoplus_{f \in K} \langle f \rangle$.
Since the lattice of kernels of $\mathbb{S}$ admit the infinite distributive law \eqref{eq_distributivity_law} we have that
$$K \cap \langle g \rangle = \left(\bigoplus_{f \in K} \langle f \rangle\right) \cap \langle g \rangle = \prod_{f \in K}\left( \langle f \rangle  \cap \langle g \rangle\right)$$
from which the statement follows at once.
\end{proof}

\begin{rem}
Let $K$ be a kernel of \ $\mathbb{S}$. The polar of $K$ is the set
\begin{equation}
K^{\bot} = \{ x  \in \mathbb{S} : \langle x  \rangle \cap  K = \{1\} \}.
\end{equation}
\end{rem}
\begin{proof}
For any $a,b \in \mathbb{S}$ since  $\langle a \rangle \cap \langle b \rangle = \langle |a| \wedge |b| \rangle$,
we have that $$|a| \wedge |b| = 1 \Leftrightarrow \langle a \rangle \cap \langle b \rangle = \{ 1\}.$$
Moreover, for any kernel $K$ and any $x \in \mathbb{S}$, $K \cap \langle x \rangle = \{1\}$ if and only if $\langle a \rangle \cap \langle x \rangle =\{1\}$ for all $a \in K$.
\end{proof}

\begin{thm}\label{thm_properties of polars}
\begin{enumerate}
  \item For any subset $A$ of \ $\mathbb{S}$,  $A^{\bot}$ is a kernel of \ $\mathbb{S}$.
  \item If $L \in \Con(\mathbb{S})$, then for any $K \in \Con(\mathbb{S})$
  $$L \cap K = \{1\} \Leftrightarrow K \subseteq L^{\bot}.$$
  \item $(\mathscr{B}(\mathbb{S}), \cdot, \cap, \bot, \{1\},\mathbb{S})$ is a complete Boolean algebra.
\end{enumerate}
\end{thm}
\begin{proof}
1. \ By Theorem (2.2.4)(e) in \cite{OrderedGroups3}, $A^{\bot}$ is a convex $\ell$-subgroup. Since $\mathbb{S}$ is also commutative, we have that $A^{\bot}$ is normal and thus a kernel of $\mathbb{S}$.\\
2. See Theorem (2.2.4)(d) in \cite{OrderedGroups3}.\\
3. See \cite{OrderedGroups3}, Theorem (2.2.5). If  $\{A_i\}$ is a collection of subsets of $\mathbb{S}$
   over some index set,  then  $$\left(\bigcup_{i} A_i\right)^{\bot} = \bigcap_{i}A_i^{\bot}.$$
   Closure under negation is a consequence of (2).
\end{proof}

\begin{rem}\label{rem_polar_of_kernel}\cite{OrderedGroups3}
Let $A$ be a subset of the $\mathbb{S}$. Then $$A^{\bot} = G(A)^{\bot} = SF(A)^{\bot} = \langle A \rangle^{\bot}$$
where $G(A)$ and $SF(A)$ are the subgroup and subsemifield generated by $A$ in $\mathbb{S}$.
\end{rem}

\begin{rem}\label{rem_double polar of a polar}
For any polar $P$, \ $P^{\bot \bot} = P$.\\
Indeed, $P$ is a polar then $P = V^{\bot}$ for some $V \subseteq \mathbb{S}$. So $P^{\bot \bot} = (V^{\bot})^{\bot \bot} =V^{\bot} = P$.
\end{rem}

\begin{prop}\label{prop_minimal_polar_containing_a_set}
For any subset $S \subseteq \mathbb{S}$ , $S^{\bot \bot}$ is the minimal polar containing $S$.
\end{prop}
\begin{proof}
By the definition of a polar,  $S^{\bot \bot}$ is a polar of $S^{\bot}$ and $S \subset S^{\bot \bot}$.
Let $P$ be a polar such that $S \subseteq P$. Then since polar is inclusion reversing and  $S \in P$ we have that
$S^{\bot} \supseteq P^{\bot}$ and so $S^{\bot \bot} = (S^{\bot})^{\bot} \subseteq (P^{\bot})^{\bot}$. As $P$ is a polar we have by Remark \ref{rem_double polar of a polar} that $(P^{\bot})^{\bot} = P$ and thus $S^{\bot \bot} \subseteq P$.
\end{proof}

\begin{defn}
Let $S \subseteq \mathbb{S}$. We say that a polar $P$ is \emph{generated} by $S$ if $P~=~S^{\bot \bot}$. If $S = \{a \}$ then we also write $a^{\bot \bot}$ for the polar generated by $\{a\}$.
\end{defn}

\begin{rem}
The following statements hold:

\begin{enumerate}
  \item A polar $B$ is generated by itself.
  \item For any subset $A \subset \mathbb{S}$
            $$\langle A \rangle^{\bot \bot} = A^{\bot \bot} = G(A)^{\bot \bot} = SF(A)^{\bot \bot}$$
            where $G(A)$ and $SF(A)$ are the subgroup and subsemifield generated by $A$ in $\mathbb{S}$.

\end{enumerate}
\end{rem}
\begin{proof}
The first assertion follows directly from Remark \ref{rem_double polar of a polar} while the second one is a direct consequence of Remark \ref{rem_polar_of_kernel}.
\end{proof}

\begin{defn}
A polar $P$ of $\mathbb{S}$ is said to be \emph{principal} if there exists some $a \in \mathbb{S}$ such that $P  = a^{\bot \bot}$, i.e, $P$ is the polar generated by $a$. We denote the collection of principal polars of a semifield $\mathbb{S}$ by $\mathscr{PB}(\mathbb{S})$.
\end{defn}

\begin{prop}\label{prop_principal_polars_operations}\cite[Section (2.2) q-13]{OrderedGroups3} \ \\
The function $\phi : |a| \mapsto |a|^{\bot \bot}$ is a lattice homomorphism from $\mathbb{S}^{+}$ to $\mathscr{B}(\mathbb{S})$.
Namely, \linebreak for any $a,b \in \mathbb{S}$
$$(|a| \wedge |b|)^{\bot \bot} =  \phi(|a| \wedge |b|) = \phi(|a|) \cap \phi(|b|) = |a|^{\bot \bot} \cap |b|^{\bot \bot}$$  and
$$(|a| \dotplus |b|)^{\bot \bot} =  \phi(|a| \dotplus |b|) = \phi(|a|) \cdot \phi(|b|) = |a|^{\bot \bot} \cdot |b|^{\bot \bot}.$$
\end{prop}
\begin{proof}
Indeed, by Theorem \ref{thm_properties of polars} and Remark \ref{rem_polar_of_kernel} the following holds: \\
$(|a| \wedge |b|)^{\bot \bot} = (\langle |a| \wedge |b| \rangle)^{\bot \bot} = (\langle a \rangle \cap \langle b \rangle)^{\bot \bot} =(\langle a \rangle^{\bot} \cup \langle b \rangle^{\bot})^{\bot} = \langle a \rangle^{\bot \bot} \cap \langle b \rangle^{\bot \bot} = |a|^{\bot \bot} \cap |b|^{\bot \bot} = |a|^{\bot \bot} \cap |b|^{\bot \bot}$. The second assertion is proved analogously.
\end{proof}

\begin{rem}\label{rem_lattice_of_principal_polars}
$(\mathscr{PB}(\mathbb{S}), \cdot, \cap)$ is a sublattice of $(\mathscr{B}(\mathbb{S}), \cdot, \cap)$.
\end{rem}
\begin{proof}
Since for any  $a,b \in \mathbb{S}$ \ $(|a| \wedge |b|)^{\bot \bot} = (\langle |a| \wedge |b| \rangle)^{\bot \bot} = (\langle a \rangle \cap \langle b \rangle)^{\bot \bot}$, $(|a| \dotplus |b|)^{\bot \bot} = (\langle |a| \dotplus |b| \rangle)^{\bot \bot} = (\langle a \rangle \cdot \langle b \rangle)^{\bot \bot}$  and since $|a|^{\bot \bot} = a^{\bot \bot}$ we have that
$$(\langle a \rangle \cap \langle b \rangle)^{\bot \bot} =  a^{\bot \bot} \cap b^{\bot \bot}$$
and
$$(\langle a \rangle \cdot \langle b \rangle)^{\bot \bot} =  a^{\bot \bot} \cdot b^{\bot \bot}.$$
\end{proof}

\begin{cor}
$\mathscr{PB}(\mathbb{H}(x_1,...,x_n)), \cdot, \cap)$ is a sublattice of  $(\mathscr{B}(\mathbb{H}(x_1,...,x_n)), \cdot, \cap)$ having $\mathbb{H}(x_1,...,x_n)$ and $\{ 1\}$ as its maximal and minimal elements respectively.
\end{cor}
\begin{proof}
$\mathbb{H}(x_1,...,x_n) = \alpha^{\bot \bot}$ for any $\alpha \in \mathbb{H} \setminus \{1\}$ since $|f| \wedge |\alpha| = 1$ if and only if \linebreak $f = 1$ and $1^{\bot} =  \mathbb{H}(x_1,...,x_n)$.
$\{1\} = 1^{\bot \bot}$ as $1^{\bot \bot} =\mathbb{H}(x_1,...,x_n)^{\bot} \subseteq \alpha^{\bot} = \{1\}$.
Thus $\{1\}, \mathbb{H}(x_1,...,x_n)  \in \mathscr{PB}(\mathbb{H}(x_1,...,x_n))$. The rest of the assertion is a special case of Remark \ref{rem_lattice_of_principal_polars}.
\end{proof}

\ \\
\ \\

\begin{prop}\cite[section (2.2)q-13]{OrderedGroups3}
The following statements hold  for any idempotent semifield $\mathbb{S}$:
\begin{enumerate}
  \item If $\{ B_i : i \in I \} \subseteq \mathscr{B}(\mathbb{S})$ then $\prod_{i}B_i = (\bigcup_{i}B_i)^{\bot \bot}$.
  \item If $B \in \mathscr{B}(\mathbb{S})$ then $B = \prod_{b \in B}b^{\bot \bot}$.
  \item For any $a,b \in \mathbb{S}$ \ $(|a|\cdot|b|)^{\bot \bot} = (|a| \dotplus |b|)^{\bot \bot}$.
\end{enumerate}
\end{prop}

Recall that a kernel $K$ of an semifield $\mathbb{S}$ is a large kernel if $K \cap L \neq \{ 1 \}$ for each kernel $L$ of $\mathbb{S}$.\\

\begin{rem}
Let $K$ be a kernel of an idempotent semifield $\mathbb{S}$. Then $K$ is large as a subkernel of $K^{\bot \bot}$.
\end{rem}
\begin{proof}
First $K \subset K^{\bot \bot}$ and $K^{\bot \bot} \in \Con(\mathbb{S})$ thus $K \in \Con(K^{\bot \bot})$. If $L~\in~\Con(K^{\bot \bot})$ then in particular $L \in \Con(\mathbb{S})$ (since  $\Con(K^{\bot \bot}) \subseteq \Con(\mathbb{S})$). If  $K \cap L = \{1\}$ then
$L \subseteq  K^{\bot}$ but also $L \subseteq  K^{\bot \bot}$, thus $L \cap L = \{1\}$which yields that $L = \{ 1 \}$. So, $K \cap L = \{1\} \Rightarrow  L = \{1\}$ for any $L \in \Con(K^{\bot \bot})$ and $K$ is a large kernel in $K^{\bot \bot}$.
\end{proof}

\begin{thm}\label{thm_completely_closed_polar_equivalence}\cite[Theorem (2.3.7)]{OrderedGroups3}
Consider the following conditions on the \linebreak kernel $K$ of the idempotent semifield $\mathbb{S}$:
\begin{enumerate}
  \item $K$ is a polar.
  \item $K$ is completely closed in $\mathbb{S}$.
\end{enumerate}
Then (1) implies (2), and, if $\mathbb{S}$ is complete, (2) implies (1).
\end{thm}

\newpage

\subsection{The polar-skeleton correspondence}
\ \\

Recall $\mathscr{R}$ is our designated semifield defined to be bipotent, divisible, archimedean and complete.
Also recall that the semifield of fractions $\mathscr{R}(x_1,...,x_n)$ is considered as mapped to the semiring of functions $Fun(\mathscr{R}^n,\mathscr{R})$ as defined in \ref{def_semiring_of_functions}.\\

Recall that since $\mathscr{R}$ is idempotent and archimedean semifield $Fun(\mathscr{R}^n,\mathscr{R})$ is an idempotent and archimedean semifield and so is $\mathscr{R}(x_1,...,x_n)$ (cf. Proposition \ref{prop_semifield_of_fractions_is_archimedean}). Moreover, $Fun(\mathscr{R}^n,\mathscr{R})$ is also complete since $\mathscr{R}$ is complete (cf. Remark \ref{rem_complete_semifield_of_functions}).\\
By Theorem \ref{thm_completion_of_l_group}, $\mathscr{R}(x_1,...,x_n)$ has a unique completion to a complete archimedean idempotent semifield $\overline{\mathscr{R}(x_1,...,x_n)}$ in $Fun(\mathscr{R}^n,\mathscr{R})$. By Theorem \ref{thm_completion_of_a_semifield}, $\mathscr{R}(x_1,...,x_n)$ is dense in $\overline{\mathscr{R}(x_1,...,x_n)}$. \\

In this section we concentrate our attention to $\overline{\mathscr{R}(x_1,...,x_n)}$. Doing so, we consider the natural extensions to $\overline{\mathscr{R}(x_1,...,x_n)}$ of the operators $Skel$ and $Ker$ defined in Subsection \ref{Subsection:Kernels_of_Skeletons} with respect to $\mathscr{R}(x_1,...,x_n)$. We write $Ker_{\mathscr{R}}(x_1,...,x_n)$ to denote the \linebreak restriction of $Ker$ to $\mathscr{R}(x_1,...,x_n)$.\\

\begin{prop}\label{prop_polarity_and_skeletons_infinite_connection}
Let $K$ be a kernel of $\mathscr{R}(x_1,...,x_n)$. If $\langle g \rangle \cap K = \{ 1 \}$ for some \linebreak $g \in \mathscr{R}(x_1,...,x_n)$ then $$Skel(g) \cup Skel(K) = \mathscr{R}^{n}.$$
\end{prop}
\begin{proof}
If $K \cap \langle g \rangle = \{1\}$ then $\langle f \rangle \cap \langle g \rangle = \{ 1 \}$ for every $f \in K$. Thus as \linebreak $Skel(K) =\bigcap_{f \in K}Skel(f)$ we have that $Skel(K) \cup Skel(g) = (\bigcap_{f \in K}Skel(f)) \cup Skel(g)$ \linebreak $=\bigcap_{f \in K}(Skel(f) \cup Skel(g)) = \bigcap_{f \in K}\mathscr{R}^n = \mathscr{R}^n$. Note that infinite distributive laws hold for sets with respect to intersections and unions.
\end{proof}

Note that the same arguments hold taking $\overline{\mathscr{R}(x_1,...,x_n)}$ instead of $\mathscr{R}(x_1,...,x_n)$.

\begin{rem}\label{rem_carachterization_of_completions}
As any kernel $K \in \Con(\mathscr{R}(x_1,...,x_n))$ is a bipotent archimedean semifield in its own right, it has a completion too, which we denote by $\bar{K}$ and $\bar{K} \subset \overline{\mathscr{R}(x_1,...,x_n)}$.
By Corollary \ref{cor_completion_of_kernels} $\bar{K}$ is a kernel of $\overline{\mathscr{R}(x_1,...,x_n)}$.\\ 

The question arises which kernels of $\overline{\mathscr{R}(x_1,...,x_n)}$ are completions of kernels of $\mathscr{R}(x_1,...,x_n)$. The answer to this question is all kernels of $\overline{\mathscr{R}(x_1,...,x_n)}$ that are \linebreak completely closed. Indeed,  $\mathscr{R}(x_1,...,x_n)$ is a subsemifield of $\overline{\mathscr{R}(x_1,...,x_n)}$, so by \linebreak Theorem  \ref{thm_nother_1_and_3}(1) we have that
$$\Con(\mathscr{R}(x_1,...,x_n)) = \{ K \cap \mathscr{R}(x_1,...,x_n) :  K \in \Con(\overline{\mathscr{R}(x_1,...,x_n)}).$$
Since $\mathscr{R}(x_1,...,x_n)$ is dense in $\overline{\mathscr{R}(x_1,...,x_n)}$ for every kernel $K$ of $\overline{\mathscr{R}(x_1,...,x_n)}$, the kernel $L = K \cap \mathscr{R}(x_1,...,x_n)$ of $\mathscr{R}(x_1,...,x_n)$ is dense in $K$. Thus $\bar{L} \supseteq K$ where $\bar{L}$ is the completion of $L$ and so, since $L \subseteq K$, we have that $\bar{L} = \bar{K}$, i.e., $\bar{L} = K$ if and only if $K$ is completely closed.
\end{rem}

\begin{cor}\label{cor_correspondence_of_kernels_with_the_completions}
By Remark \ref{rem_carachterization_of_completions}, each completely closed kernel $K$ of $\overline{\mathbb{H}(x_1,...,x_n)}$ defines a unique kernel of $\mathbb{H}(x_1,...,x_n)$ given by $L = K \cap \mathbb{H}(x_1,...,x_n)$ for which $\overline{L} = K$.
\end{cor}

\begin{exmp}\label{exmp_completion_of_bounded_kernel}
Consider the kernel $K = \langle |x| \wedge |\alpha| \rangle \in \PCon(\mathscr{R}(x))$ and the subset \linebreak $X = \{|x| \wedge \alpha^n : n \in \mathbb{N}  \}$. Then since $X \subset K$ one has that $|x| = \bigvee_{f \in X}f \in \overline{K}$, thus \linebreak $\langle x \rangle  \subset \overline{K}$ which yields that $\overline{\langle |x| \wedge |\alpha| \rangle} = \overline{\langle |x| \rangle}$.
\end{exmp}

\begin{rem}
For every $K \in \mathscr{R}(x_1,...,x_n)$
$$Skel(\bar{K}) = Skel(K)$$
where $\bar{K}$ is the completion of $K$ in $\overline{\mathscr{R}(x_1,...,x_n)}$.
\end{rem}
\begin{proof}
First note that since $K \subset \bar{K}$ in $\overline{\mathscr{R}(x_1,...,x_n)}$ we have that \linebreak $Skel(\bar{K}) \subseteq Skel(K)$. Now,  let $Z = Skel(K)$ and let $S$ be any nonempty subset of $K$. If $\bigvee_{s \in S}s \in \overline{\mathscr{R}(x_1,...,x_n)}$ then for any $a \in Z$ \ $(\bigvee_{s \in S}s) (a) = \bigvee_{s \in S}s(a) = \bigvee_{s \in S}1 =1$, yielding that $Skel(\bigvee_{s \in S}) \supseteq Z$ (cf. Remark \ref{rem_complete_semifield_of_functions}). Similarly, if $\bigwedge_{s \in S}s \in \overline{\mathscr{R}(x_1,...,x_n)}$ then for any $a \in Z$ we have that $(\bigwedge_{s \in S}s) (a) = \bigwedge_{s \in S}s(a) = \bigwedge_{s \in S}1 =1$, yielding that $Skel(\bigwedge_{s \in S}) \supseteq Z$. Thus as none of its supplementary elements reduces the size of $Z$, we have that $Skel(\bar{K}) = Skel(K)$.
\end{proof}

\begin{note}
Let $\mathbb{S}$ be an idempotent semifield and let $H$ be a subset of $\mathbb{S}$. For a subset $A$ of $\mathbb{S}$ let
$$A^{\bot  H} = \{ h \in H : |h| \wedge |a| = 1  \ \ \forall a \in  A \},$$
then $A^{\bot  H} = A^{\bot \mathbb{S}} \cap H$.\\
We only write $A^{\bot H}$ when $H$ is a proper subset of the semifield in which $A$ is considered. For example if $K \in \Con(\overline{\mathscr{R}(x_1,...,x_n)})$ then $K^{\bot} = K^{\bot \overline{\mathbb{H}(x_1,...,x_n)}}$.
\end{note}

\begin{defn}
Let $K$ be a kernel of $\overline{\mathscr{R}(x_1,...,x_n)}$. Then $K$ is said to be a \emph{$\mathcal{K}$-kernel}
if $$K = Ker(Skel(K)).$$
In other words, $K$ is a preimage of its skeleton with respect to the map $$Ker : \mathbb{P}(\mathbb{H}^n) \rightarrow \Con(\overline{\mathbb{H}(x_1,...,x_n)})$$  where $\mathbb{P}(X)$ is the powerset of the set $X$.
\end{defn}

\begin{prop}\label{prop_polar_is_a_K_kernel}
A polar in  $\mathscr{B}(\overline{\mathscr{R}(x_1,...,x_n)})$ is a $\mathcal{K}$-kernel.
\end{prop}
\begin{proof}
Let $K = V^{\bot}$ for some $V \subset \overline{\mathscr{R}(x_1,...,x_n)}$ and let $g \in \overline{\mathscr{R}(x_1,...,x_n)}$ such that $Skel(g) \supseteq Skel(K)$. Assume $g \not \in K$, then there exists some $v \in V$ such that $|g| \wedge |v| \neq 1$ so $Skel(g) \cup Skel(v) = Skel(\langle g \rangle \cap \langle v \rangle) = Skel(|g| \wedge |v|) \neq \mathscr{R}^n$. Since $K = V^{\bot}$, by definition $K \cap \langle v \rangle = \{1\}$, so, Proposition \ref{prop_polarity_and_skeletons_infinite_connection} implies that $Skel(K) \cup Skel(v) = \mathscr{R}^n$.    But $Skel(g) \supseteq Skel(K)$ implies that $Skel(g)\cup Skel(v) \supseteq Skel(K) \cup Skel(v) = \mathscr{R}^n$. A contradiction. Thus $g \in K$.
\end{proof}

\begin{prop}\label{prop_mathcal_K_kernel_is_completely_closed}
Every $\mathcal{K}$-kernel of $\overline{\mathscr{R}(x_1,...,x_n)}$ is  completely closed.
\end{prop}
\begin{proof}
By Proposition \ref{prop_skel_ker_relations} if  $Z =Skel(K)$ then  $Skel(Ker(Z)) = Z$. \linebreak
Let $Z = Skel(K) \subseteq \mathscr{R}^n$ where  $K = Ker(Z) \in \Con(\overline{\mathscr{R}(x_1,...,x_n)})$. Let $S \subseteq K$ be any nonempty subset of $K$ then for any $s \in S$ we have that  $Skel(s) \supseteq Z$. \\
Now, if $\bigvee_{s \in S}s \in \overline{\mathscr{R}(x_1,...,x_n)}$ and $\bigwedge_{s \in S}s \in \overline{\mathscr{R}(x_1,...,x_n)}$, then by Remark \ref{rem_complete_semifield_of_functions},\linebreak
for any $a \in Z$
$$(\bigvee_{s \in S}s) (a) = \bigvee_{s \in S}s(a) = \bigvee_{s \in S}1 =1 \ \ \text{and} \ \
(\bigwedge_{s \in S}s) (a) = \bigwedge_{s \in S}s(a) = \bigwedge_{s \in S}1 =1.$$
Thus
$$Skel(\bigvee_{s \in S}s), Skel(\bigwedge_{s \in S}s) \supseteq Z.$$
So $\bigvee_{s \in S}s, \bigwedge_{s \in S}s \in K$ and $K$ is completely closed in $\overline{\mathscr{R}(x_1,...,x_n)}$.
\end{proof}

\begin{rem}
Proposition \ref{prop_mathcal_K_kernel_is_completely_closed} is not true when considering $\mathcal{K}$-kernels of \linebreak $\mathscr{R}(x_1,...,x_n)$ instead of $\overline{\mathscr{R}(x_1,...,x_n)}$. Let $\alpha \in \mathscr{R}$ such that $\alpha > 1$. Consider the subset
$$X = \{  |x^n| \wedge \alpha : n \in \mathbb{N}\},$$
then $X \subset \langle x \rangle$, \ $\bigvee_{f \in X}f = \alpha$ (where $\alpha$ is the constant function) and $Skel(f) = \{ 1 \} \subset \mathscr{R}$ for every $f \in X$. Thus $\alpha = (\alpha)(1) = (\bigvee_{f \in X}f) (1) \neq \bigvee_{f \in X}f(1) = 1$ and $\alpha$ is not in the preimage of $Skel(x)$ (see figure \ref{fig:Expol}).\\ So we deduce that $\mathcal{K}$-kernels of $\mathscr{R}(x_1,...,x_n)$ are not necessarily completely closed \linebreak in $\mathscr{R}(x_1,...,x_n)$. \\
Since every polar is completely closed (cf. Theorem \ref{thm_completely_closed_polar_equivalence}), by the example given above we have that $\alpha \in x^{\bot \bot}$ which yields that $x^{\bot} = x^{\bot \bot \bot} = (x^{\bot \bot})^{\bot} =\{1\}$.
\end{rem}

\begin{figure}
\centering

\begin{tikzpicture}[scale=3,
    axis/.style={thin, color=gray, ->},
    ar important line/.style={very thick, ->},
    important line/.style={very thick},
    dashed line/.style={thick,dashed}]

    \draw[axis] (-1,0)  -- (1,0) node[right, color=black] {$x$} ;
    \draw[axis] (0,-1) -- (0,1);
    \draw[important line] (0,0)  -- (-0.7,0.7)  node[below= 0.5cm] {$|x^1| \wedge \alpha$};
    \draw[dashed line] (0,0)  -- (-0.35,0.7)  node[above= 0.3cm,left=0cm] {$|x^2| \wedge \alpha$};
    \draw[dashed line] (0,0)  -- (-0.1,0.7)  node[above= 0.3cm,right=0.2cm] {$|x^k| \wedge \alpha$};
    \draw[ar important line] (0,0)  -- (0,0.7) -- (1,0.7);
    \draw[ar important line] (0,0)  -- (0,0.7) -- (-1,0.7);
    \draw [black, fill=black] (0,0) circle (1pt);
    \draw [black, fill=black] (0,0.7) circle (0.3pt)node[above= 0.15cm,left=0.01cm] {$\alpha$};
    \draw[important line] (0,0) -- (0.7,0.7) -- (1,0.7) node[below=0.2cm,right=0.2cm]{} ;
    \draw[dashed line] (0,0)  -- (0.35,0.7);
    \draw[dashed line] (0,0)  -- (0.1,0.7);

 \end{tikzpicture}

\caption{$X = \{  |x^n| \wedge \alpha : n \in \mathbb{N}\}$ and its supremum.}
\label{fig:Expol}

\end{figure}
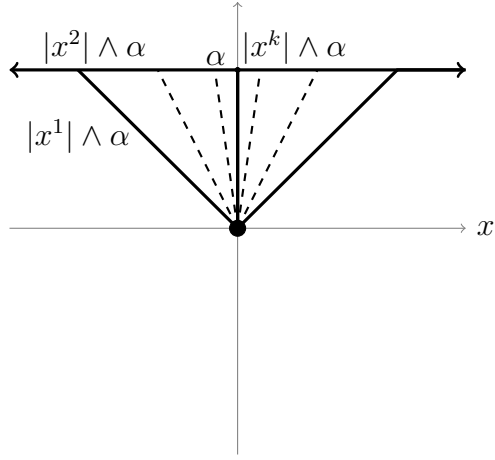

%

\begin{cor}\label{cor_mathcal_K_kernel_is_a_polar}
Every $\mathcal{K}$-kernel of $\overline{\mathscr{R}(x_1,...,x_n)}$ is a polar.
\end{cor}
\begin{proof}
By Proposition \ref{prop_mathcal_K_kernel_is_completely_closed}, every $\mathcal{K}$-kernel of $\overline{\mathscr{R}(x_1,...,x_n)}$ is completely closed and  thus by Theorem \ref{thm_completely_closed_polar_equivalence} is a polar since $\overline{\mathscr{R}(x_1,...,x_n)}$ is complete.
\end{proof}

Summarizing the above assertions we have
\begin{cor}\label{cor_polar_k_kernel_completion_equivalence}
Let $K$ be a kernel of $\overline{\mathscr{R}(x_1,...,x_n)}$. Then the following statements are equivalent:
 \begin{enumerate}
   \item $K$ is a $\mathcal{K}$-kernel.
   \item $K$ is a polar.
   \item $K$ is completely closed.
 \end{enumerate}
\end{cor}

\begin{prop}\label{prop_principal_polars_and_principal_skeletons}
For any $f \in \overline{\mathscr{R}(x_1,...,x_n)}$,
$$Skel(f) = Skel(f^{\bot \bot}) \  \text{and} \ f^{\bot \bot} = Ker(Skel(f)).$$
\end{prop}
\begin{proof}
Since $f^{\bot \bot} \supseteq \langle f \rangle$ we have that $Skel(f^{\bot \bot}) \subseteq Skel(f)$. Let $K$ be the $\mathcal{K}$-kernel such that $Skel(K) = Skel(f)$. Then $f \in K$ and $K$ is a polar by Corollary \ref{cor_mathcal_K_kernel_is_a_polar}. By Proposition \ref{prop_minimal_polar_containing_a_set}, \ $f^{\bot \bot}$ is the minimal polar containing $f$ thus  $K \supseteq f^{\bot \bot}$ and so $Skel(f) = Skel(K) \subseteq Skel(f^{\bot \bot}) \subseteq Skel(f)$.  Finally, as $f^{\bot \bot}$ is a polar it is a $\mathcal{K}$-kernel and so $f^{\bot \bot}= Ker(Skel(f^{\bot \bot})) = Ker(Skel(f))$.
\end{proof}

\begin{thm}\label{thm_polar_skeleton_corrsepondence}
There is a $1:1$ correspondence
\begin{equation}\label{eq_polar_skeleton_corrsepondence}
\mathscr{B}(\overline{\mathscr{R}(x_1,...,x_n)}) \leftrightarrow Skl(\mathscr{R}^n)
\end{equation}
between the skeletons in $\mathscr{R}^{n}$ and the polars of $\overline{\mathscr{R}(x_1,...,x_n)}$
given by $B \mapsto Skel(B)$ and $Z \mapsto Ker(Z)$.

This correspondence restricts to a correspondence
\begin{equation}\label{eq_polar_skeleton_corrsepondence}
\mathscr{PB}(\overline{\mathscr{R}(x_1,...,x_n)}) \leftrightarrow PSkl(\mathscr{R}^n)
\end{equation}
between the principal skeletons in $\mathscr{R}^{n}$ and the principal polars of $\overline{\mathscr{R}(x_1,...,x_n)}$.
\end{thm}
\begin{proof}
By Corollary \ref{cor_mathcal_K_kernel_is_a_polar} and Proposition \ref{prop_polar_is_a_K_kernel}, $B$ is a polar if and only if $B$ is  a \linebreak $\mathcal{K}$-kernel thus $Ker(Skel(B)) = B$.
If $B = f^{\bot \bot}$ for some $f \in \overline{\mathscr{R}(x_1,...,x_n)}$, so by Proposition \ref{prop_principal_polars_and_principal_skeletons}  $Skel(B) = Skel(f)$ and $Ker(Skel(f^{\bot \bot})) = Ker(Skel(f)) = f^{\bot \bot}$.
\end{proof}

\ \\
\begin{cor}\label{cor_polar_skeleton_corrsepondence}
Let $$\mathrm{B} = \{ B \cap \mathscr{R}(x_1,...,x_n) : B \in \mathscr{B}(\overline{\mathscr{R}(x_1,...,x_n)}) \}$$
and $$\mathrm{P}\mathrm{B} = \{ B \cap \mathscr{R}(x_1,...,x_n) : B \in \mathscr{PB}(\overline{\mathscr{R}(x_1,...,x_n)})) \}.$$
Since $(\mathscr{B}(\mathbb{H}(x_1,...,x_n)), \cdot, \cap)$ is a lattice and $\mathscr{PB}(\mathbb{H}(x_1,...,x_n)), \cdot, \cap)$ is a sublattice of  $(\mathscr{B}(\mathbb{H}(x_1,...,x_n)), \cdot, \cap)$, $\mathrm{B}$ is a lattice and $\mathrm{P}\mathrm{B}$ is a sublattice of $\mathrm{B}$.\\
By Corollary \ref{cor_correspondence_of_kernels_with_the_completions}, the correspondence in Theorem \ref{eq_polar_skeleton_corrsepondence} yields a correspondence
\begin{equation}\label{eq_cor_polar_skeleton_corrsepondence}
 \mathrm{B} \leftrightarrow Skl(\mathscr{R}^n)
\end{equation}
given by $K \mapsto Skel(K)$ and $Z \mapsto Ker(Z) \cap \mathscr{R}(x_1,...,x_n) = Ker_{\mathscr{R}}(x_1,...,x_n)(Z)$.\\
By Corollary \ref{cor_polar_k_kernel_completion_equivalence} and Remark \ref{rem_carachterization_of_completions} for a kernel $K \in \Con(\mathscr{R}(x_1,...,x_n))$ we have that $Ker_{\mathscr{R}}(x_1,...,x_n)(Skel(K)) = \overline{K} \cap \mathscr{R}(x_1,...,x_n)$. \\
By Theorem \ref{thm_polar_skeleton_corrsepondence}, \eqref{eq_cor_polar_skeleton_corrsepondence} restricts to a correspondence
\begin{equation}\label{eq_polar_skeleton_corrsepondence}
\mathrm{P}\mathrm{B} \leftrightarrow PSkl(\mathscr{R}^n).
\end{equation}
\end{cor}

\ \\

\begin{exmp}
For the principal kernel $\langle x \rangle \in \PCon(\mathscr{R}(x))$ we have that \linebreak $Ker_{\mathscr{R}(x)}(Skel(\langle x \rangle)) = \overline{\langle x \rangle} \cap \mathscr{R}(x) = \langle x \rangle$. Analogously,  for the principal kernel \linebreak $\langle x \rangle \cap \langle \mathscr{R} \rangle = \langle |x| \wedge |\alpha| \rangle \in \PCon(\mathscr{R}(x))$, in view of Example \ref{exmp_completion_of_bounded_kernel},  we have that \linebreak $Ker_{\mathscr{R}(x)}\left(Skel\left(\langle |x| \wedge |\alpha| \rangle\right)\right) = \overline{\langle |x| \wedge |\alpha| \rangle} \cap \mathscr{R}(x) = \overline{\langle x \rangle} \cap \mathscr{R}(x) =  \langle x \rangle$.\\
Thus the $\mathcal{K}$-kernel corresponding to $\langle x \rangle$ is the same as the $\mathcal{K}$-kernel corresponding to the (bounded from above) kernel $\langle |x| \wedge |\alpha| \rangle$.
\end{exmp}

The following example illustrates the necessity of working in $\overline{\mathscr{R}(x_1,...,x_n)}$ instead of $\mathscr{R}(x_1,...,x_n)$.
\begin{exmp}
Consider the subset $X = \{ \alpha |x|^{-n} \wedge 1 : n \in \mathbb{N} \}$ in $\mathscr{R}(x)$ where $\alpha \in \mathscr{R}$, $\alpha > 1$. Then
$$\sup X = f(x) = \begin{cases}
\alpha  &  \  x=1; \\
1  & \ x \neq 1; \end{cases}.$$
Evidently $f \in \overline{\mathscr{R}(x)} \setminus \mathscr{R}(x)$.\\
While $x^{\bot \mathscr{R}(x)} = \{ 1 \}$ we have that $f \in x^{\bot \overline{\mathscr(x)}}$ thus we get that
$x^{\bot \mathscr{R}(x) \bot \mathscr{R}(x)} = \mathscr{R}(x)$ whereas $x^{\bot \overline{\mathscr{R}(x)} \bot \mathscr{R}(x)} \subset \mathscr{R}(x)$. For example, since $|\beta \cdot 1| \wedge |f(1)| = |\beta| \wedge |\alpha| > 1$ we get that $|\beta x| \wedge |f| \neq 1$ for any $\beta \neq 1$, so $|\beta x| \not \in x^{\bot \mathscr{R}(x) \bot \mathscr{R}(x)}$.
\end{exmp}
\textbf{I thank Prof. Kalle Karu for pointing out the example regarding the polar of $\langle x \rangle$.}

\newpage
\subsection{Appendix: The Boolean algebra of polars and the \\ \ \ \ \ Stone representation}

\ \\

As stated in Theorem \ref{thm_properties of polars} for any idempotent semifield $\mathbb{S}$, the collection of polars of $\mathbb{S}$, $\mathscr{B}(\mathbb{S})$ forms a complete Boolean algebra with respect to intersection multiplication (corresponding to union) and ${\bot}$ (corresponding to negation). There is a representation for Boolean algebras called the Stone representation. We will now provide a brief overview of this representation applied to the complete Boolean algebra of polars of $\mathbb{S}$. \\
The complete construction can be found in \cite{OrderedGroups3}, (Section (3.2)).
\ \\

\begin{defn}
Let $X$ be a topological space. The collection
of all clopen (closed which are also open) subsets of $X$ is a Boolean subalgebra of the Boolean algebra of all
subsets of $X$. This collection is called the \emph{dual algebra} of X.
\end{defn}

\begin{defn}
A compact Hausdorff space whose clopen sets form a base is called a \emph{Boolean space}.
\end{defn}

\begin{rem}
A closed subset $Y$ of a Boolean space $X$ is also a Boolean space (with respect to the induced subspace topology) and each clopen set in $Y$ is the intersection with $Y$ of a clopen set in $X$.
\end{rem}

\begin{defn}
A \emph{Stone space} is a totally disconnected Boolean space which means that every open set is the union
of clopen sets.
\end{defn}

\begin{defn}
A \emph{Boolean homomorphism} between two Boolean algebras is a lattice homomorphism
that preserves complements.
\end{defn}

\begin{defn}
 The set
$\mathbf{2} = \{ 0, 1\}$  is the totally ordered Boolean algebra which will also be considered as a topological
space, giving it the discrete topology. Given the Boolean algebra $B$, the set
$Bool[B,2]$ of all 2-valued homomorphisms on $B$ is a subspace of the product
space $\mathbf{2}^{B}$. $Bool[B,2]$ is called the \emph{dual space} of $B$.
\end{defn}

\begin{rem}
\begin{itemize}
  \item For a Boolean algebra $B$, the dual space $Bool[B,2]$ is a Boolean space. If $B$ is complete then $Bool[B,2]$ is a Stone space.
  \item Each Boolean algebra $B$ is isomorphic to its second
    dual, i.e., the dual algebra $\mathcal{B}$ of the Boolean space $Bool[B,2]$. The isomorphism $\alpha : B \rightarrow \mathcal{B}$ is given by $$ \alpha(b) = \{ f \in Bool[B,2] : f(b) = 1 \}$$

  \item An element of $Bool[B,2]$ is completely determined
    by its kernel which is a maximal ideal of $B$. Thus $Bool[B,2]$
    may be replaced by the set $Spec(B)$ which consists of all of the maximal ideals of
    $B$. The basic clopen sets of $Spec(B)$ are of the form $\{ V(b) : M \in Spec(B) : b \not \in M \}$ where $b \in B$ , $V$ is an isomorphism between $B$ and the algebra of clopen sets in $Spec(B)$. This topology on $Spec(B)$ is called the Zariski topology of $Spec(B)$.
  \item Each Boolean space is isomorphic to its second dual.
    Namely, let $A$ be the dual algebra of a Boolean space $X$, and
    let $Y = Bool[A,2]$ be the dual space of A. Then the function $\beta : X \rightarrow Y$ given by
    $$\beta(x)(P) = 1  \ \text{if} \  x \in P  \ \text{and} \ \beta(x)(P) = 0  \ \text{if} \ x \not \in P$$
    is a homeomorphism.

\end{itemize}
\end{rem}

In view of the above, $Spec(\mathscr{B}(\mathbb{S}))$ is a Stone space.
\ \\

\begin{defn}
Let $X$ be a topological space. Define
$$E(X) = \{f : X \rightarrow \bar{\mathbb{R}} \ : \ f \ \text{is continuous} \}$$
where $\bar{\mathbb{R}} = \mathbb{R} \cup \{-\infty, \infty\}$. Also define
$$D(X) =  \{f \in E(X) \ : \ f^{-1}(\mathbb{R}) \ \text{is dense in } X \}$$
the set of those continuous functions which are real-valued on a dense (and open)
subset of $X$.
\end{defn}

\begin{rem}
$E(X)$ is a poset with respect to the coordinate-wise partial order: \linebreak $f \leq g$ if
$f (x) \leq g(x)$ for each $x \in X$, and, in fact, $E(X)$ is a sublattice of the product $(\bar{\mathbb{R}})^X$. Note that $f \leq g$ provided that $f (x) \leq g(x)$ for each $x$ in some dense
subset of $X$.
The subset $D(X)$ is a sublattice of $E(X)$ since the intersection of two
dense open subsets is dense.
\end{rem}

\begin{defn}
A monomorphism $\phi : H \rightarrow U$ of idempotent semifields is a \linebreak \emph{cl-essential} monomorphism if
$\phi(H) \cap K \neq \{1\}$ for every kernel $K \in \Con(U)$.
\end{defn}

\begin{thm}\cite[Theorem (2.3.23)]{OrderedGroups3}
Let $\mathbb{S}$ be an archimedean idempotent semifield, and let $X = Spec(\mathscr{B}(\mathbb{S}))$ be the
Stone space of the Boolean algebra $\mathscr{B}(\mathbb{S})$ of polars of $\mathbb{S}$. Then there is a (complete)
cl-essential monomorphism from $\mathbb{S}$ into $D(X)$.
\end{thm}

\newpage

\section{The principal bounded kernel - principal skeleton correspondence }

\ \\

Recall that the semifield $\mathscr{R}$ is defined to be a bipotent divisible archimedean \linebreak complete semifield. By Corollary \ref{cor_totaly_ordered_isomorphic_to_reals} one can regard $\mathscr{R}$ as being $(\mathbb{R}^{+}, \dotplus, \cdot)$.\\

Throughout this section, $\mathbb{H}$ is assumed to be a bipotent divisible semifield. Any \linebreak supplementary assumptions on $\mathbb{H}$ will be explicitly stated. \\

Although we have already established a correspondence between skeletons and polars which restricts to a correspondence
between principal skeletons and principal polars, we proceed to find a correspondence between principal skeletons to principal kernels of a very special kernel of $\mathscr{R}(x_1,...,x_n)$, namely, the kernel $\langle \mathscr{R} \rangle$ presented above which will now serve us as a semifield of its own right. It turns out that $\langle \mathscr{R} \rangle$ possesses just enough distinct bounded copies of the principal kernels of $\mathscr{R}(x_1,...,x_n)$ to represent the principal skeleton without any ambiguity.\\

\begin{rem}
The restriction of the image of the operator \\ $Ker : \mathscr{R}^n \rightarrow \mathbb{P}(\mathscr{R}(x_1,...,x_n))$ to  $\mathbb{P}(\langle \mathscr{R} \rangle)$ is
\begin{equation}
Ker_{\langle \mathscr{R} \rangle}(Z) = \{ f \in \langle \mathscr{R} \rangle \ : \ f(a_1,...,a_n) = 1,  \ \forall (a_1,...,a_n) \in Z \} = Ker(Z) \cap \langle \mathscr{R} \rangle.
\end{equation}
Additionally, let $Skel|_{\langle \mathscr{R} \rangle} : \mathbb{P}(\langle \mathscr{R} \rangle) \rightarrow \mathscr{R}^{n}$ be the restriction of \linebreak $Skel : \mathbb{P}(\mathscr{R}(x_1,...,x_n)) \rightarrow \mathscr{R}^{n}$  to  $\mathbb{P}(\langle \mathscr{R} \rangle)$. Then all the statements introduced in the section `Skeletons and kernels of skeletons' apply to  $Ker_{\langle \mathscr{R} \rangle}$ and  $Skel|_{\langle \mathscr{R} \rangle}$.
\end{rem}

\begin{note}
As $Ker_{\langle \mathscr{R} \rangle}$  and  $Skel|_{\langle \mathscr{R} \rangle}$ are our actual object of interest, we also denote them by $Ker$ and $Skel$. The distinction will be carried out through the context of discussion. If an ambiguity arises we will explicitly say to which $Ker$ or $Skel$ we refer.
\end{note}

We now show that restricting $Skel$ and $Ker$ does not affect the collection of resulting skeletons and that each $\mathcal{K}$-kernel of $\mathscr{R}(x_1,...,x_n)$ has a $\mathcal{K}$-kernel (with respect to the restriction) in $\langle \mathscr{R} \rangle$.
\ \\
\begin{prop}
If $K \in \Con(\mathscr{R}(x_1,...,x_n))$ is a $\mathcal{K}$-kernel, then there exists a unique kernel $K' \in \Con(\langle \mathscr{R} \rangle)$ such that $Skel(K) = Skel(K')$.
\end{prop}
\begin{proof}
Let $\alpha \in \mathscr{R} \setminus \{1\}$.
By assumption $K$ is a $\mathcal{K}$-kernel, thus $K = Ker(Skel(K))$. Define $K' = K \cap \langle \mathscr{R} \rangle$.
By Proposition \ref{prop_algebra_of_generators_of_kernels}(2) we have that  $K \cap \langle \mathscr{R} \rangle = \langle X \rangle \cap \langle \alpha \rangle = \langle \{|f| \wedge |\alpha| : f \in X \} \rangle$ where $X$ is any set generating $K$ as a kernel, in particular, one can take $X=K$. Now, for any $f \in \mathscr{R}(x_1,...,x_n)$, in particular in $K$ we have that $f(x)=1$ for some $x \in \mathscr{R}^n$ if and only if $f(x) \wedge |\alpha| = 1$ (since $|\alpha| > 1$) so $Skel(K') = Skel(K)$. Thus $K' = K \cap \langle \mathscr{R} \rangle = Ker(Skel(K')) \cap \langle \mathscr{R} \rangle = Ker_{\langle \mathscr{R} \rangle}(Skel(K'))$, and so  $K'$ is a  $\mathcal{K}$-kernel in $\Con(\langle \mathscr{R} \rangle)$.
\end{proof}

Note that these restricted kernels are exactly the kernels of the (domain) restriction to $\langle \mathscr{R} \rangle$ of the restriction homomorphism, $\mathscr{R}(x_1,...,x_n) \rightarrow \mathscr{R}(x_1,...,x_n)/\langle f \rangle$ defined by  $$g \mapsto g|_{Skel(f)}.$$

By Proposition \ref{prop_correspondence_between_K_kernels_and_skeletons} and the above discussion
we have

\begin{prop}
There is a $1:1$ order reversing correspondence
\begin{equation}
\{ \text{skeletons of } \ \mathscr{R}^{n}  \} \leftrightarrow \{ \mathcal{K}- \text{kernels of } \ \langle \mathscr{R} \rangle \},
\end{equation}
given by $Z \mapsto Ker(Z) \cap \langle \mathscr{R} \rangle$;  the reverse map is given by $K \mapsto Skel(K)$ with \linebreak $K \in \Con(\langle \mathscr{R} \rangle)$.
\end{prop}
\ \\

Let $\mathbb{S}$ be an idempotent semifield. Recall that the positive cone of $\mathbb{S}$ is $$\mathbb{S}^{+} = \{ a \in \mathbb{S}~:~a~\geq~1~\}~= \{|a|  :  a \in \mathbb{S} \}.$$
By Theorems \ref{thm_cones_and_orderes} and \ref{thm_cones_as_generators} (applied to idempotent semifields),  $\mathbb{S}^{+}$ is a sub-semigroup of $\mathbb{S}$ which is a lattice (i.e., closed with respect to $\dotplus$ (i.e., $\vee$) and $\wedge$.\\

Let $f \in \mathbb{H}(x_1,...,x_n)$ and fix $\alpha \neq 1$ in $\mathbb{H}$. Then $$ \langle |f| \wedge |\alpha| \rangle  = \langle f \rangle \cap \langle \alpha \rangle \subseteq \langle \alpha \rangle = \langle \mathbb{H} \rangle.$$
Note that for any $f \in \mathbb{H}(x_1,...,x_n)$, $|f|$ is a generator of $\langle f \rangle$.  \\

We define the map
$$\omega : \mathbb{H}(x_1,...,x_n)^{+}  \rightarrow  \langle \mathbb{H} \rangle^{+}$$
by
$$\omega(|f|) = |f| \wedge |\alpha|.$$
Then $\omega$ is a lattice homomorphism. Indeed, $$\omega (|f| \dotplus |g|) = (|f| \dotplus |g|) \wedge |\alpha| = (|f| \wedge |\alpha|) \dotplus (|g| \wedge |\alpha|) = \omega(|f|) \dotplus \omega(|g|)$$ and $$\omega (|f| \wedge |g|) = (|f| \wedge |g|) \wedge |\alpha| = (|f| \wedge |\alpha|) \wedge (|g| \wedge |\alpha|) = \omega(|f|) \wedge \omega(|g|).$$
$\omega$ induces a map
$$\Omega : \PCon(\mathbb{H}(x_1,...,x_n)) \rightarrow \PCon(\mathbb{H}(x_1,...,x_n)) \cap \langle \mathbb{H} \rangle$$
such that $\Omega(\langle f \rangle) = \langle \omega(|f|) \rangle = \langle |f| \wedge |\alpha| \rangle = \langle f \rangle \cap \langle \mathbb{H} \rangle$.\\

Let us study the map $\Omega$ and its image.

\begin{rem}\label{rem_properties_of_H_kernels}
\begin{enumerate}
  \item By the above discussion we get that $\Omega$ respects both \linebreak intersection and product of kernels, and thus it maps the lattice of principal kernels of $\mathbb{H}(x_1,...,x_n)$, $\PCon(\mathbb{H}(x_1,...,x_n))$, onto the lattice of kernels $$\{ \langle f \rangle \cap \langle \mathbb{H} \rangle : f \in \PCon(\mathbb{H}(x_1,...,x_n)) \} = \PCon(\langle \mathbb{H} \rangle)$$ by Remark \ref{rem_the_lattice_of_kernels_of_the_bounded_kernel}.
  \item If $\langle f \rangle$ is a bounded from below kernel, then $\langle f \rangle \cap \langle \mathbb{H} \rangle = \langle \mathbb{H} \rangle$  by Remark \ref{rem_down_bounded_kernel_contains_H}. In fact, by the above discussion of bounded from below functions, we see that every principal kernel whose skeleton is the empty set is mapped to $\langle \mathbb{H} \rangle$.
  \item As $Skel(\langle \alpha \rangle) = \emptyset$ and since $Skel(\langle f \rangle \cap \langle g \rangle) =      Skel(\langle f \rangle) \cup Skel(\langle g \rangle)$, for any principal kernel $\langle f \rangle$ we have that
      $$Skel(\Omega(\langle f \rangle)) = Skel(\langle f \rangle \cap \langle H \rangle) = Skel(\langle f \rangle) \cup \emptyset = Skel(\langle f \rangle).$$ Thus $\Omega$ does not affect the skeleton of a kernel, the skeleton is fixed.
  \item As any $\alpha \neq 1$ generates $\langle \mathbb{H} \rangle$, any proper subkernel $K$ of $\langle \alpha \rangle$ must admit \linebreak $K \cap \mathbb{H} = \{ 1 \}$, for otherwise, if there exists $\alpha \in K \cap \mathbb{H}$ such that $\alpha \neq 1$, \linebreak  we get $\langle \mathbb{H} \rangle \supseteq K \supseteq \langle \alpha \rangle = \langle \mathbb{H} \rangle$.
  \item $\langle \mathbb{H} \rangle$ is an idempotent semifield of $\mathbb{H}(x_1,...,x_n)$ as the latter is idempotent (since $\mathbb{H}$ is idempotent). Now, by Remark \ref{rem_the_lattice_of_kernels_of_the_bounded_kernel}, we get that
      any kernel $K$ of $\mathbb{H}(x_1,...,x_n)$ such that  $K  \subseteq \langle \mathbb{H} \rangle$ is a kernel of $\langle \mathbb{H} \rangle$. In particular, $\langle \mathbb{H} \rangle$ is a semifield with a  generator (generated as a kernel over itself by a single element) with any $\alpha \neq 1$ as a generator.

   \item By Proposition \ref{prop_unbounded_generator} we have that  $$\Omega ( \PCon(\mathbb{H}(x_1,...,x_n)) \setminus \PCon(\langle \mathbb{H} \rangle)) = \PCon(\langle \mathbb{H} \rangle).$$

\end{enumerate}
\end{rem}

Summarizing the results introduced in Remark \ref{rem_properties_of_H_kernels} for the designated semifield $\mathscr{R}$ we have that
$$\Omega : \PCon(\mathscr{R}(x_1,...,x_n)) \rightarrow \PCon(\langle \mathscr{R} \rangle)$$
is a lattice homomorphism of $(\PCon(\mathscr{R}(x_1,...,x_n)), \cdot, \cap)$ onto $(\PCon(\langle \mathscr{R} \rangle), \cdot, \cap)$, such that $Skel(\langle f \rangle) = Skel(\Omega(\langle f \rangle))$.\\
Let $f \in \PCon(\langle \mathscr{R} \rangle)$ and let $A = \{ g \in \mathscr{R} : \Omega(\langle g \rangle) = f \}$.\\ Define  $K = \langle A \rangle \in \Con(\mathscr{R}(x_1,...,x_n))$. Then by Remark \ref{rem_properties_of_H_kernels}, if $g \in A$ then  \linebreak $Skel(g) = Skel(f)$.\\

As we shall shortly show, there exists a correspondence
$$ \langle f \rangle \in \PCon(\langle \mathscr{R} \rangle) \leftrightarrow Skel(f)$$
between the principal skeletons in $\mathscr{R}^n$ and the kernels in  $\PCon(\langle \mathscr{R} \rangle)$.\\
If $Skel(g) = Skel(f)$ then since $Skel(g) = Skel(\Omega(g))$ we have that $Skel(\Omega(g)) = Skel(f) \in \PCon(\langle \mathscr{R} \rangle)$. Thus in view of the above $\Omega(\langle g \rangle) = \langle f \rangle$.
Consequently we have that $Skel(K) = Skel(f)$ and $K$ is the maximal kernel of $\mathscr{R}(x_1,...,x_n)$ having this property.\\

The choice of $\langle \mathscr{R} \rangle$ for our algebraic infrastructure is a natural choice, as we have $\mathscr{R}$ as our basic semifield, with respect to which homomorphisms are taken. Taking the semifield in $\mathscr{R}(x_1,...,x_n)$ generated by $\mathscr{R}$ as a kernel, we get the only semifield homomorphism $ \langle \mathscr{R} \rangle \rightarrow \mathbb{H}$ (where $\mathbb{H}$  is a semifield containing $\mathscr{R}$)  sending $\alpha$ to $1$ is the trivial one $ \langle \mathscr{R} \rangle \rightarrow \{1\}$. All proper subkernels $K$ of  $\langle \mathscr{R} \rangle$ admit $K \cap \mathscr{R} = \{ 1 \}$ which is a necessary condition for a kernel of an $\mathscr{R}$-homomorphism (which is otherwise not well defined, as any $\alpha \neq 1$ is required to be mapped to itself).  One can view the elements of $\langle \mathscr{R} \rangle \setminus \{1 \}$ as playing the role of invertible elements in rings, in the sense that any ideal containing an invertible element is the ring itself.

%

\begin{note}
Since the principal kernels in $\PCon(\langle \mathscr{R} \rangle)$ are in particular principal \linebreak kernels in $\PCon(\mathscr{R}(x_1,...,x_n))$, in the remainder of this paper we generally study\\ $\PCon(\mathscr{R}(x_1,...,x_n))$ where the results are true in particular for $\PCon(\langle \mathscr{R} \rangle)$. In our subsequent discussions we develop the notion of reducibility, regularity and corner-integrality for a principal kernel in $\PCon(\mathscr{R}(x_1,...,x_n))$. We develop the theory of reducibility for  general sublattices of $\PCon(\mathscr{R}(x_1,...,x_n))$, thus in particular for sublattices of $\PCon(\langle \mathscr{R} \rangle)$. In the sections to follow, we introduce the notions of regularity and corner-integrality of kernels which are oriented to skeletons, in the sense that we consider points for which a generator $f$ (equivalently $|f|$) of the kernel attains the value $1$. Since passing to $|f| \wedge |\alpha|$ for some $\alpha \neq 1$ in $\mathscr{R}$ does not affect these points,  these notions stay intact restricting them to $\PCon(\langle \mathscr{R} \rangle)$, in the sense that $|f|$ is regular (corner-integral) if and only if $|f| \wedge |\alpha|$ is regular (corner-integral). In those places where it is necessary, we explicitly restrict ourselves to $\PCon(\langle \mathscr{R} \rangle)$.
\end{note}

\begin{prop}\label{generators_and_skeletons}
 If $\langle f \rangle$ is a principal kernel generated by $f \in \mathscr{R}(x_1,...,x_n)$, then $Skel(f) = Skel(\langle f \rangle)$.
Moreover, $Skel(f) = Skel(f')$ for any generator $f'$ of~$\langle f \rangle$.
\end{prop}
\begin{proof}
Although the first statement was proved in Proposition \ref{prop_skel_property1}(2), we present here a somewhat more elegant proof. By Corollary \ref{cor_principal_ker_by_order} we have that $g \in \mathscr{R}(x_1,...,x_n)$ is in $\langle f \rangle$ if and only if there exists some $n \in \mathrm{N}$ such that $(f+f^{-1})^{-n} \leq g \leq (f + f^{-1})^{n}$, or by different notation, $|f|^{-n}  \leq g \leq  |f|^{n}$. Now, since over $Skel(f)$, $f = f^{-1} =1$, we have  \ $1 = 1^{-n} \leq g \leq 1^{n} = 1$ over $Skel(f)$ and thus $g=1$ for every $g \in \langle f \rangle$.\\
For the second assertion, let $f'$ be a generator of $\langle f \rangle$ then by  Corollary \ref{cor_principal_ker_by_order} for some $k \in \mathrm{N}$, $|f'|^{-k}  \leq f \leq  |f'|^{n}$, which yields, using the first statement that for any $x \in \mathscr{R}^n$, $f'(x) = 1$ if and only if $f(x) = 1$, so $Skel(f') = Skel(f)$.
\end{proof}

\pagebreak

\begin{prop}\label{prop_generator_of_skel1}
Let $h \in \mathscr{R}(x_1,...,x_n)$. If $h$ is a generator of $\langle f \rangle$, then \linebreak $Skel(h) = Skel(f)$.
\end{prop}
\begin{proof}
First, note that by Proposition \ref{prop_skel_property1}(1), we have that $Skel(f) \subseteq Skel(h)$ since $\langle h \rangle \subseteq \langle f \rangle$. The `only if' part of the assertion follows from the fact that if $h$ is a generator then $\langle h \rangle = \langle f \rangle$ and thus their skeletons coincide.
\end{proof}

\begin{note}
In the following proposition we use the property of $\mathscr{R}$ being complete, in the sense that
the underlying lattice of $\mathscr{R}$ is conditionally complete (see Definition \ref{defn_divisible_bipotent_base_semifield}).
\end{note}

\begin{prop}\label{prop_generator_of_skel2}
Let $\langle f \rangle \subseteq \langle \mathscr{R} \rangle$. If $h \in \langle f \rangle$ is such that $Skel(h) = Skel(f)$, then $h$ is a generator of $\langle f \rangle$.
\end{prop}
\begin{proof}
The assertion is obvious in the case where $\langle f \rangle = \langle 1 \rangle = \{1\}$. So, as \linebreak $Skel(h) = Skel(f)$ we can assume $f$ and $h$ to be not equal to $1$. If $\langle f  \rangle = \langle \alpha \rangle$ for some \linebreak $\alpha \neq 1$, then $Skel(h) = Skel(f) = \emptyset$ implies by Remark \ref{rem_properties_of_H_kernels} (3) that $\langle h \rangle = \langle \mathscr{R} \rangle = \langle f  \rangle$.


Before we continue, note that $\mathscr{R}$ is a totally ordered semifield, thus for any $a~\in~\mathscr{R}^n$, either $h(a) \leq f(a)$ or $f(a) \leq h(a)$, for any pair of functions $f,h$ in $\mathscr{R}(x_1,...,x_n)$. \linebreak Let $h \in \langle f \rangle$ such that $h$ is not a generator of $\langle f \rangle$ while $Skel(h) =Skel(f)$. By \linebreak Corollary~\ref{cor_principal_ker_by_order} we have that for each $k \in \mathbb{N}$ there exists some $x_k \in \mathscr{R}^{n}$ for which \linebreak $|f(x_k)| > |h(x_k)|^{k}$ (*). Note that $f$ and $h$ are rational polynomials thus continuous and so are $|f|^{s}$ and $|h|^{s}$ (by definition of $|\cdot|$) for any $s \in \mathbb{N}$.\\
For any $k \in \mathbb{N}$, define the set $U_k = \{x \ : \ |f(x)| > |h(x)|^k \}$.  As $\mathscr{R}$ is assumed to be (ordered) divisibly closed semifield, it is dense, so, for any $x \in U_k$  there exists a  neighborhood $B_x \subset U_k$ containing $x$ such that for all $x' \in B_x$,  $|f(x')| > |h(x')|^k $.
Now, since $h$ and $f$ are bounded from above rational functions both not equal to $1$, $U_k$ are bounded regions inside $\mathscr{R}^n$.  Taking the closure of $U_k$ we may assume it is closed.
Since $Skel(h) =Skel(f)$, $|f(x)| > |h(x)|^k $ implies that $|h(x)|,|f(x)| > 1$ , so, by the definition of $U_k$ we get the sequence of strict inclusions $U_1 \supset U_2 \supset \dots \supset U_k \supset \dots$ where by our assumption (*), $U_i \neq \emptyset$. Thus, since $\mathscr{R}$ is complete, there exists an element $y \in \mathscr{R}$ such that $y \in \mathrm{B} = \bigcap_{\mathbb{N}} B_k$. Now, for $x \not \in Skel(h)$,  $|h(x)| > 1$ thus there exists some $r = r(x) \in \mathbb{N}$ such that $|h(x)|^r > |f(x)|$ thus $x \not \in \mathrm{B}$ thus $y \not \in \mathscr{R}^n \setminus Skel(h)$ . On the other hand, if $y \in Skel(h)$ then $y \in Skel(f)$ so $ 1=|f(y)| \leq |h(y)|=1$. Thus $\bigcap_{\mathbb{N}} B_k = \emptyset$. A contradiction.
\end{proof}

\begin{note}
In the last proof we could analogously argue that since $y \not \in Skel(h)$, \linebreak $f(y) > h(y)^k$ for every natural number $k$ where $h(y) > 1$, which yields that \linebreak $f(y)~\not\in~\mathscr{R} = \langle h(y) \rangle$.
\end{note}

\begin{prop}\label{prop_generator_of_skel3}
Let $\langle f \rangle \subseteq \mathscr{R}(x_1,...,x_n)$. If $h \in \langle f \rangle$ is such that $Skel(h) = Skel(f)$ and $|h|$ has an essential expansion as $|h| = \sum_{i=1}^{k}s_i |f|^{d(i)}$ with $d(i) \in \mathbb{Z}$ and $s_1,...,s_k \in \mathscr{R}(x_1,...,x_n)$ such that $\sum_{i=1}^{k}s_i = 1$, then $h$ is a generator of $\langle f \rangle$.
\end{prop}
\begin{proof}
As $|f|$ is a generator of $\langle f \rangle$ and $|h|$ a generator of $\langle h \rangle$, we may consider $|f|$ and $|h|$ instead of $f$ and $h$ and moreover, we may assume $|f|$ and $|h|$ to be in  essential form. By Proposition \ref{prop_ker_stracture_by group}, there exist some $s_1,...,s_k \in \mathscr{R}(x_1,...,x_n)$ such that $\sum_{i=1}^{k}s_i = 1$ and $|h| = \sum_{i=1}^{k}s_i |f|^{d(i)}$ with $d(i) \in \mathbb{Z}_{\geq 0}$ (since $|h| \geq 1$, $d(i) \geq 0$). As $|h|$ is in essential form, $k$ is minimal. Now, since $Skel(|f|) = Skel(|h|)$ we have that $1=|h(x)| = \sum_{i=1}^{k}s_i(x) |f(x)|^{d(i)} = \sum_{i=1}^{k}s_i(x)\cdot 1 = \sum_{i=1}^{k}s_i(x)$ for every $x \in Skel(h)$. Thus for every $x \in Skel(h)$ there exists $1 \leq j \leq k$ such that $h(x) = s_j(x) = 1$, so $h(x) = 1$ if and only if $h(x) = s_j(x) = 1$ for some $1 \leq j \leq k$. Let $i_0$ be such that $d(i_0) = 0$. Since $|h|$ is in essential form, and $s_{i_0} \leq \sum_{i=1}^{k}s_i = 1$, $s_{i_0}$ must be dominate at some point $x_0 \in Skel(h)$, in the sense that $s_{i_0} > s_i$ for every $i \neq i_{0}$. As $|h|$ is continuous there exists a neighborhood $\epsilon$ of $x_0$ such that  $s_{i_0}(y) \geq \sum_{i \neq i_0}s_i(y) |f(y)|^{d(i)}$ for every $y \in \epsilon$. But $\sum_{i=1}^{k}s_i = 1$ thus $s_{i_0}(y) =1$ for every $y \in \epsilon$.  Without loss of generality, take $i_0 = 1$. If $\epsilon \not \subseteq Skel(f)$ then, as $\mathscr{R}$ is divisibly closed, there exists a point $y_1 \in \epsilon \setminus Skel(f)$ such that $ |h(y_1)| = s_{1}(y_1) + \sum_{2}^k s_i(y_1) |f(y_1)|^{d(i)} = s_{1}(y_1) = 1$, which contradicts the assumption that $Skel(h) = Skel(f)$. So $\epsilon \subseteq Skel(f)$. Thus over $R = Cl(Skel(f)^c)$ we have that $|h(x)| = \sum_{i=2}^{k}s_i(x) |f(x)|^{d(i)}$ where $\sum_{i=1}^{k}s_i(x) = 1$ for every $x \in R$ (since $\epsilon \cap R =\emptyset$ there always exists some $2 \leq j \leq k$ for which $s_j(x) = 1$).  Now, take $x \in R$. Then there exists some $2 \leq j_0 \leq k$ such that $s_{j_0}(x) = 1$, thus we have that $|h(x)| = \sum_{i=2}^{k}s_i(x) |f(x)|^{d(i)} \geq  s_{j_0}(x)|f(x)|^{d(j_0)} = |f(x)|^{d(j_0)}$. Consequently, since $|f(x)| \geq 1$ we have that $|h(x)| \geq |f(x)|^{d}$ with $d = \min \{d(j) \ : \ 2 \leq j \leq k \}$, $d > 0$. As also $1= |h(x)| \geq |f(x)|^{d} = 1^{d} = 1$ over $Skel(f)$, we get that $|h(x)| \geq |f(x)|^{d} \geq |f(x)|$ over $R \cup Skel(f) = \mathscr{R}^n$, i.e., $|h| \geq |f|$. Thus, by Remark \ref{rem_kernel_by_abs_value} $|f| \in \langle |h| \rangle$, so \linebreak $\langle f \rangle  = \langle |f| \rangle \subseteq \langle |h| \rangle = \langle h \rangle$. Finally, as $h \in \langle f \rangle$, we have that $\langle h \rangle =\langle f \rangle$, i.e., $h$ is a generator of $\langle f \rangle$ as desired.
\end{proof}

The proof of Proposition \ref{prop_generator_of_skel3} does not apply to any element of $\langle f \rangle$, just to the element which can be written \textbf{essentially} in the form $a= \sum_{i=1}^{k}s_i |f|^{d(i)}$ with $\sum_{i=1}^{k}s_i =1$. For example, the element $|x| \wedge \alpha \in \langle x \rangle$ with $ \alpha > 1$ can be written as $(\frac{|x|}{\alpha + |x|})\cdot 1 + (\frac{\alpha}{\alpha + |x|})\cdot |x|$ with $a_1(x) = \frac{|x|}{\alpha + |x|}$ and $a_2(x) = \frac{\alpha}{\alpha + |x|}$, but this is not in an essential form since the first term never dominates.



\begin{rem}
By Proposition \ref{prop_skel_property1}, we have that a skeleton $\mathcal{S}$ is a principal skeleton, i.e., $\mathcal{S} = Skel({f})$ for some $f \in \langle \mathscr{R} \rangle$, if and only if $\mathcal{S} = Skel(\langle f \rangle)$.
\end{rem}

\begin{prop}\label{prop_kernel_is_k_kernel}
Let   $\langle f \rangle$ be a principal kernel in $\PCon(\langle \mathscr{R} \rangle)$. Then $\langle f \rangle$ is \linebreak a $\mathcal{K}$-kernel.
\end{prop}
\begin{proof}
We need to show that ($Ker(Skel(f))=$) $Ker(Skel(\langle f \rangle)) \subseteq \langle f \rangle$.\\
Let $h \in \langle \mathscr{R} \rangle $ such that $h \in Ker(Skel(f))$. Then $h(x) = 1$  for every $x \in Skel(f)$ and so $Skel(f) \subseteq Skel(h)$. If $|h| \leq |f|^{k}$ for some $k \in \mathbb{N}$ then $h \in \langle f \rangle$. Thus in particular we may assume that $h \neq 1$. Now, by Corollary \ref{cor_principal_skel_correspondence} we have that $Skel(\langle f \rangle \cap \langle h \rangle) = Skel(f) \cup Skel(h) = Skel(h)$. Since $h \neq 1$, $Skel(h) \neq \mathscr{R}^n$ and thus $\langle f \rangle \cap \langle h \rangle \neq \{1\}$. Again by Corollary \ref{cor_principal_skel_correspondence} we have that $Skel(\langle f \rangle \cdot \langle h \rangle) = Skel(f) \cap Skel(h) = Skel(f)$. Thus $\langle f,h \rangle = \langle f \rangle \cdot \langle h \rangle \neq \langle \mathscr{R} \rangle$ for otherwise $Skel(f) = \emptyset$ . Consequently the kernel $K = \langle g \rangle = \langle f \rangle \cap \langle h \rangle$,  where $g~=~|f| \wedge |h|$, admits $\{ 1 \} \neq K \subseteq \langle f \rangle$. So, we have that $g \in \langle f \rangle$ and $Skel(g)= Skel(h)$. Thus By Proposition \ref{prop_generator_of_skel2}, $g$ is a generator of $ \langle h \rangle$, so, we have that  $\langle h \rangle = K \subseteq \langle f \rangle$ \ as desired.
\end{proof}

\begin{cor}\label{cor_correspondence_bounded_pkernels_pskeletons}
There is a $1:1$ order reversing correspondence
\begin{equation}
\{ \text{principal skeletons of } \ \mathscr{R}^{n}  \} \leftrightarrow \{ \text{principal kernels of } \ \langle \mathscr{R} \rangle \},
\end{equation}
given by $Z \mapsto Ker_{\langle \mathscr{R} \rangle}(Z)$;  the reverse map is given by $K \mapsto Skel(K)$.
\end{cor}
\begin{proof}
Every principal kernel gives rise to a principal skeleton by the definition of $Skel$. The reverse direction follows Proposition \ref{prop_generator_of_skel2} as every principal kernel which produces a principal skeleton using $Skel$ is in fact a $\mathcal{K}$-kernel.
\end{proof}

In Proposition \ref{prop_maximal_kernels_in_semifield_of_fractions_part1},
we have shown using a substitution homomorphism $\psi$ that any point $a = (\alpha_1,...,\alpha_n) \in \mathscr{R}^n$ corresponds to the maximal kernel
$$\left\langle \frac{x_1}{\alpha_1} , ... , \frac{x_n}{\alpha_n} \right\rangle = \left\langle \left|\frac{x_1}{\alpha_1}\right| + ....+ \left|\frac{x_n}{\alpha_n}\right| \right\rangle = \left\langle \left|\frac{x_1}{\alpha_1}\right| \cdot .... \cdot \left|\frac{x_n}{\alpha_n}\right| \right\rangle = \left\langle \frac{x_1}{\alpha_1} \right\rangle \cdot \dots \cdot \left\langle \frac{x_n}{\alpha_n} \right\rangle.$$
Let $\psi : \mathbb{H}(x) \rightarrow \mathbb{H}$ be defined by sending $x \mapsto 1$.
Consider the restriction homomorphism $\psi|_{\langle \mathscr{R} \rangle} :\langle  \mathscr{R} \rangle  \rightarrow \psi(\langle  \mathscr{R} \rangle) = \mathscr{R}$. Then by Theorem \ref{thm_kernels_hom_relations}, we have that \linebreak $Ker\psi|_{\langle \mathscr{R} \rangle} = Ker \psi \cap \langle  \mathscr{R} \rangle = \langle x \rangle \cap \langle  \mathscr{R} \rangle$. Thus, the result applies to $\langle \mathscr{R} \rangle$  where the maximal kernel is $\langle x \rangle \cap \langle  \mathscr{R} \rangle$.

We will now show that any maximal kernel of $\langle \mathscr{R} \rangle $ is of that form.

\begin{prop}\label{rem_maximal_kernels_in_semifield_of_fractions_part2}
If $K$ is a maximal kernel in $\Con(\langle H \rangle)$, then $K = \Omega(\langle \frac{x_1}{\alpha_1}, ... , \frac{x_n}{\alpha_n} \rangle)$ for some $\alpha_1,...,\alpha_n \in \mathscr{R}$.
\end{prop}
\begin{proof}
Denote $L_a = (|\frac{x_1}{\alpha_1}| + ....+ |\frac{x_n}{\alpha_n}|) \wedge |\alpha|$ with $\alpha \neq 1$, for $a = (\alpha_1, ..., \alpha_n)$. \linebreak
By Remark \ref{rem_properties_of_H_kernels} we may assume $Skel(K) \neq \emptyset$, since the only kernel corresponding to the empty set is $\langle \mathscr{R} \rangle$ itself. If $a \in Skel(K)$, then as $Skel(L_a) = \{a\} \subseteq Skel(K)$, we have that  $\langle L_a \rangle \supseteq K$. Thus, the maximality of $K$ implies that $K = \langle L_a \rangle$.
\end{proof}

\newpage

\section{The coordinate semifield of a skeleton}

\ \\

In this section we define the coordinate semifield corresponding to a skeleton. Being the most relevant to the development achieved in the reminder of this work, we first perform the construction for principal skeletons, using the principal kernels  of $\langle \mathscr{R} \rangle$. Later on we introduce a more general construction of the coordinate semifield of a (principal) skeleton using its corresponding (principal) polar.

\ \\

\begin{defn}
Let $V = Skel(\langle f \rangle)$ be a (principal) skeleton in $\mathscr{R}^n$. Define
\begin{equation}
\mathscr{R}[V] = \{ f|_{V}(x_1,...,x_n) \ : \ f \in \langle \mathscr{R} \rangle  \}.
\end{equation}
We call $\mathscr{R}[V]$ the \emph{coordinate semifield} of $V$.
\end{defn}

\begin{prop}\label{prop_coord_semifield_1}
For any (principal) skeleton $V = Skel(\langle f \rangle) \subseteq \mathscr{R}^{n}$, define
$$\phi_{V}~:~\langle \mathscr{R} \rangle~\rightarrow~\mathscr{R}[V]$$
to be the restriction map $f \mapsto f|_{V}$. Then $\phi_{V}$ is a homomorphism and
\begin{equation}
\langle \mathscr{R} \rangle /\langle f \rangle \cong \mathscr{R}[V].
\end{equation}
\end{prop}
\begin{proof}
For any $g,h \in \mathscr{R}(x_1,...,x_n)$, since $\phi_{V}$ is a restriction map, we have that $\phi_V(g + h) = (g+h)|_V = g|_V + h|_V = \phi_V(g) + \phi_V(h)$ and $\phi_V(g \cdot h) = (g \cdot h)|_V = g|_V \cdot h|_V = \phi_V(g) \cdot \phi_V(h)$ so $\phi_V$ is a semiring homomorphism. It is trivially onto, by the definition of $\mathscr{R}[V]$.
Now, By Proposition \ref{prop_kernel_is_k_kernel} we have that  $Ker(\phi_V) = \{  g \in \langle \mathscr{R} \rangle \ : \ g|_V = 1 \} = \{  g \in \langle \mathscr{R} \rangle : g \in \langle f \rangle \} = \langle f \rangle$.
Thus by the isomorphism theorem \ref{thm_nother_1_and_3} we have that $\langle \mathscr{R} \rangle/\langle f \rangle \cong \Im(\phi_{V}) = \mathscr{R}[V]$, as desired.
\end{proof}

\begin{prop}\label{prop_direct_prod_of_coord_semifields}
Let $K_1,K_2$ be kernels of the semifield $\langle \mathscr{R} \rangle$ such that $\langle \mathscr{R} \rangle = K_1 \cdot K_2$. Then $$\langle \mathscr{R} \rangle / (K_1 \cap  K_2) \cong \langle \mathscr{R} \rangle/K_1 \times \langle \mathscr{R} \rangle/K_2$$ as groups.
\end{prop}
\begin{proof}
$\langle \mathscr{R} \rangle /(K_1 \cap K_2) = (K_1/(K_1 \cap K_2))\cdot (K_2/(K_1 \cap K_2)) = \bar{K_1} \cdot \bar{K_2}$ where $\bar{K_i}$ is the homomorphic image of $K_i$ under them quotient map $\langle \mathscr{R} \rangle \rightarrow \langle \mathscr{R} \rangle/(K_1 \cap K_2)$ .\\ Since $\bar{K_1} \cap \bar{K_2} = \{1\}$ we have that as groups $\bar{K_1} \cdot \bar{K_2} \cong \bar{K_1} \times \bar{K_2}$. Now, by the second isomorphism theorem for kernels $\bar{K_1} = (K_1/(K_1 \cap K_2)) \cong (K_1 \cdot K_2)/K_2 = \langle \mathscr{R} \rangle/K_2$ and similarly $\bar{K_2} = \langle \mathscr{R} \rangle/K_1$. Thus we have that $G/(K_1 \cap K_2) \cong G/K_1 \times G/K_2$ as groups.
\end{proof}

\begin{cor}
If $V_1,V_2$ are principal skeletons in $\mathscr{R}^n$ such that $V_1 \cap V_2 = \emptyset$, then
$$\mathscr{R}[V_1 \cup V_2] \cong \mathscr{R}[V_1] \times \mathscr{R}[V_2]$$
as groups.
\end{cor}
\begin{proof}
Let $\langle f_1 \rangle$ and $\langle f_2 \rangle$ be the kernels in $\PCon(\langle \mathscr{R} \rangle)$ such that $V_1 = Skel(\langle f_1 \rangle)$ and $V_2 = Skel(\langle f_2 \rangle)$. Then $Skel( \langle f_1 \rangle \cdot \langle f_2 \rangle) = Skel(\langle f_1 \rangle) \cap Skel(\langle f_2 \rangle) = V_1 \cap V_2 = \emptyset$ thus $\langle f_1 \rangle \cdot \langle f_2 \rangle = \langle \mathscr{R} \rangle$. Since $V_1 \cup V_2 = Skel( \langle f_1 \rangle \cap \langle f_2 \rangle)$ we have by Proposition \ref{prop_direct_prod_of_coord_semifields} that $\mathscr{R}[V_1 \cup V_2] \cong \langle \mathscr{R} \rangle /(\langle f_1 \rangle \cap \langle f_2 \rangle) \cong \langle \mathscr{R} \rangle/\langle f_1 \rangle \times \langle \mathscr{R} \rangle/\langle f_1 \rangle \cong \mathscr{R}[V_1] \times \mathscr{R}[V_2]$.
\end{proof}

The following result is analogue to Proposition \ref{prop_coord_semifield_1} using $\mathrm{R}(x_1,...,x_n)$ and $\mathcal{B}$ defined in Corollary \ref{cor_polar_skeleton_corrsepondence} instead of $\langle \mathscr{R} \rangle$ and $Con(\langle \mathscr{R} \rangle)$, respectively.
\begin{prop}\label{prop_coord_semifield_2}
For any  skeleton $V = Skel(\mathcal{S}) \subseteq \mathscr{R}^{n}$ with $\mathcal{S} \subseteq \mathscr{R}(x_1,...,x_n)$  define
$$\phi_{V}~:~\mathscr{R}(x_1,...,x_n)~\rightarrow~\mathscr{R}[V]$$
to be the restriction map $f \mapsto f|_{V}$. Then $\phi_{V}$ is a homomorphism and
\begin{equation}
\mathscr{R}(x_1,...,x_n)  / K_{\mathcal{S}} \cong \mathscr{R}[V].
\end{equation}
where $K_{\mathcal{S}} = \mathcal{S}^{\bot \bot} \cap \mathscr{R}(x_1,...,x_n)$ with $\mathcal{S}^{\bot \bot}$ is taken in the completion  $\overline{\mathscr{R}(x_1,...,x_n)}$ of $\mathscr{R}(x_1,...,x_n)$ in $Fun(\mathscr{R}^n, \mathscr{R})$.
\end{prop}
\begin{proof}
Note that by Corollary \ref{cor_polar_skeleton_corrsepondence} we have that  $Ker(\phi_V) = \{  g \in \mathscr{R}(x_1,...,x_n) \ : \ g|_V = 1 \} = \{  g \in \mathscr{R}(x_1,...,x_n) : g \in \mathcal{S}^{\bot \bot} \} = K_\mathcal{S}$. The rest of the proof is as in Proposition~\ref{prop_coord_semifield_1}.
\end{proof}

\newpage

\section{Basic notions : Essentiality, reducibility, regularity \\ \ \ \ \ and corner-integrality}

\ \\

In the previous sections we have proved some correspondences between skeletons and kernels. In particular we have shown a correspondence between principal skeletons and principal kernels.  We now turn to find the connection between tropical varieties (which we call corner loci) and skeletons. \\

Before diving into the core of the theory we develop some tools and notion to facilitate our construction. \\

Throughout this section, $\mathbb{H}$ is assumed to be a bipotent divisible semifield. Any supplementary assumptions on $\mathbb{H}$ will be explicitly stated. We also recall that our designated semifield $\mathscr{R}$ is defined to be bipotent, divisible, archimedean and complete the prototype being  $(\mathbb{R}^{+}, \dotplus, \cdot)$ by Corollary \ref{cor_totaly_ordered_isomorphic_to_reals}.

\ \\
\subsection{Essentiality of elements in the semifield of fractions}
\ \\

%
%

In the theory of tropical geometry, there exists a notion of essentiality of monomials in a given polynomial.

\begin{defn}\label{defn_tropical_essentiality}
A monomial $m_1(x_1,...,x_n) \in \mathbb{H}[x_1,...,x_n]$ is said to be inessential in a polynomial $p(x_1,...,x_n) = \sum_{i=1}^{k}m_i(x_1,...,x_n) \in \mathbb{H}[x_1,...,x_n]$ over a domain $D \subset \mathbb{H}^n$,  if at any $x \in D$ there exists some $j \neq 1$ such that $m_j(x) \geq m_1(x)$, i.e., $m_1$ never solely dominates $p$ over $D$.
\end{defn}

\begin{note}
In the following section, we introduce a notion of essentiality for elements of $\mathbb{H}(x_1,...,x_n)$. This notion differs from the tropical one, in fact it generalizes it in some sense, as we will show shortly. Generally, the context will imply the relevant notion among the two. In case ambiguity arises, we will explicitly indicate the one we refer to.
\end{note}

Let $f \in \mathbb{H}(x_1,...,x_n)$, consider the skeleton defined by $f$, $Skel(f)$. We will now characterize for which $g \in \mathbb{H}(x_1,...,x_n)$, $Skel(f + g) = Skel(f)$.\\
For this purpose, we introduce the following definition.

\begin{defn}
Let $g \in \mathbb{H}(x_1,...,x_n)$. Define the following sets
\begin{equation}
Skel_{-}(g) = \{ x \in \mathbb{H}^n \ : \ g(x) < 1 \}, \ \ \ Skel_{+}(g) = \{ x \in \mathbb{H}^n \ : \ g(x) > 1 \}.
\end{equation}
Notice that $\mathbb{H}^n = Skel_{-}(g) \cup Skel_{+}(g) \cup Skel(g)$.
We call $Skel_{+}(g)$ and $Skel_{-}(g)$ the \emph{positive} and \emph{negative} regions of $g$, respectively.
\end{defn}

\begin{rem}
If $f,g \in \mathbb{H}(x_1,...,x_n)$, then $Skel(f + g) = Skel(f)$ if and only if the following holds:
\begin{equation}\label{cond_additive_invariance}
Skel(f) \setminus Skel(g) \subset Skel_{-}(g) \ \text{and} \  Skel(g) \setminus Skel(f) \subset Skel_{+}(f).
\end{equation}
Note that when $Skel(f) \subseteq Skel(g)$, i.e., $\langle g \rangle \subseteq \langle f \rangle$, condition \ref{cond_additive_invariance} takes the form $Skel(g) \setminus Skel(f) \subset Skel_{+}(f)$. When $Skel(g) \subseteq Skel(f)$, i.e., $\langle f \rangle \subseteq \langle g \rangle$ condition \ref{cond_additive_invariance} takes the form $Skel(f) \setminus Skel(g) \subset Skel_{-}(g)$ and when $Skel(f) \cap Skel(g) = \emptyset$ then $Skel(g) \subset Skel_{+}(f)$ and $Skel(f) \subset Skel_{-}(g)$.
\end{rem}
\begin{proof}
This statement is a direct consequence of the definitions.
\end{proof}

\ \\

\begin{defn}\label{def_inessentiality}
Let $f,g \in \mathbb{H}(x_1,...,x_n)$. $g$ is said to be \emph{inessential} for $f$ if  $$Skel(f + g) = Skel(f);$$
otherwise $g$ is \emph{essential} for $f$. Let $f = \sum_{i=1}^{k}f_i \in \mathbb{H}(x_1,...,x_n)$ and let $j \in \{1,...,n \}$. Then $f_j$ is said to be inessential in $f$ if $f_j$ is inessential for $\sum_{i \neq j}f_i$. Otherwise $f_j$ is \emph{essential} in $f$.
\end{defn}

\begin{note}
Note that inessentiality defined in Definition \ref{def_inessentiality} differs from  the notion of inessentiality in tropical geometry.
In tropical geometry, a monomial of a polynomial is considered inessential if it is not dominant anywhere, in the sense that it does not attain \emph{solely} the maximal value of the polynomial.
\end{note}

\ \\

\begin{defn}\label{defn_essensiality_condition}
Let $f \in \mathbb{H}(x_1,...,x_n)$. Then we say $f$ has the \emph{essentiality property} if for any additive decomposition of $f$, \ $f = \sum_{i=1}^{k}f_i$ with $f_i \in \mathbb{H}(x_1,...,x_n)$, the following condition holds:
$$\text{For any} \ 1 \leq j \leq k,  \ \  Skel(f) \neq Skel(h_{j}) \  \ \text{where} \ \ h_{j} = \sum_{i=1; i \neq j}^{k}f_i.$$
In words each $f_i$ is essential in $f$.
\end{defn}

\begin{defn}\label{defn_reduced_fractional_elements}
Let $f \in \mathbb{H}(x_1,...,x_n)$. Write $f= \frac{h}{g} = \frac{\sum_{i=1}^{k}h_i}{\sum_{j=1}^{m}g_j}$ where $g_j$ and $h_i$ are monomials in $\mathbb{H}[x_1,...,x_n]$. Then $f$  is said to be of \emph{reduced form} or \emph{essential form} if for any $I \subseteq \{1,...,k\}$ and  $J \subseteq \{1,...,m \}$ where $I$ and/or $J$ is a proper subset,
$$Skel(f) \neq Skel(\tilde{f})  \ \ \text{ where } \ \ \tilde{f} = \frac{\sum_{I}h_i}{\sum_{J}g_j}.$$
\end{defn}

\begin{rem}
Note that in order for a monomial $h_i$ to be essential, there is no need for the occurrence of some $g_j$ such that
$h_i(x) = g_j(x)$ for some $x \in \mathbb{H}^n$ and viceversa. The reason for this is that $h_i$ can affect the skeleton by preventing another monomial $h'$ of the numerator to dominate $h$ in a point $x \in \mathbb{H}^n$ where $h'(x)=g(x)$.\\
The inessential monomials $h_i$ and $g_j$ in $f = \frac{h}{g}$ are characterized as follows:\\
A monomial $h'$ of  $h$ is inessential in $f$ if one of the following two \linebreak conditions holds for all $x \in \mathbb{H}^n$:
\begin{enumerate}
  \item $h'(x) < h(x)$.
  \item $h'(x) = h(x)$ and $h'(x) \neq g(x)$.
\end{enumerate}
Analogously, a monomial $g'$ of $g$ is inessential in $f$ if one of the following two conditions holds for all $x \in \mathbb{H}^n$:
\begin{enumerate}
  \item $g'(x) < g(x)$.
  \item $g'(x) = g(x)$ and $g'(x) \neq h(x)$.
\end{enumerate}
It can easily be seen that $h'$ and $g'$ admitting the above criterion do not affect the \linebreak skeleton of $f$.
Moreover, if a monomial $g'$ is inessential in $g$ then taking $\tilde{f} = \frac{h}{\tilde{g}}$ where $\tilde{g}$ is defined to be $g$ with $g'$ omitted, we have that any monomial $g'' \neq g'$ of $g$ and any monomial  $h'$ of $h$ are essential in $\tilde{f}$ if and only if they are essential in $f$. \\
In view of the above, we can define $f_e \in \mathbb{H}(x_1,...,x_n)$ to be the rational function obtained from $f = \frac{g}{h}$ by omitting all inessential monomials in $f$. By the above, $f_e$ is well-defined regardless of the order with respect to which the monomials are omitted.
\end{rem}

\ \\

In view of the above discussion the following observations hold:
\begin{rem}
$f \in \mathbb{H}(x_1,...,x_n)$ is of essential form if and only if both $f$ and $f^{-1}$ admit the essentiality property.
\end{rem}

\begin{rem}
Let $f = \frac{h}{g} \in \mathbb{H}(x_1,...,x_n)$ where $h,g \in \mathbb{H}[x_1,...,x_n]$ are \linebreak polynomials.
If $h$ or  $g$ are inessential (in the tropical sense) then $f$ is not of essential form.
\end{rem}
\begin{proof}
Indeed,  if $h$ is inessential one of the composing monomials of $h$, say $h'$, does not affect the values $h$ obtains and thus does not affect the skeleton of $f$ and the summand $\frac{h'}{g}$ can be omitted from $f$ without changing $Skel(f)$. If $g$ is inessential, then since $Skel(f) = Skel(f^{-1})$ we can consider $\frac{g}{h}$ what brings us back to the previous case considered. Namely, if $g'$ is an inessential monomial of $g$, $\frac{g'}{h}$ can be omitted from $\frac{g}{h}$. Now, taking the inverse of the resulting fraction brings us back to $f$ with the monomial $g'$ omitted.
\end{proof}


\ \\

\subsection{Reducibility of principal kernels and skeletons}\label{Subsection:Reducibility_of_principal_kernels_and_skeletons}

\ \\

In this section we consider the notion of  reducibility with respect to a sublattice of kernels of the semifield $\mathbb{H}(x_1,...,x_n)$. The sublattice of kernels that is of interest to us are actually contained inside the lattice $\PCon(\langle \mathscr{R} \rangle)$. Note that $\PCon(\langle \mathscr{R} \rangle)$  is both a sublattice of $\Con(\langle \mathscr{R} \rangle)$ and of $\Con(\mathscr{R}(x_1,...,x_n))$.

\begin{defn}\label{defn_sublattice_of_kernels}
Let $\mathbb{S}$ be a semifield. A subset $\Theta$ of $\Con(\mathbb{S})$ is said to be a \emph{sublattice of kernels} if for every pair of kernels $K_1, K_2 \in \Theta$,
$$K_1 \cap K_2 \in \Theta \ \ \text{and} \ \ K_1 \cdot K_2 \in \Theta.$$
\end{defn}

\begin{exmp}
$\Con(\mathscr{R}(x_1,...,x_n)), \PCon(\mathscr{R}(x_1,...,x_n))$ are sublattices of kernels of $\mathscr{R}(x_1,...,x_n)$.
$\Con(\langle \mathscr{R} \rangle), \PCon(\langle \mathscr{R} \rangle)$ are sublattices of kernels of $\mathscr{R}(x_1,...,x_n)$ and of~$\langle \mathscr{R} \rangle$.
\end{exmp}

\begin{defn}\label{defn_theta_irreducible_maximal_kernels}
Let $\Theta$ be a sublattice of kernels of a semifield $\mathbb{S}$.
A proper (non-trivial) kernel $K \in \Theta$ is called \emph{$\Theta$-irreducible} if for any pair of kernels $A,B \in \Theta$
\begin{equation}
A \cap B \subseteq K \Rightarrow A \subseteq K \ \text{or} \ B \subseteq K.
\end{equation}
A kernel $K$ is called \emph{weakly $\Theta$-irreducible}  if for any pair of kernels $A,B$ of $\mathbb{S}$
\begin{equation}
A \cap B = K \Rightarrow A = K \ \text{or} \ B = K.
\end{equation}
$K$ is called \emph{$\Theta$-maximal} if for any kernel $A \in \Theta$
\begin{equation}
K \subseteq A \Rightarrow K=A \ \text{or} \ A = K.
\end{equation}
\end{defn}
Note that if a kernel $K \in \Theta$ is $\Theta$-irreducible then $K$ is weakly $\Theta$-irreducible.

\begin{defn}
Let $\mathbb{S}$ be a semifield and let $\Theta$ be a sublattice of kernels. Then $\mathbb{S}$ is said to be \emph{$\Theta$-irreducible} if for any pair of kernels $K_1 \in \Theta$ and $K_2 \in \Theta$ such that $K_1 \cap K_2 = \{1\}$, either $K_1 = \{ 1 \}$ or $K_2 = \{1\}$.
\end{defn}

\begin{rem}
 If $K$ is an $\Theta$-irreducible kernel of $\mathbb{S}$, then the quotient semifield $\mathbb{U}~=~\mathbb{S}/K$ is $\Theta$-irreducible.
\end{rem}
\begin{proof}
Let $K_1$ and $K_2$ be two kernels of $\mathbb{U}$. Then $\phi^{-1}(K_1) = K \cdot K_1$ and $\phi^{-1}(K_2) = K \cdot K_2$  are in $\Con(\mathbb{S})$, where $\phi :\mathbb{S} \rightarrow \mathbb{U}$ is the quotient map. Assume $K_1 \neq \{1\}$ and $K_2 \neq \{1\}$ are distinct kernels such that $K_1 \cap K_2 = \{ 1 \}$. Then $L_1 = K \cdot K_1$ and $L_2 = K \cdot K_2$ are two distinct kernels in $\Theta$ properly containing $K$ and
$$L_1 \cap L_2 =  \phi^{-1}(K_1) \cap \phi^{-1}(K_2) \subseteq \phi^{-1}(K_1 \cap K_2) = \phi^{-1}(1) = K;$$ contradicting the irreducibility of $K$.
\end{proof}

\begin{note}
In particular, taking $\mathbb{S}$ to be an idempotent semifield one can take $\theta$ to be the sublattice $\PCon(\mathbb{S})$, i.e., restrict the notion of reducibility to the principal kernels of $\mathbb{S}$.
\end{note}

Having definitions for reducibility and irreducibility of (principal) kernels, we now turn to define the analogous geometric notion for principal skeletons (for the case $\mathbb{S}~=~\mathbb{H}(x_1,...,x_n)$).

\begin{defn}\label{defn_theta_skeleton}
Let $\Theta$ be a sublattice of kernels in $\mathbb{H}(x_1,...,x_n)$. A skeleton $S$ is said to be a \emph{$\Theta$-skeleton} if there exists some kernel $K \in \Theta$ such that $S = Skel(K)$.
\end{defn}

\begin{defn}\label{defn_irreducible_skel}
Let $\Theta$ be a sublattice of kernels in $\mathbb{H}(x_1,...,x_n)$.
A $\Theta$-skeleton $S$ is said to be \emph{$\Theta$-reducible} if there exist some $\Theta$-skeletons $S_1$ and $S_2$ such that $S = S_1 \cup S_2$ and $S \neq S_1$ and $S \neq S_2$; otherwise $S$ is \emph{$\Theta$-irreducible}.\\
Let $\Theta \subseteq \PCon(\mathbb{H}(x_1,...,x_n))$ be a sublattice of kernels.
A principal (finitely generated) skeleton $Skel(f)$ is said to be \emph{$\Theta$-reducible} if there exist principal $\Theta$-skeletons $Skel(g)$ and $Skel(h)$ such that $Skel(f) = Skel(g) \cup Skel(h)$ and  $Skel(f) \neq Skel(h)$ and $Skel(f) \neq Skel(h)$; otherwise $Skel(f)$ is \emph{$\Theta$-irreducible}.
\end{defn}

\begin{rem}
By Definition \ref{defn_irreducible_skel}, $Skel(f)$ is $\Theta$-irreducible if for any pair of \linebreak $\Theta$-skeletons $Skel(g)$ and $Skel(h)$ such that $Skel(f) = Skel(g) \cup Skel(h)$, either \linebreak $Skel(f) = Skel(g)$ or $Skel(f) = Skel(h)$. Translating this last statement to the kernels \linebreak $\langle f \rangle$,  $\langle g \rangle$ and $\langle h \rangle$ in $Pcon(\langle \mathscr{R} \rangle)$,  and using the principal kernels - principal skeletons correspondence, we get the condition stated in Definition \ref{defn_theta_irreducible_maximal_kernels} of $\Theta$-irreducible \linebreak kernels.
\end{rem}

By the last remark we have
\begin{cor}
For $\Theta \subseteq \PCon(\langle \mathscr{R} \rangle)$ a sublattice of kernels of $\langle \mathscr{R} \rangle$.
$\langle f \rangle$ is \\ $\Theta$-irreducible if and only if $Skel(f)$ is $\Theta$-irreducible.
\end{cor}

\ \\

We now turn to study more closely the notion of reducibility of a kernel, and \linebreak introduce adequate geometric interpretations for reducibility of its skeleton. We \linebreak begin our discussion by introducing the notion of a reducible element of $\mathscr{R}(x_1,...,x_n)$ corresponding to  reducibility of the principal kernel it defines.\\

In the theory of commutative rings a generator of a principal ideal is unique up to multiplication by an invertible element, i.e., the association class of a generator of an ideal is unique. Recall that for two elements $a$ and $b$ of a commutative ring $R$, $a$ and $b$ are associates if and only if $a|b$ and $b|a$. In our setting things are slightly more complicated. We consider $\sim_K$ equivalence classes of generators of kernels. The equivalence $\sim_K$ will be shown to be  induced by a certain order relation, $ \succeq $, defined on the elements of the semifield and plays analogous role to that of $|$.

\begin{defn}\label{defn_theta_elements}
Let $\mathbb{S}$ be a semifield and let $\Theta$ be a sublattice of kernels in $\mathbb{S}$.\linebreak An element $a \in \mathbb{S}$ is said to be a \emph{$\Theta$-element} if $\langle a \rangle \in \Theta$.
\end{defn}

\begin{rem}
Following Definition \ref{defn_theta_elements}, any generator $b \in \mathbb{S}$ of $\langle a \rangle$ is \linebreak a $\Theta$-element.
\end{rem}

\begin{rem}
Let $\Theta \subseteq \PCon(\mathbb{H}(x_1,...,x_n))$ be a sublattice of kernels of $\mathbb{H}(x_1,...,x_n)$.
If $f, g \in \mathbb{H}(x_1,...,x_n)$ are $\Theta$-elements then so are
$$|f| \dotplus |g|, \ \ |f||g| \ \ \text{and} \ \ |f| \wedge |g|.$$
\end{rem}
\begin{proof}
Indeed, $\langle |f| \dotplus |g| \rangle = \langle |f||g| \rangle = \langle |f| \rangle \cdot \langle |g| \rangle  = \langle f \rangle \cdot \langle g \rangle \in \Theta$
and $\langle |f| \wedge |g| \rangle = \langle |f| \rangle \cap \langle |g| \rangle = \langle f \rangle \cap \langle g \rangle \in \Theta$.
\end{proof}

We proceed in developing a relation on elements of $\mathbb{H}(x_1,...,x_n)$, using \linebreak $\PCon(\mathbb{H}(x_1,...,x_n))$, which naturally induces a relation on $\Theta$-element for any sublattice of kernels $\Theta \subseteq~\PCon(\mathbb{H}(x_1,...,x_n))$.

\begin{nota}
Throughout the rest of this subsection we continue taking \linebreak $\Theta \subseteq~\Con(\mathbb{H}(x_1,...,x_n))$ to be a sublattice of kernels.
\end{nota}

\begin{defn}\label{defn_abstract_similarity_of_generators}
Let $\mathbb{S}$ be a semifield and let $a,b \in \mathbb{S}$. Define the following relation on $\mathbb{S}$
\begin{equation}
a \sim_K b  \Leftrightarrow \langle a \rangle = \langle b \rangle.
\end{equation}
This is clearly an equivalence relation, the classes of which are $$[a] = \{ a' \ : \ a' \ \text{is a generator of} \ \langle a \rangle \}.$$
Define the partial relation $ \succeq $ on $\mathbb{S}$ as follows:
\begin{equation}
 a \succeq b  \Leftrightarrow \exists a' \in [a]  \  \exists b' \in [b]  \ \ \text{such that} \ \ |a'| \geq |b'|.
\end{equation}
\end{defn}

\ \\

\begin{defn}
Let $\mathbb{S}$ be a semifield and let $a,b \in \mathbb{S}$. We say that $a$ and $b$ are \linebreak \emph{k-comparable} if $a \succeq b $ or $b \succeq a$, i.e., if there exist some $a' \sim_K a$ and $b' \sim_K b$ such that $|a'|$ and $|b'|$ are comparable  $|a'| \leq | b'| $ or $|b'| \leq |a'|$.
\end{defn}

We introduce explicitly the translation of Definition \ref{defn_abstract_similarity_of_generators} for the case where the semifield $\mathbb{S}$ is $\mathbb{H}(x_1,...,x_n)$ with $\mathbb{H}$ a bipotent divisible semifield.

\begin{rem}\label{rem_leq_wedge_connection}
For every $h,g \in \mathbb{H}(x_1,...,x_n)$ such that $g,h \geq 1$,
$$g \geq h \Leftrightarrow h = g \wedge w$$
for some $1 \leq w \in \mathbb{H}(x_1,...,x_n)$.
\end{rem}
\begin{proof}
If $g \geq h$ then taking $w = h \geq 1$ we have $h = g \wedge h$. Conversely, if $h = g \wedge w$ then $g \geq g \wedge w = h$. Moreover $w$ must admit $w \geq 1$ for otherwise if $w(a) < 1$ for some $a \in \mathbb{H}$ then $h(a) = g(a) \wedge w(a) \leq w(a) < 1$ contradicting the assumption that $h \geq 1$.
\end{proof}

\begin{defn}\label{defn_similarity_of_generators}
Let $f, g \in \mathbb{H}(x_1,...,x_n)$. Then
\begin{equation}
f \sim_K g  \Leftrightarrow \langle f \rangle = \langle g \rangle,
\end{equation}
the classes with respect to $\sim_K$ are $[g] = \{ g' \ : \ g' \ \text{is a generator of} \ \langle g \rangle$ \}.\\
Since $h \sim_K |h| \geq 1$ for any  $h \in \mathbb{H}(x_1,...,x_n)$, by Remark \ref{rem_leq_wedge_connection} the relation $\succeq $ on $\mathbb{H}(x_1,...,x_n)$ can be stated as follows:
\begin{equation}\label{similarity_of_generators_cond_1}
f \succeq g  \Leftrightarrow \exists w \in \mathbb{H}(x_1,...,x_n) \ \exists  f' \in [f] \ \text{such that} \  w,f' \geq 1 , \  \ g \sim_K f' \wedge w.
\end{equation}
\end{defn}

\begin{rem}
In view of Corollary \ref{cor_principal_ker_by_order} the partial relation \eqref{similarity_of_generators_cond_1} of Definition~\ref{defn_similarity_of_generators} can be rephrased as
\begin{equation}\label{similarity_of_generators_cond_2}
f \succeq g  \Leftrightarrow \exists w \in \mathbb{H}(x_1,...,x_n) \ \exists k \in \mathbb{N} \ \text{such that} \  w \geq 1, \ g \sim_K |f|^{k} \wedge w.
\end{equation}
Indeed, if  $f' \in [f]$ then there exists $k \in \mathbb{N}$ such that $ |f'|  \leq |f|^{k}$, and thus $|f'| = |f|^k \wedge  v$ for some $v \in \mathbb{H}(x_1,...,x_n)$ where $v \geq 1$ .  Assume $f \succeq g$, then, by \eqref{similarity_of_generators_cond_1},  $g \sim_K |f'| \wedge w$ for some  $w \geq 1$. So,  $g \sim_K |f'| \wedge w = (|f|^k \wedge  v) \wedge w = |f|^k \wedge (v \wedge w)$,  since $v,w \geq 1$ we have that $v \wedge w \geq 1$ which means that \eqref{similarity_of_generators_cond_2} holds for $g$. The converse direction is obvious, since $|f|^k \in [f]$ and $|f|^k \geq 1$ for any $k \in \mathbb{N}$.
\end{rem}

%
%
%
%

\begin{prop}\label{prop_wedge_inclusion}
For any $f,g \in \mathbb{H}(x_1,...,x_n)$,
\begin{equation}
g \succeq f \Leftrightarrow \langle g \rangle \supseteq \langle f \rangle.
\end{equation}
\end{prop}
\begin{proof}
Let $f, g \in \mathbb{H}(x_1,...,x_n)$ such that   $g \succeq f$. Then, by \eqref{similarity_of_generators_cond_2} we have that $g \succeq f$ if and only if  $f \sim_K |g|^{k} \wedge w$ for some $w \geq 1$ and $k \in \mathbb{N}$, if and only if there exists some $f' \sim f$ such that $f' = |g|^{k} \wedge w$. Note that $|g|^{k} \geq 1$ and $w \geq 1$ so $f' = |g|^{k} \wedge w \geq 1$ and thus $|f'| = f'$. Finally,  $|f'| = |g|^{k} \wedge w \Leftrightarrow |f'| \leq |g|^{k}  \Leftrightarrow |f'| \in \langle |g| \rangle \Leftrightarrow f' \in \langle g \rangle \Leftrightarrow$ $f \in \langle g \rangle$.
\end{proof}

\begin{cor}
For any $f,g \in \mathbb{H}(x_1,...,x_n)$
\begin{equation}
|g| \succeq |f| \Leftrightarrow \langle g \rangle \subseteq \langle f \rangle.
\end{equation}
\end{cor}
\begin{proof}
Follows from Remark \ref{prop_wedge_inclusion} along with the property that $|f| \in [f]$ or equivalently  $\langle f \rangle =\langle |f| \rangle$.
\end{proof}

\begin{rem}
For  $f, w \in \mathbb{H}(x_1,...,x_n)$, the following conditions are equivalent
\begin{enumerate}
  \item $f \succeq w  \ \text{and}  \ w \not \succeq f$.
  \item $\langle w \rangle \subset \langle f \rangle$ (with strict inclusion).
\end{enumerate}
\end{rem}
\begin{proof}
By condition \eqref{similarity_of_generators_cond_2} and Proposition \ref{prop_wedge_inclusion}.
\end{proof}

\begin{cor}
For any $f,g \in \mathbb{H}(x_1,...,x_n)$
\begin{equation}
f \sim_K g \Leftrightarrow f \succeq g  \  \  \text{and} \  \  g \succeq f.
\end{equation}
\end{cor}
\begin{proof}
$f \sim_K g \Leftrightarrow \langle f \rangle = \langle g \rangle$, which, by Proposition \ref{prop_wedge_inclusion} is equivalent to  $f \succeq g$ and $g \succeq f$.
\end{proof}

As stated above, the set of principal kernels of $\mathscr{R}(x_1,...,x_n)$, $\PCon(\mathbb{H}(x_1,...,x_n))$,  forms a lattice with respect to multiplication and intersection. Irreducibility of a kernel is translated there as follows:
\begin{defn}\label{defn_irreducible_ker}
A principal kernel $\langle f \rangle$ of the semifield $\mathbb{H}(x_1,...,x_n)$ is \emph{$\Theta$-irreducible} if
for $\Theta$ kernels $\langle g \rangle, \langle h \rangle \supseteq \langle f \rangle$
\begin{equation}
\langle f \rangle = \langle g \rangle \cap \langle h \rangle \Rightarrow \langle f \rangle = \langle g \rangle  \ \text{or} \ \langle f \rangle = \langle h \rangle;
\end{equation}
otherwise, it is \emph{$\Theta$-reducible}.
When $\Theta = \PCon(\mathbb{H}(x_1,...,x_n))$ we simply say reducible (irreducible).
\end{defn}

\begin{rem}
Let $f,g,h \in \mathbb{H}(x_1,...,x_n)$ such that $g \succeq f$ and $h \succeq f$. Then  $$f \succeq  |g| \wedge |h|  \Rightarrow f \sim |g| \wedge |h|.$$
\end{rem}
\begin{proof}
A consequence of Remark \ref{rem_wedge_is gcd}.
\end{proof}

\begin{defn}
Let $f \in \mathbb{H}(x_1,...,x_n)$. $f$ is said to be \emph{irreducible} if
\begin{equation}\label{cond_ereducibility_of_principals}
f \succeq |g| \wedge |h| \Rightarrow f \succeq |g| \ \text{or} \ f \succeq |h|.
\end{equation}
Otherwise $f$ is \emph{reducible}. A $\Theta$-element $f$ is said to be \emph{$\Theta$-irreducible} if both $g$ and $h$ are restricted to being $\Theta$ elements of
$\mathbb{H}(x_1,...,x_n)$. Otherwise $f$ is \emph{$\Theta$-reducible}.
\end{defn}

\begin{rem}
Since $\mathbb{H}(x_1,...,x_n)$ is an idempotent semifield it is distributive, thus
Proposition \ref{prop_dist_semifield1}(1) holds for $\mathbb{H}(x_1,...,x_n)$ implying that the following condition is equivalent to condition \ref{cond_ereducibility_of_principals}:
\begin{equation}
f \sim_K g \wedge h \Rightarrow f \sim_K g \ \text{or} \ f \sim_K h.
\end{equation}
\end{rem}

\begin{rem}
By transitivity of $\succeq$, if $f_1 \succeq g_1$, $f_1 \sim_K f_2$ and $g_1 \sim_K g_2$, then  $f_2 \succeq g_2$
\end{rem}

\begin{rem}\label{rem_invariance_of_irreducibility_wrt_representor}
For $\Theta$-elements $f$ and $f'$, such that $f' \sim_K f$, $f'$ is $\Theta$-irreducible if and only if $f$ is $\Theta$-irreducible.
\end{rem}
\begin{proof}
$f' \in [f]$ (or equivalently $f' \sim_K f$) if and only if $f \succeq f'$ and $f' \succeq f$. Assume $f$ is $\Theta$-irreducible. If $g$ and $h$ are $\Theta$-elements such that $g \wedge h \succeq f'$, then $g \succeq f'$ and $h \succeq f'$. Since  $f' \succeq f$  we get that $g \succeq f$ and $h \succeq f$. By $\Theta$-irreducibility of $f$ we have that $f \succeq g$ or $f \succeq h$. Now, as  $f' \succeq f$ we get that $f' \succeq g$ or $f' \succeq h$ as desired. The arguments of the proof in the opposite direction of the assertion are symmetric.
\end{proof}

\begin{note}
Remark \ref{rem_invariance_of_irreducibility_wrt_representor} actually follows irreducibility being defined by $\succeq$ (while $\succeq$ respects $\sim_K$). Nevertheless, we prove it explicitly.
\end{note}

\ \\

The following Proposition establishes the connection between irreducible principal kernels and irreducible elements of $\mathbb{H}(x_1,...,x_n)$.
\begin{prop}
For any principal kernel $\langle f \rangle$, with  $f \in \mathbb{H}(x_1,...,x_n)$,
$\langle f \rangle$ is irreducible if and only if $f$ is irreducible.
\end{prop}
\begin{proof}
If $f \sim_K |g| \wedge |h|$ then  $\langle f \rangle = \langle |g| \wedge |h| \rangle = \langle g \rangle \cap \langle h \rangle$. Since $\langle f \rangle$ is irreducible we have that either $\langle g \rangle \subseteq \langle f \rangle$ or $\langle h \rangle \subseteq \langle f \rangle$, which is equivalent, by Proposition \ref{prop_wedge_inclusion} to $f \succeq |g|$ or $f \succeq |h|$. Now, as $|g| \succeq |g| \wedge |h|$ and $|h| \succeq |g| \wedge |h|$ we get that $f \sim_K |g|$ or $f \sim_K |h|$ thus $f$ is irreducible. Conversely, assume $f$ is irreducible. Let $\langle f \rangle = \langle g \rangle \cap \langle h \rangle$. Then $\langle f \rangle = \langle |g| \wedge |h| \rangle$, and thus $f \sim_K |g| \wedge |h|$. As $f$ is irreducible, we have that $f \sim_K |g|$ or $f \sim_K |h|$ and so $\langle f \rangle = \langle |g| \rangle = \langle g \rangle$ or $\langle f \rangle = \langle |h| \rangle = \langle h \rangle$ as desired.
\end{proof}

The proof of the following Corollary is completely analogous.
\begin{cor}
For any principal kernel $\langle f \rangle \in \Theta$, with  $f \in \mathbb{H}(x_1,...,x_n)$\linebreak (a $\Theta$-element) ,
$\langle f \rangle$ is $\Theta$-irreducible if and only if $f$ is $\Theta$-irreducible.
\end{cor}

\begin{exmp}
Let $a = (\alpha_1,....,\alpha_n) \in \mathbb{H}^n$ and let $m_1,m_2 \in \PCon(\mathbb{H}(x_1,...,x_n)$ be  $m_1(x_1,...,x_n) = |\alpha_1^{-1}x_1| \dotplus \dots  \dotplus |\alpha_n^{-1}x_n|$ and $m_2(x_1,...,x_n) =  |\alpha_1^{-1}x_1| \cdot \dots  \cdot |\alpha_n^{-1}x_n|$
both admit $m_1(a) = 1$ and $m_2(a) = 1$. Since $|\alpha_i^{-1}x_i| \geq 1$ for each $1 \leq i \leq n$ and due to the property that for every $w_1,w_2 \geq 1$, \ $w_1 w_2 \leq (w_1 \dotplus w_2)^2$  and $w_1 \dotplus w_2 \leq w_1 w_2$, we get that $m_1 \sim_K m_2$.
\end{exmp}

\begin{rem}
$f \in \mathbb{H}(x_1,...,x_n)$ is reducible if and only if there exist $g,h \in \mathbb{H}(x_1,...,x_n)$, $g,h \leq 1$, such that $f \sim_K g \dotplus h$ where $f \not \sim_K g$ and $f \not \sim_K h$. Equivalently, $f$ is irreducible if for any such $g$ and $h$, $f \sim_K g$ or $f \sim_K h$
\end{rem}
\begin{proof}
The definition of $\wedge$ implies that  $g^{-1} \wedge h^{-1} = (g \dotplus h)^{-1} \sim_K  g + h$. Now, since $s^{-1} \sim_K s$ for every $s \in \mathbb{H}$, by definition of irreducibility we have that $f \sim_K g^{-1}$ or $f \sim_K h{-1}$ if and only if  $f \sim_K g$ or $f \sim_K h$ which yields the stated conclusion.
\end{proof}

\begin{lem}\label{lem_absolute_invariance_condition}
Let $f \in \mathbb{H}(x_1,...,x_n)$ be a rational function. We can write $f~=~\sum_{i=1}^{k}f_i$ where each $f_i$, with  $i=1,...,k$, is of the form $\frac{g_i}{h_i}$ where \linebreak $g_i,h_i \in \mathbb{H}[x_1,...,x_n]$ and $g_i$ is a monomial.
Then $Skel(f) = Skel(\wedge_{i=1}^{k}|f_i|)$ if and only if  for every $1 \leq i \leq k$ the following condition holds:
\begin{equation}\label{cond_domination}
f_i(x) = 1  \Rightarrow   \ \ f_j(x) \leq 1, \ \forall j \neq i.
\end{equation}
\end{lem}
\ \\
\begin{proof}
Denote $\tilde{f} = \wedge_{i=1}^{k}|f_i|$.
If $x \in \mathbb{H}^n$ such that $f(x)=1$, then there exists some \linebreak $i \in \{ 1,...,n \}$ such that $f_i(x) = 1$ and $f_j(x) \leq 1$ for every $j \neq i$. Thus $|f_i|(x)=|f_i(x)| = 1$ and $|f_j|(x)=|f_j(x)| \geq 1$, yielding that  $$\tilde{f}(x) = \inf\left\{|f_1(x)|,...,|f_k(x)|\right\} = \min\left\{|f_1(x)|,...,|f_k(x)|\right\} = |f_i(x)| = 1.$$ Conversely, if  $\tilde{f}(x) =1$ then there exists some $i \in \{ 1,...,n \}$ such that $|f_i(x)| = 1$ and $|f_j(x)| \geq 1$ for every $j \neq i$. Now, $|f_i(x)| = 1$ if and only if $f_i(x) = 1$ and as condition \eqref{cond_domination} holds, we get that $f_j(x) \leq 1$ for all $j \neq i$. Thus $$f(x) = \sup \{f_1(x),...,f_k(x)\} = \max\{f_1(x),...,f_k(x)\} = f_i(x) = 1.$$
\end{proof}

\subsection{Decompositions}

\ \\

Lemma \ref{lem_absolute_invariance_condition} provides us with some insight about reducible kernels.
Let $f$  be a rational function in $\mathbb{H}(x_1,...,x_n)$. We can write $f = \sum_{i=1}^{k}f_i$ where each $f_i$ is of the form $\frac{g_i}{h_i}$ with $g_i,h_i \in \mathbb{H}[x_1,...,x_n]$ and $g_i$ is a monomial.
If each time the value $1$ is attained by one of the terms $f_i$ in this expansion and all other terms attain values smaller or equal to $1$, then $\tilde{f} = \bigwedge_{i=1}^{k}|f_i|$  defines the same skeleton as $f$. Moreover, if $f \in \langle \mathbb{H} \rangle$ then $\tilde{f} \wedge |\alpha| \in \langle \mathbb{H} \rangle$, for $\alpha \in \mathbb{H} \setminus \{1\}$  is also a generator of $\langle f \rangle$. The reason we take $\tilde{f} \wedge |\alpha|$ is that we have no guarantee that each of the $f_i$'s in the above expansion is bounded.\\
We can generalize this idea as follows:\\
Let $f \in \mathbb{H}(x_1,...,x_n)$ be essential. Then $f$ is reducible if and only if there exists an additive expansion of $f$ of the form $f = \sum_{i=1}^{k}f_i$, where each $f_i= \frac{g_i}{h}$,  $g_i,h~\in~\mathscr{R}[x_1,...,x_n]$, ($h$ being the common denominator derived from the $h_i$'s above),  such that
for every $1 \leq i \leq k$ the following condition holds:
\begin{equation*}
f_i(x) = 1  \Rightarrow  \ \ f_j(x) \leq 1 \ \forall j \neq i.
\end{equation*}

\pagebreak

\begin{defn}\label{defn_decomposition}
Let $f \in \mathbb{H}(x_1,...,x_n)$. A \emph{decomposition} of $f$ is an equality of the form
\begin{equation}\label{eq_decomposition}
|f| = |u| \wedge |v|
\end{equation}
with $u,v \in \mathbb{H}(x_1,...,x_n)$. \\
The decomposition \eqref{eq_decomposition} is said to be \emph{trivial} if  \ $f \sim_K u$ \ or \ $f \sim_K v$ (equivalently $|f|~\sim_K~|u|$ or $|f|~\sim_K~|v|$). Otherwise, if  $f \not \sim_K u, v$ (equivalently $|f| \not \sim_K |u|,|v|$), \eqref{eq_decomposition}, is said to be \emph{non-trivial}.\\
A decomposition \eqref{eq_decomposition} is said to be a \emph{$\Theta$-decomposition} if both $u$ and $v$ (equivalently $|u|$ and $|v|$) are $\Theta$-elements (and thus, so is $f$) .
\end{defn}
%
%
%

\begin{lem}\label{lem_reducibility_and_decomposition}
If $f \in \mathbb{H}(x_1,...,x_n)$ is a $\Theta$-element, then $\langle f \rangle$ is $\Theta$-reducible if and only if there exists some generator $f'$ of $\langle f \rangle$ such that $f'$ has a nontrivial  $\Theta$-decomposition.
\end{lem}
\begin{proof}
If $\langle f \rangle$ is $\Theta$-reducible then there exist some kernels $\langle u \rangle$ and $\langle v \rangle$ in $\Theta$ such that $\langle f \rangle = \langle u \rangle \cap \langle v \rangle$ where $\langle f \rangle \neq \langle u \rangle$ and $\langle f \rangle \neq \langle v \rangle$. Since $\langle u \rangle \cap \langle v \rangle = \langle |u|~\wedge~|v|\rangle$ we have that $f' = |u| \wedge |v|$ is a generator of $\langle f \rangle$ which by the above is a nontrivial $\Theta$-decomposition of $f'$. Conversely, assume $f' = |u| \wedge |v|$ is a nontrivial $\Theta$- decomposition for some $f' \sim_K f$. Then $\langle f \rangle = \langle f' \rangle = \langle |u| \wedge |v| \rangle = \langle u \rangle \cap \langle v \rangle$.  Since $f' = |u| \wedge |v|$ is nontrivial, we have that $u \not \sim_K f'$ and $v \not \sim_K f'$ thus $\langle |u| \rangle = \langle u \rangle \neq \langle f' \rangle = \langle f \rangle$ and similarly $\langle v \rangle \neq \langle f \rangle$. Thus, by definition, $\langle f \rangle$ is $\Theta$-reducible.
\end{proof}

\begin{flushleft} We can equivalently rephrase Lemma \ref{lem_reducibility_and_decomposition} as follows: \end{flushleft}
\begin{rem}
$f$ is $\Theta$-reducible if and only if there exists some $f' \sim_K f$ such that $f'$ has a $\Theta$-decomposition.
\end{rem}

There is an immediate question arising from Definition \ref{defn_decomposition} and Lemma \ref{lem_reducibility_and_decomposition}:\\
If $f \in \mathbb{H}(x_1,...,x_n)$ has a non-trivial $\Theta$-decomposition and $g \sim_K f$, does $g$ have a non-trivial $\Theta$-decomposition too, and if so, what is the relation between this pair of decompositions?\\

In the following few paragraphs we will give an answer to both of these questions in the case $\theta = \PCon(\langle \mathscr{R} \rangle)$.

\pagebreak

\begin{rem}
By definition $a \wedge b = \inf(a,b)$  and $a \dotplus b = \sup(a,b)$ for any \\ $a,b~\in~\mathbb{H}(x_1,...,x_n)$. The following properties are immediate:
\begin{enumerate}
  \item $(a \wedge b)^k = a^k \wedge b^k$ for any $k \in \mathbb{Z}_{\geq 0}$.
  \item For any $s_1, ..., s_k , a_1,....,a_k  , b_1,....,b_k \in \mathbb{H}(x_1,...,x_n)$, The following holds: $$\sum_{i=1}^{k}s_i(a_i \wedge b_i) = (\sum_{i=1}^{k}s_i a_i ) \wedge (\sum_{i=1}^{k}s_i b_i).$$
  \item For any $s_1, ..., s_k \in \mathbb{H}(x_1,...,x_n)$ and  $d(i) \in \mathbb{Z}_{\geq 0}$,  $$\sum_{i=1}^{k}s_i(a \wedge b)^{d(i)} = \left(\sum_{i=1}^{k}s_i a^{d(i)} \right) \wedge \left(\sum_{i=1}^{k}s_i b^{d(i)}\right).$$
\end{enumerate}

Note that (3) is a consequence of (1) and (2). Indeed, by (1) we have that
$$\sum_{i=1}^{k}s_i(a \wedge b)^{d(i)} = \sum_{i=1}^{k}s_i((a^{d(i)}) \wedge (b^{d(i)}))$$
and by (2), $$\sum_{i=1}^{k}s_i((a^{d(i)}) \wedge (b^{d(i)})) = \left(\sum_{i=1}^{k}s_i a^{d(i)} \right) \wedge \left(\sum_{i=1}^{k}s_i b^{d(i)}\right).$$
\end{rem}

\begin{rem}
Let $h_1,...,h_k \in \mathbb{H}(x_1,...,x_n)$ be such that $h_i \geq 1$. Then $\sum_{i=1}^{k}s_ih_i \geq 1 $ for every $s_1,...,s_k \in \mathbb{H}(x_1,...,x_n)$ such that $\sum_{i=1}^{k}s_i = 1$.
Indeed, the statement holds since we have that $\sum_{i=1}^{k}s_ih_i \geq (\bigwedge_{i = 1}^{k}h_i)(\sum_{i=1}^{k}s_i) \geq \sum_{i=1}^{k}s_i = 1$.
\end{rem}

\ \\

\begin{thm}\label{thm_every_gen_is_reducible}
Let $\theta = \PCon(\langle \mathscr{R} \rangle)$.
If $\langle f \rangle$ is a (principal) $\Theta$-reducible kernel, then there exists a pair of $\Theta$-elements $g,h \in \mathbb{H}(x_1,...,x_n)$ such that $|f| = |g| \wedge |h|$ and \linebreak $|f| \not \sim_K |g|,|h|$.
\end{thm}

\begin{proof}
If $\langle f \rangle$ is a principal reducible kernel, then there exists some $f' \sim_K f$ such that $f' = |u| \wedge |v| = \min(|u|,|v|)$ for some $\Theta$-elements $u,v \in \langle \mathscr{R} \rangle $ where $f'~\not\sim_K~|u| , |v|$. Since $f'$ is a generator of $\langle f \rangle$, we have that  $|f| \in \langle f' \rangle$, so there exists some $s_1,...,s_k \in \mathbb{H}(x_1,...,x_n)$ such that $\sum_{i=1}^{k}s_i = 1$ and $|f| = \sum_{i=1}^{k}s_i (f')^{d(i)}$ with $d(i) \in \mathbb{Z}_{\geq 0}$ ($d(i) \geq 0$ since $|f| \geq 1$). \\
Thus $$f = \sum_{i=1}^{k}s_i (|u| \wedge |v|)^{d(i)} = \sum_{i=1}^{k}s_i (\min(|u|,|v|))^{d(i)} = \min\left(\sum_{i=1}^{k}s_i|u|^{d(i)},\sum_{i=1}^{k}s_i|v|^{d(i)}\right)$$  $$=|g| \wedge |h|$$ where $g = |g| = \sum_{i=1}^{k}s_i|u|^{d(i)}, h = |h| = \sum_{i=1}^{k}s_i|v|^{d(i)}$.
Now, by the above we have that  $\langle |f| \rangle \subseteq \langle |g| \rangle \subseteq \langle |u| \rangle$  and $\langle |f| \rangle \subseteq  \langle |h| \rangle \subseteq \langle |v| \rangle$, thus $Skel(f) \supseteq Skel(g) \supseteq Skel(u)$ and $Skel(f) \supseteq Skel(h) \supseteq Skel(v)$. We claim that $|g|$ and $|h|$ generate $\langle |u| \rangle $ and $\langle |v| \rangle$, respectively. Since $f' \sim_K |f|$ we have that $Skel(f') =Skel(|f|)$, thus for any $x \in \mathbb{H}^n$, $f'(x)=1 \Leftrightarrow |f|(x)=1$. Let $s_j (f')^{d(j)}$ be a dominant term of $|f|$ at $x$, i.e.,
$$|f| = \sum_{i=1}^{k}s_i(x) (f'(x))^{d(i)} = s_j(x) (f'(x))^{d(j)}.$$
Then we have that $f'(x) = 1 \Leftrightarrow s_j(x) (f'(x))^{d(j)}=1$, thus $f'(x) = 1 \Leftrightarrow s_j(x) = 1$ (since by Remark \ref{rem_torsion_free}, $(f'(x))^{d(j)}=1$ if and only if $f'(x)=1$). Now, consider $x \in Skel(g)$.  Then we have that $g(x) = 1$, i.e., $\sum_{i=1}^{k}s_i|u|^{d(i)} = 1$. Let $s_t|u|^{d(t)}$ be a dominant term of $g$ at $x$. If $s_t(x) =1$ then $|u|^{d(t)} = 1$ and thus $u = 1$ and $x \in Skel(u)$ otherwise $s_t(x) < 1$ (since $\sum_{i=1}^{k}s_i = 1$) and so, by the above $s_t (f')^{d(t)}$ is not a dominant term of $|f|$ at $x$. Since $s_t(x) < 1$ we have that
$$|u(x)|^{d(j)} = s_j(x)|u(x)|^{d(j)} < s_t(x)|u(x)|^{d(t)} = g(x) = 1,$$
for any dominant term of $|f|$ at $x$. Thus
\begin{equation}\label{eq_1_thm_every_gen_is_reducible}
s_j(x)(f'(x))^{d(j)}= s_j(x) (|u|(x) \wedge |v|(x))^{d(j)} \leq s_j(x)|u(x)|^{d(j)} < 1.
\end{equation}
On the other hand, as $Skel(f) \supseteq Skel(g)$, we have that $f'(x)=1$  and thus \newline $s_j(x)(f'(x))^{d(j)}=1$, contradicting \eqref{eq_1_thm_every_gen_is_reducible}. So, we have that $Skel(g) \subseteq Skel(u)$, so, by the above $Skel(g) = Skel(u)$ which in turn yields that $g$ is a generator of $\langle |u| \rangle = \langle u \rangle$. The proof for $h$ and $|v|$ is analogous. Consequently, we have that $g \sim_K |g| \sim_K |u|$ and $h \sim_K |h| \sim_K |v|$, so, as $|f|~\sim_K~f' \not \sim_K |u| , |v|$ we have that $|f| \not \sim_K |g|,|h|$.
\end{proof}

\ \\

\begin{cor}\label{cor_generator_structure}
If $f \in \langle \mathscr{R} \rangle$ and if $|f| = \bigwedge_{i=1}^{s} |f_i|$  for some $f_i \in \langle \mathscr{R} \rangle$, then for any $g$ such that $g \sim_K f$, $|g| =  \bigwedge_{i=1}^{s} |g_i|$ with $g_i \sim_K  f_i$ for $i = 1,...,s$.
\end{cor}
\begin{proof}
Follows successive application of Theorem \ref{thm_every_gen_is_reducible}.
\end{proof}
%

\begin{rem}\label{rem_expansion_coef_criterion}
Let $\mathbb{H}$ be a bipotent divisible semifield. Let $f \in \mathbb{H}(x_1,...,x_n)$, and let $g \in \langle f \rangle$ such that $g \sim_K f$. As $g \in \langle f \rangle$, $g$ can be written as $g = \sum_{i=1}^{k}s_i  f^{d(i)}$ for some $s_1,...,s_k \in \mathbb{H}(x_1,...,x_n)$ such that $\sum_{i=1}^{k}s_i = 1$ with $d(i) \in \mathbb{Z}$. We call $\sum_{i=1}^{k}s_i  f^{d(i)}$ a \emph{convex expansion} of $g$ with respect to $f$. Then for any $x \in \mathbb{H}^n$, $g(x) = 1$ if and only if $s_j(x) = 1$ for any leading term $s_j f^{d(j)}$ of $g$ at $x$.
\end{rem}

\ \\

\begin{cor}\label{cor_norm_decomposition}
If $\langle f \rangle$ is a kernel in $\Theta$, then $\langle f \rangle$ has a nontrivial $\Theta$-decomposition $\langle f \rangle = \langle g \rangle \cap \langle h \rangle$ if and only if $|f|$ has a non-trivial decomposition $|f|~=~|g'|~\wedge~|h'|$ with $|g'| \sim_K g$ and $|h'| \sim_K h$.
\end{cor}
\begin{proof}
If $|f| = |g'| \wedge |h'|$  then, as $|g'| \sim_K g$ and $|h'| \sim_K h$ we have that\linebreak $\langle f \rangle  = \langle |f| \rangle =  \langle |g'| \wedge |h'| \rangle = \langle |g'| \rangle \cap \langle  |h'| \rangle = \langle g \rangle \cap \langle h \rangle$. The converse follows the proof of Theorem~\ref{thm_every_gen_is_reducible}.
\end{proof}

Corollary \ref{cor_norm_decomposition} ensures us that there is a $\Theta$-decomposition of $|f|$ for every generator $f$ of a $\Theta$-reducible kernel in $\Theta$.

\begin{rem}\label{rem_wedge_equivalence}
By Corollary \ref{cor_max_semifield_principal_kernels_operations} we have that
$$\langle f \rangle \cap \langle g \rangle = \langle (f \dotplus f^{-1}) \wedge (g \dotplus g^{-1}) = \langle |f| \wedge |g| \rangle.$$
But, in fact, as for any $g' \sim_K g$ and $h' \sim_K h$, $\langle f \rangle \cap \langle g \rangle = \langle f' \rangle \cap \langle g' \rangle$,  we could have taken $|g'| \wedge |f'|$ instead of $|g| \wedge |f|$ on the righthand side of the equality, e.g., $\langle |f^k| \wedge |g^m| \rangle$ for any $m,k \in \mathbb{Z} \setminus \{ 0 \}$.
\end{rem}

\pagebreak

Since $|g| \wedge |h| = \min(|g|,|h|)$ we can utilize Remark \ref{rem_wedge_equivalence} to get the following observation:

\begin{prop}
A $\Theta$-element $f  \sim_K |g| \wedge |h| \in \mathbb{H}(x_1,...,x_n)$ is $\Theta$-reducible if the $\Theta$-elements $g$ and $h$ are not K-comparable.
\end{prop}
\begin{proof}
If $g$ and $h$ are k-comparable, then there exist some $g' \sim_K g$ and $h' \sim_K h$ such that $|g'| \geq |h'|$ or $|h'| \geq |g'|$. Without loss of generality, assume $|g'| \geq |h'|$. Then $\langle  |g| \wedge |h| \rangle  = \langle |g| \rangle \cap \langle |h| \rangle = \langle |g'| \rangle \cap \langle |h'| \rangle = \langle  |g'| \wedge |h'| \rangle = \langle  \min(|g'|,|h')| \rangle = \langle  |g'| \rangle = \langle g' \rangle = \langle  g \rangle$. Thus $\langle f \rangle = \langle  g \rangle$ so $f \sim_K g$ yielding that $f$ is $\Theta$-irreducible. Conversely, if  $g$ and $h$ are not k-comparable then $g \not \succeq h$ and $h \not \succeq g$, so by Proposition \ref{prop_wedge_inclusion}, $\langle h \rangle \not \subseteq \langle g \rangle$ and $\langle g \rangle \not \subseteq \langle h \rangle$ respectively. Then $\langle f \rangle = \langle g \rangle \cap \langle h \rangle \neq \langle g \rangle, \langle h \rangle$, which yields that  $f \not \sim_K g$ and $f \not \sim_K h$, and so $f$ is $\Theta$-reducible.
\end{proof}

\ \\

\ \\

\subsection{Regularity}
\ \\

Recall that our designated semifield $\mathscr{R}$ is defined to be bipotent, divisible, \linebreak archimedean and complete, the prototype being  $(\mathbb{R}^{+}, \dotplus, \cdot)$  by Corollary \ref{cor_totaly_ordered_isomorphic_to_reals}.\linebreak Thus $\mathscr{R}$ ( respectively $\mathscr{R}^n$) can be thought of as a topological metric subspace of $\mathbb{R}$ (respectively $\mathbb{R}^n$) with the induced topology derived from the usual (Euclidian) topology of $\mathbb{R}$ (respectively $\mathbb{R}^n$).\ \\


There are two general types of nontrivial principal skeletons in $\mathscr{R}^n$: Skeletons not containing a region of dimension $n$ and skeletons that do contain a region of dimension $n$.
These two types of principal skeletons emerge from two distinct types of kernels, characterized by their generators. The first type of skeletons correspond to principal kernels generated by an element of $\mathscr{R}(x_1,...,x_n)$ which we call \emph{regular} while the second type of skeletons correspond to principal kernels generated by an irregular element of $\mathscr{R}(x_1,...,x_n)$.\\
Principal kernels encapsulate a relation of the form $f = 1$ for some $f \in \mathscr{R}(x_1,...,x_n)$, which is induced on the quotient semifield. The relation $f = \frac{\sum g_i}{\sum h_j} = 1$ is local by nature in the following sense:
Let $x \in \mathscr{R}^n$ be any point. Then there is at least one monomial $f_{i0}$ of the numerator and at least one monomial $g_{j0}$ of the denominator which are dominant at $x$. If more than one monomial is dominant in each case, say $\{f_{ik}\}_{k=1}^{s}$ and $\{g_{jm}\}_{m=1}^{t}$, then we have some set of additional relations of the form $f_{i0} = f_{ik}$ and $g_{j0} = g_{jm}$. Now, the non-dominant monomials of both numerator and denominator define some order relations on the variables, which in turn define a region of $\mathscr{R}^n$ over which the relations $\{ f_{ik} = g_{jm} \ : \ 0 \leq k \leq s, \ 0 \leq m \leq t \}$ hold. Every such relation translates by multiplying by inverses of variables to a relation of the form $1 = \phi(x_1,...,x_n)$ with $\phi \in \mathscr{R}(x_1,...,x_n) \setminus \mathscr{R}$ a Laurent monomial, and thus reduces the dimension. Note that in the special case in which $f_{i0}$ and $g_{j0}$ singly dominate and are the same monomial, no relation is imposed on the region described above, so we are left only with order relations defining the region. In essence,  leaving only dominant monomials at a neighborhood of a point, such local relations fall into two distinct cases:
\begin{itemize}
  \item An order relation of the form $1 \dotplus g = 1$ with $g \in \mathscr{R}(x_1,...,x_n)$, which in fact describes a relation of the form $s \dotplus t = s$ with $s,t \in \mathscr{R}[x_1,...,x_n]$. The resulting quotient semifield $\mathscr{R}(x_1,...,x_n)/\langle 1 \dotplus g \rangle$ does not reduce the dimension of $\mathscr{R}(x_1,...,x_n)$, but only imposes new order relations on the variables.
  \item A `regular' relation, in the sense that it is not an order relation. Such a relation reduces the dimensionality of the image of $\mathscr{R}(x_1,...,x_n)$ in the quotient semifield.
\end{itemize}
Both of these categories are illustrated below and will be studied in Section \ref{Section:HR_decomp_HyperDimension} concerning dimensionality of kernels and skeletons.
We define a regular element of $\mathscr{R}(x_1,...,x_n)$ to be such that does not translate (locally) to order relations but only to regular relations (locally).  This observation will allow us to characterize those relations which correspond to corner loci (tropical varieties in tropical geometry). The kernels corresponding to these relations will be shown to form a sublattice of the lattice of principal kernels (which is itself a sublattice of the lattice of kernels). \\

In the following, we will characterize the generators of principal kernels of \\ $\mathscr{R}(x_1,...,x_n)$ which correspond to corner loci, which we call corner integral rational functions. In this subset of elements of $\mathscr{R}(x_1,...,x_n)$, the regular elements correspond to the traditional tropical varieties considered in tropical geometry, which are precisely the supertropical varieties defined by tangible polynomials (see subsection \ref{section:Cornerloci} and  \cite{SuperTropicalAlg}), while the irregular correspond to supertropical varieties defined by supertropical nontangible polynomials. Evidently `tangible' polynomials form a multiplicative subset in the domain of supertropical polynomials.

\ \\

\begin{exmp}
Consider the quotient map $\phi : \mathscr{R}(x) \rightarrow \mathscr{R}(x)/\langle x \dotplus 1 \rangle$. This map imposes the relation $x \dotplus 1 = 1$ on $\mathscr{R}(x)$, which is just the order relation $x \leq 1$. Under the map $\phi$,  $x$ is sent to $\bar{x} = x \langle x \dotplus 1 \rangle$, where now, in $\Im(\phi) = \mathscr{R}(\bar{x})$, $\bar{x}$ and $\bar{1}$ are comparable as opposed to the situation in $\mathscr{R}$, where $x$ and $1$ are not comparable, i.e., do not admit any order relation.  If instead of considering $x \dotplus 1$ we consider $|x| \dotplus 1 = x \dotplus x^{-1} \dotplus 1$, then as $|x| \geq 1$ the relation $|x| \dotplus 1 = 1$ is in fact $|x| = 1$ which yields the substitution map sending $x$ to $1$. Note that $|x|$ and $1$ are comparable in $\mathscr{R}(x)$, as mentioned above, $|x| \geq 1$, which is equivalent to the relation imposed by the equality $|x| \dotplus 1 = |x|$ or equivalently by $|x|^{-1} \dotplus 1 = 1$. The kernel $ \langle |x|^{-1} \dotplus 1 \rangle$, as $|x|^{-1} \dotplus 1 = 1$, is just the trivial kernel $\langle 1 \rangle = \{1 \}$.  \\
As seen in Proposition \ref{prop_principal_ker} and Corollary \ref{cor_principal_ker_by_order}, order relations affect the structures of kernels in a semifield. For instance, consider a principal kernel in a semifield $\mathscr{R}$ generated by an element $a \in \mathscr{R}$. Then $b \not \in \langle a \rangle$ for any element $b \in \mathscr{R}$ such that $b$ is not comparable to $a$.
\end{exmp}

\ \\

\begin{defn}
Let $\mathbb{S}$ be a bipotent semifield. Let $f = \frac{h}{g} = \frac{\sum_{i=1}^{k} h_i}{\sum_{j=1}^{m} g_j} \in \mathbb{S}(x_1,...,x_n)$ such that $h_i$ and $g_j$ are monomials for all $1 \leq i \leq k$ and $1 \leq j \leq m$. Then $f$ is said to be \emph{regular} if for each $x \in \mathbb{S}^{n}$, there exist $h_i$ and $g_j$ such that $h(x) = h_i(x)$, $g(x) = g_j(x)$, and $h_i \neq g_j$. Otherwise $f$ is called \emph{irregular}.
\end{defn}

\begin{nota}
Let $Reg(\mathscr{R}(x_1,...,x_n))$ denote the set of regular elements in $\mathscr{R}(x_1,...,x_n)$.
\end{nota}

\begin{rem}\label{rem_regular_closed_operations}
If $f,g \in Reg(\mathscr{R}(x_1,...,x_n))$ such that $f \neq 1$ and $g \neq 1$, then the following elements are also in $Reg(\mathscr{R}(x_1,...,x_n))$:
$$f^{-1}, \ \ f^{k} \ \text{with} \  k \in \mathbb{Z}, \  f \dotplus g, \  |f|= f \dotplus f^{-1}, \ f \wedge g.$$
\end{rem}
\begin{proof}
$f^{-1}$ is regular since the definition of regularity is invariant taking inverses.\linebreak
$f^k$ is regular follows easily from the regularity of $f$.
We will now prove that $f \dotplus g$ is regular.
Let $f,g \in Reg(\mathscr{R}(x_1,...,x_n))$. Write $f = \frac{h}{w} =  \frac{\sum_{i=1}^{k} h_i}{\sum_{j=1}^{m} w_j}$ and \linebreak $g = \frac{t}{s} = \frac{\sum_{i=1}^{l} t_i}{\sum_{j=1}^{r} s_j}$. Then
$$f \dotplus g = \frac{u}{v}= \frac{hs \dotplus tw}{ws} =  \frac{\sum_{i=1}^{k}\sum_{j=1}^{r} h_is_j \dotplus \sum_{i=1}^{l}\sum_{j=1}^{m} t_iw_j }{\sum_{i=1}^{m}\sum_{j=1}^{r} w_is_j}.$$ Let $x \in \mathscr{R}^n$.
If $u(x) = h_i(x)s_j(x)$ for some $1 \leq i \leq k$ and $1 \leq j \leq r$. Then in particular $s_j$ dominates $s$ at $x$ and since $f \neq 1$ there exists some $w_{i'} \neq h_i$ such that $w_{i'}(x) = w(x)$ thus $w_{i'}(x)s_j(x) = ws(x) = v(x)$. Since we have multiplicative cancellation (invertibility) in $\mathscr{R}(x_1,...,x_n)$ , $h_is_j \neq w_{i'}s_j$. If $u(x) = t_i(x)w_j(x)$ for some $1 \leq j \leq r$ then in particular $w_j$ dominates $w$ at $x$ and since $f \neq 1$ there exists some $s_{i'} \neq t_{i}$  such that $s_{j'}(x) = s(x)$ thus $s_{i'}(x)w_j(x) = ws(x) = v(x)$. Again, as noted above we have that $t_iw_j \neq s_{i'}w_j$ as desired.
Finally, $f \dotplus f^{-1}$ and  $f \wedge g = (f^{-1} \dotplus g^{-1})^{-1}$ which yields that they are regular.
\end{proof}


%
%

\begin{defn}
Let $\mathbb{S}$ be a bipotent semifield. A principal kernel $K = \langle f \rangle$ of $\mathbb{S}(x_1,...,x_n)$ is said to be \emph{regular} if $f \in \mathbb{S}(x_1,...,x_n)$ is regular. The Skeleton $Skel(f)$ corresponding to $K$ is said to be \emph{regular} also.
We denote the family of regular skeletons by $RegSkl(\mathbb{S}^n)$ and the family regular principal skeletons by $RegPSkl(\mathbb{S}^{n})$.
\end{defn}

We need to show that regularity of a principal kernel is well-defined, i.e., that it is independent of the choice of the generator of the kernel.

\begin{rem}\label{rem_every_generator_is_regular}
Let $f' \in \langle f \rangle$ be a generator of $\langle f \rangle$. Then as we have previously shown $Skel(f') = Skel(f)$. Let $f' = \frac{h}{g} = \frac{\sum_{i=1}^{k} h_i}{\sum_{j=1}^{m} g_j}$. Assume that $f'$ is irregular. Then there exists some $a \in \mathscr{R}^n$ such that $h(a) = h_{i_0}(a) = g_{j_0}(a) = g(a)$ and $h_{i_0}~=~g_{j_0}$ where for every $i \neq i_0$ and $j \neq j_0$ \ $h_i(a) < h_{i_0}(a)$ and $g_j(a)~<~g_{j_0}(a)$, respectively. Now, since $f'$ is continuous and by the definition of $\mathscr{R}$ (being isomorphic to $\mathbb{R}^{+}$) there is a neighborhood $\varepsilon(a)$ of $a$, $\{ a \} \subset \varepsilon(a) \subseteq \mathscr{R}^n$ for which the above holds implying that $Skel(f') \neq Skel(f)$ since by the definition of regularity no such neighborhood exists.
\end{rem}

\begin{exmp}
If $f \in \mathscr{R}(x_1,...,x_n)$ such that $f \neq f \dotplus 1$ (i.e., $1$ is essential in $f \dotplus 1$ ), then $f \dotplus 1 = \frac{f \dotplus 1}{1}$ is not regular since $1$ being essential in the numerator coincides with the denominator over some nonempty region.
\end{exmp}

\begin{note}
Note that there may exist a mutual monomial of $h$ and $g$, but in such a case, it can not dominate both $h$ and $g$ at the same point.
\end{note}

\begin{cor}\label{cor_regular_lattice}
The set of regular principal kernels forms a sublattice of \linebreak
$\PCon(\mathscr{R}(x_1,...,x_n))$.
\end{cor}
\begin{proof}
This follows directly from Corollary \ref{cor_max_semifield_principal_kernels_operations} and Remark \ref{rem_regular_closed_operations}.
\end{proof}

\begin{defn}
We denote by $Reg\PCon(\mathscr{R}(x_1,...,x_n))$ the sublattice of regular principal kernels in $\PCon(\mathscr{R}(x_1,...,x_n))$ .
\end{defn}

\subsection{Corner-Integrality}\label{subsection:corner_integrality}

\ \\

\begin{defn}\label{defn_corner_integral}
Let $\mathbb{S}$ be a bipotent semifield and let $f$ be an element of $\mathbb{S}(x_1,...,x_n)$. Write $f = \frac{h}{g}= \frac{\sum_{i=1}^{k}h_i}{\sum_{j=1}^{m}g_j}$ where $h_i$ and~$g_j$ for $i=1,...,k$ and $j=1,...,m$ are the component monomials in $\mathbb{S}[x_1,...,x_n]$ of the numerator and denominator of $f$, respectively. We say $f$ is \emph{corner-integral} if the following pair of conditions holds for every $x \in \mathbb{S}^n$:
\begin{equation}\label{corner_int_cond_1}
 \exists i \neq j \in \{1,...,k\} \ \text{such that} \ h_i(x) = h_j(x) \Rightarrow
\end{equation}
$$ \exists \ t \in \{1,...,m\} \ \text{s.t} \ h_i(x) \leq g_t(x) \ \text{or} \ \exists s \in \{1,...,k\} \setminus \{i,j\} \ \text{s.t} \ h_s(x) > h_i(x).$$

\begin{equation}\label{corner_int_cond_2}
\exists i \neq j \in \{1,...,m\} \ \text{such that}  \ g_i(x) = g_j(x) \Rightarrow
\end{equation}
$$\exists \ t \in \{1,...,k\} \ \text{s.t} \ g_i(x) \leq h_t(x)  \ \text{or} \ \exists s \in \{1,...,m\} \setminus \{i,j\} \ \text{s.t} \ g_s(x) > g_i(x).$$

In other words, $f \in \mathbb{S}(x_1,...,x_n)$ is corner integral if for any $x \in \mathbb{S}^n$, if $x$ is a corner root of $h$ then $g(x) \geq h(x)$ (i.e., $g$ surpasses $h$ at $x$) and if $x$ is a corner root of $g$ then $h(x) \geq g(x)$ (i.e., $h$ surpasses $g$ at $x$).
\end{defn}

\begin{defn}\label{defn_corner_integrality}
Let $\mathbb{S}$ be a bipotent semifield. A principal kernel $K= \langle f \rangle$ of $\mathbb{S}(x_1,...,x_n)$  is said to be \emph{corner-integral} if $f \in \mathbb{S}(x_1,...,x_n)$ is corner-integral. In other words, a principal kernel is corner-integral if it has a corner-integral generator.
The skeleton $Skel(K)$ corresponding to~$K$ is said to be a corner-integral skeleton.
\end{defn}

\begin{rem}
When $f \in \mathbb{S}(x_1,...,x_n)$ is said to be corner integral we always take $f$ to be a reduced fraction.
For example, $x \in \mathscr{R}(x)$ is trivially corner-integral, though $\frac{(x+ \alpha)x}{x+\alpha}$ for $\alpha > 1$ is not - since substituting $\alpha$ for $x$ we get $\frac{\alpha^2 + \alpha^2}{\alpha + \alpha}$. Thus $\alpha$ is a corner root of the numerator which is not surpassed by the denominator since $\alpha^2 > \alpha$.
\end{rem}

%

\begin{nota}
We denote by $CI(\mathbb{S}(x_1,...,x_n))$ the set of corner-integral elements of $\mathbb{S}(x_1,...,x_n)$.
\end{nota}

\begin{rem}\label{rem_corner_integrality_of_an_expansion}
Let $f \in \mathscr{R}(x_1,...,x_n)$ be corner-integral. Then $f^{-1}$,  $f^{k}$, for any $k \in \mathbb{N}$, and $\sum_{i=1}^{m}f^{d(i)}$ with $d(i) \in \mathbb{Z}$, also are corner integrals.
\end{rem}
\begin{proof}
First, $f^{-1}$ is corner integral as the definition of corner integrality is invariant w.r.t taking inverses.
Write $f = \frac{h}{g}= \frac{\sum_{i=1}^{k}h_i}{\sum_{j=1}^{m}g_j}$ where $h_i, g_j \in \mathscr{R}[x_1,..., x_n]$ are monomials composing $f$'s numerator and denominator. Then $f^{k} = \frac{h^{k}}{g^{k}}$. Since the corner roots of $h$ and $g$ coincide with the corner roots of $h^k$ and $g^k$, respectively, corner integrality is preserved. For the last assertion, we consider two different cases: (1) If $x \in \mathscr{R}^n$ then $ f'(x) =\sum_{i=1}^{m}f(x)^{d(i)} = (\frac{h(x)}{g(x)})^{d(j)}$ for some $j \in \{1,...,m\}$. If $x$ is such that $f' =\sum_{i \in I}f(x)^{d(i)}$ for some subset $I \subseteq \{1,...,m\}$ where $|I| \geq 2$ where for any $s,t \in I$, $f(x)^{d(s)} = f(x)^{d(t)}$ and\linebreak $d(s) \neq d(t)$. W.l.o.g., assume $d(t) > d(s)$, thus $f(x)^{d(t)-d(s)} =1$ which by \linebreak Remark \ref{rem_torsion_free}, yields that $f(x) = 1$, so $h(x) = g(x)$. Now, write $d = \max_{i \in I}\{d(i)\}$, then $f' =\frac{\sum_{i \in I}h(x)^{d(i)}g(x)^{d-d(i)}}{g(x)^{d}}$ and since $h(x) = g(x)$, $f' =(\frac{\sum_{i \in I}g(x)}{g(x)})^{d} $ and integrality obviously holds. (2) If $x$ is not as in (1) then $f'(x) =\sum_{i=1}^{m}f(x)^{d(i)} = f(x)^{d(j)}$ for exactly one  $j \in \{1,...,m\}$. If $d(j) = 0$ corner-integrality at $x$ is trivial, thus we can assume $d(j) \geq 1$ for otherwise just take $f^{-1}$, and by the above $f'$ is corner integral at $x$ as $f^{d(j)}$ is corner-integral (at any point). Since one of the above options is true for any $x \in \mathscr{R}^n$ we get that $f'$ is corner-integral.
\end{proof}

\begin{rem}
It can be shown that if  $f,g \in \mathscr{R}(x_1,...,x_n)$ are corner-integral then \linebreak $|f| + |g|$ may not be corner-integral. Thus the collection of corner-integral principal \linebreak kernels is not a lattice. In our study we thus take the lattice \emph{generated} by principal corner-integral kernels which contains elements which are not corner-integral. These elements will be shown to correspond to finitely generated corner loci (to be introduced shortly) which are not principal.
\end{rem}

\ \\

\subsection{Appendix :  Limits of skeletons}

\ \\

The following discussion gives some of the flavor of the structure and behavior of irregular principal kernels of an archimedean and bipotent semifield $\mathbb{H}$.
We show that every principal kernel is a certain kind of limit of irregular principal kernels.

\begin{rem}\label{rem_limit_of_skeletons}

Consider the polynomial $f(x) = x$. As $|x| = x \dotplus x^{-1}$ generates the kernel $\langle x \rangle$, the skeleton corresponding to $f$ is
$$Skel(f) = Skel(|f|)=\{ x \in \mathbb{H}^{n} \ : \ |f(x)| = 1 \},$$
i.e., the vertical line for which $x=1$. We will now consider a pair of skeletons related to $Skel(f)$. Let $\alpha, \beta \in \mathbb{H}$ such that $\alpha, \beta \geq 1$. Define the following rational functions :
$$ f_{\alpha}(x) = \frac{1}{\alpha} f(x) \dotplus 1 = \begin{cases}
1  &  \  f(x)  \leq  \alpha; \\
\frac{f(x)}{\alpha}  & \ f(x) \geq \alpha;
\end{cases}$$

$$ _{\beta}f(x) = \frac{1}{\beta} f^{-1}(x) \dotplus 1 = \begin{cases}
\frac{f(x)^{-1}}{\beta}  &  \  f(x)  \leq  \frac{1}{\beta} ; \\
1  & \ f(x) \geq \frac{1}{\beta}.
\end{cases}$$

Define the rational function $f_{\alpha,\beta}(x) = f_{\alpha}(x) \dotplus  _{\beta}f(x)$, i.e.,
$$f_{\alpha,\beta}(x) = \begin{cases}
\frac{f(x)^{-1}}{\beta}  &  \  f(x)  \leq  \frac{1}{\beta} ; \\
1  & \ \frac{1}{\beta} \leq f(x)  \leq \alpha ; \\
\frac{f(x)}{\alpha}  & \ f(x) \geq \alpha.
\end{cases} $$
Then $Skel(f_{\alpha,\beta})$ is the stripe $\{ x  \ : \ \frac{1}{\beta} \leq f(x)  \leq \alpha\}$ containing $Skel(f)$.
Taking $\alpha = \beta$, we get that
$$f_{\alpha,\alpha}(x) = f_{\alpha}(x) \dotplus  _{\alpha}f(x) = \left(\frac{1}{\alpha} f(x) \dotplus 1\right) \dotplus \left(\frac{1}{\alpha} f^{-1}(x) \dotplus 1\right)$$ $$= \frac{1}{\alpha}( f(x) \dotplus f^{-1}(x)) \dotplus 1 = \frac{1}{\alpha}|f(x)| \dotplus 1.$$
Note that  $f_{\alpha}$, \ $ _{\beta}f$ are irregular functions, and so is $f_{\alpha,\beta}$ if either $\alpha \neq 1$ or $\beta \neq 1$. \linebreak Moreover, when $\alpha = \beta = 1$, we have that $f_{\alpha,\beta} = |f|$.\\
Assume $\mathbb{H}$ is divisible then $\mathbb{H}$ is dense. In such a case for \newline $\alpha_i, \beta_i \in \mathbb{H}$ such that $\alpha_i, \beta_i \geq 1, \ i=1,2$, $\alpha_1 > \alpha_2, \beta_1 > \beta_2$  we have that \linebreak  $Skel(f_{\alpha_1,\beta_1}) \supset Skel(f_{\alpha_2,\beta_2})$ (proper containment).
\end{rem}

\begin{lem}\label{lem_formal_limits}
If $f \in \mathbb{H}(x_1,...,x_n)$, then
\begin{equation}
|f| = \lim_{\alpha \rightarrow_+ 1, \beta \rightarrow_+ 1} f_{\alpha,\beta}.
\end{equation}
(Note that $f_{\alpha,\beta}$ converge uniformly to $|f|$ and $\alpha \rightarrow_+ 1$ means $\alpha \geq 1$, i.e., one sided (`positive') limit.) Consequently,
\begin{equation}
\langle f \rangle = \bigcup_{\substack{\alpha > 1 \\ \beta > 1}} \langle f_{\alpha,\beta} \rangle,
\end{equation}
and
\begin{equation}
Skel(f) = \bigcap_{\substack{\alpha > 1 \\ \beta > 1}} Skel(f_{\alpha,\beta}).
\end{equation}
Moreover, if \ $\mathbb{H}$ is divisible, we have that for $\alpha_1, \alpha_2, \beta_1, \beta_2 \in \mathbb{H}$ such that \linebreak $\alpha_i, \beta_i~\geq~1, \ i=1,2$, $\alpha_1 > \alpha_2, \beta_1 > \beta_2$  we have that
\begin{equation}
Skel(f_{\alpha_1,\beta_1}) \supset Skel(f_{\alpha_2,\beta_2}).
\end{equation}
\begin{flushleft}and thus\end{flushleft}
\begin{equation}
Skel(f) \subset Int(Skel(f_{\alpha,\beta})) \ \ \forall \alpha, \beta > 1 .
\end{equation}
Here $Int(A)$ is the interior of $A$ for a set $A$.
\end{lem}

\begin{proof}
The assertions follow directly from the construction of Remark \ref{rem_limit_of_skeletons} and \linebreak Corollary \ref{cor_max_semifield_principal_kernels_operations}. We note that, for $\alpha, \beta \in \mathbb{H}$ such that $\alpha > \beta$, since $\alpha^{-1}\beta < 1$ we have that $(\alpha^{-1}\beta) \dotplus 1 = 1$ and so $\beta f \dotplus 1 = (\alpha^{-1}\beta) \alpha f \dotplus 1 \in \langle f \rangle$, thus $\langle f_{\beta,\beta} \rangle \subseteq \langle f_{\alpha,\alpha} \rangle$.
\end{proof}

\newpage
We conclude with the following consequences for $\mathscr{R}(x_1,...,x_n)$:

\begin{cor}\label{cor_formal_limit}
Every principal kernel in $\mathscr{R}(x_1,...,x_n)$ is a limit of irregular principal kernels.
Every skeleton in $\mathscr{R}^{n}$ is a limit (with respect to inclusion) of irregular skeletons.
\end{cor}

\begin{cor}\label{cor_irregular_neigbourhood}
Let $\langle f \rangle$ be a principal regular kernel in $\mathscr{R}(x_1,...,x_n)$. Then there exists an irregular principal kernel $\langle g \rangle$ such that $\langle g \rangle \subset \langle f \rangle$ and \linebreak $Int(Skel(\langle g \rangle)) \supset Skel(\langle f \rangle)$.
\end{cor}
\begin{proof}
The existence of $\langle g \rangle$ follows from the above discussion.
\end{proof}

\newpage
\section{The corner loci - principal skeletons correspondence}\label{section:Cornerloci}

\ \\

In this section we will establish a connection between a geometric object, which is a subset of $\mathscr{R}^n$, called `corner-locus' and a certain kind of principal skeletons. In fact, corner locus and what will be shown to be its corresponding principal kernel are two distinct ways to define the exact same subset of $\mathscr{R}^n$.

\ \\


\subsection{Corner loci}

\ \\

\begin{defn}\label{defn_supertropical_polynomial_semiring}
Let the \emph{supertropical semiring of polynomials} $\mathcal{F}(\mathscr{R}[x_1,...,x_n])$  be the \emph{supertropical} polynomial semiring (defined in \cite{SuperTropicalAlg}, Definition (4.1))
$$(R[x_1,...,x_n] , \mathcal{G}[x_1,...,x_n] ,  \nu)$$
where $ R = \mathscr{R} \cup \nu(\mathscr{R})$, $\mathcal{G} = \nu(\mathscr{R})$ is called the \emph{ghost ideal} with  $\nu : R \rightarrow \mathcal{G}$ an idempotent endomorphism of semirings such that $\nu|_\mathcal{G} = id_\mathcal{G}$. $\nu$ is called the \emph{ghost map}.
The elements of $\mathcal{G}$ are called \emph{ghosts} while the elements of $\mathscr{R} = R \setminus \mathcal{G}$ are called \emph{tangibles}.
For any monomial $f \in \mathcal{F}(\mathscr{R}[x_1,...,x_n])$, $f^{\nu} = \nu(f) = f \dotplus f \in \mathcal{G}[x_1,...,x_n]$ is called a ghost monomial. A monomial $f \in \mathcal{F}(\mathscr{R}[x_1,...,x_n])$ is called \emph{tangible} if  $f \in \mathscr{R}[x_1,...,x_n]$ (equivalently, $f \not \in \mathcal{G}[x_1,...,x_n]$). For monomials $f \in \mathcal{F}(\mathscr{R}[x_1,...,x_n])$ and $g \in \mathcal{G}[x_1,...,x_n]$ one has that  $$f^{\nu} = f^{\nu} \dotplus f \ \text{and} \  f \cdot g \in \mathcal{G}[x_1,...,x_n].$$
A polynomial $f \in \mathcal{F}(\mathscr{R}[x_1,...,x_n])$ is called \emph{tangible} if none of its component \linebreak monomials is ghost; otherwise we say $f$ is \emph{non-tangible}.
\end{defn}

\begin{note}
In our study we consider the evaluations of a supertropical polynomial $f \in \mathcal{F}(\mathscr{R}[x_1,...,x_n])$ on the tangibles, i.e., over $\mathscr{R}^n$, i.e., we consider $\mathcal{F}(\mathscr{R}[x_1,...,x_n])$ as mapped to the semiring of functions $Fun(\mathscr{R}^n,R)$ where \linebreak $ R = \mathscr{R} \cup \nu(\mathscr{R})$ as in Definition \ref{defn_supertropical_polynomial_semiring}.
\end{note}

\ \\

\begin{rem}
Let $f \in \mathcal{F}(\mathscr{R}[x_1,...,x_n])$ be a supertropical polynomial. Then $f$ can be written uniquely in the form
\begin{equation}\label{eq_goast_tangible_decomposition}
f = \sum_{i=1}^{t} h_i \dotplus \sum_{j=1}^{s} g_j^{\nu}
\end{equation}
where $h_i, g_j \in \mathscr{R}[x_1,...,x_n]$ are distinct monomials for $1 \leq i \leq t$ and $1 \leq j \leq s$. The polynomial $h = \sum_{i=1}^{t} h_i \in \mathscr{R}[x_1,...,x_n]$ is called the \emph{tangible part} of $f$ and the polynomial  $g^{\nu} = g \dotplus g \in \mathcal{G}[x_1,...,x_n]$ for $g = \sum_{j=1}^{s} g_j$ is called the \emph{ghost part} of $f$.\linebreak Thus $f$ can be expressed uniquely in the form:
$$f = h \dotplus g^{\nu} = h \dotplus (g \dotplus g) = \sum_{i=1}^{t} h_i \dotplus \sum_{j=1}^{s} (g_j \dotplus g_j).$$
In view of this  we can consider an element of $\mathcal{F}(\mathscr{R}[x_1,...,x_n])$ as a polynomial each of whose component monomials occurs either once, if it is a component monomial of the tangible part, or twice, if it is a component monomial of the ghost part.
\end{rem}

\begin{note}\label{note_super_tropical_representation}
In what follows we consider a supertropical polynomial $f$ as a sum of tangible monomials in $\mathscr{R}[x_1,...,x_n]$ where a component monomial occurs once if it belongs to the tangible part of $f$ and twice if it belongs to the ghost part of $f$.
\end{note}

\begin{exmp}
The supertropical polynomial $f(x,y) = y^{\nu} \dotplus x \dotplus 1^{\nu}$ is considered as $x \dotplus y \dotplus y  \dotplus 1 \dotplus 1$ where $h(x,y) = x$ is its tangible part and $g(x,y) = y \dotplus y  \dotplus 1 \dotplus 1$ is its ghost part.
\end{exmp}

We begin by introducing the well known notion of tropical geometry called \linebreak `corner root'.

\begin{defn}\label{def_corner_root}
Let $f \in \mathscr{R}[x_1,...,x_n]$ be a polynomial. Then  $f = \sum_{i=1}^{k}f_i$ where each $f_i$ is a monomial.  A point $a \in \mathscr{R}^n$ is said to be a \emph{corner-root} of $f$ if there exist two distinct monomials $f_t$ and $f_s$ of $f$ such that $f(a) = f_s(a) = f_t(a)$.
\end{defn}

In \cite[Section (5.2)]{SuperTropicalAlg} Izhakian and Rowen have generalized the notion of (tangible) corner-root to $\mathcal{F}(\mathscr{R}[x_1,...,x_n])$ as follows:

\begin{defn}\label{def_st_corner_root}
Let $f \in \mathcal{F}(\mathscr{R}[x_1,...,x_n])$ be a supertropical polynomial. Write \linebreak $f=\sum_{i=1}^{k}f_i$ where each $f_i$ is a monomial.  A point $a \in \mathscr{R}^n$ is said to be a \linebreak \emph{corner-root} of $f$ if $f(a) \in \mathcal{G}$, i.e., if $f$ obtains a ghost value at $a$. This happens in one of the  following cases:
\begin{enumerate}
  \item There exist two distinct monomials $f_t$ and $f_s$ of $f$ such that $f(a) = f_s(a) = f_t(a)$.
  \item There exists a ghost monomial $f_t$ of $f$ such that $f(a) = f_t(a)$.
\end{enumerate}

\end{defn}

\begin{defn}\label{def_cor_locus}
A set $A \subseteq \mathscr{R}^n$ is said to be a \emph{generalized corner-locus} if $A$ is a set of the form
\begin{equation}
A  = \{ x \in \mathscr{R}^{n} \ : \ \forall f \in S , \ x \ \text{is a corner root of} \ f \}
\end{equation}
for some $S \subset \mathcal{F}(\mathscr{R}[x_1,...,x_n])$.\\
We write $Cor(S)$ for the corner locus defined by $S$.
In the case where $S = \{ f_1,...,f_r \}$ is finite, we write $A = Cor(f_1,...,f_n)$ to indicate that $A$ is a corner locus defined by the mutual corner roots of  $f_1,...,f_r \in \mathcal{F}(\mathscr{R}[x_1,...,x_n])$, and say that $A$ is a \emph{finitely generated} corner locus.\\
A corner locus $A \subseteq \mathscr{R}^n$ is called \emph{principal} if there exists a supertropical polynomial $f~\in~\mathcal{F}(\mathscr{R}[x_1,...,x_n])$ such that $A = Cor(f)$.
A corner locus $A \subseteq \mathscr{R}^n$ is called \emph{regular} if $A = Cor(S)$ for $S \subset \mathscr{R}[x_1,...,x_n]$ (i.e., $S$ contains only tangible polynomials). A corner locus not indicated to be regular is assumed to be generalized.
\end{defn}

In view of  Definition \ref{def_cor_locus}, we define an operator $Cor : \mathbb{P}(\mathcal{F}(\mathscr{R}[x_1,...,x_n])) \rightarrow \mathscr{R}^n$ (where $\mathbb{P}(\mathcal{F}(\mathscr{R}[x_1,...,x_n]))$ is the power set of $\mathcal{F}(\mathscr{R}[x_1,...,x_n])$)
\begin{equation}
Cor : S \subset \mathcal{F}(\mathscr{R}[x_1,...,x_n]) \mapsto Cor(S).
\end{equation}
%

We now proceed to study the behavior of the $Cor$ operator.

As a trivial consequence of Definition \ref{def_cor_locus} we have the following
\begin{rem}\label{rem_cor_roots_of_prod}
If $a,b \in \mathscr{R}^{n}$ are corner roots of $f,g \in \mathcal{F}(\mathscr{R}[x_1,...,x_n])$ respectively, then both $a$ and $b$ are corner roots of $f \cdot g$.
\end{rem}

\begin{rem}\label{rem_intersection_union_corner_loci}
If $S_A,S_B \subseteq \mathcal{F}(\mathscr{R}[x_1,...,x_n])$ then
\begin{equation}\label{cor_loci_prop1}
S_A \subseteq S_B \Rightarrow Cor(S_B) \subseteq Cor(S_A).
\end{equation}

Let $\{ S_i \}_{i \in I}$ be a family of subsets of $\mathcal{F}(\mathscr{R}[x_1,...,x_n])$ for some index set $I$. Then $\bigcap_{i \in I}Cor(S_i)$ is a corner locus and
\begin{equation}\label{cor_loci_prop2}
\bigcap_{i \in I}Cor(S_i) = Cor\left(\bigcup_{i \in I}S_i\right) \ ; \ \bigcup_{i \in I}Cor(S_i) \subseteq Cor\left(\bigcap_{i \in I}S_i\right).
\end{equation}
In particular, $Cor(S)= \bigcap_{f \in S}Cor(f)$.
\end{rem}
\begin{proof}
First, equality \eqref{cor_loci_prop1} is a set-theoretical direct consequence of the definition of corner loci.
In turn this implies that $Cor(\bigcup_{i \in I}S_i) \subseteq Cor(S_i)$ for each $i \in I$ and thus $Cor(\bigcup_{i \in I}S_i) \subseteq \bigcap_{i \in I}Cor(S_i)$. Conversely, if $x \in \mathscr{R}^n$ is in $\bigcap_{i \in I}Cor(S_i)$ then $x \in Cor(S_i)$ for every $i \in I$, which means that $x$ is a common corner root of $\{f \ : \ f \in S_i \}$. Thus $x$ is a common corner root of $\{f \ : \ f \in \bigcup_{i \in I} S_i \}$ which yields that  $x \in Cor(\bigcup_{i \in I}S_i)$. For the second equation (inclusion) in \eqref{cor_loci_prop2}, for each $j \in I$, $\bigcap_{i \in I}S_i \subseteq S_j$. Thus, by \eqref{cor_loci_prop1}, $ Cor(S_j) \subseteq Cor(\bigcap_{i \in I}S_i)$, and so, $\bigcup_{i \in I}Cor(S_i) \subseteq Cor(\bigcap_{i \in I}S_i)$.
\end{proof}

\begin{lem}\label{lem_intersection_union_finitely_generated_corner_loci}
For the case where $A$ and $B$ are finitely generated. If \linebreak $A=Cor(f_1,...,f_s)$ and $B = Cor(g_1,...,g_t)$, then $A \cap B = Cor(f_1,...,f_s,g_1,...,g_t)$, and $A \cup B = Cor(\{f_i \cdot g_j\}_{i=1,j =1}^{s,t})$.
\end{lem}
\begin{proof}
This is a consequence of Definition \ref{def_cor_locus} and Remarks \ref{rem_cor_roots_of_prod} and \ref{rem_intersection_union_corner_loci}. \linebreak As the first equality is obvious, we only prove $A \cup B = Cor\left(\{f_i \cdot g_j\}_{i=1,j =1}^{s,t}\right)$.\linebreak
Indeed, for each  $1 \leq i \leq s$ and each $1 \leq j \leq t$, \ $A \subseteq Cor(f_i)$ and $B \subseteq Cor(g_j)$. \linebreak Thus $A \cup B \subseteq Cor(f_i) \cup Cor(g_j) = Cor\left(\{f_i \cdot g_j\}\right)$ and so
$$A \cup B \subseteq \bigcap_{i=1,j=1}^{s,t}Cor\left(\{f_i \cdot g_j\}\right) = Cor\left(\{f_i \cdot g_j\}_{i=1,j =1}^{s,t}\right).$$ On the other hand, if $a \not \in A \cup B$ then there exist some $i_0$ and $j_0$ such that \linebreak $a \not \in Cor(f_{i_0})$ and $a \not \in Cor(g_{j_0})$. Thus $a \not \in Cor(f_{i_0}) \cup Cor(g_{j_0}) = Cor(f_{i_0} \cdot g_{j_0})$.\linebreak So   $a \not \in \bigcap_{i=1,j =1}^{s,t}Cor\left(\{f_i \cdot g_j\}\right) = Cor\left(\{f_i \cdot g_j\}_{i=1,j =1}^{s,t}\right)$, proving the opposite inclusion.
\end{proof}

\begin{rem}\label{rem_finite_unions_and_intersections_of_reg_loci}
In view of Lemma \ref{lem_intersection_union_finitely_generated_corner_loci}, it is obvious that if  $A=Cor(f_1,...,f_s)$ and $B = Cor(g_1,...,g_t)$  are regular then $A \cap B$ and $A \cup B$ are regular (i.e., defined by tangible polynomials).
\end{rem}


\begin{defn}
Denote the collection of corner loci in $\mathscr{R}^n$ by $CL(\mathscr{R}^n)$ and by $RCL(\mathscr{R}^n)$ the family of regular corner loci.  Denote the collection of finitely generated corner loci by $FCL(\mathscr{R}^n) \subset CL(\mathscr{R}^n)$ and the collection of principal corner loci by $PCL(\mathscr{R}^n) \subseteq FCL(\mathscr{R}^n)$.
Analogously we denote the collection of finitely generated regular corner loci by $FRCL(\mathscr{R}^n) \subset RCL(\mathscr{R}^n)$ and the collection of principal regular corner loci by $PRCL(\mathscr{R}^n) \subseteq FRCL(\mathscr{R}^n)$.
\end{defn}

\begin{rem}
By Remark \ref{rem_intersection_union_corner_loci} and Lemma \ref{lem_intersection_union_finitely_generated_corner_loci}, $CL(\mathscr{R}^n)$ is closed under intersections, while $FCL(\mathscr{R}^n)$ is also closed under finite unions.\\
Taking $f=1+1 = 1^{\nu}$ and  $g= \alpha$ with $\alpha \neq 1$, we get that $\mathscr{R}^n = Cor(f)$ and $\emptyset = Cor(g)$ are in $FCL(\mathscr{R}^n)$.
By Remark \ref{rem_finite_unions_and_intersections_of_reg_loci} $RFCL(\mathscr{R}^n)$ is a sublattice of $FCL(\mathscr{R}^n)$.
\end{rem}

\ \\

\subsection{Corner loci and principal skeletons}

\ \\

Having defined a corner locus, we proceed to construct the connection between corner loci and principal skeletons.
We begin by constructing it for the special case of principal corner locus.


\begin{prop}\label{prop_ker_corner}
Any principal corner locus of $\mathscr{R}^{n}$ (a set of corner roots of a \linebreak supertropical polynomial) is a principal skeleton. In fact, there is a map $f \mapsto \widehat{f}$ sending a supertropical polynomial $f$ to a rational function $\widehat{f}$ such that $x \in \mathscr{R}^n$ is a corner root of $f$ if and only if $\widehat{f}(x) = 1$.
\end{prop}
\begin{proof}
Let $f = \sum_{i=1}^{k}f_i \in \mathcal{F}(\mathscr{R}[x_1,...,x_n])$
be a polynomial represented as the sum of its component monomials $f_i, \ i=1,...,k$, where some of them may not be distinct (in fact, appear exactly twice, see Note \ref{note_super_tropical_representation}). Define the following element of $\mathscr{R}(x_1,...,x_n)$:
\begin{equation}
\widehat{f} = \sum_{i=1}^{k} \frac{f_i}{\sum_{j \neq i}f_j}.
\end{equation}
Then $x \in \mathscr{R}^n$ is a corner root of $f$ if and only if $\widehat{f}(x) = 1$. Moreover, if $x$ is a corner root of multiplicity $m$ then $\widehat{f}(x)=^{[m+1]}1$ (the notation $^{[m+1]}1$ indicates that $1$ occurs $m+1$ times).  \\
If $x$ is a corner root of $f$, then there exists a subset of $2 \leq r \leq m$ monomials, say $f_1,...,f_r$ such that $f_1(x) = ... = f_r(x) = f(x)$ and $f_i(x) > f_j(x)$ for every $i=1,...,r$ and  $j = r+1,...,m$. Notice that, as $r > 1$, each denominator in $\widehat{f}$ must contain as a summand some $f_i$ with $1 \leq i \leq r$. Thus, every denominator equals $f(x)$. Now, the number of numerators in $\widehat{f}$ obtaining the value $f(x)$ at $x$ is exactly $r$, one for each monomial $f_i$ such that $1 \leq i \leq r$, so there are exactly $r$ summands of $\widehat{f}$ obtaining the value $f(x)/f(x) = 1$ at $x$. As every other summand is of the form $\frac{f_s(x)}{\sum_{j \neq s}f_j(x)} = \frac{f_s(x)}{f(x)}$ with $s \in \{ r+1,...,m \}$, since $f_s(x) < f(x)$, all these summands are inessential at $x$, yielding the first direction of our claim. Conversely, let $\widehat{f}(x) = ^{[m+1]}1$  for some $x \in \mathscr{R}$, then there are $m+1$ summands $g_{i} = \frac{f_i}{\sum_{j \neq i}f_j}$, say $g_1,...,g_{m+1}$ such that $g_1(x) = ... = g_{m+1}(x) = 1$. Then, for each $i = 1,...,m+1$, we have  $f_i(x) = \sum_{j \neq i}f_j(x)$ so there exists at least one essential $f_j$ with $j \neq i$ such that $f_j(x) = f_i(x)$. Notice that this last observation yields that $f_i(x)$ is essential in $\sum_{i=1}^{m}f_i(x)$ and that also $g_j(x) =1$. Consequently, there are exactly $m+1$ essential monomials at $x$, $f_1,...,f_{m+1}$ obtaining the same value at $x$, which in turn yields that $x$ is a corner root of multiplicity $m = (m+1)-1$, as desired.\\
Finally, by the above, $Skel(\langle \widehat{f} \rangle) = Cor(f)$, i.e., the skeleton defined by $\widehat{f}$ is exactly  the corner locus of $f$.
\end{proof}

\begin{cor}\label{cor_regularity_correspondence}
$f \in \mathcal{F}(\mathscr{R}[x_1,...,x_n])$ is non-tangible iff $\widehat{f} \in \mathscr{R}(x_1,...,x_n)$ is irregular.
\end{cor}
\begin{proof}
If we take $f \in \mathcal{F}(\mathscr{R}[x_1,...,x_n])$ to be non-tangible, say $f_k$ and $f_s$ are the same monomial $g$ for some $1 \leq s,k \leq m$, then the terms $\frac{f_k}{\sum_{j \neq k}f_j}$ and $\frac{f_s}{\sum_{j \neq s}f_j}$ of $\widehat{f}$ would be irregular since $g$ occurs as the monomial in the numerator and as one of the summands of the denominator. The essentiality of $g$ in $f$ implies that at least one of those terms is essential in $\widehat{f}$, making it an irregular element of $\mathscr{R}(x_1,...,x_n)$. Reversing the last arguments yields that if $\widehat{f}$ is irregular then $f$ is non-tangible, i.e., has a component monomial which is ghost.
\end{proof}

\pagebreak

\begin{prop}\label{prop_absolute_invariance_condition_of_hatf}
Write $\widehat{f} = \sum_{i=1}^{k}h_i$ where each $h_i = \frac{f_i}{\sum_{j \neq i}f_j}$ for $i=1,...,k$. \linebreak
Then, for every $1 \leq i \leq k$,
\begin{equation}
h_i(x) = 1  \Rightarrow  \ \ h_j(x) \leq 1, \ \forall j \neq i.
\end{equation}
\end{prop}
\begin{proof}
If $x \in \mathscr{R}^n$ be such that $h_i(x) = 1$, then there exists $k \neq i$ such that $f_k(x)~=~f_i(x)$ and $f_k(x) \geq f_t(x)$ for every $t \in \{1,...,n\} \setminus \{i,k\}$. Thus \linebreak $h_k(x) = h_i(x) = 1$ and  $h_t(x) = \frac{f_t}{\sum_{j \neq t}f_j} \leq \frac{f_t}{f_k} \leq 1$. So  $h_j(x) \leq 1$ for every $j \in \{1,...,n\}$.
\end{proof}

\begin{rem}\label{rem_corner_to_skel}
Let $\Phi^{\star}: PCL(\mathscr{R}^n) \rightarrow PSkl(\mathscr{R}^n)$ denote the map
\begin{equation}
\Phi^{\star} : Cor(f) \mapsto Skel(\widehat{f})
\end{equation}
induced by the map $f \mapsto \widehat{f}$ given in Proposition \ref{prop_ker_corner}.
For $S \subset \mathcal{F}(\mathscr{R}[x_1,...,x_n])$,\linebreak let $\widehat{S} = \{ \widehat{f} = \Phi^{\star}(f) \ : \ f \in S \} \subseteq \mathscr{R}(x_1,...,x_n)$. Then by Proposition \ref{prop_skel_property1} and \linebreak Remark~\ref{rem_intersection_union_corner_loci} we have that
$$Cor(S) = \bigcap_{f \in S}Cor(f) = \bigcap_{g \in \widehat{S}}Skel(g) = Skel(\widehat{S}).$$
Thus, $\Phi^{\star}$ extends to a map
$$\Phi:~CL(\mathscr{R}^n)~\rightarrow~Skl(\mathscr{R}^n)$$
where $\Phi : CL(f) \mapsto Skl(\widehat{f})$. In particular, taking only finite generated corner loci, and recalling that finite intersections  and unions of principal skeletons are \linebreak principal skeletons, $\Phi$ restricts to the map $\Phi|_{FCL(\mathscr{R}^n)} : FCL(\mathscr{R}^n) \rightarrow PSkl(\mathscr{R}^n)$. \\
As our interest is in the latter map we will denote $\Phi|_{FCL(\mathscr{R}^n)}$ by $\Phi$.\\
Note that since the map $f \mapsto \widehat{f}$ sends tangible elements of $\mathcal{F}(\mathscr{R}[x_1,...,x_n])$ (i.e., \linebreak elements of $\mathscr{R}[x_1,...,x_n]$) to regular elements of $\mathscr{R}(x_1,...,x_n)$ we also have that \linebreak
$\Phi|_{FRCL(\mathscr{R}^n)} : FRCL(\mathscr{R}^n) \rightarrow RegPSkl(\mathscr{R}^n)$.
\end{rem}

%

\begin{lem}\label{lem_ker_prop}
Let $f = \sum_{i=1}^{k}f_i \in \mathcal{F}(\mathscr{R}[x_1,...,x_n])$
be a polynomial represented as the sum of its component monomials $f_i, \ i=1,...,k$. For any $i=1,...,k$, denote the $i$'th summand of $\widehat{f}$ defined in Proposition \ref{prop_ker_corner}, by $A_i = \frac{f_i}{\sum_{j \neq i}f_j} \in \mathscr{R}(x_1,...,x_n)$ .  Then for $1 \leq i,s \leq k$ such that $i \neq s$
\begin{equation}\label{eq_lem_ker_prop}
 A_i = A_s \ \Leftrightarrow A_i = A_s = 1 \  \ \text{or} \ A_i = A_s \ \text{are inessential in} \ \widehat{f}.
\end{equation}
\end{lem}
\begin{proof}
Assume $A_i = A_s$, i.e., $\frac{f_i}{\sum_{j \neq i}f_j} = \frac{f_s}{\sum_{l \neq s}f_l}$. Since we have multiplicative cancellation we can multiply both sides by the product of the denominators without changing the equality. Thus we can rephrase this equation as
\begin{equation}\label{eq_lem_ker_prop}
f_i(f_i \dotplus A) = f_s(f_s \dotplus A) \ ; \  \ A = \sum_{j \neq i,s}f_j
\end{equation}
Now, we can partition $\mathscr{R}^n$ into the following distinct regions:
\begin{enumerate}
  \item $f_i$ is the only essential monomial, in which case \eqref{eq_lem_ker_prop} has the form $f_i^2 = f_s \cdot f_i$ and thus $f_i = f_s$ which cannot hold (since $f_i$ is the only essential monomial).
  \item $f_s$ is the only essential monomial, in which case \eqref{eq_lem_ker_prop} has the form $f_i \cdot f_s = f_s^2$ and thus $f_i = f_s$ which cannot hold.
  \item $f_i = f_s$ are both essential.
  \item Both $f_i$ and $f_s$ are non-essential, in which case \eqref{eq_lem_ker_prop} has the form $f_i \cdot A = f_s \cdot A$ and thus $f_i = f_s$.
\end{enumerate}
Thus, the only possible solutions are $f_i = f_s$ both essential and $f_i = f_j < A$. When $f_i = f_s$ are essential we have that $A_i = \frac{f_i}{f_s} = \frac{f_s}{f_i} = A_s = 1$. When $f_i = f_j < A$ we have that  $A_i = \frac{f_i}{A} = \frac{f_s}{A} = A_s < 1$. Since there is always some $1 \leq j \leq k$ such that $A_j \geq 1$ (see the proof of Proposition \ref{prop_ker_corner}) we have that both $A_i$ and $A_s$ are inessential in $\widehat{f}$. The converse direction of \eqref{eq_lem_ker_prop} is obvious.
\end{proof}

\begin{cor}
By Proposition \ref{prop_absolute_invariance_condition_of_hatf} we can strengthen \eqref{eq_lem_ker_prop} as follows:
\begin{equation}\label{eq_lem_ker_prop_strengthened}
A_i = A_s \ \text{iff} \ A_i = A_s = 1 \ \text{and are essential in} \ \widehat{f} \ \text{or} \ A_i = A_s \ \text{are inessential in} \ \widehat{f}.
\end{equation}

\end{cor}

We now establish the connection between tropical essentiality as described in \linebreak Definition \ref{defn_tropical_essentiality} and the essentiality introduced in Definition \ref{def_inessentiality}.
We show that the map $\Phi$ of Remark \ref{rem_corner_to_skel} sending a finitely generated corner locus to its corresponding principal skeleton, sends tropical inessential terms of a defining polynomial $f$ of a principal locus to inessential terms in the defining element $\widehat{f}$ of the skeleton, and can be generally be omitted from the definition of $\widehat{f}$ without changing the skeleton defined by $\widehat{f}$.

\begin{prop}
Let $f = \sum_{i=1}^{k}f_i \in \mathcal{F}(\mathscr{R}[x_1,...,x_n])$
be a polynomial represented as the sum of its component monomials $f_i, \ i=1,...,k$. Let
$\widehat{f} = \sum_{i=1}^{k} \frac{f_i}{\sum_{j \neq i}f_j}$.
If $f_m$ is inessential in $f$, then $\frac{f_m}{\sum_{j \neq m}f_j}$ is inessential in $\widehat{f}$.
\end{prop}
\begin{proof}
 Denote $f'= \sum_{i=1, i \neq m}^{k} \frac{f_i}{\sum_{j \neq i}f_j}$, for each $1 \leq i \leq k$ denote $g_i = \frac{f_i}{\sum_{j \neq i}f_j}$ and for $i \neq m$ define $g'_i=\frac{f_i}{\sum_{j \neq t, m}f_j}$ ($g_i$ after omitting $f_m$ from the denominator).
If $f_m$ is inessential in $f$, then for every $x \in \mathscr{R}^n$ there exists a monomial $f_{t}$, \ $t=t(x) \neq m$ of $f$ such that $f_m(x) \leq f_t(x)$. Thus $g'_i(x) = g_i(x)$ for every $x \in \mathscr{R}^n$ and we can write $f' = \sum_{i=1; i \neq m}^{k}g'_i$.
Consequently $g_m = \frac{f_m}{\sum_{j \neq m}f_j} \leq \frac{f_m}{f_t} \leq 1$. So $g_m$ does not surpass any `roots' of $f'$. Now we have to show that $g_m$ does not contribute `roots' either. Indeed, if $g_m(x) = 1$ then there exists some $f_s$ in $f$ with $s \neq m$ such that $f_s(x) = f_m(x)$ and $f_s(x) \geq f_j(x)$ for all $j \in \{1,...,k\} \setminus \{m,s\}$. Since $f_m$ is inessential there must exists some $l \not \in \{s,m \}$ such that $f_l(x) \geq f_s(x), f_m(x)$ (and so, $f_l(x) = f_s(x)$), for otherwise $f_s$ and $f_m$ are the only monomials defining a corner root of $f$ at $x$ which yields that $f_m$ is essential in $f$, contradicting our assumption. Thus $g'_l(x) = \frac{f_l(x)}{\sum_{j \neq l, m}f_j(x)} = \frac{f_l(x)}{f_s(x)} = 1$. Note that we could also take $g'_s$ instead of $g'_l$. We have proved that $g_m$ and $f_m$ could be omitted from $\widehat{f}$ without affecting its skeleton and thus inessential by definition.
\end{proof}

\ \\
\begin{lem}
Any one of the summands $\frac{f_i}{\sum_{j \neq i}f_j}$ of the fractional function $$\widehat{f}~=~\sum_{i=1}^{k} \frac{f_i}{\sum_{j \neq i}f_j} \in \mathscr{R}(x_1,...,x_n)$$ of Proposition \ref{prop_ker_corner} can be omitted, without affecting the skeleton of $\widehat{f}$.
\end{lem}
\begin{proof}
For each $1 \leq i \leq k$ denote $g_i =\frac{f_i}{\sum_{j \neq i}f_j}$. Without loss of generality, consider the term $g_1 = \frac{f_1}{\sum_{j \neq 1}f_j}$. Then if $g_1(x) = 1$ for $x \in \mathscr{R}^n$ , then there exists $s \neq 1$ such that $f_s(x) = f_1(x)$ and $f_s(x)$ is dominant in the denominator. Thus $f_i(x) \leq f_1(x) = f_s(x)$ for all $i \not \in \{1,s\}$. This last observation yields that also $g_s(x)=1$. So, we have shown that $Skel(g_1) \subseteq Skel(\sum_{i=2}^{k}g_i)$. Now, it remains to show that $g_1$ does not surpass any point of the skeleton defined by $\sum_{i=2}^{k}g_i$. Assume $g_j(x) = 1$ for some $j \neq 1$. Then as above there exists $s \neq j$ such that $f_s(x) = f_j(x)$ and $f_s(x)$ is dominant in the denominator and $f_i(x) \leq f_j(x) = f_s(x)$ for all $i \neq j,s$. If $s \neq 1$ then we get that $g_1(x) \leq 1$ and thus $g_j$ is not surpassed by $g_1$ at $x$ . Thus assume $s = 1$. In such a case we get $g_1(x) = \frac{f_1(x)}{f_j(x)} = 1$, and once again $g_j$ is not surpassed by $g_1$ at $x$. As the occurrence of $g_1$ does not add or delete any point of the skeleton of $\sum_{i=2}^{k}g_i$, it can be omitted. Since  $Skel(\sum_{i=2}^{k}g_i) = Skel(\widehat{f})$ we have that $\langle \sum_{i=2}^{k}g_i \rangle = \langle f \rangle$, thus $\sum_{i=2}^{k}g_i$ is a generator of $\langle \widehat{f} \rangle$.
\end{proof}

\begin{rem}\label{rem_wedge_corr_to_corner_loci}
Let $f = \sum_{i=1}^{k}f_i \in \mathcal{F}(\mathscr{R}[x_1,...,x_n])$
be a supertropical polynomial represented as the sum of its component monomials $f_i, \ i=1,...,k$. In Proposition \ref{prop_ker_corner}, we considered the fractional function $\widehat{f} = \sum_{i=1}^{k} \frac{f_i}{\sum_{j \neq i}f_j} \in \mathscr{R}(x_1,...,x_n)$ and showed that $Skel(\widehat{f}) = Cor(f)$. Now, consider the element  $\tilde{f}~=~\bigwedge_{i=1}^{k} \left|\frac{f_i}{\sum_{j \neq i}f_j}\right| = \min_{\ i=1}^{\ k} \left\{ \left|\frac{f_i}{\sum_{j \neq i}f_j}\right| \right\}$ of $\mathscr{R}(x_1,...,x_n)$. By definition $\tilde{f} \geq 1$. Moreover, by Proposition \ref{prop_absolute_invariance_condition_of_hatf}, $\widehat{f}$ fulfills the condition \eqref{cond_domination} of Lemma \ref{lem_absolute_invariance_condition}, so   $\tilde{f}(x) = 1 \Leftrightarrow \widehat{f}(x) = 1$. Thus $Skel(\widehat{f}) = Skel(\tilde{f})$. \\
In view of the above, we can say the following: the skeleton of $\tilde{f}$ is $Cor(f)$. By the correspondence between principal kernels and skeletons we have that $\langle \tilde{f} \rangle = \langle \widehat{f} \rangle$, or equivalently, $\tilde{f} \sim_K \widehat{f}$.
\end{rem}

As a consequence of Remark \ref{rem_wedge_corr_to_corner_loci}, we get the following corollary:
\begin{cor}
A corner locus of $f = \sum_{i=1}^{k}f_i \in \mathcal{F}(\mathscr{R}[x_1,...,x_n])$ such that $\widehat{f}$ \linebreak contains more than $2$ (distinct) essential terms $g_i = \frac{f_i}{\sum_{j \neq i}f_j}$,  is a reducible skeleton (with respect to principal skeletons),  i.e., the kernel $\langle \widehat{f} \rangle$ is reducible (with respect to the sublattice of principal kernels $\PCon(\mathscr{R}(x_1,...,x_n))$).
\end{cor}

\begin{proof}
Let $\widehat{f}$ be in reduced form. $\tilde{f} = \bigwedge_{i=1}^{k} \left|\frac{f_i}{\sum_{j \neq i}f_j}\right|$ implies that $\langle \tilde{f} \rangle= \bigcap_{i=1}^{k} \left\langle \frac{f_i}{\sum_{j \neq i}f_j} \right\rangle$ and so $Skel(\widehat{f}) = Skel(\tilde{f}) =  \bigcup_{i=1}^{k} Skel(\frac{f_i}{\sum_{j \neq i}f_j})$. For each $1 \leq i \leq k$, denote $g_i = \frac{f_i}{\sum_{j \neq i}f_j}$. Assume that $\widehat{f}$ is irreducible, then there exists  $1 \leq t\leq k$ such that $Skel(\widehat{f})~=~Skel(g_t)$. Then $Skel(g_t \dotplus \sum_{i=1 ; i \neq t}^{k}g_i) = Skel(\widehat{f}) = Skel(g_t)$, and thus $\sum_{i=1 ; i \neq t}^{k}g_i$ is inessential, contradicting our assumption that $\widehat{f}$ contains more than two essential terms $g_i$.
\end{proof}

\begin{rem}
We now describe the kernels corresponding to regular corner loci composed of at most two monomials. Assume $f \in \mathscr{R}[x_1,...,x_n]$ is of the form \linebreak $f = f_1 \dotplus f_2$ where $f_1$ and $f_2$ are two distinct monomials. Then by the construction \linebreak introduced in Proposition \ref{prop_ker_corner}, $f$ translates to $\widehat{f} = \frac{f_1}{f_2} \dotplus \frac{f_2}{f_1} = (\frac{f_1}{f_2}) \dotplus (\frac{f_1}{f_2})^{-1} = |\frac{f_1}{f_2}|$.\linebreak Note that $|\frac{f_1}{f_2}| \not \sim_K |g| \wedge |h| = \frac{|h|\dotplus |g|}{|h|\cdot|g|}$ for any $|g|,|h| \in \mathscr{R}(x_1,...,x_n)$ such that \linebreak $|g| \wedge |h| \not \sim_K \{ |g|,|h| \}$ (i.e., $Skel(g) \not \subseteq Skel(h)$ and $Skel(h) \not \subseteq Skel(g)$), since the monomials $f_1$ and $f_2$ dominate the numerator and denominator, respectively, over all of $\mathscr{R}^n$ while $|g| \wedge |h|$ has at least two monomials dominating its numerator. Thus $\widehat{f}$ is irreducible and so corresponds to an irreducible kernel. In the case where $f = f_1$ is a single monomial, $f$ translates to $\widehat{f} = \frac{f_1}{0} \dotplus \frac{0}{f_1} = |0|$, and thus $\{ \widehat{f} = 1 \} = \emptyset$.
\end{rem}

We now give an example of applying our theory to the well-known tropical line $y \dotplus x \dotplus 1$.
\begin{exmp}\label{exmp_tropical_line_transformation}
Consider the polynomial $f = y \dotplus x \dotplus 1$ in $\mathscr{R}(x,y)$. Then $$\widehat{f} = \frac{y}{x \dotplus 1} \dotplus \frac{x}{y \dotplus 1} \dotplus \frac{1}{x \dotplus y} \ .$$ This form is not reduced, since each of the component fractions can be omitted from the sum without affecting the skeleton of $\widehat{f}$. Geometrically these three terms \linebreak correspond to the three angles formed by the tropical line, of which one can be evidently omitted. $\widehat{f}$ is corner integral and regular and thus $\tilde{f} = \left|\frac{y}{x \dotplus 1}\right| \wedge \left|\frac{x}{y \dotplus 1}\right| \wedge \left|\frac{1}{x \dotplus y}\right|$ is also a generator of $\langle \widehat{f} \rangle$, i.e, $\widehat{f} \sim \tilde{f}$. In reduced form, we can take $\widehat{f}= \frac{y}{x \dotplus 1} \dotplus \frac{x}{y \dotplus 1}$ \linebreak or $\widehat{f}= \frac{y}{x \dotplus 1} \dotplus \frac{1}{x \dotplus y}$ or $\widehat{f}= \frac{1}{x \dotplus y} \dotplus \frac{x}{y \dotplus 1}$. Say we take the former. Then $\tilde{f} = \left|\frac{y}{x \dotplus 1}\right| \wedge \left|\frac{x}{y \dotplus 1}\right|$. It can be easily seen that no more than one of the three terms introduced above can be omitted from $\widehat{f}$ without changing the skeleton.
\end{exmp}

\ \\

\begin{defn}
Let $f \in \mathscr{R}(x_1,...,x_n)$. Write $f= \sum_{i=1}^{k}f_i$, with each \\ $f_i \in \mathscr{R}(x_1,...,x_n)$ having monomial numerator for $i=1,...,k$. We say that $f$ admits the \emph{bound property}  if for any $x \in \mathscr{R}^n$ the following condition holds:
\begin{align*}
\forall\ 1 \leq  i,j \leq k  \ \  \text{such that}  \ \ i \neq j \ : \ f_i(x) = f_j(x) \Leftrightarrow   f_i(x) = f_j(x) = 1  \\  \text{or} \  \exists s \in \{1,...,k\} \setminus \{ i,j \} \ \text{such that} \ f_i(x), f_j(x) < f_s(x).
\end{align*}
The latter option means that $f_i$ and $f_j$ are not essential in $f$ at $x$.
\end{defn}

\begin{rem}\label{rem_eqiv_cond_cor_int}
Let $h \in \mathscr{R}(x_1,...,x_n)$. Then, if $h$ admits the bound property then $h$ admits condition \eqref{corner_int_cond_1} and $h^{-1}$ admits condition \eqref{corner_int_cond_2}.
\end{rem}
\begin{proof}
Clearly, the second assertion follows the first, by taking inverse elements. Write $h = \frac{f}{g}= \frac{\sum_{i=1}^{k}f_i}{\sum_{j=1}^{m}g_j}$ where $f_i, g_j \in \mathscr{R}[x_1,..., x_n]$ are monomials composing $h$'s numerator and denominator. Then $h = \sum_{i=1}^{k}(\frac{f_i}{g})$. Assume there exists some $x \in \mathscr{R}^n$ such that $f_i(x) = f_j(x)$, then we have that $\frac{f_i(x)}{g(x)}= \frac{f_j(x)}{g(x)}$. So by the boundary property this yields that either $\frac{f_i(x)}{g(x)} = 1 = \frac{f_j(x)}{g(x)}$ or $\frac{f_i(x)}{g(x)} = \frac{f_j(x)}{g(x)}$  are inessential in $h$. The former yields that $f_i(x) = g(x) = f_j(x)$ so there is some $t \in \{1,...,m\}$ such that $f_i(x) \leq g_t(x)$. The latter implies that $f_i(x)$ is inessential, as desired.
\end{proof}

\begin{prop}\label{prop_corner_integrality_of hatf}
Let $f \in \mathcal{F}(\mathscr{R}[x_1,...,x_n])$ be a supertropical polynomial. Then $\widehat{f}$ defined in Proposition \ref{prop_ker_corner} is corner integral. Moreover, if $f$ is tangible (no monomial occurs twice) then $\widehat{f}$ is regular.
\end{prop}
\begin{proof}
As a consequence of Lemma \ref{lem_ker_prop}, we get that $\widehat{f}$ admits the bound property. Thus by Remark \ref{rem_eqiv_cond_cor_int} we have that condition \eqref{corner_int_cond_1} holds. For condition \eqref{corner_int_cond_2}, since $\widehat{f} = \sum_{i=1}^{k} \frac{f_i}{\sum_{j \neq i}f_j}$, writing $\widehat{f}=\frac{h}{g}$ we have that $h = \sum_{i=1}^{k}f_i(\sum_{j \neq i}(\sum_{l \neq j}f_j))$ and $g = \prod_{i=1}{k}(\sum_{j \neq i}f_j)$. We want to show that every corner root of $g$ is surpassed by some summand of $h$. First note that the corner roots of $g$ are exactly those of the polynomial $w=\sum_{i=1}^{k}f_i$. Indeed, if $x \in \mathscr{R}^n$ is a corner root of $w$ then is a corner root of precisely $k-2$ factors $(\sum_{j \neq i}f_j)$ with $i$ an index of any chosen pair of essential monomials. Conversely, if $x \in \mathscr{R}^n$ is a corner root of $g$ then it is a corner root of  $k-2$ of its factors not involving the essential monomials, say $f_i$ and $f_j$, which yields essentiality of $f_i$ and $f_j$ in $w$, proving our claim. In view of the last assertion, we can restrict our attention to the corner roots of $w$. Now, assume $f_i(x)~=~f_j(x)$ for some $i,j \in \{ 1,...,k\}$. If there exists $s \in \{ 1,...,k\} \setminus \{i,j\}$ such that $f_s(x) > f_i(x)$, take the one attaining the maximal value at $x$, the $f_s(\sum_{q \neq s}(\sum_{l \neq q}f_q(x))) \geq g(x)$. If no such $s$ exists, then $f_i(x)=f_j(x)$ attains maximal value in $f$ and $f_i(x)=f_j(x) = f_i(\sum_{q \neq i}(\sum_{l \neq q}f_q(x))) = f_j(\sum_{q \neq j}(\sum_{l \neq q}f_q(x)))$ as desired.\\
By the construction of $\widehat{f}$ and Corollary \ref{cor_regularity_correspondence}, irregularity of $\widehat{f}$ implies a multiple \linebreak occurrence of the same monomial in $f$, so $f$ is non-tangible thus the second assertion follows.
\end{proof}

\pagebreak

\begin{prop}\label{prop_formal_poly_loci_correspondence}
For $h = \frac{\sum_{i=1}^{k}f_i}{\sum_{j=1}^{m}g_j} \in \mathscr{R}(x_1,...,x_n)$ where $f_i, g_j \in \mathscr{R}[x_1,...,x_n]$, define the polynomial
\begin{equation}\label{eq_locus}
\underline{h} = \sum_{i=1}^{k}f_i \dotplus \sum_{j=1}^{m}g_j \in \mathcal{F}(\mathscr{R}[x_1,...,x_n]).
\end{equation}
Let $Z=Skel(h) \subseteq \mathscr{R}^{n}$ be a principal corner integral skeleton. Then $Z$ corresponds to a corner locus of the supertropical polynomial $\underline{h} \in \mathcal{F}(\mathscr{R}[x_1,...,x_n])$, in the sense that the generalized corner locus $Cor(\underline{h})$ coincides with $Z$  in $\mathscr{R}^n$.  If $h$ is also regular then $\underline{h}$ is tangible so $Cor(\underline{h}) \in PRCL(\mathscr{R}^n)$ is a regular corner locus.
\end{prop}

\begin{proof}
The corner locus  of \eqref{eq_locus} is

\begin{align*}
U = \ & \bigcup_{i=1,j=1}^{k,m} \{ f_i = g_j , \ f_i,g_j \geq f_s,g_t \ \forall s \neq i, \ t \neq j \}  \\
 \ \cup \ & \bigcup_{\substack{i=1,j=1 \\ i \neq j}}^{k} \{ f_i = f_j , \ f_i,f_j \geq g_t \ \forall \ t \in \{1,...,m \} \} \\
 \ \cup \ & \bigcup_{\substack{i=1,j=1 \\ i \neq j}}^{m} \{ g_i = g_j , \ g_i,g_j \geq f_s \ \forall \ s \in \{1,...,k \} \}.
\end{align*}

As $h$ is corner integral we have that
$$U = \bigcup_{i=1,j=1}^{k,m} \{ f_i = g_j , \ f_i,g_j \geq f_s,g_t \ \forall s \neq i, \ t \neq j \}$$
as all corner roots of the numerator and denominator of $h$ are surpassed.
Now, the skeleton defined by $h$ is
\small{\begin{align*}
\{(\alpha_1,...,\alpha_n)  :  h(\alpha_1,...,\alpha_n) = 1 \} = \ & \{(\alpha_1,...,\alpha_n)  :  \frac{\sum_{i=1}^{k}f_i(\alpha_1,...,\alpha_n)}{\sum_{j=1}^{m}g_j(\alpha_1,...,\alpha_n)} = 1 \} \\
 = \ & \{(\alpha_1,...,\alpha_n)  :  \sum_{i=1}^{k}f_i(\alpha_1,...,\alpha_n) \} = \sum_{j=1}^{m}g_j(\alpha_1,...,\alpha_n) \} \\
 = \ & \bigcup_{i=1,j=1}^{k,m}\{ f_i = g_j , \ f_i,g_j \geq f_s,g_t \ \forall s \neq i, \ t \neq j \} = U.
\end{align*}}

If $h$ is taken to be also regular, then all multiple occurrences of monomials in \linebreak $\underline{h}=\sum_{i=1}^{k}f_i \dotplus \sum_{j=1}^{m}g_j$ are inessential and can be omitted, and thus $\underline{h}$ is tangible.
\end{proof}

\begin{defn}
Let $f,g \in \mathcal{F}(\mathscr{R}[x_1,...,x_n])$ be a pair of supertropical polynomials. Define the following relation
\begin{equation}
f \sim_L g  \ \ \Leftrightarrow Cor(f) = Cor(g).
\end{equation}
The relation $\sim_L$ is obviously  reflexive, symmetric and transitive, thus is an equivalence relation on $\mathcal{F}(\mathscr{R}[x_1,...,x_n])$.
\end{defn}

\begin{prop}\label{prop_regular_ci_skeletons_corner_loci_correspondence}
There is a $1:1$ correspondence between principal corner-integral skeletons and principal corner-loci which restricts to a correspondence between \linebreak principal regular corner-integral skeletons and principal regular corner-loci. This correspondence induces a correspondence between principal (regular) corner-integral kernels and principal (regular) corner-loci.
\end{prop}
\begin{proof}
Let $f \in \mathscr{R}(x_1,...,x_n)$ be corner integral, and let $f' \sim_K f$. Then
$$Skel(f') = Skel(f) = Cor(\underline{f}) = Skel(\widehat{\underline{f}}).$$
So, $f' \sim_K \widehat{(\underline{f})}$.\\
Conversely, Let $g \in \mathcal{F}(\mathscr{R}[x_1,...,x_n])$ and let $g' \sim_L g$. Then
$$Cor(g') = Cor(g) = Skel(\widehat{g}) = Cor(\underline{\widehat{g}}),$$
where the last equality holds since we have shown $\widehat{g}$ to be corner integral.
So $g' \sim_L  \underline{\widehat{g}}$.\\
The restriction to regular skeletons, their corresponding kernels and corner loci follows
the propositions introduced above.
\end{proof}

\begin{defn}
Let $\Omega$  be the lattice generated by principal corner integral kernels with respect to (finite) multiplications and intersections.
\end{defn}

\pagebreak

\begin{rem}
Now, using the procedure introduced in Remark \ref{rem_corner_to_skel} and  the correspondence introduced in Corollary \ref{cor_correspondence_bounded_pkernels_pskeletons}, we get that
$\Omega$ corresponds to the lattice of finitely generated generalized corner loci. Intersecting $\Omega$ with the lattice of regular kernels yields a lattice   $\Theta \subset \PCon(\mathscr{R}(x_1,...,x_n))$ generated by regular corner-integral principal kernels which corresponds to the lattice of regular finitely generated corner loci  $FRCL(\mathscr{R}^n)$.
\end{rem}
\begin{note}
We are mainly interested in the sublattice $\Theta \subset \PCon(\langle \mathscr{R} \rangle)$ of corner-integral principal kernels in $\Con(\langle \mathscr{R} \rangle)$, for which the above holds too. $\Theta$ also is a sublattice of the lattice of  corner-integral principal kernels in $\PCon(\mathscr{R}(x_1,...,x_n))$.
\end{note}

\begin{rem}
$\Theta$ defined above is a sublattice of kernels of $\PCon(\mathscr{R}(x_1,...,x_n))$ (or of $\PCon(\langle \mathscr{R} \rangle)$). Thus all results of the section concerning reducibility is applicable \linebreak to it.
\end{rem}

\textbf{In the subsequent section concerning kernel dimension, we will show that the lattice generated by corner integral regular  kernels is in fact the lattice of regular kernels! Thus the lattice of corner loci corresponds to the lattice of regular kernels.}

\newpage\subsection{Example: The tropical line}


\ \\

\begin{note}
In the following example we consider the rational function \linebreak $\widehat{f} = \left|\frac{x}{y  \dotplus  1}  \dotplus  \frac{y}{x  \dotplus  1}  \dotplus  \frac{1}{x  \dotplus  y}\right| \wedge |\alpha| \in \langle \mathscr{R} \rangle$ for any $\alpha \in \mathscr{R} \setminus \{1 \}$. As taking $\wedge |\alpha|$  does not affect the computations below, we prefer omitting it, and work with $\frac{x}{y  \dotplus  1}  \dotplus  \frac{y}{x  \dotplus  1}  \dotplus  \frac{1}{x  \dotplus  y}$ instead of its `copy' in $\langle \mathscr{R} \rangle$.
\end{note}

\begin{exmp}
Let $f = x  \dotplus  y  \dotplus  1$ be the tropical line. Its corresponding skeleton is defined by the rational function $\widehat{f} = \frac{x}{y  \dotplus  1}  \dotplus  \frac{y}{x  \dotplus  1}  \dotplus  \frac{1}{x  \dotplus  y}$, and so, its corresponding kernel in $\mathscr{R}(x,y)$ is $\langle \widehat{f} \rangle$. As shown above $\widehat{f} \sim_K \left|\frac{x}{y  \dotplus  1}\right| \wedge \left|\frac{y}{x  \dotplus  1}\right| \wedge \left|\frac{1}{x  \dotplus  y}\right|$. Moreover, any of the three terms above can be omitted. Thus we have that
$$\langle \widehat{f} \rangle = \left\langle \frac{x}{y  \dotplus  1} \right\rangle \cap \left\langle \frac{y}{x  \dotplus  1} \right\rangle \cap \left\langle x  \dotplus  y \right\rangle,$$
where each of the kernels comprising the intersection is contained in both of the remaining kernels (in the last kernel we chose to take $x \dotplus y$ as a generator instead of its inverse). Now, taking logarithms, it can be seen that $Skel(\langle x  \dotplus  y \rangle)$ is exactly the union of the bounding rays of the third quadrant. As
$$x \dotplus y \sim_K |x \dotplus y| = |x \dotplus y| \dotplus  (|x| \wedge |y|) = \left(|x \dotplus y| \dotplus |x|\right) \wedge \left(|x \dotplus y|  \dotplus  |y|\right),$$
 we have that  $\langle x \dotplus y \rangle = \langle |x \dotplus y| \dotplus |x| \rangle \cap \langle  |x \dotplus y|  \dotplus  |y| \rangle$. Notice that  $\frac{1}{x  \dotplus  y}$ does not admit the corner integrality condition, as the corner roots of $x$ and $y$ are not surpassed by $1$, thus does not correspond to a tropical hypersurface. We proceed by considering $\frac{x}{y  \dotplus  1}  \dotplus  \frac{y}{x  \dotplus  1}$ omitting the last term of $\widehat{f}$. Computing its representation as a single fraction, we get $\frac{x^2  \dotplus  x  \dotplus  y^2  \dotplus  y}{(x \dotplus 1)(y \dotplus 1)} = \frac{(x \dotplus y)(x \dotplus  y  \dotplus  1)}{(x \dotplus 1)(y  \dotplus  1)}$. The corner roots of the numerator are $\{x=1,y=1\}$ and $\{x=y\}$, and the corner roots of the denominator are $\{x=1\}$ and $\{y=1\}$. The corner root $\{x=y\}$ is surpassed by the denominator as $(x \dotplus 1)(y \dotplus 1) = (x \dotplus y)  \dotplus  (xy  \dotplus  1) \geq (x \dotplus y)$ thus this is a regular skeleton (or equivalently a regular kernel) as expected. We urge the reader not to try to put negative values into the last equation, since they, of course, do not exist in a semifield. Finally, we discuss the above wedge decomposition of $|x  \dotplus  y|$. It is a quite natural one. The geometric locus of the equation $|x| \wedge |y| = \min(|x|,|y|) = 1$ in a logarithmic scale is exactly the axes, as a union of the $x$-axis corresponding to $|x|=1$ and the $y$-axis corresponding to $|y|=1$. Intersecting it with the geometric locus of $|x \dotplus y|$ leaves the latter untouched as the former locus contains it. In fact, using such methods of intersections we can define any segment and ray in $\mathscr{R}^2$ using principal kernels, so only points in the plane are irreducible skeletons. This of course is not a problem, since we are still free to consider lattices inside the lattice of principal kernels, generated by designated subsets, which will be considered prime or irreducible. In fact, principal kernels leave us with maximal `flexibility' in our hands.
\end{exmp}

\begin{figure}
\centering

\begin{tikzpicture}[scale=3,
    axis/.style={thin, ->},
    important line/.style={very thick, ->},
    gap line/.style={thin},
    dline/.style={thick,dashed}]


    \draw[axis] (0,0)  -- (1,0) node[right, color=black] {$x$} ;
    \draw[axis] (0,0) -- (0,1) node[above, color=black] {$y$};
    \draw[important line] (0,0)  -- (-1,0) ;
    \draw[important line] (0,0) -- (0,-1) ;
    \draw[important line] (0,0) -- (0.7,0.7) ;

    \draw[dline] (-0.1,-1.1)  -- (-0.1,-0.1) -- (-1.1,-0.1) ;
    \draw[dline] (0.1,-1.1)  -- (0.1,-0.07) -- (1,0.8) ;
    \draw[dline] (-1.1,0.1)  -- (-0.1, 0.1) -- (0.8,1)   ;

    \node[left] at (0.5,-0.25) {$|\frac{y}{x \dotplus 1}|$};
    \node[right] at (-0.5,0.25) {$|\frac{x}{y \dotplus 1}|$};
    \node[left] at (-0.25, -0.25) {$|\frac{1}{x \dotplus y}|$};

 \end{tikzpicture}
\caption{$\tilde{f} = |\frac{x}{y  \dotplus  1}| \wedge |\frac{y}{x  \dotplus  1}| \wedge |\frac{1}{x  \dotplus  y}|$} \label{fig:Ex1}
\end{figure}

\newpage

\section{Corner-integrality revisited}

\ \\

In this section we study the notion of corner-integrality introduced in sub section \ref{subsection:corner_integrality}. We specify a procedure for finding a corner-integral kernel containing a given kernel or equivalently a supertropical hypersurface containing a given skeleton. We show that under this procedure corner-integral kernels are left unchanged.

\ \\

Let $f \in \mathscr{R}(x_1,...,x_n)$ be a rational function. We start by expressing corner-integrality of $f$ and thus of $\langle f \rangle$, the kernel generated by $f$, by a condition involving kernels.
Write $f = \frac{h}{g}$, where $h = \sum_{i=1}^{k}h_i$ and $g = \sum_{j=1}^{m}g_j$ where $h_i$ and $g_j$ for $i=1,...,k$ and $j=1,...,m$ are the component monomials in $\mathscr{R}[x_1,...,x_n]$ of the numerator and denominator of $f$, respectively. For simplicity, we assume that $f$ is in essential form, so each monomial affects $Skel(f)$, the skeleton defined by $f$.
As noted in Subsection \ref{subsection:corner_integrality}, $f$ is corner-integral if the following two conditions hold for every $x \in \mathscr{R}^n$:
\begin{equation}\label{corner_int_cond_1}
\exists i \neq j \in \{1,...,k\} \ \text{such that} \ h_i(x) = h_j(x) \Rightarrow
\end{equation}
$$ \exists \ t \in \{1,...,m\} \ \text{s.t} \ h_i(x) \leq g_t(x) \ \text{or} \ \exists s \in \{1,...,k\} \setminus \{i,j\} \ \text{s.t} \ h_s(x) > h_i(x).$$

\begin{equation}\label{corner_int_cond_2}
\exists i \neq j \in \{1,...,m\} \ \text{such that}  \ g_i(x) = g_j(x) \Rightarrow
\end{equation}
$$\exists \ t \in \{1,...,k\} \ \text{s.t} \ g_i(x) \leq h_t(x)  \ \text{or} \ \exists s \in \{1,...,m\} \setminus \{i,j\} \ \text{s.t} \ g_s(x) > g_i(x).$$

Equivalently, $f$ is corner integral if condition \eqref{corner_int_cond_1} holds for both $f$ and $f^{-1}$ (exchanging the $g_i$'s and the $h_j$'s in the condition).

\begin{prop}\label{prop_absolute_corner_integral}
Let $f \in \mathscr{R}(x_1,...,x_n)$. Then $f$ is corner integral if and only if $|f|$ is corner-integral.
\end{prop}
\begin{proof}
Write $f = \frac{h}{g}$ where $f,g \in \mathscr{R}[x_1,...,x_n]$. Then $|f| = \frac{h}{g} + \frac{g}{h} = \frac{h^2 + g^2}{gh}$. Since $|f| \geq 1$ we only need to check that the corner roots of the numerator of $|f|$ are surpassed by its denominator.  If $f$ is corner integral then in the numerator all corner roots of $g$ are surpassed by $h$ and vise versa so we are left with the scenario of a corner root $a$ such that $h^2(a) = g^2(a)$ which yields that $h(a) = g(a)$ (note that it is possible that $a$ will also be a corner root of $g^2$ or $h^2$). In such a case $h^2(a) = g^2(a) = gh(a)$ thus the value of the numerator of $|f|$ at $a$ is surpassed by the value of the denominator, proving that $|f|$ is corner-integral.
Conversely, assume $|f|$ is corner integral then $h^2 + g^2$ is surpassed by $gh$, in particular any corner root of $h^2$ is surpassed by $gh$ which yields that any corner root of $h$ is surpassed by $g$, and any corner root of $g^2$ is surpassed by $gh$ which yields that any corner root of $g$ is surpassed by $h$, thus $f$ is corner-integral.
\end{proof}

\begin{prop}\label{prop_wedge_corner_integrality}
If $f,g \in \mathscr{R}(x_1,...,x_n)$ are corner-integral, then $|f| \wedge |g|$ is corner-integral.
\end{prop}
\begin{proof}
Since $f$ and $g$ are corner integral we have that $|f|$ and $|g|$ is corner-integral. Write $|f| = \frac{f_1}{f_2}$ and $|g| = \frac{g_1}{g_2}$ where $f_1, f_2,g_1,g_2 \in \mathscr{R}[x_1,...,x_n]$. Now, $|f| \wedge  |g| = \frac{fg}{f+g} = \frac{f_1g_1}{f_1g_2 + g_1f_2}$. Since $f$ and $g$ are corner-integral the corner roots of $f_1$ are surpassed by $f_2$ and the corner roots of $g_1$ are surpassed by $g_2$. Thus the corner roots of $f_1g_1$ are surpassed by either $f_2g_1$ or $f_1g_2$. We note that there may be equalities of the form $f_{1,u}(x) g_{1,v}(x) = f_{1,t}(x) g_{1,s}(x)$ where $f_{1,u}(x) \neq f_{1,t}(x)$ and $g_{1,v}(x) \neq g_{1,s}(x)$ for some $x \in \mathscr{R}^n$, but these `corner roots' are surpassed by $f_1'(x)g_1'(x)$ where $f_1'$ and $g_1'$ are dominant monomials of $f_1$ and $g_1$ at $x$.
In the opposite direction, since $|f| \geq 1$ and $|g| \geq 1$ we have that $f_1 \geq f_2$ and $g_1 \geq g_2$, respectively. Thus $f_1g_1 \geq f_1g_2, f_2g_1$ so $f_1g_1 \geq f_1g_2 + f_2g_1$ and so in particular $f_1g_1$ surpasses all corner roots of $f_1g_2 + f_2g_1$.
\end{proof}

\begin{cor}\label{cor_intersection_of_CI}
An intersection of corner-integral kernels is a corner-integral kernel.
\end{cor}
\begin{proof}
Let $K_1$ and $K_2$ be corner-integral (principal) kernels and let $u_1$ and $u_2$ be a pair of corner-integral generators for $K_1$ and $K_2$, respectively. By Proposition \ref{prop_absolute_corner_integral} we may assume $u_1,u_2 \geq 1$. Take $f = u_1 \wedge u_2$, then by Proposition \ref{prop_wedge_corner_integrality} $f = u \wedge v = |u| \wedge |v|$ is corner-integral thus by definition $\langle f \rangle$  is corner-integral.
\end{proof}

\begin{cor}
Let $f \in \mathscr{R}(x_1,...,x_n)$. Then $f$ is corner-integral if and only if $|f| \wedge |\alpha|$ is corner integral for any $\alpha \neq 1$ in $\mathscr{R}$. \\
\end{cor}
\begin{proof}
By Proposition \ref{prop_absolute_corner_integral} we may assume $f \geq 1$, i.e. $f = |f|$.
Since $|\alpha|$ is trivially corner-integral by Proposition \ref{prop_wedge_corner_integrality} the corner-integrality of $f$ implies the corner-integrality of  $f \wedge |\alpha|$. Conversely, if $f$ is not corner-integral then $f = \frac{h_1 + h_2}{g}$ such that $h_1(x_0) = h_2(x_0) > g(x_0)$ for some $x_0 \in \mathscr{R}^n$. Then $f \wedge |\alpha| = \frac{|\alpha|(h_1 + h_2)}{|\alpha|g + (h_1 + h_2)}$. Since $|\alpha| >1$ we have that $|\alpha|(h_1(x_0) + h_2(x_0)) > (h_1(x_0) + h_2(x_0))$ and by assumption $|\alpha|(h_1(x_0) + h_2(x_0)) > |\alpha|g(x_0)$ so $(f \wedge |\alpha|)$ is not corner-integral.
\end{proof}

Recall that a $k$-kernel is a kernel in $\Con(\mathscr{R}(x_1,...,x_n))$ which is a preimage of a skeleton $Z = Skel(S) \subseteq \mathscr{R}$ for some subset $S \subset \mathscr{R}(x_1,...,x_n)$ under the map $Ker : \mathbb{P}(\mathscr{R}^n) \rightarrow \mathscr{R}(x_1,...,x_n)$.
We have shown that $Skel(S) = Skel(\langle S \rangle)$ where $\langle S \rangle$ is the kernel generated by the elements of $S$.
We have found these $k$-kernels  for the  restriction $Ker|_{\langle \mathscr{R} \rangle} : \mathbb{P}(\mathscr{R}^n) \rightarrow \langle \mathscr{R} \rangle$.
We have shown that the preimage in $\langle \mathscr{R} \rangle$ of a principal skeleton to be $Ker(Skel(f)) = \langle f \rangle \cap \langle \mathscr{R} \rangle$ so that for $\alpha \neq 1$ \ $Ker(Skel(\langle |f| \wedge |\alpha| \rangle)) = \langle |f| \wedge |\alpha| \rangle =  \langle f \rangle \cap \langle \mathscr{R} \rangle$.
Since corner integrality of $f$  implies corner integrality of $|f| \wedge |\alpha|$ we  refer to $\langle f \rangle$ rather then to $\langle f \rangle \cap \langle \mathscr{R} \rangle$ in our computations.

We will now translate condition \eqref{corner_int_cond_1} introduced above to the language of kernels.\\
Let $C_i(h)$ be the set of corner roots of $h$ attained by the monomial $h_i$ (and some other monomials) and denote by $\underline{h_i} = \sum_{j \neq i}h_j$. Then
condition \eqref{corner_int_cond_1} is equivalent to saying that if $x \in \mathscr{R}^n$ is a corner root of $h$ then $h(x) \leq g(x)$ (i.e $g$ surpasses $h$ at $x$). So, for every $1 \leq i \leq k$ we can formulate the condition by
\begin{align}\label{eq_step1}
\ & \ \{ x \in \mathscr{R}^n : x \in C_i(h) \} \subseteq  \{ x \in \mathscr{R}^n : h(x) \leq  g(x) \} \nonumber \\
= & \ \{ x \in \mathscr{R}^n : h_i(x) = \underline{h_i}(x) \} \subseteq  \{ x \in \mathscr{R}^n : h(x) \leq  g(x) \}  \nonumber \\
= & \ \left\{ x \in \mathscr{R}^n : \frac{h_i}{\underline{h_i}}(x) = 1 \right\} \subseteq  \left\{ x \in \mathscr{R}^n : \frac{h(x)}{g(x)} + 1 = 1 \right\}  \nonumber \\
= & \ Skel\left(\frac{h_i}{\underline{h_i}}\right) \subseteq  Skel\left(\frac{h}{g} + 1\right) = Ske(f+1) \nonumber.\\
\end{align}
This means that $f$ is corner-integral if all the corner roots of its numerator $h$ are contained in the skeleton defined by $f+1$, i.e. in the region of $\mathscr{R}$ over which $f \leq 1$.
Since $Skel(K_1) \subseteq Skel(K_2) \Leftrightarrow K_2 \subset K_1$ for any pair of $K$-kernels $K_1, K_2$ we have
\begin{equation}\label{eq_inclusion_1}
Skel\left(\frac{h_i}{\underline{h_i}}\right) \subseteq  Skel(f + 1) \Leftrightarrow \langle f + 1 \rangle \subseteq \left\langle \frac{h_i}{\underline{h_i}} \right\rangle
\end{equation}

\pagebreak
Now, intersecting both sides of the expression obtained in \eqref{eq_step1} by \linebreak  $Skel(f^{-1} +1)$ we get
\begin{align*}
\ \ & \ Skel(\frac{h_i}{\underline{h_i}}) \subseteq Skel(f + 1)\\
\Leftrightarrow & \ Skel\left(\frac{h_i}{\underline{h_i}}\right) \cap Skel(f^{-1} + 1) \subseteq  Skel(f + 1) \cap Skel(f^{-1} + 1) \\
\Leftrightarrow & \ Skel\left(\frac{h_i}{\underline{h_i}}\right) \cap Skel(f^{-1} + 1) \subseteq  Skel(f). \\
\end{align*}
Note that the second transition is an equivalence rather then implication since $Skel(f^{-1} +1) \cup Skel(f+1) = \mathscr{R}^n$.

Again, translating the resulting skeletons expression to kernels yields
\begin{equation*}
\ Skel\left(\frac{h_i}{\underline{h_i}}\right) \cap Skel(f^{-1} + 1) \subseteq  Skel(f)  \Leftrightarrow \langle f \rangle \subseteq \left\langle \frac{h_i}{\underline{h_i}} \right\rangle \cdot \langle f^{-1} + 1 \rangle.
\end{equation*}

Moreover, since the above inclusion holds for every $i \in \{ 1,...,k \}$ we have that
\begin{equation*}
\left(\bigcup_{i=1}^{k}Skel\left(\frac{h_i}{\underline{h_i}}\right)\right) \cap  Skel(f^{-1} + 1) \subseteq Skel(f)
\end{equation*}
and so  $$ \langle f \rangle \subseteq \langle f^{-1} + 1 \rangle \cdot ( \bigcap_{i=1}^{k} \langle \frac{h_i}{\underline{h_i}} \rangle)$$

Interchanging the roles of $h$ and $g$ we get that the second condition \ref{corner_int_cond_2}  translates to
\begin{equation*}
Skel\left(\frac{g_j}{\underline{g_j}}\right) \subseteq  Skel(f^{-1} + 1)
\end{equation*}
which is equivalent to
\begin{equation}\label{eq_inclusion_in_skel2}
\ Skel(\frac{g_j}{\underline{g_j}}) \cap Skel(f + 1) \subseteq  Skel(f).
\end{equation}
and so since the inclusion in \eqref{eq_inclusion_in_skel2} holds for every $j \in \{ 1,...,m \}$ we have that
\begin{equation*}
\left(\bigcup_{j=1}^{m}Skel\left(\frac{g_j}{\underline{g_j}}\right)\right) \cap  Skel(f + 1) \subseteq Skel(f)
\end{equation*}

Using the notation introduced above we have the following necessary and \linebreak sufficient conditions hold:
\begin{prop}\label{prop_CI_conditions}
$f = \frac{h}{g} \in \mathscr{R}(x_1,...,x_n)$ where $g,h \in \mathscr{R}[x_1,...,x_n]$ is corner-integral if and only if the following conditions hold
\begin{equation}\label{eq_prop_CI_conditions_1}
\left(\bigcup_{i=1}^{k}Skel\left(\frac{h_i}{\underline{h_i}}\right)\right) \cap  Skel(f^{-1} + 1) \subseteq Skel(f)
\end{equation}
\begin{equation}\label{eq_prop_CI_conditions_2}
\left(\bigcup_{j=1}^{m}Skel\left(\frac{g_j}{\underline{g_j}}\right)\right) \cap Skel(f + 1) \subseteq  Skel(f).
\end{equation}
\end{prop}
 
\begin{rem}\label{rem_alternative_expression}
Recall that for $f = \sum_{i=1}^{k}f_i \in \mathcal{F}(\mathscr{R}[x_1,...,x_n])$ where $\mathcal{F}(\mathscr{R}[x_1,...,x_n])$ is the supertropical semiring of polynomials, the map $f \mapsto \widehat{f} = \sum_{i=1}^{k} \frac{f_i}{\sum_{j \neq i}f_j}$ is sending a supertropical polynomial $f$ to a rational function $\widehat{f}$ such that $x \in \mathscr{R}^n$ is a corner root of $f$ if and only if $\widehat{f}(x) = 1$, i.e $Skel(\widehat{f}) = Cor(f)$.
Recall also that $Skel(\widehat{f}) =Skel(\tilde{f})$ where $\tilde{f} = \bigwedge_{i=1}^{k} \left|\frac{f_i}{\sum_{j \neq i}f_j}\right|$.
Notice that $$\bigcup_{i=1}^{k}Skel\left(\frac{h_i}{\underline{h_i}}\right) = Skel(\bigwedge_{i=1}^{k} \left|\frac{h_i}{\sum_{j \neq i}h_j}\right|) =  Skel(\tilde{h})$$ and similarly  $$\bigcup_{j=1}^{m}Skel\left(\frac{g_j}{\underline{g_j}}\right) = Skel(\tilde{g}).$$
Thus we can rewrite \eqref{eq_prop_CI_conditions_1} and \eqref{eq_prop_CI_conditions_2} in the form
$$Skel\left( \tilde{h} \right) \cap  Skel(f^{-1} + 1) \subseteq Skel(f) \ \text{and} \ Skel\left( \tilde{g} \right) \cap Skel(f + 1) \subseteq  Skel(f).$$
or in the form
$$Skel\left( \widehat{h} \right) \cap  Skel(f^{-1} + 1) \subseteq Skel(f) \ \text{and} \ Skel\left( \widehat{g} \right) \cap Skel(f + 1) \subseteq  Skel(f).$$
Finally, we have proved that $\widehat{f}$ is corner-integral for any supertropical polynomial $f \in \mathcal{F}(\mathscr{R}[x_1,...,x_n])$.
\end{rem}

\pagebreak

In view of Proposition \ref{prop_CI_conditions}, given $f = \frac{h}{g} \in \mathscr{R}(x_1,...,x_n)$, in order to obtain a corner-integral fraction whose skeleton contains $Skel(f)$ one must adjoin both $$\left(\bigcup_{i=1}^{k}Skel\left(\frac{h_i}{\underline{h_i}}\right)\right) \cap  Skel(f^{-1} + 1)$$ and
$$\left(\bigcup_{j=1}^{m}Skel\left(\frac{g_j}{\underline{g_j}}\right)\right) \cap Skel(f + 1)$$ to the skeleton of $f$.

Define the map $\Phi_{CI} : \mathscr{R}(x_1,...,x_n) \rightarrow \mathscr{R}(x_1,...,x_n)$ by taking $\Phi_{CI}(f)$, where $f = \frac{h}{g}$, to be the fraction whose skeleton is formed by adjoining all the necessary points required for $f$ to admit corner integrality to $Skel(f)$. Namely
\begin{equation}
\Phi_{CI}(f) = |f| \wedge \left(|f^{-1} + 1| + \tilde{h} \right) \wedge \left(|f + 1| + \tilde{g} \right).
\end{equation}
Then since $|f^{-1} + 1|, |f+1| \leq |f^{-1} + 1|+|f+1|= |f|$ we have that
$$ \Phi_{CI}(f) = \left(|f^{-1} + 1| + \left(|f| \wedge \tilde{h} \right)\right) \wedge \left(|f + 1| + \left(|f| \wedge \tilde{g}\right)\right).$$

By this definition we have that $$\langle \Phi_{CI}(f) \rangle =  \left(\langle f^{-1} + 1 \rangle \cdot \left(\langle f \rangle \cap \bigcap_{i=1}^{k} \left\langle \frac{h_i}{\underline{h_i}} \right\rangle\right)\right)  \cap \left(\left(\langle f + 1 \rangle \cdot \left(\langle f \rangle \cap \bigcap_{j=1}^{m} \left\langle \frac{g_j}{\underline{g_j}} \right\rangle\right)\right)\right).$$

\begin{prop}\label{prop_CI_functor}
Let $f = \frac{h}{g} \in \mathscr{R}(x_1,...,x_n)$ be a rational function, where \linebreak $h = \sum_{i=1}^{k}h_i$ and $g = \sum_{j=1}^{m}g_j$ where $h_i$ and $g_j$ for $i=1,...,k$ and $j=1,...,m$ are the component monomials in $\mathscr{R}[x_1,...,x_n]$ of the numerator and denominator of $f$, respectively. Then
\begin{align}\label{eq_prop_CI_functor}
Skel(\widehat{h+g}) \ = & \ Skel(f) \nonumber \\
\  \ \cup & \ \left(\left(\bigcup_{i=1}^{k}Skel\left(\frac{h_i}{\underline{h_i}}\right)\right)  \cap Skel(f^{-1} + 1)\right) \nonumber \\
\  \ \cup & \ \left( \left(\bigcup_{j=1}^{m}Skel\left(\frac{g_j}{\underline{g_j}}\right)\right) \cap Skel(f + 1) \right). \nonumber \\
\end{align}
Thus $Skel(\widehat{h+g}) = Skel(\Phi_{CI}(f))$.
\end{prop}
\begin{proof}
Let $a \in \mathscr{R}^n$ be a corner roots of $h + g$ then $a$ admits one of the following disjoint characterizations:
\begin{enumerate}
  \item $h(a) = g(a)$ which is equivalent to saying that $a \in Skel(f)$.
  \item $a$ ia a corner root of $h$ and $g(a) \leq h(a)$ (i.e. $f^{-1}(a)+1 = 1$) which is equivalent to saying that $a \in Skel(\widehat{h}) \cap Skel(f^{-1} + 1) = Skel(|\widehat{h}| + |f^{-1} + 1|)$.
  \item $a$ is a corner root of $g$ and $h(a) \leq g(a)$ (i.e. $f(a)+1 = 1$) which is equivalent to saying that $a \in Skel(\widehat{g}) \cap Skel(f + 1) = Skel(|\widehat{h}| + |f + 1|)$.
\end{enumerate}
Consequently
$$Skel(\widehat{h+g}) = Skel(f) \cup \left(Skel(\widehat{h}) \cap Skel(f^{-1} + 1)\right) \cup \left(Skel(\widehat{h}) \cap Skel(f + 1)\right).$$
Thus by Remark \ref{rem_alternative_expression} the equality in \eqref{eq_prop_CI_functor} holds.
\end{proof}

\begin{cor}\label{cor_CI_map_properties_for_skeletons}
Let $f \in \mathscr{R}(x_1,...,x_n)$. Then  $\langle \Phi_{CI}(f) \rangle$ is corner-integral, \linebreak $Ske(\Phi_{CI}(f)) \supseteq Skel(f)$ and $Ske(\Phi_{CI}(f)) = Skel(f)$ if and only if $f$ is  \linebreak corner-integral.
\end{cor}
\begin{proof}
The first claim follows Proposition \ref{prop_CI_functor} from which we have that \linebreak $\Phi_{CI}(f) \sim_K \widehat{f}$ where $\widehat(f)$ is corner-integral by Proposition \ref{prop_corner_integrality_of hatf}. The second claim is straightforward from the definition of $\Phi_{CI}(f)$ since
\small{\begin{align*}
Skel\left(|f| \wedge \left(|f^{-1} + 1| + \tilde{h} \right) \wedge \left(|f + 1| + \tilde{g} \right)\right)  = & \ Skel(f) \\ \cup & \ Skel\left(|f^{-1} + 1| + \tilde{h} \right) \cup Skel\left(|f + 1| + \tilde{g} \right).
\end{align*}}
The last statement follows Proposition \ref{prop_CI_conditions}.
\end{proof}

We can rephrase Corollary \ref{cor_CI_map_properties_for_skeletons} as follows:
\begin{cor}\label{cor_CI_map_properties_for_kernels}
Let $f \in \mathscr{R}(x_1,...,x_n)$. Then $\langle \Phi_{CI}(f) \rangle$ is a corner-integral kernel contained in $\langle f \rangle$. Moreover, $f$ is corner-integral if and only if \linebreak $\langle \Phi_{CI}(f) \rangle = \langle f \rangle$. Equivalently $Skel(\Phi_{CI}(f))$ is a principal corner-integral skeleton containing $Skel(f)$ which yields that $Skel(\Phi_{CI}(f))$ is supertropical hypersurface containing $Skel(f)$.
\end{cor}

By Remark \ref{rem_corner_integrality_of_an_expansion}, if $f$ is corner-integral then so is $f' = \sum_{i=1}^{m}f^{d(i)}$ with $d(i) \in \mathbb{Z}$ and thus $\langle \Phi_{CI}(f) \rangle =  \langle f \rangle =  \langle f' \rangle =  \langle \Phi_{CI}(f') \rangle$. In particular this applies to $f' = f^{-1} + f =|f|$. It turns out that this equality $\langle \Phi_{CI}(f) \rangle =  \langle \Phi_{CI}(|f|) \rangle$ holds for any $f$ as we prove in the following remark.
\begin{rem}\label{rem_corner_integrality_powers}
For any $f =\frac{h}{g} \in \mathscr{R}(x_1,...,x_n)$
\begin{equation}
\langle \Phi_{CI}(\sum_{i=1}^{k}f^{d(i)}) \rangle = \langle \Phi_{CI}(f) \rangle
\end{equation}
where $d(i) \in \mathbb{Z}$ is monotonically increasing for $i = 1,...,k$, $d(1) < 0, d(k) > 0$.
\begin{equation}
\langle \Phi_{CI}(f^k) \rangle = \langle \Phi_{CI}(f) \rangle
\end{equation}
for any $k \in \mathbb{Z} \setminus \{0\}$.
\end{rem}
\begin{proof}
Due to the Frobenious property   $\sum_{i=1}^{k}f^{d(i)} = \frac{\sum_{i=1}^{k}h^{s + d(i)}g^{t - d(i)}}{h^{s}g^{t}}= \frac{h^{s+t} + g^{s+t}}{h^{s}g^{t}}$ where $t=|d(k)|$ and $s=|d(1)|$. So 
$$\Phi_{CI}(\sum_{i=1}^{k}f^{d(i)}) = \Phi_{CI}(\frac{h^{s+t} + g^{s+t}}{h^{s}g^{t}}) = \widehat{h^{s+t} + g^{s+t} + h^{s}g^{t}} = \widehat{h^{s+t} + g^{s+t}} = \widehat{h + g}^{s+t}.$$ Since $s+t \neq 0$ we have that  $\widehat{h + g}^{s+t}$ is a generator of $\langle \widehat{h + g} \rangle$ thus \linebreak $\langle \widehat{h + g}^{s+t} \rangle = \langle \widehat{h + g} \rangle = \langle \Phi_{CI}(f) \rangle$.
The latter equality holds since \linebreak $\Phi_{CI}(\frac{h^{k}}{g^{k}}) =  \widehat{h^{k} + g^{k}}= \widehat{h + g}^{k}$. Since $k \neq 0$, $\widehat{h + g}^{k}$ is a generator of $\langle \widehat{h + g} \rangle$ thus \linebreak $\langle \widehat{h + g}^{k} \rangle = \langle \widehat{h + g} \rangle = \langle \Phi_{CI}(f) \rangle$.
\end{proof}

\begin{cor}
Let $f \in \mathscr{R}(x_1,...,x_n)$ be such that  $f = u_1 \wedge \dots \wedge u_k$ where $u_1,...,u_k \in \mathscr{R}(x_1,...,x_n)$ are each corner-integral. The $f$ is corner-integral and
\begin{equation}\label{eq_phi_wedge_similarity}
\Phi_{CI}(f) \sim_K \Phi_{CI}(u_1) \wedge \dots \wedge \Phi_{CI}(u_k)
\end{equation}
\end{cor}
\begin{proof}
$f$ is corner-integral by induction on Proposition \ref{prop_wedge_corner_integrality}. First we assume $f \geq 1$ (thus so are the $u_i$'s). Since $u_i$ is corner-integral $\Phi_{CI}(u_i) \sim_K u_i$ for $i=1,...,k$, thus $\Phi_{CI}(u_1) \wedge \dots \wedge \Phi_{CI}(u_k) \sim_K u_1 \wedge \dots \wedge u_k = f$ (note that $a \sim_K b \Leftrightarrow \langle a \rangle  = \langle b \rangle$ thus $\langle a \rangle \cap \langle b \rangle = \langle |a'| \wedge |b'| \rangle$ for any $a' \sim_K a, b' \sim_K b$). Since $f$ is corner-integral  $f \sim \Phi_{CI}(f)$ so \eqref{eq_phi_wedge_similarity} holds. For any $g \in \mathscr{R}(x_1,...,x_n)$ taking $| \dot |$ does not change corner-integrality and $\Phi_{CI}(|g|) \sim_ K \Phi_{CI}(g)$. Thus for any $f$ we can consider $|f|$ instead and apply the first case.
\end{proof}

\begin{exmp}
Let $f = \sum_{i=1}^{k}f_i \in \mathcal{F}(\mathscr{R}[x_1,...,x_n])$
be a supertropical \linebreak polynomial represented as the sum of its component monomials $f_i, \ i=1,...,k$.
Then $\tilde{f} = \bigwedge_{i=1}^{k} \left|\frac{f_i}{\sum_{j \neq i}f_j}\right|$ is corner-integral.\\
First note that $\tilde{f} \geq 1$. Now, for any given $i \in \{1,...,k\}$ denote $D_i = \frac{f_i}{\sum_{j \neq i}f_j}$, then by Remark \ref{rem_corner_integrality_powers} we have that  $\Phi_{CI}(|D_i|) \sim_K \Phi_{CI}(D_i) = \widehat{f}$. Then \linebreak  $\Phi_{CI}(\tilde{f}) \sim_K \bigwedge_{i=1}^{k}\Phi_{CI}(|D_i|) = \bigwedge_{i=1}^{k}\widehat{f} = \widehat{f} \sim_K \tilde{f}$ so $\Phi_{CI}(\tilde{f}) \sim_K \tilde{f}$ and $\tilde{f}$ is \linebreak corner-integral.
\end{exmp}

\begin{rem}
An $\mathscr{R}$-homomorphic image of corner-integral element of \linebreak $\mathscr{R}(x_1,...,x_n)$ is corner-integral. Thus the $\mathscr{R}$-homomorphic image of a corner-integral kernel is corner-integral.
\end{rem}
\begin{proof}
Let $f = \frac{h}{g}$ and let $a \in \mathscr{R}^n$ a corner root of $h$ so that $h_1(a) = h_2(a)$ for some component monomials $h_1,h_2$ of $h$. Then for a semifield epimorphism $\phi$, $h_1(\phi(a)) + h_2(\phi(a)) = \phi(h_1(a)) + \phi(h_2(a)) = \phi(h_1(a) + h_2(a)) = \phi(h(a)) = h(\phi(a))$ thus $\phi(a)$ is a corner root of $\phi(h)$ and since $\phi$ is onto every corner root of $\phi(h)$ is of the form $\phi(a)$ for a corner root $a$ of $h$ . Since $f$ is corner-integral $h(a) \leq g(a)$, and as $\phi$ is order preserving $h(\phi(a)) = \phi(h(a)) \leq \phi(g(a)) = g(\phi(a))$ i.e., $\phi(g)$ surpasses $\phi(h)$. The symmetric argument switching $h$ and $g$ along with the assertions above yield the corner-integrality of $\phi(f)$.
\end{proof}

\newpage

\section{Composition series of kernels of an idempotent semifield}

\ \\
In this section, we restrict our discussion to idempotent semifields. We write an analogue to the theory of composition series of modules, just for kernels of an idempotent semifield. The kernels of an idempotent semifield are also subsemifields, thereby allowing us to utilize the isomorphism theorems to prove our assertions.
\ \\

\begin{rem}
Let $\mathbb{S}$ be an idempotent semifield and let $K_1, K_2$ be kernels of $\mathbb{S}$ such that $K_1 \subseteq K_2$. Then $K_1$ is a kernel of $K_2$. In such a case we say that $K_1$ is a \emph{subkernel} of $K_2$ and write $K_1 \leq K_2$. If $K_1$ is strictly contained in $K_2$, we write $K_1 < K_2$.
\end{rem}
\begin{proof}
$K_2$ is a subsemifield of $\mathbb{S}$, so, by Theorem \ref{thm_nother_1_and_3}.(1), $K_1~=~K_2 \cap K_1$ is a kernel of $K_2$.
\end{proof}

\begin{rem}
As it was previously shown, in Corollary \ref{cor_principal_kernels_sublattice}, the family of principal kernels of an idempotent semifield is a sublattice of kernels. Moreover, By remark \ref{rem_image_of_principal_kernel}, homomorphic images of principal kernel are principal kernels.
Thus the Isomorphism Theorems \ref{thm_nother_1_and_3} and \ref{thm_nother2} hold for principal kernels of an idempotent semifield. Consequently, all subsequent assertions hold for principal kernels too (considering idempotent semifields).
\end{rem}

\begin{defn}
Let $L$ be a kernel of an idempotent semifield $\mathbb{S}$.
A descending chain
\begin{equation}\label{eq_desc_chain}
L = K_{0}  \supset K_{1} \supset \dots  \supset K_{t}
\end{equation}
of subkernels $K_i$ of $L$ for $1 \leq i \leq t$ is said to have \emph{length} \ $t$. The \emph{factors} of the chain are the kernels $K_{i-1}/K_i$, for $1 \leq i \leq t$. The chain in \eqref{eq_desc_chain} is said to be a \emph{composition series} for $L$ if $K_{t} = \langle 1 \rangle$ and each factor is a simple kernel. \\
We say that two chains of kernels are equivalent if they have isomorphic factors.
\end{defn}

\begin{rem}\label{rem_composition_mutual_subkernel}
Let $L$ be a kernel of an idempotent semifield $\mathbb{S}$. \linebreak
If $L = K_0 \supset K_1 \supset \dots \supset K_t$ is a chain $\mathcal{C}$ of kernels and $K \leq K_t$, then the chain $\mathcal{C}'$ given by  $L/K = K_0/K \supset K_1/K \supset \dots \supset K_t/K$ is equivalent to $\mathcal{C}$.
\end{rem}
\begin{proof}
By Theorem \ref{thm_nother2} we have that $K_{i-1}/K_{i} \cong (K_{i-1}/K)/(K_{i}/K)$, and thus the factors of $\mathcal{C}$ and $\mathcal{C}'$ are isomorphic.
\end{proof}

\begin{rem}
If $K_{i-1} \supset K_{i}$ are kernels such that the semifield $K_{i-1}/K_{i}$ is not simple, then there exists some kernel $N$ between them, i.e., $K_{i-1} \supset N \supset K_{i}$. The process of inserting such an extra subkernel $N$ into the chain is called \emph{refining} the chain. \linebreak Consequently, any chain that is not a composition series can be refined.
\end{rem}

\begin{rem}
For any simple subkernel $S$ and any $K < L$, by Theorem \ref{thm_nother_1_and_3}.(2), we have that $$(K\cdot S)/K \cong  S/(K \cap S)$$ which is either isomorphic to $S$ or $\langle 1 \rangle = \{1 \}$. It follows that $L$ is a finite product of simple subkernels $\{ S_i \ : \ 1 \leq i \leq t \}$. Then letting $L_k = \prod_{i=1}^{t-k}S_i$ we get a composition series $$L = L_0 \supset L_1 \supset \dots \supset L_{t-1} \supset \langle 1 \rangle$$ (discarding duplications).
\end{rem}

\begin{rem}\label{rem_composition_chains}
Let $L$ be a kernel of an idempotent semifield $\mathbb{S}$.
Define a composition chain $\mathcal{C}(L,K)$ from $L$ to a subkernel $K$ to be a chain
$$ L = L_0 \supset L_1 \supset \dots \supset L_t = K$$
such that each factor is simple. By Remark \ref{rem_composition_mutual_subkernel}, $\mathcal{C}(L,K)$ is equivalent to the composition series
$$ L/K \supset L_1/K \supset \dots \supset L_t/K = 1$$
of $L/K$.
It follows that if $P$ is a subkernel of a kernel $N$ for which $N/P \cong L/K$, then there is a composition chain $\mathcal{C}(N,P)$ equivalent to $\mathcal{C}(L,K)$.
\end{rem}

The following is the well-known Jordan – H\"{o}lder theorem for kernels:

\begin{thm}\label{thm_Jordan–Holder}
Let $L$ be a kernel of an idempotent semifield $\mathbb{S}$.
Suppose $L$ has a composition series $$L = L_0 \supset L_1 \supset \dots \supset L_t = \langle 1 \rangle$$
which we denote by $\mathcal{C}$. Then
\begin{enumerate}
  \item Any arbitrary finite chain of subkernels $$L = K_0 \supset K_1 \supset \dots \supset K_s$$
  (denoted as $\mathcal{D}$), can be refined to a composition series equivalent to $\mathcal{C}$. In particular $s \leq t$.
  \item Any two composition series of $L$ are equivalent.
  \item $\ell(L)  = \ell(K) + \ell(L/K)$ for every subkernel $K$ of $L$. In particular, every subkernel and every homomorphic image of a kernel with composition series has a composition series.
\end{enumerate}
\end{thm}

\begin{proof}
We prove the theorem by induction on $t$. If $t=1$ then $L$ is simple and the theorem is trivial, so we assume the whole theorem is true for kernels having a composition series of length~$\leq~t-1$.\\
(1) Let $\mathcal{C}_1 = \mathcal{C}_1(L_1)$ denote the composition series $L_1 \supset \dots \supset L_t = \langle 1 \rangle$ of $L_1$; $\mathcal{C}_1$~has length $t-1$. If $K_1 \subseteq  L_1$, then by induction on $t$, the chain
$L_1~\supseteq~K_1~\supset~K_2 \supset \dots \supset K_s$ can be refined to a composition series of $L_1$ equivalent to $\mathcal{C}_1$, yielding (1) at once (by tagging on $L \supset L_1$).
Thus, we may assume $K_1 \not \subseteq  L_1$, so $L_1 \cap K_1 \subset K_1$. Also, by Remark~\ref{rem_maximal_kernel_property}, $L_1 \cdot K_1 = L$ since $L_1$ is maximal in $L$. Note that we have two ways of descending from $L$ to $L_1 \cap K_1$; either via $L_1$ or via $K_1$. But these two routes are equivalent, in the sense that
\begin{equation}\label{eq_composition_1}
L/K_1 = (M_1 \cdot K_1)/K_1 \cong L_1/(L_1 \cap K_1),
\end{equation}

\begin{equation}\label{eq_composition_2}
K_1/(L_1 \cap K_1) \cong (K_1 \cdot L_1)/L_1 = L/L_1.
\end{equation}

By induction on $t$, the chain $L_1 \supset L_1 \cap K_1 \supset \langle 1 \rangle$ refines the composition series $\mathcal{E}_1(L_1)$ equivalent to $\mathcal{C}_1$ (of length $t-1$). This is comprised of $\mathcal{E}'_1(L_1, L_1 \cap K_1)$, a composition series from $L_1$ to $L_1 \cap K_1$ of some length $t_1$ followed by a composition series $\mathcal{E}'_2(L \cap K_1)$ of some length $t_2$, where $t_1 + t_2 = t-1$.
Since $t_1 \geq 1$, we see $t_2 \leq t-2$. Furthermore, since $L_1$ is maximal in $L$, by Corollary \ref{cor_max_ker_simple_corr} we have that $L/L_1$ is simple, thus \eqref{eq_composition_2} shows that $K_1/(L_1 \cap K_1)$ is simple.  So the chain $K_1 \supset L_1 \cap K_1$ followed by $\mathcal{E}'_2$ is a composition series $\mathcal{F}_1$ of $K_1$ having length $t_2 +1 \leq t-1$. By induction, the chain $K_1 \supset K_2 \supset \dots \supset K_s$ refines to a composition series $\mathcal{D}_1(K_1)$ equivalent to $\mathcal{F}_1$.
The isomorphism  \eqref{eq_composition_1} enables us to transfer $\mathcal{E}'_1(L_1, L_1 \cap K_1)$ to an equivalent composition series $\mathcal{E}{''}_1(L, K_1)$ from $L$ to $K_1$, also of length $t_1$. Tracking this onto $\mathcal{D}_1(K_1)$ yields the desired composition series refining $\mathcal{D}$ which is equivalent to $\mathcal{C}$.
In conclusion, we have passed from our original composition series $\mathcal{C}_1$ through the following equivalent composition series:\\
1. $L \supset L_1$ followed by $\mathcal{E}'_1(L_1, L_1 \cap K_1)$ and $\mathcal{E}'_2$;\\
2. $\mathcal{E}{''}_1(L, K_1)$ followed by $K_1 \supset L_1 \cap K_1$ and $\mathcal{E}'_2$;\\
3. $\mathcal{E}{''}_1(L, K_1)$ followed by $\mathcal{D}_1$,
which defines $\mathcal{D}$, as desired.\\

\begin{flushleft}(2) is immediate from (1).\end{flushleft}
\begin{flushleft}(3) Refine the chain $L > K > \langle 1 \rangle$ to a composition series, and apply Remark \ref{rem_composition_chains}.\end{flushleft}
\end{proof}

%

\begin{defn}
Given a kernel $L$ of an idempotent semifield $\mathbb{S}$, we define its \emph{composition length} $\ell(L)$ to be the length of a composition series for $L$, if such exists.
\end{defn}

\begin{rem}
Let $\mathbb{S}$ be an idempotent semifield. All the results stated in this section hold taking any sublattice of kernels $\Theta$ of $\Con(\mathbb{S})$ in the sense of Definition \ref{defn_sublattice_of_kernels}, instead of $\Con(\mathbb{S})$. Note that, given $\Theta$, maximal kernels are taken to be maximal $\Theta$-kernels, i.e., maximal elements of $\Theta$. For example, one can consider the sublattice of principal kernels of $\mathbb{S}$, $\PCon(\mathbb{S})$.
\end{rem}

\ \\

\newpage
\section{The Hyperspace-Region decomposition and the Hyperdimension}\label{Section:HR_decomp_HyperDimension}

\ \\

\subsection{Hyperspace-kernels and region-kernels.}

\ \\

\begin{rem}
Though we consider the semifield of fractions $\mathscr{R}(x_1,...,x_n)$, most of the results introduced in this section are applicable  to any finitely generated semifield $\mathscr{R}(a_1,...,a_n)$ over $\mathscr{R}$, where $\{a_1,...,a_n\}$ are generators of $\mathscr{R}(a_1,...,a_n)$ as a semifield. We explicitly indicate whenever a condition needs to be imposed on $\{a_1,...,a_n\}$ to hold for the semifield $\mathscr{R}(a_1,...,a_n)$.
In particular, $\langle \mathscr{R} \rangle \subset \mathscr{R}(x_1,...,x_n)$ is just another case of a finitely generated semifield over $\mathscr{R}$, taking $a_i = x_i \wedge |\alpha|$ for $1 \leq i \leq n$ and $\alpha \in \mathscr{R} \setminus \{1\}$. In this case, the generators $a_i$ are bounded from above (or simply $|a_i|$ are bounded), and we specifically designate the results that are true only for unbounded generators.
\end{rem}

\begin{defn}
An element $f \in \mathscr{R}(x_1,...,x_n)$ is said to be a \emph{hyperplane-fraction}, or HP-fraction, if $f \sim_K \frac{h}{g}$ with  $h,g \in \mathscr{R}[x_1,...,x_n]$ distinct monomials and $\langle \frac{h}{g} \rangle~\cap \mathscr{R}~\subseteq~\{1\}$.
\end{defn}

\begin{rem}
The condition $\langle \frac{h}{g} \rangle \cap \mathscr{R} \subseteq \{1\}$ ensures us that $Skel(f) \neq \emptyset$.
Moreover, this condition can be rephrased as $\langle  \mathscr{R} \rangle \not \subseteq \langle \frac{h}{g} \rangle$. This is consistent with our interest in proper subkernels in $\langle  \mathscr{R} \rangle$. We could equivalently take $f$ to be an element of $\langle \mathscr{R} \rangle$.
\end{rem}

\begin{rem}
One can choose to view an HP-fraction simply as a nonconstant Laurent monomial in $\mathscr{R}(x_1,...,x_n)$.
\end{rem}

\begin{rem}\label{rem_HP_not_bounded}
HP-fractions in $\mathscr{R}(a_1,...,a_n)$ where $a_i$ are unbounded fractions, are not bounded; i.e., for any HP-fraction $f$ there exists no $\alpha \in \mathscr{R}$ such that $|f| \leq |\alpha|$.
\end{rem}

Analogously we can prove the following assertion:

\begin{rem}\label{rem_HP_fraction_not_bounded}
HP-fractions in $\mathscr{R}(x_1,...,x_n)$ are not bounded; i.e., for any HP-fraction $f$ there exists no $\alpha \in \mathscr{R}$ such that
$|f| \leq |\alpha|$.
\end{rem}
\begin{proof}
Let $f' = \frac{h}{g}$ where $h,g \in \mathscr{R}[x_1,...,x_n]$ are distinct monomials such that \\ $\langle \frac{h}{g} \rangle \cap \mathscr{R} \subseteq \{1\}$. Then $f'$ contains at least one free (unbounded) indeterminate. Thus $f' \neq 1$ and so $|f'| \not \leq 1$. So we may assume $\alpha \neq 1$. Now, since $\mathscr{R}$ is divisibly closed there exists some $a_k \in \mathscr{R}^n$ such that $f'(a_k) = |\alpha|^k$ for any natural $k$. As $\alpha \neq 1$ is a generator of $\mathscr{R}$ as a kernel, for any $|\beta|$ with $\beta \in \mathscr{R}$ there exists some $k$ for which $|f'(a_k)| = f'(a_k) = |\alpha|^k > |\beta|$. Thus $f'$ is not bounded. By Remark \ref{rem_generator_of_down_bounded_kernel}, for $f \sim_K f'$ we have that $f$ is bounded if and only if $f'$ is bounded, concluding our claim.
\end{proof}

\begin{cor}
Since the condition for a function $f$ to be bounded depends solely on $|f|$, if $f$ is not bounded then both $|f|$ and  $f^{-1}$ are not bounded since  $||f|| = |f^{-1}| = |f|$ (here $||f||$ denotes $| \cdot |$ applied to the function $|f|$).
\end{cor}

\begin{defn}
An element $f \in \mathscr{R}(x_1,...,x_n)$ is said to be a \emph{hyperspace-fraction}, or HS-fraction, if $f \sim_K \sum_{i=1}^{t} |f_i|$ where for each $1 \leq i \leq t$,  $f_i = \frac{h_i}{g_i}$ with  $h_i,g_i \in \mathscr{R}[x_1,...,x_n]$ distinct monomials such that $\langle \frac{h_i}{g_i} \rangle \cap \mathscr{R} \subseteq \{1\}$.
\end{defn}

%

\begin{rem}\label{rem_HS_not_bounded}
HS-fractions in $\mathscr{R}(a_1,...,a_n)$ with $a_i \in \mathscr{R}(x_1,...,x_n)$ unbounded, are not bounded.
\end{rem}
We can analogously prove
\begin{rem}\label{rem_HS_fraction_not_bounded}
HS-fractions in $\mathscr{R}(x_1,...,x_n)$ are not bounded.
\end{rem}
\begin{proof}
The claim follows from the inequality  $|f_j| \leq \sum_{i=1}^{t} |f_i| = f'$ and the invariance of boundness under $\sim_K$.
\end{proof}

\begin{defn}
A principal kernel of $\mathscr{R}(x_1,...,x_n)$ is said to be a \emph{hyperplane-fraction kernel} (or shortly, \emph{HP-kernel}) if it is generated by a hyperplane fraction.
\end{defn}

\begin{rem}
An HP-Kernel is regular.\\
Indeed, as regularity is preserved under $\sim_K$ we may assume $f = \frac{h}{g}$ with $h$ and $g$ distinct monomials  as both numerator and denominator of a hyperplane fraction are monomials, the condition for regularity holds  trivially.
\end{rem}

\begin{defn}
A principal kernel of $\mathscr{R}(x_1,...,x_n)$ is said to be a \emph{hyperspace-fraction kernel} (or shortly, \emph{HS-kernel}) if it is generated by a hyperspace fraction.
\end{defn}

\pagebreak

\begin{rem}\label{rem_HS_prod_HP}
A principal kernel is an HS-kernel if and only if it is a product of HP-kernels.
\end{rem}
\begin{proof}
Every hyperplane fraction is a hyperspace fraction comprised of a single summand, thus every HP-kernel is an HS-kernel.
Conversely, if $\langle f \rangle$ is an HS-kernel, then $\langle f \rangle = \left\langle \sum_{i=1}^{t} |f_i| \right\rangle = \prod_{i=1}^{t} \langle f_i \rangle$ with  $f_i =\frac{h_i}{g_i}$ where $h_i,g_i \in \mathscr{R}[x_1,...,x_n]$ are distinct monomials for each $1 \leq i \leq t$. Thus, by definition $\langle f_i \rangle$ is an HP-kernel for each $i \in  \{1,...,t \}$ proving our claim.
\end{proof}

\begin{cor}
An HS-Kernel is regular.
\end{cor}
\begin{proof}
Since an HP-kernel is regular and since, by Remark \ref{rem_HS_prod_HP}, an HS-kernel is a product of HP-kernels, the assertion follows by Remark \ref{rem_regular_closed_operations} and Corollary~\ref{cor_regular_lattice}.
\end{proof}

\begin{defn}
A skeleton in $\mathscr{R}^n$ is said to be a \emph{hyperplane-fraction skeleton} (shortly HP-skeleton) if it is defined by a hyperplane fraction.
A skeleton in $\mathscr{R}^n$ is said to be a \emph{hyperspace-fraction skeleton} (shortly HS-skeleton) if it is defined by a hyperspace fraction.
\end{defn}

\begin{cor}
A skeleton is an HS-skeleton if and only if it is an intersection of HP-skeletons.
\end{cor}
\begin{proof}
As $Skel(\langle f \rangle  \cdot  \langle g \rangle ) = Skel(\langle f \rangle)  \cap  Skel(\langle g \rangle )$, the assertion follows directly from Remark \ref{rem_HS_prod_HP}.
\end{proof}

\begin{prop}\label{prop_HP_element_is_a_generator}
Let $\langle f \rangle$ be an HP-kernel. Then $w \in \langle f \rangle$ is an HP-fraction if and only if $w^s = f^k$ for some $s,k \in \mathbb{Z} \setminus\{ 0 \}$.
\end{prop}
\begin{proof}
If \ $w^s = f^k$ \ then \ $w \in  \langle f \rangle$ \ and \  $w^s = f^k \sim_K (\frac{h}{g})^k = \frac{h^k}{g^k}$ where $h,g~\in~\mathscr{R}[x_1,...,x_n]$ are distinct monomials such that $\{ \frac{h}{g} \} \cap \mathscr{R} \subseteq \{1\}$. Since  $h,g \in \mathscr{R}[x_1,...,x_n]$ are distinct monomials such that $\{ \frac{h}{g} \} \cap \mathscr{R} \subseteq \{1\}$, we have that $h^k, g^k$ are distinct monomials and since $\mathscr{R}$ is divisibly closed we also have that $\{ \frac{h^k}{g^k}= (\frac{h}{g})^k \} \cap \mathscr{R} \subseteq \{1\}$ (for otherwise the condition will not hold for $\frac{h}{g}$ too). Thus $w^s$ and so also $w$ is an HP-fractions. Conversely, let $w \in \langle f \rangle$ be an HP-fraction, then $w \sim_K w' = \frac{r}{s}$ with  $r,s \in \mathscr{R}[x_1,...,x_n]$ are distinct monomials such that $\{ \frac{r}{s} \} \cap \mathscr{R} \subseteq \{1\}$. As $ w \sim_K w' $ we can prove the claim for $w =w'$. Similarly, we can assume $f = \frac{h}{g}$ where $h,g \in \mathscr{R}[x_1,...,x_n]$ are distinct monomials such that $\{ \frac{h}{g} \} \cap \mathscr{R} \subseteq \{1\}$.
By assumption $\langle w \rangle \subseteq \langle f \rangle$, thus $Skel(w) \supseteq Skel(f)$.
Assume $w^s \neq f^k$ for any $s,k \in \mathbb{Z} \setminus\{ 0 \}$. We will show that there exists some $a \in \mathscr{R}^n$ such that $a \in Skel(f) \setminus Skel(g)$ for some $g \in \langle f \rangle$.

Let $(p_1,...,p_n), (q_1,...,q_n) \in \mathbb{Z}^n$ be the vectors of powers of $x_1,...,x_n$ in the Laurent monomials $f$ and $w$. Since $w$ and $f$ are nonconstant, $(p_1,...,p_n) \neq (0)$,$(q_1,...,q_n)~\neq~(0)$. By Remark \ref{rem_radicality_of_ker} we may assume that $gcd(p_1,...,p_n) = gcd(q_1,...,q_n) = 1$ (since $\mathscr{R}$ is divisible, the constant terms of $f$ and $w$ can be adequately adjusted). Since $w \in \langle f \rangle$ we have that $|w| \leq |f|^m$ for some $m \in \mathbb{N}$, so if $x_i$ occurs in $w$ it must also occur in $f$. Finally, since $w^s \neq f^k$ for any $k \in \mathbb{Z} \setminus\{ 0 \}$ we can also assume that $(p_1,...,p_n) \neq (q_1,...,q_n)$, for otherwise $w = \alpha f$ for some $\alpha \neq 1$ and thus
$$Skel(f) \cap Skel(w) = Skel(f) \cap Skel(\alpha f) =\emptyset,$$
contradicting our assumption that $Skel(w) \supseteq Skel(f)$ and $Skel(w) \neq \emptyset$. Let ${x_j}^l$ occur in $w$ for some $l \in \mathbb{Z} \setminus \{ 0 \}$ such that $x_j$ is not identically $1$ over $Skel(w)$ (and thus also on $Skel(f)$). Thus there exists some  $k \in \mathbb{Z} \setminus \{ 0 \}$ such that ${x_j}^{l+k}$ occurs in $f$. Define the Laurent monomial $g = w^{-1}f^k \in \langle f \rangle$, then $x_j$ does not occur in $g$. Without loss of generality, assume $j = 1$ where the exponent of $x_1$ in $f$ is $p_1$. If $a = (\alpha_1,...,\alpha_n) \in Skel(f)$, then $g(a)= w(a)^{-1}f(a) = 1$. By our assumption that $w^s \neq f^k$ there exists $x_t$ occurring in $f$ with some power $p_t \in \mathbb{Z} \setminus \{ 0 \}$ and not in $w$ (for otherwise they both contain only $x_j$). Take $b = (1,\alpha_2,...,\beta,..., \alpha_n)$ with $\beta \in \mathscr{R}$ occurring in the $t$'th place, such that $\beta^{p_t}=\frac{\alpha_{t}^{p_t}}{\alpha_{1}^{p_1}}$ (there exists such $\beta$ since $\mathscr{R}$ is divisible). Then as $x_j$ is not identically $1$ over $Skel(f)$, we can choose $a \in Skel(w)$ such that  $f(b) = 1$ and $g(b) \neq 1$.

\end{proof}

\begin{cor}\label{cor_HP-element_is_a_generator}
Let $\langle f \rangle$ be an HP-kernel. Then $w \in \langle f \rangle$ is an HP-fraction if and only if $w$ is a generator of $\langle f \rangle$.
\end{cor}
\begin{proof}
The claim follows from Remark \ref{prop_HP_element_is_a_generator} and the property that $\langle g^k \rangle = \langle g \rangle$ for any principal kernel of a semifield.
\end{proof}

\begin{defn}
We define an \emph{order-fraction} $o$ in the semifield $\mathscr{R}(x_1,...,x_n)$ to be an element of the form $o=1  \dotplus  f$ for some HP-fraction $f \neq 1$. We say that $o$ is the order fraction defined by $f$.
\end{defn}

\begin{defn}
We define an \emph{order-Kernel} of the semifield $\mathscr{R}(x_1,...,x_n)$ to be a principal kernel of the form $\langle o \rangle$ for some order-fraction $o = 1  \dotplus  f$. We say that $\langle o \rangle$ is the order kernel defined by $f$.
\end{defn}

\begin{rem}\label{rem_complementary_order}
Let $O = \langle 1  \dotplus  f \rangle$ be an order kernel of $\mathscr{R}(x_1,...,x_n)$, where $f$ is its defining HP-fraction. Then there exists an order kernel $O^{c}$ of $\mathscr{R}(x_1,...,x_n)$ such that
$$O \cap O^{c} = \langle 1 \rangle \ \text{and} \  O \cdot O^{c} = \langle f \rangle.$$
\end{rem}
\begin{proof}
Take $O^{c} = 1  \dotplus  f^{-1}$ then since  $f$ is an HP-fraction so is $f^{-1}$ and thus $O^{c}$ is an order kernel. Now,
$O \cap O^{c} = \langle |1  \dotplus  f| \wedge |1  \dotplus  f^{-1}| \rangle = \langle \min (\max(1,f), \max(1,f^{-1})) \rangle = \langle 1 \rangle$ and $O \cdot O^{c}  = \langle |1  \dotplus  f|  \dotplus  |1  \dotplus  f^{-1}| \rangle = \langle 1  \dotplus  f  \dotplus  f^{-1} \rangle = \langle 1 \dotplus |f| \rangle = \langle |f| \rangle = \langle f \rangle$ (noting that $(1  \dotplus  f),(1  \dotplus  f^{-1}) \geq 1$ implies $|1  \dotplus  f| = 1  \dotplus  f$  and $|1  \dotplus  f^{-1}|= 1  \dotplus  f^{-1}$).
\end{proof}

\begin{defn}
In the notation of Remark \ref{rem_complementary_order}, $O^c$ is said to be the \emph{complementary order kernel} of $O$ and $ 1  \dotplus  f^{-1}$ the \emph{complementary order fraction} of $o = 1  \dotplus  f$, denoted $o^{c}$.
\end{defn}

\begin{rem}
By definition $(O^c)^c = O$.
\end{rem}

\begin{defn}
An element $r \in \mathscr{R}(x_1,...,x_n)$ is said to be a \emph{region-fraction},  if $r \sim_K \sum_{i=1}^{t} |o_i|$ where $o_i$ is an order-fraction for every $1 \leq i \leq t$ and $o_j \neq o_i^{c}$ for any $1 \leq i,j \leq t$.
\end{defn}

\begin{rem}
We will now clarify the reason behind the last condition posed on the order-fractions comprising the region fraction.
For $i=1,...,t$ let $o_i = 1  \dotplus g_i$ with $g_i$ the HP-fraction defining the order fraction $o_i$.\\ Then since $1 \dotplus  f_i \geq 1$ for every $i$, we have that
$$\sum_{i=1}^{t} |o_i| = \sum_{i=1}^{t} |1  \dotplus  f_i| = \sum_{i=1}^{t} (1  \dotplus  f_i) = 1  \dotplus  \sum_{i=1}^{t}f_i.$$
 Thus a region-fraction $r$ can be defined as $r \sim_K 1  \dotplus  \sum_{i=1}^{t}f_i$, so the last condition of the definition can be stated as $f_j \neq f_i^{-1}$ for any $1 \leq i,j \leq t$. One can see that if there exist $k$ and $m$ for which $f_m \neq f_k^{-1}$, we get that
 $$\sum_{i=1}^{t} |o_i| = |f_k|  \dotplus  (1   \dotplus  \sum_{i \neq k,m}f_i) = |f_k|  \dotplus  \Big|1   \dotplus  \sum_{i \neq k,m}f_i \Big|,$$
thus $Skel(r) = Skel(\sum_{i=1}^{t} |o_i|) = Skel(f_k) \cap Skel(1   \dotplus  \sum_{i \neq k,m}f_i) \subseteq Skel(f_k)$.\\
We aim for a region fraction to define a skeleton containing some neighborhood in $\mathscr{R}^n$, thus
the latter condition is required by the above discussion.
\end{rem}

\begin{defn}
A principal kernel $R \in \PCon(\mathscr{R}(x_1,...,x_n))$ is said to be a \linebreak \emph{region kernel} if it is generated by a region fraction.
\end{defn}

\begin{rem}\label{rem_Region_prod_Order}
$R$ is a region kernel if and only if it is of the form
$$R = \prod_{i=1}^{v}O_i$$
for some order kernels $O_1,...,O_v \in \PCon(\mathscr{R}(x_1,...,x_n))$.
\end{rem}
\begin{proof}
Let $r$ be a generating region fraction of $R$. So,
$$R = \langle r \rangle = \langle \sum_{i=1}^{t} |o_i| \rangle~=~\prod_{j=1}^{v}\langle o_i \rangle.$$
Since $o_i$ is an order fraction for each $ 1 \leq i \leq v$, we have by definition that the \linebreak $O_i = \langle o_i \rangle$ are order kernels. Conversely, if $R = \prod_{i=1}^{v}O_i$ then taking $o_i$ to be an \linebreak order fraction generating $O_i$, we get that $r = \sum_{i=1}^{t} |o_i|$ is a region fraction generating $R$, since $\left\langle \sum_{i=1}^{t} |o_i| \right\rangle  = \prod_{i=1}^{v}O_i = R$.
\end{proof}

%

\begin{lem}\label{lem_HP_and_Oreder_kernels_are_CI}
Any HP-kernel is corner-integral, and any order kernel is corner-integral.
\end{lem}
\begin{proof}
As an HP-fraction has a single monomial in its numerator and denominator, there are no corner roots to surpass, and thus it is trivially corner integral. For an order kernel, as it is generated by an element of the form $1  \dotplus  \frac{h}{g} = \frac{g  \dotplus  h}{g}$ for monomials $g$ and $h$, it has a single corner root at the numerator ($x$ such that $g(x) = h(x)$) which is surpassed by the denominator $g(x)$. Finally as corner integrality does not depend of the choice of the generator, our claim is proved.
\end{proof}

\begin{rem}\label{rem_regularity_of_HP_times_order_kernels}
If $\langle f \rangle \neq \langle 1 \rangle$ is an HP-kernel and $\langle g \rangle$ is an order kernel, then  $$\langle f \rangle \cdot \langle g \rangle = \left\langle |f| \dotplus |g| \right\rangle$$
is regular.
\end{rem}
\begin{proof}
As regularity does not depend on the choice of a generator of the kernel and since $\langle |f| \dotplus |g| \rangle = \langle f \rangle \cdot \langle g \rangle$, we may consider $|f|  \dotplus  |g|$. Write $f = \frac{u}{v}$ and $g = 1  \dotplus  \frac{a}{b}$ with $u,v,a$ and $b$ monomials in $\mathscr{R}[x_1,...,x_n]$. Then  $|f|  \dotplus  |g| = |\frac{u}{v}|  \dotplus  |1  \dotplus  \frac{a}{b}|$. Since  $1  \dotplus  \frac{a}{b} \geq 1$, we have that $|f|  \dotplus  |g| = |\frac{u}{v}|  \dotplus  1  \dotplus  \frac{a}{b}$ and since $|\frac{u}{v}| \geq 1$ we get that
$$|f|  \dotplus  |g| = \left(\left|\frac{u}{v}\right|  \dotplus  1\right)  \dotplus  \frac{a}{b} = \left|\frac{u}{v}\right|  \dotplus  \frac{a}{b}.$$ Now, $\left|\frac{u}{v}\right|  \dotplus  \frac{a}{b}  = \frac{u}{v}  \dotplus  \frac{v}{u}  \dotplus  \frac{a}{b} = \frac{(u^2  \dotplus  v^2)b  \dotplus  auv}{buv}$. Since $(u^2  \dotplus  v^2) = (u  \dotplus  v)^2 = u^2  \dotplus  uv  \dotplus  v^2 \geq uv$, we have that $(u^2  \dotplus  v^2)b \geq uvb$. Thus either $u(x)^{2}b(x) \geq u(x)v(x)b(x)$ or $v(x)^{2}b(x) \geq u(x)v(x)b(x)$ for any $x \in \mathscr{R}^n$. By assumption, $f \neq 1$ and thus $u \neq v$, so  $u^{2}b \neq uvb$ and $v^{2}b \neq uvb$. Thus for any $x \in \mathscr{R}^n$ there is always a monomial in the numerator distinct from the one in the denominator. Note that $auv$ dominates the numerator if $a(x)u(x)v(x) >  (u(x)^{2} \dotplus  v(x)^{2})b(x) \geq u(x)v(x)b(x)$, and thus essentiality of $auv$ implies $auv \neq buv$, i.e., $a \neq b$ concluding our proof.
\end{proof}

\begin{rem}\label{rem_regularity_of_HP_times_order_kernels_generalization}
Induction yields that
$$\langle f \rangle \cdot \langle g_1 \rangle \cdot \dots \cdot \langle g_k \rangle$$
is regular, for any HP-kernel $\langle f \rangle \neq 1$ and order kernels $\langle g_1 \rangle, ... , \langle g_k \rangle$.
\end{rem}

\subsection{Geometric interpretation of HS-kernels and region kernels and the use of logarithmic scale}\label{subsection:geometric_interpretation}
\ \\

\begin{defn}
Let $\gamma \in \mathbb{H}$, and let $k \in \mathbb{N}$. A \emph{$k$-th root} of $\gamma$, if exists, is an element $\beta \in \mathbb{H}$ such that $\beta^k = \alpha$.
\end{defn}
%

\begin{rem}\label{rem_uniqueness_of_root}
For any $k \in \mathbb{N}$ and $\gamma \in \mathbb{H}$ the  $k$-th root of $\gamma$ is unique.
\end{rem}
\begin{proof}
Let $\alpha, \beta \in \mathbb{H}$ such that $\alpha^k = \beta^k$ for some $k \in \mathbb{N}$. Then
multiplying both sides of the equality by $\beta^{-k}$, we get that $(\beta^{-1}\alpha)^{k} = 1$. By Remark \ref{rem_torsion_free}, we have $\beta^{-1}\alpha = 1$ and so $\alpha = \beta$.
\end{proof}

Let $(\mathbb{H}, \cdot,  \dotplus )$ be a divisible semifield.
By Remark \ref{rem_uniqueness_of_root}, we can uniquely define any rational power of the elements of $\mathbb{H}$. In such a way, $\mathbb{H}$ becomes a vector space over $\mathbb{Q}$, rewriting the multiplicative operation $\cdot$ on $\mathbb{H}$ as addition and defining
\begin{equation}\label{eq_vector_space_translation}
(m/n) \cdot \alpha = \alpha^{\frac{m}{n}}.
\end{equation}
In this way we can apply linear algebra techniques to $(\mathbb{H},\cdot)$. When considering $\mathbb{H}$ in such a way we will denote the original addition of $\mathbb{H}$ (in the idempotent case) by $\dotplus$ or $\vee$  in order to avoid ambiguity. We call the representation given by \eqref{eq_vector_space_translation} the logarithmic representation of $(\mathbb{H}, \cdot)$.
$\mathbb{H}$ viewed as just described, any HP-fraction may be considered as a linear functional over $\mathbb{Q}$ and thus an HP-kernel  may be considered as a kernel generated by a linear functional (notice that for an HP-fraction $f$, the equation $f = 1$ over $(\mathbb{H}, \cdot)^n$ translates to $F = 0$ over $(\mathbb{H}, +)^n$ where $F$ is the linear form obtained from $f$ by applying \eqref{eq_vector_space_translation}, i.e., $F$ is the logarithmic form of $f$.)

\begin{rem}
In the special case of a semifield $\mathscr{R}$, by Theorem \ref{thm_holder_translated}, the above interpretation allows us to consider $(\mathscr{R}, \cdot)$ as being $(\mathbb{R}, +)$ and $\mathscr{R}^n$ as being $\mathbb{R}^n$ with coordinate-wise addition and scalar multiplication over $\mathbb{Q}$.
\end{rem}

\ \\

As a consequence of the above discussion, viewing an HP-fraction as a linear \linebreak functional defining an $n$-dimensional affine subspace of $\mathscr{R}^{n+1}$, we have the following statement:
\begin{rem}
If $f$ is an HP-fraction in $\mathscr{R}(x_1,...,x_n)$, then $f$ is completely determined by the set $\{ p_0, ...,p_n \}$ for any $p_i = (\alpha_{i,1},...,\alpha_{i,n},f(\alpha_{i,1},...,\alpha_{i,n})) \in \mathscr{R}^{n+1}$ where \linebreak $\{ a_i = (\alpha_{i,1},...,\alpha_{i,n}) \} \subset \mathscr{R}^n$ such that $p_0,....,p_n$ are in general position (are not contained in an $(n-1)$-dimensional affine subspace of $\mathscr{R}^{n+1}$).
\end{rem}
\begin{proof}
Writing $f(x_1,...,x_n) = \alpha \prod_{i=1}^{n}x_i^{k_i}$ with $k_i \in \mathbb{Z}$, then $\alpha  = f(a_{0})\prod_{i=1}^{n}a_{0,i}^{-k_i}$. After $\alpha$ is determined, since $p_0,....,p_n$ are in general position the set $$\left\{a_1,...,a_n, b = \left(f(a_1),...,f(a_n)\right)\right\} \subset \mathscr{R}^n$$
define a linearly independent set of $n$ linear equations in the variables $k_i$, and thus determine them uniquely.
\end{proof}

Consider an HS-kernel of $\mathscr{R}(x_1,...,x_n)$ defined by the HS-fraction $f = \sum_{i=1}^{t}|f_i|$ where $f_1,...,f_t$ are HP-fractions. Then $f = 1$ if and only if $f_i = 1$ for each $i = 1,...,t$. Thus $f = 1$ gives rise to a homogenous system of rational linear equations of the form $F_i = 0$ where $F_i$ is the logarithmic form of $f_i$. This way $Skel(f) \subset \mathscr{R}^n \cong (\mathbb{R}^{+})^{n}$ is identified with an affine subspace of $\mathbb{R}^n$ which is just the intersection of the $t$ affine hyperplanes defined by $F_i = 0, \ 1 \leq i \leq t$.
Analogously, for an order kernel defined by $o = 1  \dotplus  g$ for some HP-fraction $g$, $o = 1$ if and only if $g \leq 1$ which gives rise to the rational half space of $\mathbb{R}$ defined by the weak inequality $g \leq 0$. Thus, the region kernel defined by $r = \sum_{i=1}^{t}|o_i| = 1  \dotplus  \sum_{i=1}^{t}g_i$ where $o_1 = 1  \dotplus  g_1 ,...,o_t = 1  \dotplus  g_t$ are order fractions, yields a nondegenerate polyhedron formed as an intersection of the affine half spaces each of which is defined by $G_i \leq 0$, where $G_i$ is the logarithmic form of the HP-fraction $g_i$ defining $o_i$ (i.e., $o_i = 1  \dotplus g_i$).

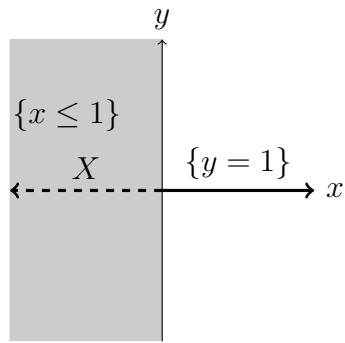
\begin{figure}
\centering

\begin{tikzpicture}[scale=2]

  \tikzstyle{axis}=[thin, ->]
  \tikzstyle{dline}=[very thick, ->]
  \tikzstyle{important line}=[very thick, dashed, ->]

    \colorlet{anglecolor}{green!50!black}

    \filldraw[fill=gray!40, draw = gray!40 ] (0,-1) -- (0,1) -- (-1,1) -- (-1,-1);
    \draw[dline] (0,0)  -- (0.5,0) node[above] {$\{y=1\}$} -- (1,0)  node[right] {$x$} ;
    \draw[axis] (0,-1) -- (0,1)   node[above] {$y$} ;
    \draw[important line]  (0,0)  -- (-0.5,0) node[above] {$X$} -- (-1,0);

    \node[left] at (-0.2,0.5) {$\{ x \leq 1 \}$};

 \end{tikzpicture}
\caption{Order relations} \label{fig:Ex2}
\end{figure}

\ \\

\newpage\subsection{A preliminary discussion}

\ \\

We introduce the following example to motivate our subsequent discussion.\\
Consider a point $a = ( \alpha_1,...,\alpha_n) \in \mathscr{R}^n$. Then $a = Skel(f_a)$ where $$f_a(x_1,...,x_n) = \left|\frac{x_1}{\alpha_1}\right|  \dotplus  \dots  \dotplus  \left|\frac{x_n}{\alpha_n}\right| \in \mathscr{R}(x_1,...,x_n).$$ We would like $\langle f_a \rangle$ to encapsulate the dimension reduction from $\mathscr{R}^n$ to $\{a \}$.
Consider the following  chain of principal HS-kernels
\small{\begin{equation}\label{eq_chain_of_origin}
\langle f_{a} \rangle = \left\langle \left|\frac{x_1}{\alpha_1}\right|  \dotplus  \dots  \dotplus  \left|\frac{x_n}{\alpha_n}\right| \right\rangle \supset \left\langle \left|\frac{x_1}{\alpha_1}\right|  \dotplus  \dots  \dotplus  \left|\frac{x_{n-1}}{\alpha_{n-1}}\right| \right\rangle  \supset \dots \supset \left\langle \left|\frac{x_1}{\alpha_1}\right| \right\rangle  \supset \langle 1 \rangle = \{1\} .
\end{equation}}
For each $1 \leq k \leq n$ denote $f_k = \left|\frac{x_1}{\alpha_1}\right|  \dotplus  \dots  \dotplus  \left|\frac{x_k}{\alpha_k}\right|$ and $f_0 = 1$.
Each of the HS-kernels $\langle f_{k-1} \rangle$ in the chain is also a semifield which is a subsemifield of the preceding kernel $\langle f_{k-1} \rangle$. The factors defined by the quotients $\langle f_{k} \rangle / \langle f_{k-1} \rangle$ are the quotient semifields
$$\left\langle \left|\frac{x_1}{\alpha_1}\right|  \dotplus  \dots  \dotplus  \left|\frac{x_k}{\alpha_k}\right| \right\rangle / \left\langle \left|\frac{x_1}{\alpha_1}\right|  \dotplus  \dots  \dotplus  \left|\frac{x_{k-1}}{\alpha_{k-1}}\right| \right\rangle = \prod_{j=1}^{k}\left\langle \frac{x_j}{\alpha_j} \right\rangle / \prod_{j=1}^{k-1}\left\langle \frac{x_j}{\alpha_j} \right\rangle$$  $$\cong \left\langle \frac{x_k}{\alpha_k} \right\rangle / \left(\left(\prod_{j=1}^{k-1}\left\langle \frac{x_j}{\alpha_j} \right\rangle \right) \cap \left\langle \frac{x_k}{\alpha_k} \right\rangle\right).$$
Note that on the right hand side of the equality we have a homomorphic nontrivial image of an HP-kernel, namely, an HP-kernel of the quotient semifield $$\mathscr{R}(x_1,...,x_n)/ \left(\left(\prod_{j=1}^{k-1}\langle \frac{x_j}{\alpha_j} \rangle \right) \cap \langle \frac{x_k}{\alpha_i} \rangle\right).$$
There are some questions arising from the above construction: \\

Can this chain of HS-kernels be refined to a longer descending chain of principal \linebreak kernels descending from $\langle f_a \rangle$? Are the lengths of descending chains of principal kernels beginning at $\langle f_{a} \rangle$ bounded, and if so, can any chain can be refined to a chain of maximal length?
The following example provides a positive answer to the first question: \\

Consider the kernels $\langle |x|  \dotplus  |y| \rangle$ and $\langle |x| \rangle = \langle x \rangle$. Both are semifields over the trivial semifield and $\langle |x| \rangle$ is a subkernel of $\langle |x|  \dotplus  |y| \rangle$. Consider the substitution map $\phi$ sending $x$ to $1$. Then $\Im(\phi) = \langle 1  \dotplus  |y| \rangle = \langle |y| \rangle_{\mathscr{R}(y)}$. The kernel $\langle |y| \rangle$ is not simple in the lattice  of principal kernels of the semifield  $\langle |y| \rangle_{\mathscr{R}(y)}$ as we have the chain  $\langle |y| \rangle \supset \langle |1  \dotplus  y| \rangle \supset \langle 1 \rangle$,  which is the image of the refinement  $$\langle |x|  \dotplus  |y| \rangle \supset \langle |x  \dotplus  y|  \dotplus  |x| \rangle \supset \langle |y| \rangle$$ (since  $\phi(|x  \dotplus  y|  \dotplus  |x|) = |\phi(x)  \dotplus  \phi(y) |  \dotplus  | \phi(x)| = |1  \dotplus  y|  \dotplus  |1|= |1 \dotplus y|$). One can notice that $\langle 1  \dotplus  y \rangle$ is an order kernel which induces the order $y \leq 1$ on the semifield $\langle |y| \rangle$. \\

In view of the above example we would like to restate the questions posed above as follows:
Can this chain of HS-kernels be refined to a longer descending chain of HS-kernels descending from $\langle f_a \rangle$?  Are the lengths of descending chains of HS-kernels beginning at $\langle f_{a} \rangle$ bounded, and if so, can any chain of  HS-kernels be refined to such a chain of maximal length? \\

In the next section we provide answers to these three questions by which the chain introduced above is of maximal unique length common to all chains of HS-kernels descending from $\langle f_a \rangle$.

\ \\

\subsection{The HO-decomposition}

\ \\

In the following we describe an explicit decomposition of a principal kernel $\langle f \rangle$ as an intersection of kernels of two types:
The first, to be named an HO-kernel, is a product of some HS-kernel and a region kernel. The second is a  product of a region kernel and a bounded from below kernel.\\
While the first type defines the skeleton of $\langle f \rangle$, the second type has no effect on it as it corresponds to the empty set.
This latter type is the source of ambiguity in relating a skeleton to a kernel, preventing the kernel corresponding to a skeleton from being principal. When intersecting with $\langle \mathscr{R} \rangle$, the kernels of the second type  in the decomposition are chopped off, in the sense that they all become equal to  $\langle \mathscr{R} \rangle$. This restriction to $\langle \mathscr{R} \rangle$ thus removes the ambiguity making each HO-kernel (intersected with $\langle \mathscr{R} \rangle$) in a $1:1$ correspondence with its skeleton (which is, in fact, a segment in the skeleton defined by $\langle f \rangle$). Subsequently, the `HO-part' is unique and independent of the choice of the kernel generating the skeleton.\\
Geometrically, the decomposition to be described below is just a fragmentation of a principal skeleton defined by $\langle f \rangle$ to the ``linear" fragments comprising it. Each \linebreak fragment is attained by bounding an affine subspace of $\mathscr{R}^n$ defined by an appropriate HS-fraction (which in turn generates an HS-kernel)  using a region fraction (generating a region kernel). Note that the HS-fraction may be $1$ and so may the region fraction. By the discussion above, although the HS-fraction and region fraction defining each segment may vary moving from one generator of the principal kernel to the other, the HS-kernels and region kernels they define stay intact as they correspond to the \linebreak fragments of the skeleton of $\langle f \rangle$. Thus we may form the next construction, though explicit, using any generator without affecting the resulting HO-kernels.\ \\

We now move forward to introduce the construction which will be illustrated by two subsequent examples.\\

\begin{constr}\label{HO_construction}
Consider an element $f \in \mathscr{R}(x_1,...,x_n)$, the kernel it generates, $\langle f \rangle$, and its corresponding skeleton $Skel(f)$. Taking $|f|$ we may assume $f \geq 1$. Write $f = \frac{h}{g} = \frac{\sum_{i=1}^k h_i}{\sum_{j=1}^m g_j}$ where $h_i$ and $g_j$ are monomials in $\mathscr{R}[x_1,...,x_n]$. Assume $Skel(f) \neq \emptyset$.
Let $a$ be a point of $Skel(f)$. Since $f(a) = 1$,  there exists a subset
$$H_a \subseteq H = \{ h_i \ : \ 1 \leq i \leq k \}$$
and a subset
$$G_a \subseteq G = \{ g_j \ : \ 1 \leq j \leq m \}$$
for which $g'(a) = h'(a)$ for any $h' \in H_a$ and $g' \in G_a$.  Denote by $H_{a}^{C}$ and $G_{a}^{C}$ the complementary subsets of monomials of $H$ and $G$ respectively, i.e., $H_{a}^{C} = H \setminus H_a$ and $G_{a}^{C} = G \setminus G_a$. Then, for any $h' \in H_a$ and $h'' \in H_{a}^{C}$ we have $$h'(a)  \dotplus  h''(a) = h'(a),$$ i.e., $h'(a) \geq h''(a)$. We can write these last relations equivalently as  $1  \dotplus  \frac{h''(a)}{h'(a)} = 1$ for all $h' \in H_a$ and all $h'' \in H_{a}^{C}$. Similarly, for any $g' \in G_a$ and $g'' \in G_{a}^{C}$ we have that $g'(a)  \dotplus  g''(a) = g'(a)$. We can write these last relations equivalently as  $1  \dotplus  \frac{g''(a)}{g'(a)} = 1$ for all $g' \in G_a$ and all $g'' \in G_{a}^{C}$.

Thus for any such $a$ we obtain the following relations:
\begin{equation}\label{eq_relations1}
\frac{h'}{g'} = 1, \ \ \ \ \ \forall h' \in  H_{a}, g' \in G_{a},
\end{equation}
\begin{equation}\label{eq_relations2}
1  \dotplus  \frac{h''}{h'} = 1 \ ; \ 1  \dotplus  \frac{g''}{g'} = 1, \ \ \forall h' \in H_{a}, h'' \in H_{a}^{C},  g' \in G_{a}, g'' \in G_{a}^{C}.
\end{equation}
Varying $a \in Skel(f)$, evidently there are only finitely many possibilities for relations in \eqref{eq_relations1} and \eqref{eq_relations2} \ as there are only finitely many monomials $h_i$  \ and $g_j$ \ comprising $f$.\\

Any set of relations comprised of the relations in \eqref{eq_relations1} and \eqref{eq_relations2}, which will hereby be denoted by $\theta_1$ and $\theta_2$, corresponds to a kernel generated by the corresponding elements
$$\frac{h'}{g'} \ , \left(1  \dotplus  \frac{h''}{h'}\right), \ \text{and} \  \left(1  \dotplus  \frac{g''}{g'}\right),$$ where $\{ \frac{h'}{g'} = 1 \}  \in \theta_1$ and $\{ g  \dotplus  \frac{g''}{g'} = 1\}, \{ 1  \dotplus  \frac{h''}{h'} = 1\} \in \theta_2$.  The finite collection of pairs $(\theta_1(i), \theta_2(i))$ for $i = 1,...,s$, is formed by considering all points $a$ in $Skel(f)$. Moreover, every point admitting the relations in  $(\theta_1(i), \theta_2(i))$ for any $i \in \{1,...,s\}$ is in $Skel(f)$. Thus this collection supplies a complete covering of $Skel(f)$. Denote by $K_i$, for $i = 1,...,s$, the kernel generated by the elements of $\theta_1(i)$ and $\theta_2(i)$.
Then since
$$Skel\left(\langle f \rangle \cap \langle \mathscr{R} \rangle \right)=Skel(f) = \bigcup_{i=1}^{s}Skel(K_i) = \bigcup_{i=1}^{s}Skel\left(K_i \cap \langle \mathscr{R} \rangle\right)$$ $$= Skel\left(\bigcap_{i=1}^{s}K_i \cap \langle \mathscr{R} \rangle\right),$$
where $\langle f \rangle \cap \langle \mathscr{R} \rangle$ and $\bigcap_{i=1}^{s}K_i \cap \langle \mathscr{R} \rangle$ are in $\PCon(\langle \mathscr{R} \rangle)$ we have that \linebreak $\langle f \rangle \cap \langle \mathscr{R} \rangle = \bigcap_{i=1}^{s}K_i \cap \langle \mathscr{R} \rangle$.\\
In essence $\bigcap_{i=1}^{s}K_i$ provides a local description of $f$ in a neighborhood of is skeleton.\\
We now proceed to supply an insight of the structure of the kernel $\langle f \rangle$ to put the construction above into a broader context.
\end{constr}

In Construction \ref{HO_construction}, we used the skeleton of $\langle f \rangle$ to construct $\bigcap_{i=1}^{s}K_i$. Considering all points $a$ in $\mathscr{R}^n$  might add some regions, complementary to the regions defined by \eqref{eq_relations2} in $\theta_2(i)$ for $i=1,...,s$, over which $\frac{h'}{g'} \neq 1, \ \ \ \forall h' \in  H_{a},\\ \forall g' \in G_{a}$, i.e., regions over which the dominating monomials never meet. Continuing the construction above using  $a \in \mathscr{R}^n \setminus Skel{f}$ similarly produces a finite collection of, say $t \in \mathbb{Z}_{\geq 0}$, kernels generated by elements from \eqref{eq_relations2} and their complementary order fractions and by elements of the form \eqref{eq_relations1} (where now $\frac{h'}{g'} \neq 1$ over the considered region) . A principal kernel $N_j = \langle q_j \rangle$  $1 \leq j \leq t$, of this complementary set of kernels has the property that $Skel(N_j) = \emptyset$, thus by Corollary \ref{cor_empty_kernels_correspond_to_bfb_kernels}, $N_j$ is bounded from below. As there are finitely many such kernels there exists $\gamma \in \mathscr{R}$, $\gamma > 1$ and small enough, such that $|q_j| \wedge \gamma = \gamma$ for $j = 1,...,t$. Thus $\bigcap_{j=1}^{t}N_j$ is bounded from below and thus by Remark \ref{rem_bounded_from_below_contain_H_kernel} we have that  $\bigcap_{j=1}^{t}N_j \supseteq \langle \mathscr{R} \rangle$. \\
As now the extension coincides with $f$ over all of $\mathscr{R}^n$ and as $\mathscr{R}$ is divisible, we have
\begin{equation}\label{full_expansion}
\langle f \rangle = \bigcap_{i=1}^{s}K_i \cap \bigcap_{j=1}^{t}N_j.
\end{equation}

So, $\langle f \rangle \cap \langle \mathscr{R} \rangle = \bigcap_{i=1}^{s}K_i \cap \bigcap_{j=1}^{t}N_j \cap \langle \mathscr{R} \rangle = \bigcap_{i=1}^{s}K_i \cap \langle \mathscr{R} \rangle$.\\

In view of the last discussion, we see that intersecting a principal kernel $\langle f \rangle$  with $\langle \mathscr{R} \rangle$ `chops off' all of its comprising bounded from below kernels (the $N_j$'s above). This way it eliminates ambiguity in the kernel corresponding to $Skel(f)$.\\
Finally we note that if $Skel(f) = \emptyset$ then $\langle f \rangle = \bigcap_{j=1}^{t}N_j$ for appropriate kernels $N_j$ and   $\langle f \rangle \cap \langle \mathscr{R} \rangle = \langle \mathscr{R} \rangle$.\\

\ \\

A few notes concerning the construction:
\begin{rem}
\begin{enumerate}
  \item If $K_1$ and $K_2$ are such that $K_1 \cdot K_2 \cap \mathscr{R} = \{1\}$ (i.e., $Skel(K_1) \cap Skel(K_2) \neq \emptyset$), then the sets of HP-fractions $\theta_1$ of $K_1$ and of $K_2$ are not equal (though one may contain the other), for otherwise they would be combined via the construction to form a single kernel.

  \item Let $\langle f \rangle \cap \langle \mathscr{R} \rangle = \bigcap_{i=1}^{s}(K_i \cap \langle \mathscr{R} \rangle)  = \bigcap_{i=1}^{s}\langle |k_i| \wedge |\alpha| \rangle = \bigwedge_{i=1}^{s} \langle |k_i| \wedge |\alpha| \rangle$ with $\alpha \in \mathscr{R} \setminus \{ 1 \}$.
By Corollary \ref{cor_norm_decomposition}, for any generator $f'$ of $\langle f \rangle \cap \langle \mathscr{R} \rangle$ we have that $|f'| = \bigwedge_{i=1}^{s} |k_i'|$ with $k_i' \sim_K |k_i| \wedge |\alpha|$ for every $i = 1,...,s$. In particular, $Skel(k_i') = Skel(|k_i| \wedge |\alpha|) = Skel(k_i)$. Thus the above construction of the $K_i$ kernels is independent of the choice of the generator as it is totally defined by the fragments $Skel(k_i)$ of the skeleton $Skel(f)$.
\end{enumerate}
\end{rem}

\ \\
\ \\

We now provide two examples for the construction introduced above. We make use of the notation above for the different types of kernels involved in the construction.
\begin{exmp}\label{exmp1_for_construction}
Let $f = |x| \wedge \alpha \in \mathscr{R}(x,y)$ for some $\alpha > 1$ in $\mathscr{R}$. Then $ f = \frac{\alpha|x|}{\alpha  \dotplus  |x|}$. The order relation $ \alpha \leq |x|$ translates to the relation $\alpha  \dotplus  |x| = |x|$  or equivalently to $\alpha|x|^{-1}  \dotplus  1 = 1$. Over the region defined by the last relation we have $f = \frac{\alpha |x|}{\alpha} = |x|$. Similarly, the complementary order relation $ \alpha \geq |x|$ translates to $\alpha^{-1}|x|  \dotplus  1 = 1$ (via $|x|  \dotplus  \alpha = \alpha$) over which region $f = \frac{\alpha |x|}{|x|} = \alpha$ . So
$$\langle f \rangle = K_1 \cap K_2 = (R_{1,1} \cdot L_{1,1}) \cap (R_{2,1} \cdot N_{2,1})$$
where $R_{1,1} = \langle \alpha|x|^{-1}  \dotplus  1 \rangle$, $L_{1,1} = \langle |x| \rangle$, $R_{2,1} = \langle \alpha^{-1}|x|  \dotplus  1 \rangle$ and $N_{2,1} = \langle \alpha \rangle$.\\ Geometrically $R_{1,1}$ is a strip containing the axis $x = 1$ and $R_{2,1}$ is its complementary region. The restriction of $f$ to $R_{1,1}$ gives it the form $|x|$ while restricting to $R_{2,1}$, $f$ equals $\alpha$. Furthermore, we see that over $R_{2,1}$, $f$ is bounded from below by $\alpha$. Omitting $N_{2,1}$ we still have $Skel(f) = Skel(R_{1,1} \cdot L_{1,1})$ though $ R_{1,1} \cdot L_{1,1}  \supset \langle f \rangle$ and equality do not hold. Intersecting $\langle f \rangle$ with $\langle \mathscr{R} \rangle$ leaves $\langle f \rangle$ intact while  $ (R_{1,1}~\cdot~ L_{1,1})~\cap~\langle \mathscr{R} \rangle =   (R_{1,1} \cdot L_{1,1} \cap R_{2,1} \cdot N_{2,1}) \cap \langle \mathscr{R} \rangle$.  Thus $\langle f \rangle = \langle f \rangle \cap \langle \mathscr{R} \rangle = (R_{1,1} \cdot L_{1,1}) \cap \langle \mathscr{R} \rangle$.
\end{exmp}

\begin{exmp}\label{exmp2_for_construction}
Let $f = |x  \dotplus  1| \wedge \alpha \in \mathscr{R}(x,y)$ for some $\alpha > 1$ in $\mathscr{R}$. First note that since $x \dotplus 1 \geq 1$ we have that $|x \dotplus 1| = x  \dotplus  1$, allowing us to rewrite $f$ as $(x \dotplus 1) \wedge \alpha$. Then $ f = \frac{\alpha (x  \dotplus  1) }{\alpha  \dotplus  (x  \dotplus  1)} = \frac{\alpha x  \dotplus  \alpha}{\alpha  \dotplus  x}$. The order relation $ \alpha \leq x$ translates to the relation $\alpha  \dotplus  x = x$  or equivalently to $\alpha x^{-1}  \dotplus  1 = 1$. Over the region defined by the last relation, we have $f = \frac{\alpha x  \dotplus  \alpha}{\alpha  \dotplus  x} = \frac{\alpha x  \dotplus  \alpha}{x} = \alpha  \dotplus  \frac{\alpha}{x} = \alpha$. Similarly, the complementary order relation $ \alpha \geq x$ translates to $\alpha^{-1}x  \dotplus  1 = 1$  over which $f = \frac{\alpha x  \dotplus  \alpha}{\alpha  \dotplus  x} = \frac{\alpha x  \dotplus  \alpha}{\alpha} =  x  \dotplus  1$. So
$$\langle f \rangle = K_1 \cap K_2 = (R_{1,1} \cdot R_{1,2}) \cap (R_{2,1} \cdot N_{2,1}) = R_{1,2} \cap R_{2,1} \cdot N_{2,1}$$
where $R_{1,1} = \langle \alpha^{-1} x  \dotplus  1 \rangle$, $R_{1,2} = \langle x  \dotplus  1 \rangle$, $R_{2,1} = \langle \alpha x^{-1}  \dotplus  1 \rangle$ and $N_{2,1} = \langle \alpha \rangle$.
Since $N_{2,1} \subset R_{2,1} \cdot N_{2,1}$ we have that $Skel(R_{2,1} \cdot N_{2,1}) \subset Skel(N_{2,1})= \emptyset$. So
$$Skel(f) = Skel(R_{1,2}) \cup Skel(R_{2,1} \cdot N_{2,1}) = Skel(x \dotplus 1) \cup \emptyset = Skel(x \dotplus 1).$$
\end{exmp}

As can be easily seen from examples \ref{exmp1_for_construction} and \ref{exmp2_for_construction} by substituting any \linebreak  HP-fraction for $x$ and any order fraction for $x \dotplus 1$ , the intersection of a kernel $K = L \cdot R = \prod L_i \cdot \prod {O}_j$ ($L_i$ and $O_j$ are HP-kernels and order kernels respectively) with $\langle \mathscr{R} \rangle$ yields
\begin{align*}
K' = \ & K \cap \langle \mathscr{R} \rangle  = \prod (L_i \cap \langle \mathscr{R} \rangle)  \cdot \prod (O_j \cap \langle \mathscr{R} \rangle) \\
= & \prod \left((L_i \cdot R_i) \cap (N_i \cdot R_i^{\ c})\right) \cdot \prod \left((O_j \cdot {R'}_j) \cap (M_j \cdot  {{R'}_j}^{\ c})\right) =  \left( \prod L_i \cdot \prod {O'}_j \right) \cap  N\\
= &( L \cdot R' ) \cap N
\end{align*}

where $R'$, $R_i$ and $R'_j$ are region kernels. ${R_i}^{c}$ and ${R'_j}^{c}$ are $R_i$'s and $R'_j$'s complementary region kernels respectively. $N_i, M_j$ and $N$ are bounded from below kernels. Note that the ${O'}_j$s involve the $R_i$s, the ${R'}_j$s and the $O_j$s, while $N$ is derived from the bounded from below kernels, namely the $N_i$s and $M_j$s. Also note that intersecting with $\langle \mathscr{R} \rangle$ keeps the HS-kernel unchanged in the new decomposition.

\pagebreak

As the bounded from below kernels, the $N_j$s in \eqref{full_expansion},  do not affect $Skel(f)$ we leave them aside for the time being and proceed to study the structure of the kernels $K_i$ and their corresponding skeletons.\\
\ \\
Let $K$ be one of the above kernels $K_i$. First note that every element of the set generating $K$ which is specified above, is either an HP-fraction (a nonconstant Laurent monomial) of the form $\frac{h'}{g'}$, or an order element of the form $1  \dotplus  \frac{h''}{h'}$ (or $1  \dotplus  \frac{g''}{g'}$). Let $L_i$ with $i =1,..,u$ and $O_j$ with $j = 1,...,v$, be the kernels generated by each of the HP-kernels and the order kernels respectively.
Then we can write
\begin{equation}\label{eq_L_O}
K = L \cdot R = \prod_{i=1}^{u}L_i \cdot \prod_{j=1}^{v}O_j
\end{equation}
where $L=\prod_{i=1}^{u}L_i$ is an HS-kernel and $R = \prod_{j=1}^{v}O_j$ is a region kernel.
note that by the assumption of the construction above $Skel(K) \neq \emptyset$ (as there is at least one point of the skeleton used for constructing it). Moreover, there are no distinct HS-kernels $M_1$ and $M_2$ such that $L = M_1 \cap M_2$ for otherwise the construction would have produced two distinct kernels, one with $M_1$ as its HS-kernel and the other with $M_2$ as its HS-kernel instead of producing $K$ in the first place.\\

\begin{defn}\label{defn_HO_fraction}
An element $f \in \mathscr{R}(x_1,...,x_n)$ is said to be an \emph{HO-fraction} if $$f~=~l(f)~\dotplus~o(f)$$ with $l(a)=\sum_{i=1}^{t}|l_i|$ an HS-fraction and $o(a) = \sum_{j=1}^{k}|o_j|$ a region fraction, where the $l_i$'s are HP-fractions and the $o_j$'s are order-fractions. We say that $l(f)$ and $o(f)$ are an HS-fraction and a region-fraction corresponding to $f$.
\end{defn}

\begin{defn}
A principal kernel $K \in \PCon(\mathscr{R}(x_1,...,x_n))$ is said to be an \linebreak \emph{HO-kernel} if it is generated by an HO-fraction.
\end{defn}

\begin{rem}
 A kernel $K \in \PCon(\mathscr{R}(x_1,...,x_n))$ is an HO-kernel if and only if \linebreak  $K = L \cdot R$ where $R$ is a region kernel and $L$ is an HS-kernel.
\end{rem}
\begin{proof}
If $K$ is an HO-kernel, then $K = \langle f \rangle$ where $f = l(f)  \dotplus  o(f)$ is an HO-fraction. Thus $K = \langle l(f)  \dotplus  o(f) \rangle = \langle l(f) \rangle \cdot \langle o(f) \rangle = L \cdot R$ where $L = \langle l(f) \rangle$ is an HS-kernel and  $R~=~\langle o(f) \rangle$ is a region kernel. Conversely, taking the HO-fraction $f = l  \dotplus  r$ where $l$ is an HS-fraction generating $L$ and $r$ is a region fraction generating $R$, we get that $\langle f \rangle = \langle l  \dotplus  r \rangle = \langle l \rangle \cdot \langle r \rangle  =  L \cdot R = K$.
\end{proof}

\begin{rem}
Let $K = L \cdot R$ be an HO-kernel with  $R$ a region kernel and $L$ an \linebreak HS-kernel.
By Remarks \ref{rem_HS_prod_HP} and \ref{rem_Region_prod_Order}, we have that $L = \prod_{i=1}^{u}L_i$ for some HP-kernels $L_1,...,L_u$ and $R = \prod_{j=1}^{v}O_j$ for some order kernels $O_1,...,O_v$. Thus $K$ is of the form
\begin{equation}
K = L \cdot R = \prod_{i=1}^{u}L_i \cdot \prod_{j=1}^{v}O_j
\end{equation}
where $u \in \mathbb{Z}_{\geq 0}$, \ $v \in \mathbb{N}$, and $O_1,...,O_v$ are order kernels.
Note that every region kernel and every HS-kernel are by definition an HO-kernel, taking $u = 0$ for a region-kernel and $v=1$ with $O_1 = \langle 1  \dotplus  a \rangle$ where $L = \langle a \rangle$ is the HS-kernel.
\end{rem}

\begin{rem}
Let $K_1$ and $K_2$ be region-kernels (respectively HS-kernels) such that  $K_1~\cdot~K_2~\cap~\mathscr{R}~=~\{ 1 \}$. Then $K_1 \cdot K_2$ is a region-kernel (respectively HS-kernel).
Consequently, if $K_1$ and $K_2$ are HO-kernels such that $K_1 \cdot K_2 \cap \mathscr{R} = \{ 1 \}$ , then \linebreak $K_1 \cdot K_2$ is an HO-kernel.
Indeed, the assertions follow from the decomposition \linebreak $K_s = L_s \cdot O_s = \prod_{i=1}^{u_s}L_{s,i} \cdot \prod_{j=1}^{v_s}O_{s,j}$ for $s =1,2$ so that
$$K_1 \cdot K_2 = (L_1 L_2) \cdot (O_1 O_2) =  \left(\prod_{i=1}^{u_1}L_{1,i} \prod_{i=1}^{u_2}L_{2,i}\right) \cdot \left(\prod_{j=1}^{v_1}O_{1,j} \prod_{j=1}^{v_2}O_{2,j}\right) = L \cdot O,$$ with the appropriate $u_s,v_s$ taken for $s =1,2$.
\end{rem}

By the above discussion we have
\begin{thm}\label{thm_HP_expansion}
Every principal kernel $\langle f \rangle$ of $\mathscr{R}(x_1,...,x_n)$ can be written as an intersection of finitely many principal kernels $$\{K_i : i=1,...,s\} \ \text{and} \ \{N_j : j = 1,...,m \},$$ where each $K_i$ is a product of an HS-kernel and a region kernel
\begin{equation}
K_i = L_i \cdot R_i =  \prod_{j=1}^{t}L_{j,i} \prod_{l=1}^{l}O_{l,i}
\end{equation}
while each $N_i$ is a product of bounded from below kernels and (complementary) region kernels.
For $\langle f \rangle \in \PCon(\langle \mathscr{R} \rangle)$, the $N_j$ can be replaced by $\langle \mathscr{R} \rangle$ without affecting the resulting kernel.
\begin{itemize}
  \item If $\langle f \rangle$ is an HS-kernel, then the decomposition degenerates to $\langle f \rangle = K_1$ with $K_1 = L_1 = \langle f \rangle$.
  \item If $\langle f \rangle$ is a region kernel, then $\langle f \rangle = K_1$ with $K_1 = R_1= \langle f \rangle$.
  \item $\langle f \rangle$ is an irregular kernel if and only if there exists some $i_0 \in \{ 1,...,s \}$ such that $K_{i_0} =  R_{i_0} = \prod_{l=1}^{l}O_{l,i_0}$.
  \item $\langle f \rangle$ is a regular kernel if and only if $K_i$ is comprised of at least one HP-kernel, for every $i = 1,...,s$.
\end{itemize}
\end{thm}
\begin{proof}
The last three assertions are direct consequences of the construction above, namely, if $\langle f \rangle$ is either an HS-kernel or a region kernel, $\langle f \rangle$ is already in the form of its decomposition. The fourth is equivalent to the third.\\
If $\langle f \rangle$ is an HS-kernel
then by Remark \ref{rem_HS_prod_HP}, $\langle f \rangle = \prod_{j=1}^{t}L_j$ where $L_j$ is an \linebreak HP-kernel for each $j$.
Write $f = \frac{\sum h_{i}}{\sum g_{j}}$ with $h_i,g_j \in \mathscr{R}[x_1,...,x_n]$ monomials. If $\langle f \rangle$ is irregular, then by  definition of irregularity we have some $i_0$ and $j_0$ for which $h_{i_0} = g_{j_0}$ where $(g_{j_0}=) h_{i_0} > h_i , g_j$ for every $i \neq i_0$ and $j \neq j_0$, at some neighborhood of a point $a \in \mathscr{R}^n$. The kernel corresponding to (the closure) of this region
has its relation \eqref{eq_relations1} degenerating to $1 = 1$ as $\frac{h_{i_0}}{g_{j_0}} = 1$ over the region, thus is given only by its order relations of \eqref{eq_relations2}.
\end{proof}

\ \\

\begin{defn}
We call the decomposition given in Theorem \ref{thm_HP_expansion} of a principal kernel $\langle f \rangle \in \PCon(\mathscr{R}(x_1,...,x_n))$  the \emph{HO-decomposition} of $\langle f \rangle$.
In the special case where  $\langle f \rangle \in \PCon(\langle \mathscr{R} \rangle)$, all bounded from below terms of the intersection are equal to $\langle \mathscr{R} \rangle$.
\end{defn}


\begin{defn}
For a subset  $S \subseteq \mathscr{R}(x_1,...,x_n)$,
denote by  $\text{HO}(S)$ the family of \linebreak HO-fractions in $S$, by
$\text{HS}(S)$ the family of HS-fractions in $S$, and by $\text{HP}(S)$ the family of HP-fractions in $S$.
\end{defn}

\begin{rem}
Since every HP-fraction is an HS-fraction and every HS-fraction is an HO-fraction, we have that $$\text{HP}(S) \subset \text{HS}(S) \subset \text{HO}(S)$$ for any $S \subseteq \mathscr{R}(x_1,...,x_n)$.
\end{rem}

\begin{exmp}\label{exmp_HO_decomposition}
Consider the kernel $\langle f \rangle$ where $f = \frac{x}{y+1} \in \mathscr{R}(x_1,...,x_n)$. The points on the skeleton of $f$ define three distinct HS-kernels: $\langle \frac{x}{y} \rangle $ (corresponding to the equality $x=y$) over the region $\{y \geq 1\}$ which is defined by the region kernel $\langle 1 + y^{-1} \rangle$, $\langle x \rangle = \langle \frac{x}{1} \rangle$ (corresponding to $x=1$) over the region $\{y \leq 1\}$ which is defined by the region kernel $\langle 1 + y \rangle$, and $\langle |x| + |y| \rangle$ (corresponding to the point defined by $x=1$ and $x=y$). Thus by Construction \ref{HO_construction} $$\langle f \rangle =  \left(\left\langle \frac{x}{y} \right\rangle \cdot \langle 1 + y^{-1} \rangle\right) \cap \langle x \rangle \cdot \langle 1 + y \rangle \cap \langle |x| + |y| \rangle \cdot \langle 1 \rangle$$ and $$Skel(f) = \left(Skel\left(\frac{x}{y}\right) \cap Skel(1 + y^{-1})\right) \cup \left(Skel(x) \cap Skel(1 + y)\right) \cup \left(Skel(|x| + |y|) \cap \mathscr{R}^2 \right).$$
Note that the third component of the decomposition (i.e., the HS-kernel $\langle |x| + |y| \rangle$) can be omitted without effecting $Skel(f)$. \\
The decomposition is shown (in logarithmic scale) in Figure \ref{fig:Ex6} where the first two components are the rays beginning at the origin and the third component is the origin itself.

\end{exmp}

\ \\

\begin{figure}
\centering

\begin{tikzpicture}[scale=3,
    axis/.style={thin, color=gray, ->},
    important line/.style={very thick, ->}]

    \draw[axis] (-1,0)  -- (1,0) node[right, color=black] {$x$} ;
    \draw[axis] (0,-1) -- (0,1) node[above, color=black] {$y$} ;
    \draw[important line] (0,0)  -- (0,-1) node[above= 1.5cm,left=0.4cm] {$Skel(\langle x \rangle \cap \langle 1 + y \rangle)$};
    \draw [black, fill=black] (0,0) circle (1pt) node[below=0.5cm,right=0.3cm] {$Skel(\langle |x| + |y| \rangle)$};
    \draw[important line] (0,0) -- (0.7,0.7) node[below=0.2cm,right=0.2cm] {$Skel(\left\langle \frac{x}{y} \right\rangle  \cap \langle 1 + y^{-1} \rangle)$};

 \end{tikzpicture}

\caption{\small{$Skel(\frac{x}{y+1}) = Skel(\left\langle \frac{x}{y} \right\rangle  \cap \langle 1 + y^{-1} \rangle) \cup   Skel(\langle x \rangle \cap \langle 1 + y \rangle) \cup Skel(|x| + |y|)$}}
\label{fig:Ex6}

\end{figure}
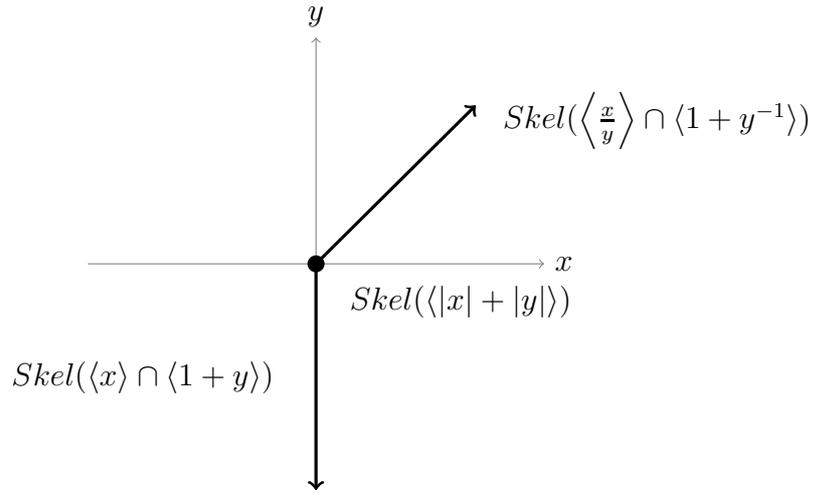
\newpage

\ \\

\subsection{The lattice generated by regular corner-integral principal \\ \ \ \ \ \ \  \ \ \ kernels}

Recall Remark \ref{rem_regularity_of_HP_times_order_kernels_generalization} which states that
the principal kernel
$$\langle f \rangle \cdot \langle g_1 \rangle \cdot \dots \cdot \langle g_k \rangle$$
is regular, for any HP-kernel $\langle f \rangle \neq 1$ and order kernels $\langle g_1 \rangle, ... , \langle g_k \rangle$.

\begin{cor}\label{cor_regularity_of_intersection_terms_of_HS_O_expansion}
Let $K \in \PCon(\langle \mathscr{R} \rangle)$ and
let $$K = (L \cdot R) \cap \langle \mathscr{R} \rangle = \left(\prod_{i=1}^{u}L_i \cdot \prod_{j=1}^{v}O_j\right) \ \cap \langle \mathscr{R} \rangle$$ be the decomposition of $K$, as given in \eqref{eq_L_O}, where $L_1,....,L_u$ are some HP-kernels and $O_1,...,O_v$ are some order kernels. If $u \neq 0$, i.e., $L \neq \langle 1 \rangle$ then $K$ is regular.
\end{cor}
\begin{proof}
Indeed, $K = \prod_{i=2}^{u}L_i \cdot (L_1 \cdot \prod_{j=1}^{v}O_j) \cap \langle \mathscr{R} \rangle$. By Remark \ref{rem_regularity_of_HP_times_order_kernels_generalization}, we have that $(L_1 \cdot \prod_{j=1}^{v}O_j)$ is regular as $L_1$ is regular as an HP-kernel. Thus since a product of regular kernels is regular and since intersection with $\langle \mathscr{R} \rangle$ does not affect regularity, we have that $K$ is  a regular kernel.
\end{proof}

\begin{thm}\label{thm_ci_reg_kernels_generate_principal_lattice}
The lattice generated by principal corner integral kernels in $\PCon(\mathscr{R})$ is the lattice of principal kernels $\PCon(\mathscr{R})$ and the lattice generated by regular \linebreak principal corner integral kernels is the lattice of regular principal kernels.
\end{thm}
\begin{proof}
Let $\langle f \rangle$ be a principal kernel and let
$$\langle f \rangle = \left(\bigcap_{i=1}^{s}K_i\right)  \cap \langle \mathscr{R} \rangle, \ \ K_i = \prod_{j=1}^{t}L_{j,i} \prod_{l=1}^{l}O_{l,i}$$
be its HO-decomposition. By Lemma \ref{lem_HP_and_Oreder_kernels_are_CI}, each HP-kernel $L_{j,i}$ and each order kernel $O_{l,i}$ are corner integral. Thus $\langle f \rangle$ as a finite product of principal corner integral kernels is in the lattice generated by principal corner-integral kernels. As any principal corner integral kernel is in particular principal, the lattice of principal kernels contains the lattice generated by principal corner-integral kernels proving the first assertion. As for the second assertion,  if $\langle f \rangle$ is regular, then by Theorem \ref{thm_HP_expansion}, for every $1 \leq i \leq s$, we have that $L_{1,i} \neq 1$. Thus by Corollary \ref{cor_regularity_of_intersection_terms_of_HS_O_expansion} we have that each $K_i$ is a product of principal regular corner-integral kernels. Thus $\langle f \rangle$ is in the lattice generated by principal regular corner-integral kernels. As any regular corner integral kernel is in particular regular, the lattice of principal regular kernels contains the lattice generated by principal regular corner-integral kernels. Thus the second assertion holds.
\end{proof}

\begin{cor}\label{cor_lattice_generated_by_CI_kernels}
By Theorem \ref{thm_ci_reg_kernels_generate_principal_lattice}, and the correspondence between principal \linebreak (regular) corner-loci and principal (regular) corner-integral kernels of $\langle \mathscr{R} \rangle$, defined by composing the correspondence between the kernels with their principal skeletons \linebreak $K \mapsto Skel(K)$ introduced in Corollary \ref{cor_correspondence_bounded_pkernels_pskeletons} with the correspondence between these principal skeletons and their corresponding corner loci introduced in \linebreak Proposition \ref{prop_regular_ci_skeletons_corner_loci_correspondence}, we have that the lattice of (regular) finitely generated corner loci corresponds to the lattices of principal (regular) kernels of $\langle \mathscr{R} \rangle$.
\end{cor}

\begin{cor}
By Corollary \ref{cor_lattice_generated_by_CI_kernels}, we have that supertropical varieties correspond to principal skeletons and kernels while tropical varieties correspond to regular principal skeletons and kernels.
\end{cor}

\ \\
\ \\
\subsection{Convexity degree and hyperdimension}\label{section:Convexity degree and Hyperdimension}
 \ \\

In this section, $\mathbb{K}$ is a semifield which is an affine extension of the bipotent semifield $\mathscr{R}$, i.e., $\mathbb{K}$ is of the form $\mathscr{R}(x_1,...,x_n)/L$ for some kernel $L \in \Con(\mathscr{R}(x_1,...,x_n))$ (see Remark \ref{rem_affine_semifields_as_images}). In particular $\mathbb{K}$ is idempotent.

\begin{prop}\label{prop_subdirect_decomposition_for_idempotent_kernels}
Let $\mathbb{S}$ be an idempotent semifield. Let $M$ be a kernel of $\mathbb{S}$
such that $M = K_1 \cap K_2 \cap \dots  \cap K_t$ for distinct kernels $K_i$  of $\mathbb{S}$.
Then $\mathbb{S}/M$ is subdirectly reducible and
$$\mathbb{S}/M \xrightarrow[s.d]{} \prod_{i=1}^{t} \mathbb{S}/K_i.$$
\end{prop}
\begin{proof}
Consider the quotient semifield $\bar{\mathbb{S}} = \mathbb{S}/M$. Let $\phi : \mathbb{S} \rightarrow \bar{\mathbb{S}}$ be the quotient map. Denote by $\bar{K_i} = \phi(K_i) = K_i/M$ the images of $K_1,...,K_n$ under $\phi$ which are kernels of $\bar{\mathbb{S}}$. Since $M = K_1 \cap K_2 \cap \dots  \cap K_t$, we have that $\bar{K_1} \cap \bar{K_2} \cap \dots  \cap \bar{K_t} = \{1\}$ in $\bar{\mathbb{S}}$ (note that by Corollary \ref{cor_qoutient_corr}  there is a lattice isomorphism between the  kernels of $\bar{\mathbb{S}}$ and the kernels of $\mathbb{S}$ that contain $M$). \\
Thus, by Remark \ref{rem_sub_direct_prod_construction}, we have $\bar{\mathbb{S}}$ is subdirectly reducible
and as \\ $\bar{\mathbb{S}}/\bar{K_i} = (\mathbb{S}/M)/(K_i/M) \cong \mathbb{S}/K_i$ (by the second isomorphism theorem \ref{thm_nother2}) we have that
$$\bar{\mathbb{S}} \xrightarrow[s.d]{} \prod_{i=1}^{t}\bar{\mathbb{S}}/\bar{K_i} \cong \prod_{i=1}^{t} \mathbb{S}/K_i.$$
\end{proof}

\begin{cor}\label{cor_subdirect_decomposition_for_semifield_of_fractions}
In particular, Proposition \ref{prop_subdirect_decomposition_for_idempotent_kernels} applies to the idempotent semifield $\mathbb{H}(x_1,...,x_n)$ where $\mathbb{H}$ is a bipotent semifield (or any of its kernels considered as a semifield) and to any principal kernel $M = \langle f \rangle \in \PCon(\mathbb{H}(x_1,...,x_n))$ such that $M = K_1 \cap K_2 \cap \dots  \cap K_t$ for some $K_i \in \PCon(\mathbb{H}(x_1,...,x_n))$.\\
\end{cor}

Let $\langle f \rangle \in \langle \mathscr{R} \rangle$ be a principal kernel and let $\langle f \rangle = \bigcap_{i=1}^{s} K_i$, where
$$K_i = (L_i \cdot R_i) \cap \langle \mathscr{R} \rangle = (L_i \cap \langle \mathscr{R} \rangle) \cdot (R_i \cap \langle \mathscr{R} \rangle) = L_i' \cdot R_i'$$
is its (full) HO-decomposition; i.e., for each $1 \leq i \leq s$ \ $R_i \in \PCon(\mathscr{R}(x_1,...,x_n))$ is a region kernel and $L_i\in \PCon(\mathscr{R}(x_1,...,x_n))$ is either an HS-kernel or bounded from below (in which case  $L_i' = \langle \mathscr{R} \rangle$). Then by Corollary \ref{cor_subdirect_decomposition_for_semifield_of_fractions}, we have the subdirect decomposition
$$\langle \mathscr{R} \rangle/ \langle f \rangle  \xrightarrow[s.d]{} \prod_{i=1}^{t} \langle \mathscr{R} \rangle /K_i = \prod_{i=1}^{t} (\langle \mathscr{R} \rangle /L_i' \cdot R_i')$$
where $t \leq s$ is the number of kernels $K_i$ for which $L_i' \neq  \langle \mathscr{R} \rangle$ (for otherwise $\langle \mathscr{R} \rangle /K_i~=~\{1 \}$ and can be omitted from the subdirect product).

The last discussion motivates us to study the semifields $\langle \mathscr{R} \rangle /K_i$ as building blocks for the algebraic structure of the quotient semifield $\langle \mathscr{R} \rangle/ \langle f \rangle$ which in turn is the \linebreak coordinate semifield corresponding to the skeleton $Skel(f)$.

\begin{exmp}\label{exmp_infinite_chain}
Consider the principal kernel $\langle x \rangle \in \PCon(\mathscr{R}(x,y))$. For $\alpha \in \mathscr{R}$ such that $\alpha > 1$, we have the following infinite strictly descending chain of principal kernels
$$\langle x \rangle \supset \langle |x| \dotplus |y+1| \rangle \supset \langle |x| \dotplus |\alpha^{-1}y \dotplus 1| \rangle \supset \langle |x| \dotplus |\alpha^{-2}y \dotplus 1| \rangle \supset \dots$$ $$\supset \langle |x| \dotplus |\alpha^{-k}y \dotplus 1| \rangle \supset \dots$$ and the strictly ascending chain of skeletons corresponding to it (see figure 2.3)
$$Skel(x) \subset Skel(|x| \dotplus |y+1|) \subset \dots \subset Skel(|x| \dotplus |\alpha^{-k}y \dotplus 1|) \subset \dots =$$ $$Skel(x) \subset Skel(x) \cap Skel(y+1) \subset \dots \subset Skel(x) \cap Skel(\alpha^{-k}y \dotplus 1) \subset \dots.$$
\end{exmp}

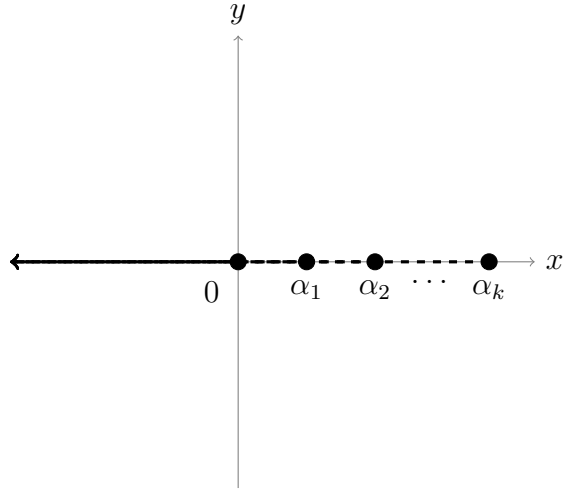
\begin{figure}\label{fig_3}
\centering

\begin{tikzpicture}[scale=3,
    axis/.style={thin, color=gray, ->},
    important line/.style={very thick,dashed, ->}]

    \draw[axis] (-1,0)  -- (1.3,0) node[right, color=black] {$x$} ;
    \draw[axis] (0,-1) -- (0,1) node[above, color=black] {$y$} ;
    \draw[important line] (0,0)  -- (-1,0) ;
    \draw [black, fill=black] (0,0) circle (1pt) node[below=0.4cm,left=0.1cm] {$0$};
    \draw[important line] (0.3,0) -- (-1, 0) ;
    \draw [black, fill=black] (0.3,0) circle (1pt) node[below=0.1cm] {$\alpha_1$};
    \draw[important line] (0.6,0) -- (-1, 0) ;
    \draw [black, fill=black] (0.6,0) circle (1pt) node[below=0.1cm] {$\alpha_2$};
    \draw [black] (0.85,0) circle (0pt) node[below=0.1cm] {$\dots$};
    \draw[important line] (1.1,0) -- (-1, 0) ;
    \draw [black, fill=black] (1.1,0) circle (1pt) node[below=0.1cm] {$\alpha_k$};

 \end{tikzpicture}

\caption{In logarithmic scale : $\alpha_t = \log(\alpha^k) = k \cdot log(\alpha) $} \label{fig:Ex3}
\end{figure}

\ \\

\begin{exmp}\label{exmp_decomposition}
Again, consider the principal kernel $\langle x \rangle \in \PCon(\mathscr{R}(x,y))$. Then $\langle x \rangle = \langle |x| \dotplus (|y \dotplus 1| \wedge |\frac{1}{y} \dotplus 1|) \rangle = \langle  (|x| \dotplus |y \dotplus 1|) \wedge (|x| \dotplus |\frac{1}{y} \dotplus 1|) \rangle = \langle |x| \dotplus |y \dotplus 1| \rangle \cap \langle |x| \dotplus |\frac{1}{y} \dotplus 1|  \rangle$. So, we have that a nontrivial decomposition of $Skel(x)$ as $Skel(|x| \dotplus |y \dotplus 1|) \cup Skel(|x| \dotplus |\frac{1}{y} \dotplus 1|)$ (note that $Skel(|x| \dotplus |y+1|) = Skel(x) \cap Skel(y+1)$ and $Skel(|x| \dotplus |\frac{1}{y} \dotplus 1|) = Skel(x) \cap Skel(\frac{1}{y} \dotplus 1)$). In a similar way, using complementary order kernels,  one can show that every principal kernel can be non-trivially decomposed to a pair of principal kernels.
\end{exmp}

\ \\

\begin{figure}
\centering

\begin{tikzpicture}[scale=3,
    axis/.style={thin, color=gray, ->},
    important line/.style={very thick,dashed, ->}]

    \draw[axis] (-1,0)  -- (1,0) node[right, color=black] {$x$} ;
    \draw[axis] (0,-1) -- (0,1) node[above, color=black] {$y$} ;
    \draw[important line] (0,0)  -- (1,0) node[below=0.4cm,left=0.4cm] {$|x| \dotplus |\frac{1}{y} \dotplus 1|$};
    \draw [black, fill=black] (0,0) circle (1pt) node[below=0.3cm,right=0.1cm] {$0$};
    \draw[important line] (0,0) -- (-1, 0) node[below=0.4cm,right=0.4cm] {$|x| \dotplus |y \dotplus 1|$};

 \end{tikzpicture}

\caption{$Skel(x) = Skel(|x| \dotplus |y \dotplus 1|) \cup Skel(|x| \dotplus |\frac{1}{y} \dotplus 1|)$} \label{fig:Ex4}

\end{figure}

Examples \ref{exmp_infinite_chain} and \ref{exmp_decomposition} demonstrate that the lattice of principal kernels \linebreak $\PCon(\mathscr{R}(x_1,...,x_n))$ (resp. $\PCon(\langle \mathscr{R} \rangle)$) is too rich to define reducibility or dimensionality. Moreover, these examples suggest that this richness is caused by order kernels.  This motivates us to consider $\Theta$-reducibility for some sublattice of kernels \linebreak $\Theta \subset \PCon(\mathscr{R}(x_1,...,x_n))$ (resp. $\Theta \subset \PCon(\langle \mathscr{R} \rangle)$).  There are various families of \linebreak kernels that one can utilize to define the notions of reducibility, dimensionality, etc. We consider the sublattice generated by HP-kernels. Our choice is made due to its connection to the (local) dimension of the linear spaces (in logarithmic scale) defined by the skeleton corresponding to a kernel. Namely, HP-kernels, and more generally \linebreak HS-kernels, define affine subspaces in $\mathscr{R}^n$ (see \ref{subsection:geometric_interpretation}).

\ \\

\begin{defn}
Let $\mathbb{K}$ be a semifield as defined in the beginning of this section. Define $S(\mathbb{K})$
to be the set elements of $\mathbb{K}$ generated by the collection $$\{ |f| : f \in \mathbb{K} \ ; \  \text{f is an HP-fraction} \ \}$$ with respect to the operations $\wedge$ and $\dotplus$ (equivalently $ \vee $). Since $\mathbb{K}$ is a semifield it is closed with respect to   $\wedge$ and $\vee$ and $| \cdot |$, thus $S(\mathbb{K}) \subset \mathbb{K}$. By definition, all HS-fractions in $\mathbb{K}$ are elements of $S(\mathbb{K})$.
Define
$$\Gamma(\mathbb{K}) = S(\mathbb{K})/ \sim_K.$$
Namely, for $f,g \in \Gamma(\mathbb{K})$, $f \sim_K g$ if and only if $\langle f \rangle = \langle g \rangle$ in $\PCon(\mathbb{K})$. Note that $1 \in   \Gamma(\mathbb{K})$. \\
Let $\Omega(\mathbb{K}) \subset \PCon(\mathbb{K})$ be the lattice of kernels generated by the collection of all \linebreak HP-kernels of $\mathbb{K}$, i.e., every element $\langle f \rangle \in \Omega(\mathbb{K})$ is obtained by a finite intersections and products of HP-kernels.\\
By definition, every HS-kernel in $\PCon(\mathbb{K})$ is an element of $\Omega(\mathbb{K})$ as it is a product of finite set of HP-kernels. Note that $1 \in   \Omega(\mathbb{K})$.
\end{defn}

\begin{rem}\label{rem_rho_correspondence}
Let $\Gamma = \Gamma(\mathbb{K})$ and $\Omega = \Omega(\mathbb{K})$.
There is a correspondence $\rho$ between $(\Gamma, \wedge, \vee)$ and $(\Omega, \cap, \cdot)$ defined by
\begin{equation}
\rho(f) = \langle f \rangle.
\end{equation}
This is a lattice homomorphism in the sense that for any $f,g \in \Gamma$
$$ \rho(f \wedge g) = \rho(f) \cap \rho(g), \ \text{and} \ \rho(f \vee g) = \rho(f) \cdot  \rho(g).$$
Note that  $f = |f|$ and $g = |g|$ by the definition of $\Gamma$.
If $f,g \in \Gamma$ such that $\rho(f) = \rho(g)$, then by definition $\langle f \rangle = \langle g \rangle$ and thus $f \sim_K g$. So, $\rho$ is injective. Finally, $\rho$ is onto by the definition of HP-kernels.
\end{rem}

As $\Omega(\mathbb{K})$ is a sublattice of kernels of $\PCon(\mathbb{K})$, we have the notion of \linebreak $\Omega(\mathbb{K})$-irreducible kernel as developed in Subsection \ref{Subsection:Reducibility_of_principal_kernels_and_skeletons} concerning `reducibility of principal kernels and skeletons' (taking $\Theta =  \Omega(\mathbb{K})$). Namely  $\langle f \rangle \in \Omega(\mathbb{K})$ is reducible if there exist some $ \langle g \rangle, \langle h \rangle \in \Omega(\mathbb{K})$ such that $\langle f \rangle  \supseteq \langle g \rangle \cap \langle h \rangle$ while $\langle f \rangle \not \supseteq \langle g \rangle$ and $\langle f \rangle \not \supseteq \langle h \rangle$.
\ \\
\ \\
\begin{lem}
By the construction of $\Omega(\mathbb{K})$, the condition stated above is equivalent to the condition  $\langle f \rangle  = \langle g \rangle \cap \langle h \rangle$ while $\langle f \rangle \neq \langle g \rangle$ and $\langle f \rangle \neq \langle h \rangle.$
\end{lem}
\begin{proof}
Assume $\langle f \rangle$ admits the stated condition. If $\langle f \rangle  \supseteq \langle g \rangle \cap \langle h \rangle$, then $\langle f \rangle = \langle f \rangle \cdot \langle f \rangle  = \langle g \rangle \cdot \langle f \rangle \cap \langle h \rangle \cdot \langle f \rangle$. Thus $\langle f \rangle = \langle g \rangle \cdot \langle f \rangle$ or $\langle f \rangle = \langle h \rangle \cdot \langle f \rangle$ and so $\langle f \rangle \supseteq \langle g \rangle$ or $\langle f \rangle \supseteq \langle g \rangle$. The converse is obvious.
\end{proof}

\begin{note}
For the rest of this section, we refer to $\Omega(\mathbb{K})$-irreducibility as irreducibility.
\end{note}

\begin{defn}
Define the \emph{irreducible hyperspace spectrum} of $\mathbb{K}$, $\text{HSpec}(\mathbb{K})$, to be the family of irreducible kernels in $\Omega(\mathbb{K})$.
\end{defn}

\begin{rem}
$HSpec(\mathbb{K})$ is the family of HS-kernels in $\Omega(\mathbb{K})$ which is exactly the family of HS-kernels of $\mathbb{K}$.
\end{rem}

\begin{defn}
A chain $P_0 \subset P_1 \subset \dots \subset P_t$ in $\text{HSpec}(\mathbb{K})$ means an ascending chain of HS-kernels of $\mathbb{K}$, and is said to have \emph{length} $t$. An HS-kernel $P$ has \emph{height} $t$ (denoted $hgt(P)=t$) if there is a chain of length $t$ in $\text{HSpec}(\mathbb{K})$ terminating at $P$, but no chain of length $t+1$ terminates at $P$.
\end{defn}

\begin{rem}\label{rem_correspondence_of_quontient_HSpec}
Let $L$ be a kernel in $\PCon(\mathbb{K})$. Consider the quotient homomorphism $\phi_L : \mathbb{K} \rightarrow \mathbb{K}/L$.
As the image of a principal kernel is the principal kernel generated by the image of any of its generators, we have that for an HP-kernel $\langle f \rangle$ we have $\phi_L(\langle f \rangle)~=~\langle \phi_L(f) \rangle$. Choosing $f$ to be in canonical HP-fraction,  since $\phi_L$ is an $\mathscr{R}$ homomorphism, we have that $\langle \phi_L(f) \rangle$ is a nontrivial HP-kernel in $ \mathbb{K}/L$ if and only if $\phi_L(f) \not \in \mathscr{R}$. Thus the set of HP-kernels of $\mathbb{K}$ mapped to HP-kernels of $\mathbb{K}/L$ is
\begin{equation}\label{eq_subset_of_Omega}
\left\{ \langle g \rangle : \langle g \rangle \cdot \langle \mathscr{R} \rangle \supseteq \phi_L^{-1}\left(\langle \mathscr{R} \rangle\right) = L \cdot \langle \mathscr{R} \rangle \right\}.
\end{equation}
Since $\phi_L$ is an $\mathscr{R}$-homomorphism and onto, we have that one of the preimages of a canonical HP-fraction $g \in \mathbb{K}/L$ (represented by $g$) is $g \in \mathbb{K}$ itself. As $\phi_L$ is an \linebreak $\mathscr{R}$-homomorphism it respects $\vee, \wedge$ and $| \cdot |$  thus $\phi_{L}(\Gamma(\mathbb{K})) = \Gamma(\mathbb{K}/L)$.\linebreak
In fact by Corollary \ref{cor_nother_4_for_principal_kernels}, we have a correspondence identifying $\text{HSpec}(\mathbb{K}/L)$ with the subset of $\text{HSpec}(\mathbb{K})$ which consists of all HS-kernels $P$ of $\mathbb{K}$ such that \linebreak $P \cdot \langle \mathscr{R} \rangle~\supseteq~L~\cdot~\langle \mathscr{R} \rangle$. Moreover, by the same considerations and since $\wedge$ is preserved under a homomorphism we have that the above correspondence extends to a correspondence identifying $\Omega(\mathbb{K}/L)$ with the subset \eqref{eq_subset_of_Omega} of $\Omega(\mathbb{K})$.  Under this correspondence, the maximal (HS) kernels of $\mathbb{K}/L$ correspond to maximal (HS) kernels of $\mathbb{K}$ and reducible kernels of $\mathbb{K}/L$ correspond to reducible kernels of $\mathbb{K}$. Indeed, the latter assertion is obvious since  $\wedge$ is preserved under a homomorphism. For the former assertion, by the second isomorphism theorem $(\mathbb{K}/L)/(P/L) \cong \mathbb{K}/P$ so simplicity of the quotients is preserved. Thus so is maximality of $P$ and $P/L$.
\end{rem}

\begin{defn}\label{defn_Hdim}
The Hyperdimension of $\mathbb{K}$, written $Hdim \mathbb{K}$ (if it exists), is the maximal height of the HS-kernels in $\mathbb{K}$.
\end{defn}

\begin{defn}\label{defn_convex_dep}
Let $A \subset \text{HS}(\mathbb{K})$ be any set of HS-fractions and let $f \in \text{HS}(\mathbb{K})$. Then $f$ is said to be \emph{$\mathscr{R}$-convexly dependent} on $A$  if
\begin{equation}\label{eq_defn_convex_dep}
f \in \left\langle \{g : g \in A \} \right\rangle \cdot \langle \mathscr{R} \rangle,
\end{equation}
otherwise $f$ is said to be \emph{$\mathscr{R}$-convexly-independent} of $A$.
A subset $A \subset \text{HS}(\mathbb{K})$ is said to be $\mathscr{R}$-convexly independent  if
for every $a \in A$, $a$ is $\mathscr{R}$-convexly independent of $A \setminus \{ a \}$  over $\mathscr{R}$.
Note that by assuming $g \in \mathbb{K} \setminus \{1\}$ for some $g \in A$ the condition in \eqref{eq_defn_convex_dep} simplifies to $f \in \left\langle \{g : g \in A \} \right\rangle $. Indeed, under this last assumption we have that  $\langle \mathscr{R} \rangle \subseteq \left\langle \{g : g \in A \} \right\rangle $ and so $\left\langle \{g : g \in A \} \right\rangle \cdot \langle \mathscr{R} \rangle = \left\langle \{g : g \in A \} \right\rangle$.
\end{defn}

\begin{note}
If $\{a_1,...,a_n \}$ is $\mathscr{R}$-convexly dependent (independent), then we also say that $a_1,...,a_n$ are $\mathscr{R}$-convexly dependent (independent).
\end{note}
\begin{rem}
By the definition, we have that an HS-fraction $f$ is $\mathscr{R}$-convexly \linebreak dependent on  $\{g_1,...,g_t\} \subset \text{HS}(\mathbb{K})$  if and only if
$$\langle |f|  \rangle = \langle f  \rangle  \subseteq  \langle g_1,..., g_t \rangle \cdot \langle \mathscr{R} \rangle = \Big\langle \sum_{i=1}^{t} |g_i| \Big\rangle \cdot \langle \mathscr{R} \rangle = \Big\langle \sum_{i=1}^{t} |g_i|  \dotplus  |\alpha| \Big\rangle$$ where $\alpha$ is any element of $\mathbb{K} \setminus \{1\}$.
\end{rem}

\begin{exmp}
For any $\alpha \in \mathscr{R}$ and any $f \in \mathscr{R}(x_1,...,x_n)$,
$$|\alpha f| \leq |f|^2  \dotplus |\alpha|^2 =(|f|  \dotplus  |\alpha|)^2.$$
Thus $\alpha f \in \left\langle (|f|  \dotplus  |\alpha|)^2 \right\rangle = \langle |f|  \dotplus  |\alpha| \rangle = \langle f \rangle \cdot \langle \mathscr{R} \rangle$.
In particular, if $f$ is an HS-fraction then $\alpha f$ is $\mathscr{R}$-convexly dependent on $f$.
\end{exmp}

As a consequence of Proposition \ref{prop_HP_element_is_a_generator}, we have that for two HP-fractions $f,g$, if $g \in \langle f \rangle$ then $\langle g \rangle = \langle f \rangle$. In other words, either $\langle g \rangle = \langle f \rangle$ or $\langle g \rangle \not \subseteq \langle f \rangle$ and $\langle f \rangle \not \subseteq \langle g \rangle$. This motivates us to restrict our attention to the convex dependence relation on the set of HP-fractions. This will be justified later by showing that for each $\mathscr{R}$-convexly independent subset of HS-fractions of size $t$ in $\mathbb{K}$, there exists an $\mathscr{R}$-convexly independent subset of HP-fractions of size $ \geq t$ in $\mathbb{K}$.

\begin{prop}\label{prop_abstract_dependence_properties}

Let $A \subseteq \text{HP}(\mathbb{K})$ and let $f \in \text{HP}(\mathbb{K})$. Then
\begin{enumerate}
  \item If $f \in A$ then $f$ is $\mathbb{H}$-convexly-dependent on $A$.
  \item If $f$ is $\mathscr{R}$-convexly dependent on $A$ and $A_1$ is a set such that $a$ is $\mathscr{R}$-convexly-dependent on $A_1$ for each $a \in A$, then $f$ is $\mathscr{R}$-convexly dependent on $S_1$.
  \item If $f$ is $\mathscr{R}$-convexly-dependent on $A$, then $f$ is $\mathscr{R}$-convexly-dependent on $A_0$ for some finite subset $A_0$ of $A$.
\end{enumerate}
\end{prop}
\begin{proof}
(1) Since $f \in A$ we have that  $f \in \langle A \rangle \subseteq  \langle A \rangle \cdot \langle \mathscr{R} \rangle $.\\
(2) If $a$ is convexly-dependent on $S_1$  for each $a \in A$, then $A \subseteq \langle A_1 \rangle \cdot \langle \mathscr{R} \rangle$. Thus $\langle A \rangle \subseteq \langle A_1 \rangle \cdot \langle \mathscr{R} \rangle$. If $f$ is $\mathscr{R}$-convexly dependent on $A$ then $f \in \langle A \rangle \cdot \langle \mathscr{R} \rangle \subseteq \langle A_1 \rangle \cdot \langle \mathscr{R} \rangle$, so, $f$ is $\mathscr{R}$-convexly dependent on $A_1$.\\
(3) $a \in \langle A \rangle \cdot \langle \mathscr{R} \rangle$, so by Proposition \ref{prop_ker_stracture_by group} there exist some $s_1,...,s_k \in \mathbb{K} $ and  $g_1,...,g_k \in G(A \cup \mathscr{R}) \subset \langle A \rangle \cdot \langle \mathscr{R} \rangle$, where $G(A \cup \mathscr{R})$ is the group generated by $A \cup \mathscr{R}$, such that $\sum_{i=1}^{k}s_i = 1$ and $a = \sum_{i=1}^{k}s_i g_i^{d(i)}$ with $d(i) \in \mathbb{Z}$. Thus $a \in \langle g_1,...,g_k \rangle$ and  $A_0  = \{g_1,...,g_k\}$.
\end{proof}

 \begin{rem}\label{rem_smallest_kernel_containing_the_generated_semifield}
 Note that $\langle g_1,..., g_t \rangle \cdot \langle \mathscr{R} \rangle $ is the smallest kernel containing the semifield \ $SF(g_1,..., g_t)$ generated (as a semifield) by  $g_1,..., g_t$ over $\mathscr{R}$.\\
 Indeed,$\mathscr{R} \subset  SF(g_1,...,g_t)$ and $g_1,...,g_t \in SF(g_1,...,g_t)$. Thus any kernel containing \linebreak $SF(g_1,..., g_t)$ must contain $\langle g_1,..., g_t \rangle \cdot \langle \mathscr{R} \rangle$, being the smallest kernel containing $\mathscr{R}$ and $\{ g_1,..., g_t \}$.\\
 \end{rem}

  Since Remark \ref{rem_smallest_kernel_containing_the_generated_semifield} is similar to algebraic dependence, we are led to try to show that convex dependence is an abstract dependence.

\begin{rem}\label{rem_HP_not_contains_contained_H}
For $f \in \text{HP}(\mathbb{K})$ the following hold:
\begin{enumerate}
  \item $\langle \mathscr{R} \rangle \not \subseteq \langle f \rangle$.
  \item If $\mathbb{K}$ is not bounded then $\langle f  \rangle \not \subseteq \langle \mathscr{R} \rangle$.
\end{enumerate}
Indeed, by definition an HP-fraction is not bounded from below. Thus $\langle f \rangle \cap \mathscr{R} = \{ 1 \}$ or equivalently $\langle \mathscr{R} \rangle \not \subseteq \langle f \rangle$.
For the second assertion, by Remark \ref{rem_HP_not_bounded}, an HP-kernel is not bounded when $\mathbb{K}$ is not bounded, so we have that $\langle f \rangle \not \subseteq \langle \mathscr{R} \rangle$.
\end{rem}
\ \\

A direct consequence of Remark \ref{rem_HP_not_contains_contained_H} is
\begin{rem}
If $\mathbb{K}$ is not bounded then any proper HS-kernel (i.e., not $\langle 1 \rangle$) is \linebreak $\mathscr{R}$-convexly independent.
\end{rem}
\begin{proof}
By Remark \ref{rem_HP_not_contains_contained_H} the assertion is true for HP-kernels, and thus for \linebreak HS-kernels, since every HS-kernel contains some HP-kernel.
\end{proof}

\begin{prop}[Exchange axiom]\label{prop_convex_dependence_of_HP_property}
Let $S = \{b_1,...,b_t\} \subset \text{HP}(\mathbb{K})$ and let $f$ and $b$ be elements of $\text{HP}(\mathbb{K})$. Then if $f$ is convexly-dependent on $S \cup \{ b \}$ and $f$ is $\mathscr{R}$-convexly independent of $S$, then $b$ is $\mathscr{R}$-convexly-dependent on $S \cup \{ f \}$.
\end{prop}
\begin{proof}
We may assume that $\alpha \in S$ for some $\alpha \in \mathscr{R}$.
Since $f$ is $\mathscr{R}$-convexly independent of $S$, by definition $f \not \in  \langle S \rangle$ this implies that $\langle S \rangle \subset \langle S \rangle \cdot \langle f \rangle$ (for otherwise $\langle f \rangle \subseteq \langle S \rangle$ yielding that $f$ is $\mathscr{R}$-convexly dependent on $S$).  Since $f$ is $\mathscr{R}$-convexly-dependent on $S \cup \{ b \}$, we have that  $f \in \langle S \cup \{ b \} \rangle = \langle S \rangle \cdot \langle b \rangle$. In particular, we get that $b \not \in \langle S \rangle \cdot \langle \mathscr{R} \rangle$ for otherwise $f$ would be dependent on $S$.
Consider the quotient map $\phi : \mathbb{K}  \rightarrow \mathbb{K}/\langle S \rangle$. Since $\phi$ is a semifield epimorphism and $f,g  \not \in \langle S \rangle \cdot \langle \mathscr{R} \rangle = \phi^{-1}(\langle \mathscr{R} \rangle)$, we have that $\phi(f)$ and $\phi(b)$ are not in $\mathscr{R}$ thus are HP-fractions in the semifield $\Im(\phi) =\mathbb{K}/\langle S \rangle$. By the above,  $\phi(f) \neq 1$ and $\phi(f) \in \phi(\langle b \rangle) = \langle \phi(b) \rangle$. Thus, by Corollary \ref{cor_HP-element_is_a_generator}  we have that $\langle \phi(f) \rangle =\langle \phi(b) \rangle$. So $\langle S \rangle \cdot \langle f \rangle = \phi^{-1}(\langle \phi(f) \rangle) = \phi^{-1}(\langle \phi(b) \rangle) = \langle S \rangle \cdot \langle b \rangle$, consequently $b \in \langle S \rangle \cdot \langle b \rangle = \langle S \rangle \cdot \langle f \rangle = \langle S \cup \{ f \} \rangle $, i.e., $b$ is $\mathscr{R}$-convexly-dependent on $S \cup \{ f \}$.
\end{proof}

\begin{defn}\label{defn_convex_span_in_semifield_of_fractions}
Let $A \subseteq \text{HP}((\mathbb{K})$.
The \emph{convex-span} of $A$ over $\mathscr{R}$ is the set
\begin{equation}
ConSpan_{\mathscr{R}}(A) = \{ a \in \text{HP}(\mathbb{K}) : a \text{ is $\mathscr{R}$-convexly dependent on $A$} \}.
\end{equation}
Let $\mathrm{K} \subseteq \mathbb{K}$ be a subsemifield such that $\mathscr{R} \subseteq \mathrm{K}$.
Then a set  $A \subseteq \text{HP}(\mathbb{K})$  is said to \emph{convexly span} $\mathrm{K}$ over $\mathscr{R}$ if  $$\text{HP}(\mathrm{K}) = ConSpan_{\mathscr{R}}(A).$$
\end{defn}

In view of Propositions \ref{prop_abstract_dependence_properties} and \ref{prop_convex_dependence_of_HP_property}, convex-dependence on  $\text{HP}(\mathbb{K})$ is a (strong) dependence relation. Then by \cite[Chapter 6]{ComView}, we have that:

\begin{cor}\label{cor_basis_for_HP}
Let $V \subset \text{HP}(\mathbb{K})$. Then $V$ contains a basis $B_V \subset V$, which is a maximal convexly independent subset of unique cardinality such that $$ConSpan(B_V) = ConSpan(V).$$
\end{cor}

\ \\

\begin{defn}
Let $V \subset \text{HP}(\mathbb{K})$ be a set of HP-fractions. We define the \emph{convexity degree} of $V$, $condeg(V)$, to be $|B|$ where $B$ is a basis for $V$.
\end{defn}

\begin{rem}\label{rem_dependence_is_a_lattice}
If $S \subset \text{HP}(\mathbb{K})$, then for any $f, g \in \mathbb{K}$ such that $f,g \in ConSpan(S)$
$$|f|  \dotplus  |g| \in ConSpan(S) \ \ \text{and} \ \ |f| \wedge |g| \in ConSpan(S).$$
\end{rem}
\begin{proof}
First we prove that $|f|  \dotplus  |g| \in ConSpan(S)$. Since  $\langle f \rangle \subseteq \langle S \rangle \cdot \langle \mathscr{R} \rangle$ and $\langle g \rangle \subseteq \langle S \rangle \cdot \langle \mathscr{R} \rangle$, we have
$\langle f, g \rangle =  \langle |f|  \dotplus  |g| \rangle = \langle f \rangle \cdot \langle g \rangle \subseteq \langle S \rangle \cdot \langle \mathscr{R} \rangle$.
For $|f| \wedge |g| \in ConSpan(S)$, $ \langle |f| \wedge|g| \rangle = \langle f \rangle \cap \langle g \rangle \subseteq \langle g \rangle \subseteq \langle  S \rangle \cdot \langle \mathscr{R} \rangle$.
\end{proof}

\begin{rem}
Let $S =\{ f_1,...,f_m\}$ a finite set of HP-fractions. Then
$$ConSpan(S) =  \langle f_1,...,f_n \rangle \cdot \langle \mathscr{R} \rangle.$$
\end{rem}
\begin{proof}
A straightforward consequence of Definition \ref{defn_convex_span_in_semifield_of_fractions}.
\end{proof}

\begin{rem}\label{rem_HS_kernel_finitely_spanned}
If $K$ is an HS-kernel, then $K$ is generated by an HS-fraction $f~\in~\mathbb{K}$ of the form $f = \sum_{i=1}^{t} |f_i|$ where $f_1,...,f_t$ are HP-fractions. So,
$$ConSpan(K) = \langle \mathscr{R} \rangle~\cdot~K =\langle \mathscr{R} \rangle \cdot \langle f \rangle = \langle \mathscr{R} \rangle \cdot \langle \sum_{i=1}^{t} |f_i| \rangle = \langle \mathscr{R} \Big\rangle \cdot \prod_{i=1}^{t} \langle f_i \Big\rangle$$ $$= \langle \mathscr{R} \rangle \cdot \langle f_1,...,f_t \rangle$$ and so, $\{f_1,..., f_t \}$ convexly spans $\langle \mathscr{R} \rangle \cdot K$.
\end{rem}

\begin{rem}\label{rem_HS_depend_on_HP}
Let $f$ be an HS-fraction. Then $f \sim_K \sum_{i=1}^{t} |f_{i}|$  where $f_{i}$ are \linebreak HP-fractions. Since $\langle f \rangle = \prod_{i=1}^{t} \langle f_i \rangle = \langle \{ f_1,...,f_t \} \rangle$, we have that
$f$ is $\mathscr{R}$-convexly dependent on $\{ f_1,...,f_t \}$ .
\end{rem}

\begin{lem}\label{lem_HS_HP_convexity_relations}
 If $\{ b_1, ...,b_m \}$ is a set of HS-fractions, such that $b_i\sim_K\sum_{j=1}^{t_i} |f_{i,j}|$  where $f_{i,j}$ are HP-fractions, then $b_1$ is $\mathscr{R}$-convexly dependent on$\{ b_2, ...,b_m \}$~if and only if all its summands $f_{1,r}$ for $1 \leq r \leq t_1$ are $\mathscr{R}$-convexly dependent on $\{ b_2, ...,b_m \}$.
\end{lem}
\begin{proof}
If $b_1$ is $\mathscr{R}$-convexly dependent on $\{ b_2, ...,b_m \}$, then
$$ \prod_{j=1}^{t_1} \langle f_{1,j} \rangle = \Big\langle \sum_{j=1}^{t_1} |f_{1,j}| \Big\rangle = \langle b_1 \rangle \subseteq \big\langle \{ b_1, ...,b_m \} \big\rangle .$$ Since $\langle f_{1,r} \rangle \subseteq \prod_{j=1}^{t_1} \langle f_{1,j} \rangle$ for every $1 \leq r \leq t_1$, we have that  $f_{1,r}$ is $\mathscr{R}$-convexly \linebreak dependent on $\{ b_1, ...,b_m \}$ and by Remark \ref{rem_HS_depend_on_HP} $f_{1,r}$ is $\mathscr{R}$-convexly dependent on $S$ where $S = \left\{ f_{i,j} : 2 \leq i \leq m ; 1 \leq j \leq t_i \right\}$.
Conversely, if each $f_{1,r}$ is $\mathscr{R}$-convexly dependent on $\{ b_2, ...,b_m \}$ for $1 \leq r \leq t_1$, then there exist some $k_1,...,k_{t_1}$ such that \linebreak $|f_{1,r}| \leq   \sum_{i=2}^{m}\sum_{j=1}^{t_i}  |f_{i,j}|^{k_r}$. Taking $k'=\max\{ k_r : r = 1,...,t_1 \}$, we get that $b_1$ is \linebreak $\mathscr{R}$-convexly dependent on $\{ b_2, ...,b_m \}$.
\end{proof}

\begin{lem}\label{lem_HS_base_gives_rise_to_HP_base}
Let $V = \{ f_1, ...,f_m \}$ be a $\mathscr{R}$-convexly independent set of HS-fractions, such that $f_i \sim_K \sum_{j=1}^{t_i} |f_{i,j}|$  where $f_{i,j}$ are HP-fractions. Then there exist a $\mathscr{R}$-convexly independent subset $S_0 \subseteq S =\left\{ f_{i,j} : 2 \leq i \leq m ; 1 \leq j \leq t_i \right\}$ such that $|S_0| \geq |V|$ and $ConSpan(S_0) = ConSpan(V)$.
\end{lem}
\begin{proof}
By Remark \ref{rem_HS_depend_on_HP}, \  $f_i$ is dependent on $\{ f_{i,j} :  1 \leq j \leq t_i \} \subset S$ for each \linebreak  $1 \leq i \leq m$, thus $ConSpan(S) = ConSpan(V)$. By Corollary \ref{cor_basis_for_HP}, \ $S$ contains a maximal  $\mathscr{R}$-convexly independent subset $S_0$ such that $ConSpan(S_0) = ConSpan(S)$. By  Lemma~\ref{lem_HS_HP_convexity_relations} \ for each $i=1,...,m$, there exists some HP-fraction $g_i=f_{i,j_i}$ such that $g_i$ is $\mathscr{R}$-convexly independent of $V \setminus \{ f_i \}$, for otherwise $f_i$ would be \linebreak $\mathscr{R}$-convexly dependent on $V \setminus \{ f_i \}$, thus $\{g_1,...,g_m \} \subset S$ is $\mathscr{R}$-convexly independent, so \linebreak $|S_0| \geq m = |V|$.
\end{proof}

Now that our restriction to HP-fractions is justified, we move forward with our \linebreak construction.

\begin{rem}\label{rem_basis_for_HS_kernel}
Let $K \in \mathbb{K}$ be an HS-kernel. Consider the following set
$$ \text{HP}(K) = \left\{ f \in K :  f \ \text{is an HP-fraction} \ \right\}.$$
$K$ is an HS-fraction, so by definition there are some HP-fractions $f_1,...,f_t$ such that $L = \langle \sum_{i=1}^{t} |f_i| \rangle$. By Remark \ref{rem_HS_kernel_finitely_spanned}, $ConSpan(K)$ is convexly-spanned by ${f_1,...,f_t}$. Now, Since $ConSpan(K) = ConSpan(f_1,...,f_t)$ and $\{f_1,...,f_t\} \subset \text{HP}(K) \subset \text{HP}(\mathbb{K})$, by \ref{cor_basis_for_HP} $\{f_1,...,f_t\}$ contains a basis $B=\{b_1,...,b_s\} \subset \{f_1,...,f_t\}$ of $\mathscr{R}$-convexly independent elements such that $ConSpan(B) = ConSpan(f_1,...,f_t)=ConSpan(K)$.  Note that $s \in \mathbb{N}$ is finite and is uniquely determined by $K$.\\
\end{rem}

\begin{defn}
Let $K \in \mathbb{K}$ be an HS-kernel. In the notation of Remark \ref{rem_basis_for_HS_kernel},
we define the \emph{convexity degree} of $ConSpan(K)$, $condeg(K)$ to be $s$, the number of elements in a basis $B$.
\end{defn}

\begin{rem}\label{rem_convexity_degree_of_semifield_of_fractions}
By Example \ref{exmp_basis_for_semifield_of_fractions} we have that $condeg\left(\mathscr{R}(x_1,...,x_n)\right) = n$.
\end{rem}

\begin{nota}
For any semifield homomorphism $\phi : \mathbb{H} \rightarrow \mathbb{S}$, we denote the image of $h \in \mathbb{H}$ under $\phi$ by $\bar{h} = \phi(h)$.
\end{nota}

\begin{prop}\label{prop_order_independence_1}
Let $O$ be an order-kernel of $\mathscr{R}(x_1,...,x_n)$. Then if a set of \linebreak HP-fractions $\{ h_1,...,h_t \}$ is $\mathscr{R}$-convexly dependent in $\mathscr{R}(x_1,...,x_n)$, then $\{ \bar{h}_1,...,\bar{h}_t \}$  is \linebreak $\mathscr{R}$-convexly dependent in the quotient semifield $\mathscr{R}(x_1,...,x_n)/O$.
\end{prop}
\begin{proof}
Denote by $\phi_O: \mathscr{R}(x_1,...,x_n) \rightarrow \mathscr{R}(x_1,...,x_n)/O$ the quotient $\mathscr{R}$-homomorphism.
First note that since $\phi_O$ is an $\mathscr{R}$ homomorphism, we have that $\phi_O( \langle \mathscr{R} \rangle) =  \langle \phi_O(\mathscr{R}) \rangle = \langle \mathscr{R} \rangle_{\mathscr{R}(x_1,...,x_n)/ O}$.
Now, if $  h_1,...,h_t  $ are $\mathscr{R}$-convexly dependent then there exist some $j$, say without loss of generality $j=1$, such that $ h_1 \in \langle h_2,...,h_t \rangle \cdot \langle \mathscr{R} \rangle$. By assumption,
\begin{align*}
\bar{h}_1 \ & =  \phi_O (h_1)  \in \phi_O \left(\langle h_2,...,h_t, \alpha \rangle \right)\\
& = \left\langle \phi_O(h_2),...,\phi_O(h_t),\phi_O(\alpha) \right\rangle = \langle \bar{h}_1,...,\bar{h}_t, \alpha \rangle\\
& = \langle \bar{h}_1,...,\bar{h}_t  \rangle \cdot \langle \mathscr{R} \rangle
\end{align*}
(the equalities hold by Remark \ref{rem_image_of_principal_kernel} and $\phi_O$ being an $\mathscr{R}$-homomorphism). Thus $\bar{h}_1 $ is $\mathscr{R}$-convexly dependent on $\{ \bar{h}_2 ,...,\bar{h}_t \}$.
\end{proof}

\begin{flushleft}Conversely, we have:\end{flushleft}

\begin{lem}\label{lem_order_independence_2}
Let $O$ be an order-kernel of $\mathscr{R}(x_1,...,x_n)$.  Let $\{ h_1,...,h_t \}$ be a set of HP-fractions. If $\bar{h}_1 ,...,\bar{h}_t$ are $\mathscr{R}$-convexly dependent in the quotient semifield \linebreak $\mathscr{R}(x_1,...,x_n)/O$  and $\sum_{i=1}^{t}\bar{|h_1|} \cap \mathscr{R} = \{1\}$, then $h_1,...,h_t$  are $\mathscr{R}$-convexly dependent \linebreak in $\mathscr{R}(x_1,...,x_n)$.
\end{lem}
\begin{proof}
Note that $\sum_{i=1}^{t}\bar{|h_1|} \cap \mathscr{R} = \{1\}$ if and only if $\bigcap_{i=1}^{t}Skel(h_1) \cap Skel(O) \neq \emptyset$. Translating the variables by a point $a \in \bigcap_{i=1}^{t}Skel(h_1) \cap Skel(O)$, we may assume that the constant coefficient of each HP-fraction $h_i$ is $1$.
Assume $\bar{h}_1 ,...,\bar{h}_t$ are $\mathscr{R}$-convexly dependent. W.l.o.g. we may assume that $\bar{h}_1$ is $\mathscr{R}$-convexly dependent on $\{ \bar{h}_2 ,...,\bar{h}_{t+1} \}$.   Taking $h_{t+1} \in \mathscr{R}$, we may write $\bar{h}_1 \in \langle \bar{h}_2,...,\bar{h}_t,\bar{h}_{t+1} \rangle$. Considering the pre-images of the quotient map, we have that $$\langle h_1 \rangle \cdot O \subseteq \langle h_2,...,h_t, h_{t+1} \rangle \cdot O.$$ Thus $|h_2|  \dotplus  \dots  \dotplus  |h_{t+1}|  \dotplus  |1  \dotplus  g| \succeq  |h_1|  \dotplus  | 1  \dotplus  g|$ with $1  \dotplus  g$ a generator of $O$. So, by Remark \ref{rem_kernel_by_abs_value}, there exists some $k \in \mathbb{N}$ such that
\begin{equation}\label{eq_order_independence_2_1}
|h_1|  \dotplus  | 1  \dotplus  g| \leq (|h_2|  \dotplus  \dots  \dotplus  |h_{t+1}|  \dotplus  |1  \dotplus  g|)^{k} = |h_2|^{k}  \dotplus  \dots  \dotplus  |h_{t+1}|^{k}  \dotplus  |1  \dotplus  g|^{k}.
\end{equation}
 As $1  \dotplus  g \geq 1$ we have that $|1  \dotplus  g | = 1  \dotplus  g$, and the right hand side of equation~ \eqref{eq_order_independence_2_1} equals $$|h_2|^{k}  \dotplus  \dots  \dotplus  |h_{t+1}|^{k}  \dotplus  (1  \dotplus  g)^{k} = |h_2|^{k}  \dotplus  \dots  \dotplus  |h_{t+1}|^{k}  \dotplus  1  \dotplus  g^{k} = |h_2|^{k}  \dotplus  \dots \dotplus  |h_{t+1}|^{k}  \dotplus  g^{k}.$$ The last equality is due to the fact that $\sum|h_i|^k \geq 1$ so that $1$ is absorbed. The same arguments applied to the left hand side of equation \eqref{eq_order_independence_2_1} \ yields that
\begin{equation}\label{eq_order_independence_2_2}
|h_1|   \dotplus  g \leq |h_2|^{k}  \dotplus  \dots  \dotplus  |h_{t+1}|^{k}  \dotplus  g^{k}.
\end{equation}
Assume on the contrary that $h_1$ is $\mathscr{R}$-convexly independent of $\{ h_2 ,...,h_t  \}$. Then $$\langle h_1 \rangle \not \subseteq \langle h_2,...,h_{t+1} \rangle = \bigg\langle \sum_{i=2}^{t+1}|h_i| \bigg\rangle.$$ Thus for any $m \in \mathbb{N}$ there exists some $x_m \in \mathscr{R}^n$ such that $$|h_1(x_m)| > \bigg|\sum_{i=2}^{t+1}|h_i(x_m)|\bigg|^{m} = \sum_{i=2}^{t+1}|h_i(x_m)|^m .$$ Thus by equation \eqref{eq_order_independence_2_2} and the last observation we get that $$\sum_{i=2}^{t}|h_i(x_m)|^m   \dotplus  g(x_m) < |h_1(x_m)|  \dotplus  g(x_m) \leq  \sum_{i=2}^{t}|h_i(x_m)|^k  \dotplus  g(x_m)^k,$$
i.e., there exists some fixed $k \in \mathbb{N}$ such that for any $m \in \mathbb{N}$,
\begin{equation}\label{eq_order_independence_2_3}
\sum_{i=2}^{t}|h_i(x_m)|^m  < \sum_{i=2}^{t}|h_i(x_m)|^k  \dotplus  g(x_m)^k.
\end{equation}
For $m > k$, since  $|\gamma|^k \leq |\gamma|^m$ for any $\gamma \in \mathscr{R}$, we get that $\sum_{i=2}^{t}|h_i(x_m)|^m \geq \sum_{i=2}^{t}|h_i(x_m)|^k$. Write $g^k(x) = g(1)g'(x)$. Since $g^k$ is an HP-kernel, $g(1)$ is the constant coefficient of $g$ and $g'$ is a Laurent monomial with coefficient $1$. Now, by the way $x_m$ were chosen we have that $\sum_{i=2}^{t}|h_i(x_m)| > 1$ and $\sum_{i=2}^{t}|h_i(x_m)|^m < g(x_m)^k$, and thus for sufficiently large $m_0$ we have that $g(1) < g'(x_{m_0})$. Since $g'$ is a Laurent monomial with coefficient $1$, we have that $g'(x_{m}^{-1}) = g'(x_{m})^{-1}$ so $g^k(x_{m_0}^{-1}) = g(1)g'(x_{m_0}^{-1}) = g(1)g'(x_{m_0})^{-1} < 1$. Thus by \eqref{eq_order_independence_2_3} we must have $\sum_{i=2}^{t}|h_i(x_{m_0}^{-1})|^m < \sum_{i=2}^{t}|h_i(x_{m_0}^{-1})|^k$. A contradiction.
\end{proof}

%

\begin{prop}\label{prop_order_independence_2}
Let $R$ be a region-kernel of $\mathscr{R}(x_1,...,x_n)$.  Let $\{ h_1,...,h_t \}$ be a set of HP-fractions such that $(R \cdot   \langle h_1, ...,h_t \rangle) \cap \mathscr{R} = \{1 \}$. Then $h_1 \cdot R,...,h_t \cdot R$ are \linebreak $\mathscr{R}$-convexly dependent in the quotient semifield $\mathscr{R}(x_1,...,x_n)/R$ if and only if $h_1,...,h_t$ are $\mathscr{R}$-convexly dependent in $\mathscr{R}(x_1,...,x_n)$.
\end{prop}
\begin{proof}
The `if' part of the assertion follows Proposition \ref{prop_order_independence_1}. As $R = \prod_{i=1}^{m}O_i$ for some order -kernels $\{ O_i \}_{i=1}^{m}$, the `only if' part follows from Lemma \ref{lem_order_independence_2} by applying it repeatedly to each of the $O_i$'s comprising $R$.
\end{proof}
%

\begin{prop}\label{prop_convexity_degree_invariance_for_HS_kernels}
Let $L \in HS(\mathbb{K})$ and let $R \in \PCon(\mathbb{K})$ be a region kernel. Let $\phi_{R} : \mathbb{K} \rightarrow \mathbb{K}/R$ be the quotient map.  Then $$condeg(L) = condeg(L \cdot R).$$
\end{prop}

\begin{proof}
Since $L$ is a subsemifield of $\mathbb{K}$, by Theorem \ref{thm_kernels_hom_relations}, we have that \linebreak $\phi_{R}^{-1}(\phi_R(L))=R \cdot L$ and $condeg(L) \geq  condeg(\phi_R(L))$ by Proposition \ref{prop_order_independence_1}, while by Lemma \ref{lem_order_independence_2} we have that  $condeg(\phi_R(L)) \geq condeg(\phi_{R}^{-1}(\phi_R(L)))$. Thus \linebreak $condeg(L) \geq  condeg(R \cdot L)$.  $L \subseteq  R \cdot L$  implies that $condeg(L) \leq  condeg(R \cdot L)$. Thus equality holds.
\end{proof}
\ \\

\begin{exmp}\label{exmp_basis_for_semifield_of_fractions}
As we have previously shown, the maximal kernels in \\$\PCon(\mathscr{R}(x_1,...,x_n))$ are HS-fractions of the form
$L_{(\alpha_1,...,\alpha_n)}=\langle \alpha_1 x_1, \dots, \alpha_n x_n \rangle$ for any $\alpha_1,...,\alpha_n \in \mathscr{R}$.
We have previously shown that $\mathscr{R}(x_1,...,x_n) =   \mathscr{R} \cdot L_{(\alpha_1,...,\alpha_n)} = \langle \mathscr{R} \rangle \cdot L_{(\alpha_1,...,\alpha_n)}$. Thus $\mathscr{R}(x_1,...,x_n) = ConSpan(\{ \alpha_1 x_1, \dots, \alpha_n x_n \})$, i.e., $\{ \alpha_1 x_1, \dots, \alpha_n x_n \}$ convexly spans $\mathscr{R}(x_1,...,x_n)$ over $\mathscr{R}$.  Now, since there is no order relation between $\alpha_i x_i$ and the elements of  $\{ \alpha_j x_j : j \neq i \} \cup \{\alpha\ : \alpha \in \mathscr{R} \}$ we have that $\alpha_k x_k \not \in \langle \bigcup_{j \neq k}\alpha_j x_j \rangle \cdot~\langle\mathscr{R}\rangle$.  Thus $\{ \alpha_1 x_1, \dots, \alpha_n x_n \}$ is $\mathscr{R}$-convexly independent constituting a basis for $\mathscr{R}(x_1,...,x_n)$ for any chosen $\alpha_1,...,\alpha_n \in \mathscr{R}$.

\end{exmp}

%
%
%

\begin{rem}\label{rem_condeg_not_effected_by_region}
Let $R$ be a region kernel of $\mathbb{K}$ and let $$A = \{ \langle g \rangle : \langle g \rangle \cdot \langle \mathscr{R} \rangle \supseteq  R \cdot \langle \mathscr{R} \rangle \}.$$
In view Remark \ref{rem_correspondence_of_quontient_HSpec} and of Proposition \ref{prop_order_independence_2}, we have that $condeg(\mathbb{K}/R) = condeg(A)$. As $L_{a} \supseteq A$ for any $a \in Skel(R) \neq \emptyset$, we have that $condeg(A) = condeg(\mathbb{K})$.\\
\end{rem}
\begin{prop}
If $R$ be a region kernel and $L \in HS(\mathbb{K})$  of $\mathbb{K}$, then
$$condeg(\mathbb{K}/LR) = condeg(\mathbb{K}) - condeg(L).$$
In particular, $$condeg(\mathscr{R}(x_1,...,x_n)/LR) = n - condeg(L).$$
\end{prop}
\begin{proof}
By the third isomorphism theorem, we have that $\mathbb{K}/LR \cong (\mathbb{K}/R)/ (L\cdot R/R)$. We can always choose a basis for HP$(\mathbb{K}/R)$ containing a basis for HP$(L\cdot R/R)$. Taking $L$ instead of $\mathbb{K}$ in Remark \ref{rem_condeg_not_effected_by_region}, we have that $condeg(L \cdot R/R) = condeg(\phi_R(L))$. Note that $L \cdot R \cap \mathscr{R} = \{ 1 \}$, thus By Proposition \ref{prop_order_independence_2} $condeg(\phi_R(L)) = condeg(L)$. So $condeg(\mathbb{K}/LR) = condeg(A)-condeg(L) = condeg(\mathbb{K}) - condeg(L)$.\\

Taking $\mathbb{K} = \mathscr{R}(x_1,...,x_n)$ in the above setting, we get that $$condeg(\mathscr{R}(x_1,...,x_n)/LR) = condeg(A)-condeg(L) = n - condeg(L).$$
\end{proof}

\begin{prop}\label{prop_kernel_descending_chain}
Let $L$ be an HS-kernel in $\mathbb{K}$. Let $\{ h_1, ...,h_t \}$ be a set of HP-fractions in $\text{HSpec}(\mathbb{K})$ such that $ConSpan(h_1, ...,h_t) = ConSpan(L)$ and let $L_i = \langle h_i \rangle$.
Consider the following descending chain of HS-kernels where $Skel(K)~\neq~\emptyset$
\begin{equation}\label{eq_desc_chain_kernels}
L = \prod_{i=1}^{u}L_i  \supseteq \prod_{i=1}^{u-1}L_i  \supseteq \dots \supseteq L_1  \supseteq \langle 1 \rangle .
\end{equation}
Then  the chain \eqref{eq_desc_chain_kernels} is a strictly descending chain of HS-kernels  if and only if $h_1,....,h_u$ are $\mathscr{R}$-convexly independent.
\end{prop}
\begin{proof}
Since $\prod_{i=1}^{t}L_i \subseteq K$,  $Skel(\prod_{i=1}^{t}L_i ) \supset Skel(L) \neq \emptyset$ for every $0 \leq t \leq u$, which in turn implies that $(\prod_{i=1}^{t}L_i ) \cap \mathscr{R} = \{ 1 \}$ for every $0 \leq t \leq u$ (otherwise $Skel(\prod_{i=1}^{t}L_i)~=~\emptyset$). If $\{h_1,....,h_u\}$ is $\mathscr{R}$-convexly independent then $h_u$ is $\mathscr{R}$-convexly independent of the set $\{h_1,....,h_{u-1}\}$ thus $L_u = \langle h_u \rangle \not \subseteq \prod_{i=1}^{u-1}L_i \cdot \langle \mathscr{R} \rangle$. Thus the inclusions of the chain \eqref{eq_desc_chain_kernels} are strict, i.e., it is strictly descending.  On the other hand if $h_u$ is $\mathscr{R}$-convexly dependent on $\{h_1,....,h_{u-1}\}$ then   $L_u = \langle h_u \rangle  \subseteq \prod_{i=1}^{u-1}L_i \cdot \langle \mathscr{R} \rangle$. Assume $L_u = \langle h_u \rangle  \not \subseteq \prod_{i=1}^{u-1}L_i$, then $\langle \mathscr{R} \rangle \subseteq \prod_{i=1}^{u}L_i$ implying that $\prod_{i=1}^{u}L_i$ is not an HS-kernel. Thus $L_u = \langle h_u \rangle   \subseteq \prod_{i=1}^{u-1}L_i$ and the chain is not strictly descending.
\end{proof}

\begin{prop}\label{prop_kernel_descending_chain_properties}
If $L \in HSpec(\mathbb{K})$, then $hgt(L) = condeg(L)$.
 Moreover, every factor of a  descending chain of maximal length is an  HP-kernel.
\end{prop}
\begin{proof}
By Proposition \ref{prop_kernel_descending_chain}, we have that the maximal length of a chain of \linebreak HS-kernels descending from an HS-kernel $L$ equals the number of elements in a basis of $ConSpan(L)$; thus we have that the chain is of unique length $condeg(K)$, i.e., \linebreak $hgt(L) = condeg(L)$. Moreover, by Theorem \ref{thm_nother_1_and_3}(2) we have that
$$\prod_{i=1}^{j}L_i / \prod_{i=1}^{j-1}L_i   \cong L_j/\bigg(L_j \cap \prod_{i=1}^{j-1}L_i \bigg).$$ Since $L_j \cdot (L_j \cap \prod_{i=1}^{j-1}L_i  ) = L_j \cap \prod_{i=1}^{j}L_i   \subset \prod_{i=1}^{j}L_i $ and $(\prod_{i=1}^{j}L_i ) \cap \mathscr{R} = \{1\}$, we have that $\left(L_j \cdot (L_j \cap (\prod_{i=1}^{j-1}L_i  ))\right) \cap \mathscr{R} = \{1\}$. So the homomorphic image of the HP-kernel $L_j$ under the quotient map $\mathscr{R}(x_1,...,x_n) \rightarrow \mathscr{R}(x_1,...,x_n)/(L_j \cap (\prod_{i=1}^{j-1}L_i))$  is an HP-kernel. Thus every factor of the chain is an HP-kernel.
\end{proof}

\begin{thm}
If $\mathbb{K}$ is an affine semifield, then
$$Hdim(\mathbb{K}) = condeg(\mathbb{K}).$$
\end{thm}
\begin{proof}
A consequence of Definition \ref{defn_Hdim} and Proposition \ref{prop_kernel_descending_chain_properties}.
\end{proof}

\begin{defn}
A set of region kernels $\{R_1,...,R_t\} \subset \PCon(\mathscr{R}(x_1,...,x_n))$ is said to be a \emph{cover of regions} if $R_1 \cap R_2 \cap \dots \cap R_t = \{1 \}$. In other words, $\{R_1,...,R_t\}$ is a cover of regions if $Skel(R_1) \cup Skel(R_2) \cup \dots \cup Skel(R_t) = \mathscr{R}^n$.
\end{defn}

\begin{rem}\label{cover_of_regions_sub_direct_product}
If $\{R_1,...,R_t\} \subset \PCon(\mathscr{R}(x_1,...,x_n))$ is a cover of regions, then we can express $\mathscr{R}(x_1,...,x_n)$ as a subdirect product
$$\mathscr{R}(x_1,...,x_n) = \mathscr{R}(x_1,...,x_n)/(R_1 \cap R_2 \cap \dots \cap R_t) \xrightarrow[s.d]{} \prod_{i=1}^{t} \mathscr{R}(x_1,...,x_n)/R_i.$$
Let $K \in \Con(\mathscr{R}(x_1,...,x_n))$. Then $R_1 \cap R_2 \cap \dots \cap R_t \cap K = \bigcap_{i=1}^t (R_i \cap K) = \{1 \}$. Since $K$ itself is an idempotent semifield, we have that
$$K = K/\bigcap_{i=1}^t (R_i \cap K) \cong \prod_{i=1}^{t} K/(R_i \cap K) \cong \prod_{i=1}^{t} R_i K/R_i.$$
\end{rem}

As we have seen for every principal regular kernel $\langle f \rangle \in (\mathscr{R}(x_1,...,x_n))$, there exists an explicitly formulated cover of regions $\{R_{1,1},...,R_{1,s},R_{2,1},...,R_{2,t} \}$ such that
$$\langle f \rangle = \bigcap_{i=1}^{s}K_i \cap \bigcap_{j=1}^{t}N_j$$
where $K_i = L_i \cdot R_{1,i}$ for $i =1,...,s$ and appropriate HS-kernels $L_i$ and
$N_j = B_j \cdot R_{2,j}$ for $j =1,...,t$ and appropriate bounded from below kernels $B_j$.
If $\langle f \rangle \in \PCon(\langle \mathscr{R} \rangle)$, as we have shown, then  $B_j = \langle \mathscr{R} \rangle$ for every $j = 1,...,t$. Note that over the different regions in $\mathscr{R}^n$, corresponding to the region kernels $R_{i,j}$, $f$ is locally represented by distinct elements of $HSpec(\mathscr{R}(x_1,...,x_n))$ (HS-fractions). In fact the regions themselves are defined such that the local HS-representation of $f$ will stay invariant over each. Thus the    $R_{i,j}$'s, defining the partition of the space, are uniquely determined as the minimal set of regions over each of which $\langle f \rangle$ comes from a unique HS-kernel. \\

For each $j = 1,...,t$, we have that  $$condeg(N_j) = Hdim(N_j) = 0$$ and $$condeg(\mathscr{R}(x_1,...,x_n)/N_j) = Hdim(\mathscr{R}(x_1,...,x_n)/N_j) = n.$$
For each $i = 1,...,s$, we have that  $$condeg(K_i)= condeg(L_i) = Hdim(L_i) \geq 1$$ and $$condeg(\mathscr{R}(x_1,...,x_n)/K_i) = Hdim(\mathscr{R}(x_1,...,x_n)/K_i) = n - Hdim(L_i) < n.$$

\begin{defn}\label{defn_hyper_dim}
Using the notations used in the discussion above,
let $\langle f \rangle$ be a regular principal kernel in $\PCon(\langle \mathscr{R} \rangle)$.
Define the \emph{Hyper-dimension} of $\langle f \rangle$ to be $$Hdim(\langle f \rangle) = (Hdim(L_1),...,Hdim(L_s))$$ and $$Hdim(\langle f \rangle) = \left(Hdim(\mathscr{R}(x_1,...,x_n)/L_1),..., Hdim(\mathscr{R}(x_1,...,x_n)/L_s)\right).$$
\end{defn}

\ \\
\begin{rem}
In view of the discussion in Subsection \ref{subsection:geometric_interpretation} each term $\mathscr{R}(x_1,...,x_n)/L_i$ in Definition \ref{defn_hyper_dim} corresponds to the linear subspace of $\mathscr{R}^n$ (in logarithmic scale) defined by the linear constraints endowed  on the quotient $\mathscr{R}(x_1,...,x_n)/L_i$ by the HS-kernel $L_i$. One can think of these terms as an algebraic description of the affine subspaces which locally comprise the skeleton $Skel(f)$.
\end{rem}

\clearpage

\bibliographystyle{amsplain}
\bibliography{TAGBib}

\end{document}